\documentclass[reqno, letterpaper, oneside]{article} 
\usepackage{amsmath,amsfonts,amssymb,amsthm}
\usepackage{cleveref,enumerate}
\crefname{equation}{}{}

\usepackage[final]{graphicx}
\usepackage[usenames,dvipsnames]{color}

\usepackage{geometry}
\numberwithin{equation}{section}
\usepackage{bbm}
\newtheorem{theorem}{Theorem}[section]
\newtheorem{lemma}[theorem]{Lemma}
\newtheorem{corollary}[theorem]{Corollary}
\newtheorem{claim}[theorem]{Claim}

\theoremstyle{definition}
\newtheorem{remark}[theorem]{Remark}
\newtheorem{definition}[theorem]{Definition}

\newtheorem*{remark*}{Remark}
\newtheorem*{claim*}{Claim}

\newcommand\ceil[1]{\lceil #1 \rceil}
\newcommand\ceilL[1]{\left\lceil{#1}\right\rceil}
\newcommand\floor[1]{\lfloor #1 \rfloor}
\newcommand\eps{\varepsilon}

\renewcommand\le{\leqslant}
\renewcommand\ge{\geqslant}
\renewcommand\preceq{\preccurlyeq}

\renewcommand\varrho{\rho} 

\newcommand\NN{\mathbb{N}}
\newcommand\NNP{\NN^+}
\newcommand\ZZ{{\mathbb Z}}

\newcommand\RR{{\mathbb R}}
\newcommand\CC{{\mathbb C}}

\newcommand\E{\operatorname{\mathbb E{}}}
\renewcommand\Pr{{\mathbb P}}
\newcommand\pto{\overset{\mathrm{p}}{\to}}
\newcommand\dtv[2]{\mathrm{d}_{\mathrm{TV}}\bigl(#1\,,#2\bigr)}
\newcommand\Po{\mathrm{Po}}

\newcommand\Var{\operatorname{Var}}

\newcommand\noproof{\qed}

\newcommand\cA{{\mathcal A}}
\newcommand\cB{{\mathcal B}}
\newcommand\cC{{\mathcal C}}
\newcommand\cD{{\mathcal D}}
\newcommand\cE{{\mathcal E}}
\newcommand\cF{{\mathcal F}}
\newcommand\cG{{\mathcal G}}
\newcommand\cH{{\mathcal H}}

\newcommand\cK{{\mathcal K}}
\newcommand\cI{{\mathcal I}}

\newcommand\cN{{\mathcal N}}
\newcommand\cP{{\mathcal P}}
\newcommand\cQ{{\mathcal Q}}
\newcommand\cR{{\mathcal R}}
\newcommand\cS{{\mathcal S}}
\newcommand\cT{{\mathcal T}}

\newcommand\cX{{\mathcal X}}
\newcommand\bcX{\overline{\mathcal X}}
\newcommand\cY{{\mathcal Y}}

\newcommand\chH{\widehat{\mathcal H}}

\newcommand\tlambda{\lambda} 
\newcommand\hlambda{\mu} 

\newcommand\tq{{\tilde{q}}}
\newcommand\tC{{\tilde{C}}}
\newcommand\tH{{\widetilde{H}}}
\newcommand\tP{{\widetilde{P}}}

\newcommand\tR{{\widetilde{R}}}
\newcommand\tY{{\tilde{Y}}}
\newcommand\tZ{{\tilde{Z}}}
\newcommand\tk{{\tilde{k}}}
\newcommand\tr{{\tilde{r}}}
\newcommand\tw{{\tilde{w}}}
\newcommand\ttt{{\tilde{t}}}

\newcommand\fD{{\mathfrak D}}
\newcommand\fE{{\mathfrak E}}

\newcommand\fH{{\mathfrak H}}

\newcommand\fN{{\mathfrak N}}

\newcommand\fR{{\mathfrak R}}
\newcommand\fS{{\mathfrak S}}

\newcommand\bp{{\mathfrak{X}}}

\newcommand\tc{t_{\mathrm{c}}}
\newcommand\tcx{t_{\mathrm{b}}}

\newcommand{\ve}{\vec{e}}
\newcommand{\vv}{\vec{v}}
\newcommand{\vc}{\vec{c}}

\newcommand{\tF}{\tilde{F}}

\newcommand\ba{\mathbf{a}} 
\newcommand\bbb{\mathbf{b}} 
\newcommand\bx{\mathbf{x}}

\newcommand\br{\mathbf{r}} 
\newcommand\by{\mathbf{y}} 

\newcommand\pp{\mathbf{p}} 
\newcommand\qq{\mathbf{q}}

\newcommand\op{o_{\mathrm{p}}}
\newcommand\Op{O_{\mathrm{p}}}
\newcommand\Thetap{\Theta_{\mathrm{p}}}

\newcommand\modp{\equiv_{\per} 0}

\newcommand\per[1][{\cR}]{p_{{#1}}}

\newcommand\trho{\tilde\rho}

\newcommand\dx{\mathrm{d}x}
\newcommand\dy{\mathrm{d}y}
\newcommand\dt{\mathrm{d}t}

\newcommand\ddt{\frac{\mathrm{d}}{\mathrm{d}t}}

\newcommand\cc{\mathrm{c}}
\newcommand\JP{J^{\Po}}
\newcommand\JPU{J^{\Po}(U^\mathrm{c})}
\newcommand\JPp{J^{+}}
\newcommand\JPm{J^{-}}
\newcommand\JPpm{J^{\pm}}

\newcommand\bb[1]{\bigl(#1\bigr)}
\newcommand\LS{S}
\DeclareMathOperator{\Bin}{Bin}

\newcommand{\indic}[1]{\mathbbm{1}_{\{{#1}\}}}

\newcommand\bSp{\overline{\cS^+}}
\newcommand\cSkr{\cS_{\cR}^*}
\newcommand\kR{k_{\cR}}
\newcommand\IR{I_{\cR}}
\newcommand\cSR{\cS_{\cR}}

\newcommand\gf{g}
\newcommand\tg{{\tilde g}}
\newcommand\hf{{\hat g}}
\newcommand\htf{{\hat{\tilde g}}}

\newcommand*\newoddpage{\clearpage\newpage}

\newenvironment{romenumerate}[1][0pt]{
\addtolength{\leftmargini}{#1}\begin{enumerate}
 }{\end{enumerate}}

\let\OLDthebibliography\thebibliography
\renewcommand\thebibliography[1]{
  \OLDthebibliography{#1}
  \setlength{\parskip}{0pt}
  \setlength{\itemsep}{0pt plus 0.3ex}
}

\usepackage{enumitem}

\begin{document}

\title{The phase transition in \\ bounded-size Achlioptas processes} 
\author{Oliver Riordan%
\thanks{Mathematical Institute, University of Oxford, Radcliffe Observatory Quarter, Woodstock Road, Oxford OX2 6GG, UK.
E-mail: {\tt riordan@maths.ox.ac.uk}.}
\ and Lutz Warnke%
\thanks{Department of Mathematics, University of California, San Diego, La Jolla CA~92093, USA. 
E-mail: {\tt lwarnke@ucsd.edu}. 
Supported by NSF~grant DMS-1703516, NSF~CAREER grant~DMS-2225631, and a Sloan Research Fellowship.}
}
\date{May~11, 2017; revised October~23, 2023}

\renewcommand{\thefootnote}{\fnsymbol{footnote}}
\footnotetext{\hspace{-0.5em}AMS 2000 Mathematics Subject Classification: 05C80, 60C05, 90B15}
\renewcommand{\thefootnote}{\arabic{footnote}} 

\maketitle

\begin{abstract}
Perhaps the best understood phase transition is that in the component structure of the
uniform random graph process introduced by Erd\H os and R\'enyi around 1960.
Since the model is so fundamental, it is very interesting to know which features of
this phase transition are specific to the model, and which are `universal', at least
within some larger class of processes (a `universality class').
Achlioptas process, a class of variants of the Erd\H os--R\'enyi process that
are easy to define but difficult to analyze, have been extensively studied
from this point of view. Here, 
settling a number of conjectures and open problems,
we show that all `bounded-size'
Achlioptas processes share (in a strong sense)
all the key features of the Erd\H os--R\'enyi phase transition. We do not expect this to
hold for Achlioptas processes in general.
\end{abstract}


\setcounter{tocdepth}{2} 
\tableofcontents 

\newpage

\section{Introduction}

\subsection{Summary} 
In this paper we study the percolation phase transition in Achlioptas processes, which have become a key example for random graph processes with dependencies between the edges. 
Starting with an empty graph on~$n$ vertices, in each step two potential edges are chosen uniformly at random. 
One of these two edges is then added to the evolving graph,
where the choice of which edge is decided by a rule that may only use the sizes of the components containing the four endvertices.%
\footnote{Here we are describing Achlioptas processes with `size rules'. This is by far the most natural
and most studied type of Achlioptas process, but occasionally more general rules are considered.}
For the widely studied class of bounded-size rules (where all component sizes larger than some constant~$K$ are treated the same), the location and existence of the percolation phase transition is nowadays well-understood. 
However, despite many partial results during the last decade (see, e.g.,~\cite{BK,SW,JS,KPS,BBW11,RWapcont,BBW12b,BBW12a,DKP}),
our understanding of the finer details of the phase transition has remained incomplete, in particular
concerning the size of the largest component. 

Our main results resolve the finite-size scaling behaviour of percolation in all bounded-size Achlioptas processes. 
We show that for any Achlioptas processes with any such rule the 
phase transition is qualitatively the same as that of the classical Erd\H os--R\'enyi random graph process
in a very precise sense: the width of the `critical window' (or `scaling window') is the same,
and so is the asymptotic behaviour of the size of the largest component
above and below this window, as well as the tail behaviour of the component size distribution
throughout the phase transition.
In particular, when $\varepsilon = \varepsilon(n) \to 0$ as $n \to \infty$ but $\varepsilon^3 n \to \infty$,
we show that, with probability tending to $1$ as $n \to \infty$,
the size of the largest component after $i$ steps satisfies
\[
L_1(i) \sim \begin{cases} C \varepsilon^{-2}\log(\varepsilon^3 n) & \text{if $i=\tc n-\eps n$,}\\ c\varepsilon n  & \text{if $i=\tc n+\eps n$,}\end{cases}
\]
where~$\tc,C,c>0$ are rule-dependent constants%
\footnote{Following standard conventions, here~$\tc$ stands for the `critical time' where the phase transition happens (with respect to the size of the largest component); we stress that the constant~$c>0$ is not related to~$\tc$.}
(in the Erd\H os--R\'enyi case we have~$t_{\mathrm{c}}=C=1/2$ and~$c=4$). 
These and our related results for the component size distribution settle a number of conjectures and open problems from~\cite{JSP,JS,KPSPC,BS,KPS,BBW12a,DKP}.
In the language of mathematical physics, they establish that all bounded-size Achlioptas processes fall in the same `universality
class' (we do not expect this to be true for general Achlioptas processes). 
Such strong results (which fully identify the phase transition of the largest component and the critical window) are known for very few random graph models.

Our proof deals with the edge-dependencies present in bounded-size Achlioptas processes 
via a mixture of combinatorial multi-round exposure arguments, the differential equation method,
PDE theory, and coupling arguments.  
This eventually enables us to  analyze the phase transition via branching process arguments.  

\subsection{Background and outline results}\label{sec:bg}
In the last 15 years or so there has been a great deal of interest in studying evolving network models, i.e.,
random graphs in which edges (and perhaps also vertices) are randomly added step-by-step, rather than generated
in a single round. Although the original motivation, especially for the Barab\'asi--Albert model~\cite{BA1999},
was more realistic modelling of networks in the real world, by now evolving models are studied in their own
right as mathematical objects, in particular to see how they differ from static models.
Many properties of these models have been studied, starting with the degree distribution. In many cases
one of the most interesting features is a \emph{phase transition} where a `giant' (linear-sized)
component emerges as a density parameter increases beyond a critical value.

One family of evolving random-graph models that has attracted a great deal of interdisciplinary
interest (see, e.g.,~\cite{DRS,dCDGM,RW,J,BF,SW,RWapcont,BBW12b}) is that of \emph{Achlioptas processes},
proposed by Dimitris Achlioptas at a Fields Institute workshop in~2000.
These `power of two choices'~\cite{Power2Paper,Power2Survey} variants of the Erd{\H o}s--R{\'e}nyi random graph process can be described as follows. 
Starting with an empty graph on~$n$ vertices and no edges, in each step
two potential edges~$e_1,e_2$ are chosen uniformly at random from all~$\binom{n}{2}$ possible edges (or from those not already present).
One of these edges is selected according to some `decision rule' $\cR$  and added to the evolving graph. 
Note that the distribution of the graph $G^{\cR}_{n,i}$ after~$i$ steps depends on the rule used, and that always adding~$e_1$ gives the classical Erd{\H o}s--R{\'e}nyi random graph process (approximately or exactly, depending on whether repeated edges and loops are allowed or forbidden).
%
Figure~\ref{fig:L1plots} gives a crude picture of the phase transition for a range of rules. 
In general, the study of
Achlioptas processes is complicated by the fact that there are non-trivial dependencies between the choices in different rounds. 
Indeed, this makes the major tools and techniques for studying the phase transition unavailable (such as tree-counting~\cite{ER1960,Bollobas1984,Luczak1990,2SAT}, branching processes~\cite{Karp1990,BJR,BR2009,BR2012}, or random walks~\cite{Aldous1997,NP2007,NP2010,BR2012RW}), since these crucially exploit independence. 

\begin{figure}[t] 
\centering
  \setlength{\unitlength}{1bp}%
\hspace{2em}\includegraphics[width=3.0in,bb=90 50 450 302]{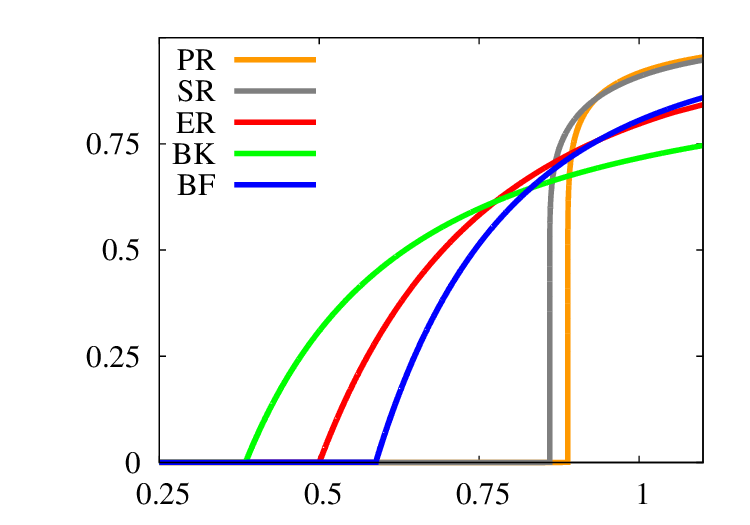}\vspace{-0.75em}
	\caption{\label{fig:L1plots} Simulation of the fraction $L_1(G^{\cR}_{n,tn})/n$ of vertices in the largest component after~$tn$ steps in various Achlioptas processes for~$n=10^{13}$, giving an approximation to the limiting curve $\rho^{\cR}(t)$. 
	The rules are Erd{\H o}s--R{\'e}nyi~(ER), Bohman--Frieze~(BF) and Bohman--Kravitz~(BK), all of which are bounded-size, and the sum and product rules (SR and PR), which are not.}
\end{figure}

The non-standard features of Achlioptas processes have made them an important testbed for developing new robust methods in the context of random graphs with dependencies, and for gaining a deeper understanding of the phase-transition phenomenon.
Here the class of \emph{bounded-size rules} has received considerable attention (see, e.g,~\cite{BK,SW,JS,KPS,BBW11,RWapcont,BBW12b,BBW12a,DKP}): 
the decision of these rules is based  
only on the sizes $c_1, \ldots, c_4$ of the components containing the endvertices of the two potential edges~$e_1$ and~$e_2$, with the restriction that all component sizes larger than some given \emph{cut-off~$K$} are treated in the same way (i.e., the rule only `sees' the truncated sizes $\min\{c_i,K+1\}$). 
Perhaps the simplest example is the \emph{Bohman--Frieze process}~(BF), the bounded-size rule with cut-off~$K=1$ in which the edge $e_1$ is added if and only if $c_1=c_2=1$ (see, e.g.,~\cite{BF,JS,KPS}). 
Figure~\ref{fig:L1plots} suggests that while the BF rule delays percolation compared to the classical Erd{\H o}s--R{\'e}nyi random graph process~(ER),
it leaves the essential nature of the phase transition unchanged.
In this paper we make this rigorous for all bounded-size rules, by showing that these exhibit Erd{\H o}s--R{\'e}nyi-like behaviour (see Theorem~\ref{thm:intro}).
Although very few rigorous results are known for rules which are not bounded-size (see~\cite{RWapsubcr,RWapunique}), as suggested
in Figure~\ref{fig:L1plots} these seem to have very different behaviour in general.

The study of bounded-size Achlioptas processes is guided by the typical questions from percolation theory (and random graph theory). 
Indeed, given any new model, the first question one asks is whether there is a phase transition
in the component structure, and where it is located. 
This was answered
in a pioneering paper by Spencer and Wormald~\cite{SW} (and for a large subclass by Bohman and Kravitz~\cite{BK}) using a blend of combinatorics, differential equations and probabilistic arguments. 
They showed that for any bounded-size rule $\cR$ there is a rule-dependent \emph{critical time} $\tc=\tc^\cR \in (0,\infty)$ at which the phase transition happens, i.e., at which the largest component goes from being of order~$O(\log n)$ to order~$\Theta(n)$. 
More precisely, writing, as usual, $L_j(G)$ for the number of vertices in the $j$th largest component of a graph $G$,
Spencer and Wormald showed that there is a constant~$\tc=\tc^\cR$  
given by  the blowup point of a certain finite system of differential equations
such that,
for any fixed $t \in [0,\infty)$, \emph{whp} (with high probability, i.e., with probability tending to $1$ as $n \to \infty$) we have
\begin{equation}\label{eq:def:tc}
L_1(G^{\cR}_{n,tn}) = \begin{cases} O(\log n) & \text{if $t < \tc$,}\\ \Theta(n) & \text{if $t > \tc$.}\end{cases}
\end{equation}
In the Erd{\H o}s--R{\'e}nyi process~\eqref{eq:def:tc} holds with $\tc=1/2$, see also Remark~\ref{rem:ER}.
\begin{remark}\label{rem:rounding}
Here and throughout we adopt the (common) \emph{rounding convention}, that any quantity (here~$tn$) indexing a step in our discrete-time process is automatically rounded down to the nearest integer.
\end{remark}
\noindent
Let $N_k(G)$ denote the number of vertices of~$G$ which are in components of size~$k$, and let
\begin{equation}\label{eq:def:Sr}
S_r(G) = \sum_{C} |C|^r/n = \sum_{k \ge 1}k^{r-1}N_k(G)/n,
\end{equation}
where the first sum is over all components $C$ of $G$ and $|C|$ is the number of vertices in $C$.
Thus $S_{r+1}(G)$ is the $r$th
moment of the size of the component containing a randomly chosen vertex.
The \emph{susceptibility} $S_2(G)$ is of particular interest since in many classical percolation
models its analogue diverges precisely at the critical point.
Spencer and Wormald~\cite{SW} showed that this holds also for bounded-size
Achlioptas processes: the $n\to\infty$ limit of $S_2(G^{\cR}_{n,tn})$ diverges at the critical time~$\tc$.

Once the existence and location of the phase transition have been established, one typically asks about finer details of the phase transition,  
 in particular about the size of the largest component.
For the Bohman--Frieze process this was addressed 
in an influential paper
by Janson and Spencer~\cite{JS}, using a mix of coupling arguments, the theory of inhomogeneous random graphs, and asymptotic analysis of differential equations.  
They showed that there is a constant $c=c^{\mathrm{BF}}>0$ such that we whp have linear growth of the form
\begin{equation}\label{eq:L1:BF}
 \lim_{n\to\infty} L_1(G^{\mathrm{BF}}_{n,\tc n + \eps n})/n = (c+o(1)) \eps
\end{equation}
as $\eps \searrow 0$, which resembles the Erd{\H o}s--R{\'e}nyi behaviour (where $\tc=1/2$ and $c=4$). 
Using work of the present authors~\cite{RWapcont} and PDE theory, this was 
extended to certain BF-like rules by Drmota, Kang and Panagiotou~\cite{DKP},  
but the general case remained open until now.
Regarding the asymptotics in~\eqref{eq:L1:BF}, note that~$\eps$ is held fixed as $n\to\infty$; only after taking the limit
in~$n$ do we allow $\eps\to 0$. 

The next questions one typically asks concern the `finite-size scaling', i.e., behaviour as a function of~$n$, usually with a focus on the size of the largest component as $\eps=\eps(n) \to 0$ at various rates.
For the `critical window' $\eps = \lambda n^{-1/3}$ (with $\lambda \in \RR$) of bounded-size rules
this was resolved
by Bhamidi, Budhiraja and Wang~\cite{BBW11,BBW12b}, using coupling arguments, Aldous' multiplicative coalescent, and inhomogeneous random graph theory. 
However, the size of the largest component outside this window has surprisingly remained open, despite considerable
attention.
For example, two papers~\cite{BBW12a,Sen} were solely devoted to the study of $L_1(G^{\cR}_{n,i})$ in the usually easier \emph{subcritical phase} ($i=\tc n - \eps n$ with $\eps^3 n \to \infty$), but 
both obtained suboptimal upper bounds (a~similar remark applies to the susceptibility, see~\cite{BBW12b}). 
In contrast, there are no rigorous results about $L_1(G^{\cR}_{n,i})$ in the more interesting \emph{weakly supercritical phase} ($i=\tc n + \eps n$ with $\eps\to 0$ but $\eps^3 n \to \infty$),
making the size of the largest component perhaps the most important open problem in the context of bounded-size~rules.

Of course, there are many further 
questions that one can ask about the phase transition, and here one central theme is:  
how similar are Achlioptas processes to the Erd{\H o}s--R{\'e}nyi reference model?
For example, concerning vertices in `small' components of size~$k$, tree counting shows that in the latter model we have 
\begin{equation}\label{eq:Nk:Er}
N_k(G^{\mathrm{ER}}_{n,\tc n \pm \eps n})/n \approx k^{-3/2}e^{-(2+o(1))\eps^2 k}/\sqrt{2\pi}
\end{equation}
as $\eps \to 0$ and $k \to \infty$ (ignoring technicalities), where $\tc=1/2$.
Due to the dependencies between the edges explicit formulae are not available for bounded-size rules, 
which motivates the development of new robust methods that 
recover the tree-like Erd{\H o}s--R{\'e}nyi asymptotics in such more complicated settings. 
Here Kang, Perkins and Spencer~\cite{KPS} presented 
an interesting PDE-based argument
for the Bohman--Frieze process, 
but this contains
an error (see their erratum~\cite{KPSE}) which does not seem to be fixable.
Subsequently, partial results have been proved by Drmota, Kang and Panagiotou~\cite{DKP}
for a restricted class of BF-like rules.

In this paper we answer the percolation questions discussed above
for all bounded-size Achlioptas processes, settling a number of open problems and conjectures~\cite{JSP,JS,KPSPC,BS,KPS,BBW12a,DKP} concerning the phase transition. 
We first present a simplified version of our main results, 
writing $L_j(i)=L_j(G^{\cR}_{n,i})$, $S_r(i)=S_r(G^{\cR}_{n,i})$ and $N_k(i)=N_k(G^{\cR}_{n,i})$ to avoid clutter.
In a nutshell, \eqref{eq:intro:Ljsub}--\eqref{eq:intro:smallcpt} of Theorem~\ref{thm:intro} 
determine the finite-size scaling behaviour of the largest component, the susceptibility, and the small components.
Informally speaking, all these key statistics 
have, up to rule-specific constants, the same asymptotic behaviour as in the 
Erd{\H o}s--R{\'e}nyi process, including the same `critical exponents'
(in ER we have $t_{\mathrm{c}}=C=1/2$, $c=4$, $B_r=(2r-5)!!2^{-2r+3}$, $A=1/\sqrt{2\pi}$ and $a=2$, see also Remark~\ref{rem:ER}).
In particular, 
\eqref{eq:intro:L1sup}--\eqref{eq:intro:L2sup} 
show that the unique `giant component' initially grows at a linear rate, as illustrated by Figure~\ref{fig:L1plots}.
\begin{theorem}
\label{thm:intro}
Let $\cR$ be a bounded-size rule with
critical time $\tc=\tc^{\cR}>0$ as in~\eqref{eq:def:tc}. 
There are rule-dependent positive constants $a,A,c,C,\gamma$ and $(B_r)_{r \ge 2}$ such that the following holds for any ${\eps=\eps(n) \ge 0}$ satisfying ${\eps \to 0}$ and ${\eps^3 n \to \infty}$ as ${n \to \infty}$.%
{\begin{enumerate}
\itemindent -0.75em 
\item\label{th1i} \emph{(Subcritical phase)} For any fixed $j \ge 1$ and $r \ge 2$, whp we have%
\begin{align}
\label{eq:intro:Ljsub}
L_j(\tc n -\eps n) & \sim C \eps^{-2}\log(\eps^3 n),\\
\label{eq:intro:Sksub}
S_r(\tc n -\eps n) & \sim B_r \eps^{-2r+3}.%
\end{align}%
\item\label{th1ii} \emph{(Supercritical phase)} Whp we have%
\begin{align}
\label{eq:intro:L1sup}
L_1(\tc n +\eps n) & \sim c\eps n , \\
\label{eq:intro:L2sup}
L_2(\tc n +\eps n) & = o(\eps n) .%
\end{align}%
\item\label{th1iii} \emph{(Small components)}
Suppose that $k=k(n)\ge 1$ and $\eps=\eps(n)\ge 0$ satisfy $k \le n^{\gamma}$, $\eps^2k \le \gamma\log n$, $k\to\infty$ and $\eps^3k\to 0$. Then whp we have
\begin{align}
\label{eq:intro:smallcpt}
N_k(\tc n \pm \eps n) & \sim A k^{-3/2}e^{-a\eps^2k}n.
\end{align}%
Here we do not assume that $\eps^3n\to\infty$; note that~\eqref{eq:intro:smallcpt} is a statement about two steps of the process, $\tc n+\eps n$ and $\tc n-\eps n$, not a range of steps.
\end{enumerate}}%
\end{theorem}
In each case, what we actually prove is stronger (e.g., relaxing $\eps\to 0$ to $\eps \le \eps_0$); see Section~\ref{sec:results}. 
To the best of our knowledge, analogous precise results, giving sharp estimates for the size of the largest component in the entire sub- and super-critical phases (see also Remark~\ref{rem:intro}), are known only for the Erd{\H o}s--R{\'e}nyi model~\cite{Bollobas1984,Luczak1990},
and percolated versions of random regular graphs~\cite{NachmiasPeres2010}, the configuration model~\cite{OR2012}, and (in the supercritical case) the hypercube~\cite{vdHN2012}.
\begin{remark}[Notation]\label{rem:panot}
Here and throughout we use the following standard notation for probabilistic asymptotics, where $(X_n)$
is a sequence of random variables and $f(n)$ a function. `$X_n\sim f(n)$ whp' 
means that there is some
$\delta(n)\to 0$ such that whp $(1-\delta(n))f(n)\le X_n\le (1+\delta(n))f(n)$. This is equivalent to 
$X_n=(1+\op(1))f(n)$, where in general $\op(f(n))$ denotes a quantity that, after dividing by $f(n)$,
tends to 0 in probability. Similarly, `$X_n=o(f(n))$ whp' simply means $X_n=\op(f(n))$.
`$X_n=\Op(f(n))$' means that $X_n/f(n)$ is bounded in probability. 
Finally, we use $a \pm b$ as shorthand for $a+b$ or $a-b$, with (as usual) a consistent choice of signs. Thus $a^\pm=b\mp c$ means $a^+=b-c$ and $a^-=b+c$, for example.
\end{remark}
\begin{remark}[Critical window]\label{rem:intro}
Assume that $\eps=\eps(n) \ge 0$ satisfies $\eps \to 0$ and $\eps^3 n \le \omega=\omega(n) \to \infty$ as $n \to \infty$. 
By~\eqref{eq:intro:Ljsub} and~\eqref{eq:intro:L1sup}, for any step $\tc n -\eps n \le i \le \tc n + \eps n$ whp 
$n^{2/3}/\omega \le L_1(i) \le \omega n^{2/3}$, say. 
It follows that when $\eps^3n=O(1)$, then $L_1(i)=\Op(n^{2/3})$ (in fact $\Thetap(n^{2/3})$). We do not discuss this critical case further,
since it is covered by the results of~\cite{BBW11,BBW12b}.
\end{remark}
The natural benchmark for our results is the classical Erd{\H o}s--R{\'e}nyi random graph process (which, as discussed, is also a bounded-size Achlioptas process). 
In their seminal~1960 paper, Erd{\H{o}}s and R{\'e}nyi~\cite{ER1960} determined the asymptotics of the 
number $L_1(G^{\mathrm{ER}}_{n,i})$ of vertices in the largest component after $i=tn$ steps for fixed $t>0$.
In~1984,
 Bollob{\'a}s~\cite{Bollobas1984} initiated the study of `zooming in' on the critical point, i.e, of the case~$t \sim 1/2$, which 
has nowadays emerged into a powerful paradigm.  
In particular, assuming $\eps =\eps(n) \to 0$ and $\eps^3 n \ge (\log n)^{3/2}$, Bollob{\'a}s identified the 
characteristic features of the phase transition. 
Namely, in the subcritical phase $i=n/2-\eps n$ there are many `large' components of size $L_j(G^{\mathrm{ER}}_{n,i}) \sim 2^{-1}\eps^{-2}\log(\eps^3n)$, 
similar to~\eqref{eq:intro:Ljsub}.  
In the supercritical phase $i=n/2+\eps n$ there is a unique `giant component' of size $L_1(G^{\mathrm{ER}}_{n,i}) \sim 4\eps n$, whereas all other components are much smaller, 
similar to~\eqref{eq:intro:L1sup}--\eqref{eq:intro:L2sup}. 
In~1990, {\L}uczak~\cite{Luczak1990} sharpened the assumptions of~\cite{Bollobas1984} to the optimal condition $\eps =\eps(n) \to 0$ and $\eps^3 n \to \infty$ (also used by Theorem~\ref{thm:intro}), thus fully identifying the phase transition picture.
Indeed, a separation between the sub- and super-critical phases requires $\eps^{-2}\log(\eps^3n) = o(\eps n)$, which is equivalent to $\eps^3 n \to \infty$ (see also Remark~\ref{rem:intro}). 
Informally speaking, Theorem~\ref{thm:intro} shows that the characteristic Erd{\H o}s--R{\'e}nyi features are 
robust in the sense that they remain valid for all bounded-size Achlioptas processes.

One main novelty of our proof approach is a combinatorial multi-round exposure argument around the critical point~$\tc$. 
From a technical perspective this allows us to avoid arguments where the process is approximated
(in some time interval) by a simpler process, which would introduce various error terms.
Such approximations are key in all previous work on this problem~\cite{SW,BK,JS,KPS,BBW11,RWapcont,BBW12b,BBW12a,Sen,DKP}. 
Near the critical~$i \approx \tc n$ we are able to track the exact evolution of our bounded-size Achlioptas process $(G^{\cR}_{n,i})_{i \ge 0}$.
This more direct control is key for our very precise results, in particular concerning the finite-size scaling behaviour as $\eps=\eps(n) \to 0$. 
In this context our high-level proof strategy for step $i=\tc n + \eps n$ 
is roughly as follows: 
(i)~we first track the evolution of $(G^{\cR}_{n,i})_{0 \le i \le i_0}$ up to step $i_0=(\tc-\sigma)n$ for some tiny constant~$\sigma > 0$, 
(ii)~we then reveal information about the steps $(\tc-\sigma) n, \ldots, (\tc+\eps) n$ in two stages
(a type of a two-round exposure), 
and (iii)~we analyze the second exposure round using branching process arguments.
The key is to find a suitable two-round exposure method in step (ii).
Of course, even having found this,
since there are dependencies between the edges, the technicalities of 
our approach are naturally 
quite involved (based on a blend of techniques, including the differential equation method, PDE theory, and branching process analysis); 
see Section~\ref{sec:overview} for a detailed overview of our arguments. 

So far we have discussed bounded-size rules. One of the first concrete rules suggested was the
\emph{product rule} (PR), where we select the potential edge minimizing the product of the sizes
of the components it joins. 
This rule belongs to the class of \emph{size rules}, which make their decisions based only on the sizes $c_1, \ldots, c_4$ of the components containing the endvertices of~$e_1,e_2$, but is not a \emph{bounded}-size rule.
The original question of Achlioptas from around 2000 was whether one can
delay the phase transition beyond $\tc^{\mathrm{ER}} = 1/2$ using an appropriate rule, and Bollob\'as
quickly suggested the product rule as most likely to do this. 
In fact, this question (which with hindsight
is not too hard) was answered affirmatively by Bohman and Frieze~\cite{BF} using a much simpler rule (a minor variant of the BF~rule). 
Under the influence of statistical mechanics, the focus quickly shifted
from the location of the critical time~$\tc$ 
to the qualitative behaviour of the phase transition (see, e.g.,~\cite{JSP}). 
In this context the product rule has received considerable attention; 
the simulation-based Figure~\ref{fig:L1plots} shows why: for this rule the growth of the largest component seems very abrupt, i.e., much steeper than in the Erd{\H o}s--R{\'e}nyi process.
In fact, based on extensive numerical data, Achlioptas, D'Souza and Spencer conjectured in \emph{Science}~\cite{DRS} that, for the product rule, the size of the largest component whp `jumps' from $o(n)$ to $\Theta(n)$ 
in $o(n)$ steps of the process, a phenomenon known as `explosive percolation'.
Although this claim was supported by many papers in the physics literature (see the references in~\cite{RW,RWapcont,J,dCDGM}), 
we proved in~\cite{RW,RWapcont} that no Achlioptas process can `jump', i.e., that they all have continuous phase transitions. 
Nevertheless, the product rule (like other similar rules) still seems
to have an extremely steep phase transition;
we believe that $L_1(G^{\mathrm{PR}}_{n,\tc n + \eps n}) \sim c\eps^{\beta} n$ for some $\beta \in (0,1)$, in contrast to the `linear growth'~\eqref{eq:intro:L1sup} of bounded-size rules; see also~\cite{dCDGM}. 
Despite much attention, general size rules have largely remained resistant to rigorous analysis; see~\cite{RWapsubcr,RWapunique} for some partial results.
Our simulations and heuristics~\cite{RWPRE} strongly suggest that $L_1(tn)/n$ can even be nonconvergent in some cases.

\subsection{Organization}
The rest of the paper is organized as follows.
In Section~\ref{sec:results} we define the class of models that we shall work with, which is more general
than that of bounded-size Achlioptas process.
Then we give our detailed results for the size of the largest component, the number of vertices in small components,
and the susceptibility (these imply Theorem~\ref{thm:intro}). 
In Section~\ref{sec:overview} we give an overview of the proofs, highlighting the key ideas and techniques --
the reader mainly interested in the ideas of the proofs may wish to read this section first.
In Section~\ref{sec:setup} we formally introduce the proof setup, including the two-round exposure, and establish some preparatory results. 
Sections~\ref{sec:cpl} and~\ref{sec:cpl2} are the core of the paper; here we relate the component size distribution of~$G^{\cR}_{n,i}$ to a certain branching process, and estimate the first two moments of~$N_k(i)$.
In Section~\ref{sec:proof} we then establish our main results for $L_1(i)$, $N_k(i)$ and $S_r(i)$, by exploiting the technical work of Sections~\ref{sec:setup}--\ref{sec:cpl2} and the branching process results proved with Svante Janson in~\cite{BPpaper}. 
In Section~\ref{sec:open} we discuss some extensions and several open problems. 
Finally, Appendix~\ref{sec:apx} contains some results and calculations that are omitted from the main text, and 
Appendix~\ref{sec:notation} gives a brief glossary of notation.

\enlargethispage{\baselineskip}

\subsection{Acknowledgements}
The authors thank Costante Bellettini and Luc Nguyen for helpful comments on analytic solutions to PDEs, 
Svante Janson for useful feedback on the branching process analysis contained in an earlier version of this paper (based on large deviation arguments using a uniform local limit theorem together with uniform Laplace method estimates), 
and Joel Spencer for his continued interest and encouragement.
We are also grateful to the referees for a careful reading of the paper, and for helpful suggestions concerning the presentation.

\newoddpage

\section{Statement of the results}\label{sec:results}

In this section we state our main results in full, and also give further details of the most relevant earlier results
for comparison.
In informal language, we show that in any bounded-size Achlioptas process, the phase transition `looks like' that in the Erd{\H o}s--R{\'e}nyi reference model (with respect to many key statistics).
In mathematical physics jargon this loosely says that all bounded-size rules belong to the same \emph{`universality class'} (while certain constants may differ, the behaviour is essentially the same).  
In a nutshell, our three main contributions
are as follows, always considering an arbitrary bounded-size rule.
\begin{enumerate}
	\item[(1)] {\bf Size of the largest component:} 
	We determine the asymptotic \emph{size of the largest component} in the sub- and super-critical phases, i.e., step $i=\tc n \pm \eps n$ with $|\eps|^3 n \to \infty$ and $|\eps| \le \eps_0$ (see Theorems~\ref{thm:L1sub} and~\ref{thm:L1} in Section~\ref{sec:intro:L1}), and show uniqueness of the `giant' component in the supercritical phase. We recover the characteristic Erd{\H o}s--R{\'e}nyi features showing, for example, 
that whp 
$L_2(\tc n +\eps n) \ll L_1(\tc n +\eps n) \sim \rho(\tc+\eps) n$ for some (rule-dependent) analytic function
$\rho$ with $\rho(\tc+\eps) \sim c \eps$ as~$\eps \searrow 0$. 
	\item[(2)] {\bf Small components:}
	We determine the whp asymptotics
	of the number of vertices in components of size~$k$ as approximately $N_k(\tc n \pm \eps n) \approx Ak^{-3/2}e^{-(a+o(1))\eps^2k}n$ (see Theorems~\ref{thm:rhok} and~\ref{thm:Nk} in Section~\ref{sec:intro:Nk}). 
	Informally speaking, in all bounded-size rules the number of vertices in \emph{small components} thus exhibits Erd{\H o}s--R{\'e}nyi tree-like behaviour, including polynomial decay at criticality (the case $\eps=0$).  
	\item[(3)] {\bf Susceptibility:} 
	We determine the $\eps=\eps(n) \to 0$ whp asymptotics of the subcritical \emph{susceptibility} as
	$S_r(\tc n-\eps n) \sim B_r \eps^{-2r+3}$ (see Theorem~\ref{thm:Sj} in Section~\ref{sec:intro:Nk}). 
	Thus the `critical exponents' associated to the susceptibility are the same for any bounded-size rule as in the Erd{\H o}s--R{\'e}nyi case.
\end{enumerate}

So far we have 
largely ignored that Achlioptas processes evolve over time. 
Indeed, Theorem~\ref{thm:intro} deals with the `static behaviour' of some particular step~$i=i(n)$, i.e., whp properties of the random graph~$G^{\cR}_{n,i}$. 
However, we are also (perhaps even more) interested in the `dynamic behaviour' of the evolving graph, i.e., whp properties of the random graph process~$(G^{\cR}_{n,i})_{i \ge 0}$. 
Our results accommodate this: Theorems~\ref{thm:L1}, \ref{thm:Nk} and~\ref{thm:Sj} apply \emph{simultaneously} to every step outside of the critical window, i.e., 
every step $i=\tc n \pm \eps n$ with $|\eps|^3 n \to \infty$ and~$|\eps| \le \eps_0$.

\begin{remark}\label{rem:ER} 
When comparing our statements with results for the classical Erd{\H o}s--R{\'e}nyi process, the reader should keep in mind that $G^{\mathrm{ER}}_{n,i}$ corresponds to the uniform size model with $i$~edges. 
In particular, results for step $i=n/2 \pm \eps n$ should be compared to the binomial model $G_{n,p}$ with edge probability $p=(1\pm 2\eps)/n$.
\end{remark}

\begin{remark}\label{rem:notation} 
Occasionally we write~$a_n \ll b_n$ for~$a_n = o(b_n)$, and~$a_n \gg b_n$ for~$a_n = \omega(b_n)$. 
\end{remark}

\begin{definition}\label{def:analytic} 
For $I\subseteq \RR$, we say that a function $f:I\to \RR$ is \emph{(real) analytic} if for every $x_0\in I$
there is an $r>0$ and a power series $g(x)=\sum_{j\ge 0} a_j (x-x_0)^j$ with radius of convergence at least $r$ such
that~$f$ and~$g$ coincide on $(x_0-r,x_0+r)\cap I$. This implies that $f$ is infinitely differentiable, but not vice versa.
A function~$f$ defined on some domain including~$I$ is \emph{(real) analytic on~$I$}
if $f|_I$ is analytic. The definitions for functions of several variables are analogous.
\end{definition}

\subsection{Bounded-size $\ell$-vertex rules}\label{sec:intro:def}
All our results apply to (the bounded-size case of) a class of processes that generalize Achlioptas
processes. As in~\cite{RW,RWapcont}
we call these \emph{$\ell$-vertex rules}.
Informally, in each step we sample~$\ell$ random vertices (instead of two random edges), 
and according to some rule~$\cR$ then add one of the $\binom{\ell}{2}$ possible edges between them to the evolving graph.
 
Formally, an \emph{$\ell$-vertex size rule~$\cR$} yields for each~$n$ a random sequence $(G^{\cR}_{n,i})_{i\ge 0}$ of graphs with vertex set $[n]=\{1, \ldots, n\}$, as follows. $G^{\cR}_{n,0}$ is the empty graph with no edges. 
In each step~$i\ge 1$ we draw~$\ell$ vertices $\vv_i={(v_{i,1},\ldots,v_{i,\ell})}$
from $[n]$ independently and uniformly at random, 
and then, writing $\vc_i={(c_{i,1},\ldots,c_{i,\ell})}$ 
for the sizes of the components containing $v_{i,1},\ldots,v_{i,\ell}$ in $G^{\cR}_{n,i-1}$, 
we obtain~$G^{\cR}_{n,i}$ by adding the edge ${v_{i,j_1}v_{i,j_2}}$ to $G_{n,i-1}^{\cR}$, where the rule $\cR$ deterministically selects the edge (between the vertices in~$\vv_i$) based only on the component sizes. 
Thus we may think of $\cR$ as a function $(\vc_i) \mapsto \{j_1,j_2\}$ from $(\NNP)^\ell$
to $\binom{[\ell]}{2}$.
Note that $G^{\cR}_{n,i}$ may contain loops and multiple edges; formally, it is a multigraph. However, there 
will be rather few of these and they do not affect the component structure, so the reader will lose nothing
thinking of $G^{\cR}_{n,i}$ as a simple~graph.

As the reader can guess, a~\emph{bounded-size $\ell$-vertex rule~$\cR$ with cut-off~$K$} is then an $\ell$-vertex size rule where all component sizes larger than~$K$ are treated in the same way. 
Following the literature (and to avoid clutter in the proofs), we introduce the convention that a component has \emph{size $\omega$} if it has size at least~$K+1$. We define the set 
\[
\cC = \cC_K := \{1, \ldots, K, \omega \} 
\]
of all `observable' component sizes. 
Thus any bounded-size rule~$\cR$ with cut-off~$K$ corresponds to a function~${\cC_K^\ell\to \binom{[\ell]}{2}}$. 
Of course, $\cR$ is a~\emph{bounded-size $\ell$-vertex rule} if it satisfies the definition above for some~$K$.

For the purpose of this paper, results for $\ell$-vertex rules routinely transfer to processes with small variations in the definition (since we only consider at most, say, $9n$ steps, exploiting that $\tc \le 1$ by~\cite{RWapcont}). 
As in~\cite{RWapcont} this includes, for example, each time picking an $\ell$-tuple of \emph{distinct} 
vertices, or picking (the ends of) $\ell/2$ randomly selected (distinct) 
edges not already present. 
We thus recover the original Achlioptas processes 
as $4$-vertex rules where $\cR$ always selects one of the pairs~$e_{i,1}=\{v_{i,1},v_{i,2}\}$ and~$e_{i,2}=\{v_{i,3},v_{i,4}\}$. 
Since we are aiming for strong results here, one has to be a little careful with this reduction; an explicit
argument is given in Appendix~\ref{apx:tfer}.

\begin{remark}\label{rem:defs}
In the results below, a number of rule-dependent constants and functions appear. To avoid repetition, we briefly describe
the key ones here. 
Firstly, for each rule $\cR$ there is a set $\cSR\subseteq \NN^+$ of component sizes that can be produced 
by the rule. For large enough $k$, we have $k\in \cSR$ if and only if $k$ is a multiple of the
\emph{period} $\per$ of the rule; see Section~\ref{sec:period}. For all Achlioptas processes,
$\cSR=\NN^+$ and $\per=1$, so the indicator functions $\indic{k\in \cSR}$ appearing in many results
play no role in this case.
Secondly, for each rule $\cR$ there is a function
$\psi(t)=\psi^{\cR}(t)$ describing the exponential rate of decay of the component size distribution at time $t$ (step $tn$),
as in Theorem~\ref{thm:rhok}.   
This function is (real) analytic on a neighbourhood $(\tc-\eps_0,\tc+\eps_0)$ of $\tc$, with $\psi(\tc)=\psi'(\tc)=0$ and $\psi''(\tc) > 0$.
Hence $\psi(\tc \pm \eps)=\Theta(\eps^2)$ as $\eps\to 0$. 
\end{remark}

\subsection{Size of the largest component}\label{sec:intro:L1}
In this subsection we discuss our results for the size of the largest component in bounded-size rules, 
which are much in the spirit of the pioneering work of Bollob{\'a}s~\cite{Bollobas1984} and {\L}uczak~\cite{Luczak1990} for the Erd{\H{o}}s--R{\'e}nyi model. 
Here Theorem~\ref{thm:L1} is perhaps our most important single result: it establishes the asymptotics of $L_1(i)=L_1(G^{\cR}_{n,i})$ in the sub- and super-critical phases (i.e., all steps $i=\tc n \pm \eps n$ with $|\eps|^3 n \to \infty$ and $|\eps| \le \eps_0$).
These asymptotics are as in the Erd{\H{o}}s--R{\'e}nyi case, up to rule-specific constants. 

One of the most basic questions about the phase transition in any random graph model is: how large is the largest component just after the transition? 
From the general continuity results of~\cite{RWapcont} it follows for bounded-size rules that $L_1(tn)/n$ converges to a deterministic `scaling limit' (see also~\cite{RWapunique}). 
More concretely, there exists a continuous function $\rho=\rho^{\cR}:[0,\infty) \to [0,1]$ such that for each $t \ge 0$ we have 
\begin{equation}\label{eq:L1:pto}
\frac{L_1(tn)}{n} \pto \rho(t) , 
\end{equation}
where $\pto$ denotes convergence in probability (i.e., for every $\eta>0$ whp 
$|L_1(tn)/n-\rho(t)| \le \eta$). 
In fact, it also follows that $\rho(t) = 0$ for $t \le \tc$ and $\rho(t)>0$ otherwise, see~\cite{SW,BK,RWapsubcr}.  
Of course, due to our interest in the size of the largest component, this raises the natural question:
what are the asymptotics of the scaling limit $\rho=\rho^{\cR}$?
For \emph{some} bounded-size rules, Janson and Spencer~\cite{JS} and Drmota, Kang and Panagiotou~\cite{DKP} 
showed that $\rho(\tc+\eps) \sim c \eps$ as $\eps \searrow 0$, where $c=c^{\cR}>0$. 
However, for Erd{\H o}s--R{\'e}nyi random graphs much stronger properties are known: $\rho=\rho^{\mathrm{ER}}$ is analytic on $[1/2,\infty)$.
In particular, it has a power series expansion of the form 
\[
\rho^{\mathrm{ER}}(1/2+\eps) = 4 \eps + \sum_{j \ge 2} a_j \eps^j
\]
for $\eps \ge 0$ having positive radius of convergence 
(in fact, $1-\rho(t)=e^{-2t\rho(t)}$).
Our next theorem shows that \emph{all} bounded-size rules have these typical Erd{\H o}s--R{\'e}nyi properties (up to rule specific constants), 
confirming natural conjectures of Janson and Spencer~\cite{JS} and Borgs and Spencer~\cite{BS} (and the folklore conjecture that $\rho(\tc+\eps) \sim c \eps$ initially grows at a linear rate).
\begin{theorem}[Linear growth of the scaling limit]%
\label{thm:L1rho}
Let $\cR$ be a bounded-size $\ell$-vertex rule. 
Let the critical time $\tc>0$ and the function $\rho=\rho^{\cR}$ be as in~\eqref{eq:def:tc} and~\eqref{eq:L1:pto}.
Then the function $\rho$ is analytic on $[\tc,\tc+\eps_0]$ for some $\eps_0>0$, with $\rho(\tc)=0$ and 
the right derivative of $\rho$ at $\tc$ strictly positive. 
More precisely, there are constants $(a_j)_{j \ge 1}$ with $a_1>0$ such that 
for all $\eps \in [0,\eps_0]$ we have 
\begin{equation}\label{eq:thm:L1rho}
\rho(\tc+\eps) = \sum_{j \ge 1} a_j\eps^j . 
\end{equation} 
\end{theorem}
Informally, \eqref{eq:L1:pto} and~\eqref{eq:thm:L1rho} show that the initial growth of the largest component is linear for any bounded-size rule, i.e., roughly that $L_1(\tc n + \eps n) \approx c \eps n$ for some rule-dependent constant $c=a_1>0$ (see also Figure~\ref{fig:L1plots}). 
\begin{remark}\label{rem:rhok:BP}
The proof shows that for $t \in [\tc-\eps_0,\tc+\eps_0]$ we have $\rho(t)=\Pr(|\bp_t|=\infty)$ for a certain branching process~$\bp_t$ defined in Section~\ref{sec:BPI}. In the Erd\H os--R\'enyi case, this would simply be a Poisson Galton--Watson process, but here it is much more complicated.
\end{remark}
The convergent power series expansion~\eqref{eq:thm:L1rho} improves and extends results of Janson and Spencer~\cite{JS} for the Bohman--Frieze process (with a $O(\eps^{4/3})$ second order error term), and one of the main results of Drmota, Kang and Panagiotou~\cite{DKP} for a restricted class of bounded-size rules (they establish $\rho(\tc+\eps) = c \eps + O(\eps^2)$ for BF-like rules). 
Theorem~\ref{thm:L1rho} also shows that $\rho'(t)$ is discontinuous at~$\tc$ (recall that $\rho(t)=0$ for $t \le \tc$), which in mathematical physics is a key feature of a `second order' phase transition.

Unfortunately, the convergence result~\eqref{eq:L1:pto} tells us very little 
about the size of the largest component just before the phase transition (recall that $\rho(t)=0$ for $t \le \tc$).
For the Erd{\H o}s--R{\'e}nyi process it is well-known that in the subcritical phase, for any fixed $j \ge 1$ we roughly have $L_j(n/2-\eps n) = \Theta(\eps^{-2}\log(\eps^3 n))$ whenever $\eps^3n \to \infty$, see, e.g.,~\cite{Bollobas1984,Luczak1990}. 
For bounded-size rules there are several partial results~\cite{SW,KPS,BBW12a,Sen} for $L_1(\tc n-\eps n)$, but none are as strong as the aforementioned Bollob\'as--{\L}uczak results from~\cite{Bollobas1984,Luczak1990}.
For the subcritical phase, our next theorem establishes the full Erd{\H o}s--R{\'e}nyi-type behaviour for all bounded-size rules (in a strong form).  
Theorem~\ref{thm:L1sub} confirms a conjecture of Kang, Perkins and Spencer~\cite{KPS}, and 
resolves a problem
of Bhamidi, Budhiraja and Wang~\cite{BBW12a}, both concerning upper bounds of the form $L_1(\tc n-\eps n) \le D \eps^{-2}\log n$. 
\begin{theorem}[Largest subcritical components]%
\label{thm:L1sub}
Let $\cR$ be a bounded-size $\ell$-vertex rule with critical time $\tc>0$ as in~\eqref{eq:def:tc}. 
There is a constant $\eps_0>0$ such that the following holds for any integer $r \ge 1$.
For any $i=i(n) \ge 0$ 
such that $\eps=\tc-i/n$ satisfies $\eps \in (0,\eps_0)$ and $\eps^3 n \to \infty$ as $n \to \infty$, we have 
\begin{equation}\label{eq:L1sub}
L_r(i) = \psi(\tc-\eps)^{-1}\left(\log(\eps^3n)-\frac{5}{2}\log\log(\eps^3n) +\Op(1) \right) ,
\end{equation}
where the rule-dependent function $\psi(t)$ is as in Theorem~\ref{thm:rhok} below.
\end{theorem}
Here, as usual, $X_n=\Op(1)$ means that for any $\delta > 0$ there are $C_{\delta},n_{0}>0$ such that $\Pr(|X_n| \le C_{\delta}) \ge 1-\delta$ for $n \ge n_0$.
See Remark~\ref{rem:defs} for the interpretation of the function $\psi$.
Since $\psi(\tc-\eps)=\Theta(\eps^2)$, this result shows in crude terms that the largest $O(1)$
components have around the same size $L_r(\tc n -\eps n) \approx L_1(\tc n -\eps n) \approx a\eps^{-2} \log(\eps^3 n)$ for some rule-dependent constant $a >0$. 

Theorem~\ref{thm:L1sub} is best possible:
the assumption $\eps^3 n \to \infty$ cannot be relaxed -- inside the critical window the sizes of the largest components are not concentrated, see~\cite{BBW12b}. 
Furthermore, as discussed in~\cite{BR2009}, the $\Op(1)$ error term  is sharp in the Erd{\H o}s--R{\'e}nyi case (where $\tc=1/2$ and $\psi(\tc-\eps)\sim 2\eps^2$ as $\eps \to 0$).
Theorem~\ref{thm:L1sub} improves several previous results on the size of the largest subcritical component due to Wormald and Spencer~\cite{SW}, Kang, Perkins and Spencer~\cite{KPS}, Bhamidi, Budhiraja and Wang~\cite{BBW12a} and Sen~\cite{Sen}. 
Most notably, in the weakly subcritical phase $i=\tc-\eps n$ with $\eps=\eps(n) \to 0$ the main results of~\cite{BBW12a,Sen} are as follows: 
for bounded-size rules~\cite{BBW12a} establishes bounds of the form $L_1(i) \le D\eps^{-2}(\log n)^4$ for $\eps=\eps(n) \ge n^{-1/4}$, 
which in~\cite{Sen} was sharpened for the Bohman--Frieze rule to $L_1(i) \le D\eps^{-2}\log n$ for $\eps=\eps(n) \ge n^{-1/3}$. 
The key difference is that~\eqref{eq:L1sub} provides \emph{matching} bounds 
(the harder lower bounds were missing in previous work) 
all the way down to the critical window. Here 
the difference between $\log n$ and $\log(\eps^3n)$ matters. 

For the supercritical phase, \eqref{eq:L1:pto} and Theorem~\ref{thm:L1rho} apply only for \emph{fixed} $\eps>0$, showing roughly that $L_1(\tc n + \eps n) \approx \rho(\tc+\eps) n$. 
Of course, it is much more interesting to `zoom in' on the critical point~$\tc$, and study the size of the largest component when $\eps=\eps(n) \to 0$. 
Despite being very prominent and interesting in 
the classical Erd{\H o}s--R{\'e}nyi model, see, e.g.,~\cite{Bollobas1984,Luczak1990,NP2007,BR2009,BB,JLR}, 
the weakly supercritical phase of bounded-size rules has remained resistant to rigorous analysis for more than a decade. 
The following theorem closes this gap in our understanding of the phase transition, 
establishing the typical Erd{\H o}s--R{\'e}nyi characteristics for all bounded-size rules.
Informally, \eqref{eq:super:L1}--\eqref{eq:super:L2} state that (whp) we have $L_2(\tc n+\eps n) \ll L_1(\tc n+\eps n) \approx \rho(\tc +\eps) n$ whenever $\eps=\eps(n)$ satisfies $\eps^3n\to\infty$ and $0 < \eps \le \eps_0$ (note that $\tau \to 0$ as $n \to \infty$, where we have not attempted to optimize~$\tau$).  
In particular, in view of~\eqref{eq:thm:L1rho} this means that, in all bounded-size rules, just after the critical window the unique `giant component' already grows with linear rate (i.e., that whp $L_1(\tc n+\eps n) \approx c \eps n$, see Figure~\ref{fig:L1plots}).
\begin{theorem}[Size of the largest component]%
\label{thm:L1}
Let $\cR$ be a bounded-size $\ell$-vertex rule. 
Let the critical time $\tc>0$ and the functions $\rho,\psi$ be as in~\eqref{eq:def:tc} and Theorems~\ref{thm:L1rho} and~\ref{thm:L1sub}. 
There is a constant $\eps_0>0$ such that the following holds for any function $\omega=\omega(n)$ with $\omega \to \infty$ as $n \to \infty$ \footnote{As usual we assume $\omega>0$, but we do not write this as it is not formally needed: any statements involving $\omega$ that we prove are asymptotic (for example any `whp' statement), and so we need consider only $n$ large enough that $\omega(n)>0$.}
and any fixed $r\ge 1$. Setting $\tau=\tau(n):=(\log \omega)^{-1/2}$, whp the following inequalities hold in all steps $i=i(n) \ge 0$. 
{\begin{enumerate}
\itemindent -0.75em 
\item \emph{(Subcritical phase)} If $\eps=\tc-i/n$ satisfies $\eps^3 n \ge \omega$ and $\eps \le \eps_0$, then%
\begin{align}
\label{eq:sub:Lj}
L_r(i) & \in [ (1\pm \tau) \psi(\tc-\eps)^{-1}\log(\eps^3 n) ],
\end{align}%
where $[a\pm b]$ denotes the interval $[a-b,a+b]$.
\item \emph{(Supercritical phase)} If $\eps=i/n-\tc$ satisfies $\eps^3 n \ge \omega$ and $\eps \le \eps_0$, then we have%
\begin{align}
\label{eq:super:L1}
L_1(i) & \in [ (1 \pm \tau) \rho(\tc+\eps) n ], \\
\label{eq:super:L2}
L_2(i) & \le \tau L_1(i) .
\end{align}
\end{enumerate}}%
\end{theorem}%
Note that we consider step $i=\tc n - \eps n$ in~\eqref{eq:sub:Lj} and step $i=\tc n + \eps n$ in~\eqref{eq:super:L1}--\eqref{eq:super:L2}. 
This parametrization may look strange, but it allows us to conveniently make whp statements about the random graph process~$(G^{\cR}_{n,i})_{i \ge 0}$. 
Indeed, \eqref{eq:sub:Lj}--\eqref{eq:super:L2} whp hold \emph{simultaneously} in every step outside of the critical window, i.e., 
every step $i=\tc n \pm \eps n$ with $|\eps|^3 n \to \infty$ and $|\eps| \le \eps_0$.
This is much stronger than a whp statement for some particular step $i=i(n)$, i.e., for the random graph $G^{\cR}_{n,i}$.
For this reason, the subcritical part of Theorem~\ref{thm:L1} does not follow from Theorem~\ref{thm:L1sub}. 
With this discussion in mind, one can argue that Theorem~\ref{thm:L1} describes the `dynamic behaviour' of the phase transition in bounded-size Achlioptas processes.

\subsection{Small components}\label{sec:intro:Nk}
During the last decade, a widely used heuristic for many `mean-field' random graph models is that most `small' components are trees (or tree-like). 
The rigorous foundations of this heuristic can ultimately be traced back to the classical Erd{\H o}s--R{\'e}nyi model, where it has been key in the discovery (and study) of the phase transition phenomenon, see~\cite{ER1960,Bollobas1984,Luczak1990}. 
As there are explicit counting formulae for trees, by exploiting the (approximate) independence between the edges this easily gives the asymptotics of $N_k(\tc n \pm \eps n)$ in the Erd{\H o}s--R{\'e}nyi process, see~\eqref{eq:Nk:Er}.

For bounded-size rules, the classical tree-counting approach breaks down due to the dependencies between the edges.
However, Spencer and Wormald~\cite{JSP,SW} already observed around~2001 that $N_k(tn)=N_k(G^{\cR}_{n,tn})$ can be approximated
via the differential equation method~\cite{DEM,DEM99}; their proof implicitly exploits that the main contribution again comes from trees, see also~\cite{RWapsubcr}. 
In particular, for the more general class of bounded-size $\ell$-vertex rules it is nowadays routine to prove that for any $t \in [0,\infty)$ and $k \ge 1$ we have 
\begin{equation}\label{eq:Nk:pto}
\frac{N_k(tn)}{n} \pto \rho_k(t) ,
\end{equation}
where the functions $\rho_k=\rho_k^{\cR}:[0,\infty) \to [0,1]$ are the unique solution of an associated system of differential equations ($\rho'_k$ depends only on $\rho_j$ with $1 \le j \le \max\{k,K\}$, see Lemmas~\ref{lem:Nk:t} and~\ref{lem:Nk2:t}). 
In fact, a byproduct of~\cite{SW,RWapsubcr} is that the $\rho_k(t)$ have exponential decay of the form $\rho_k(t) \le A_t e^{-a_t k}$ for $t < \tc$, with $a_t,A_t>0$.   
To sum up, in view of~\eqref{eq:Nk:Er} and~\eqref{eq:Nk:pto} the precise asymptotics of~$\rho_k(\tc n \pm \eps n)$ is an interesting problem (for the special case $\eps=0$ this was asked by Spencer and Wormald as early as 2001, see~\cite{JSP}).
This requires the development of new proof techniques, which recover the Erd{\H o}s--R{\'e}nyi tree-asymptotics in random graph models with dependencies.

This challenging direction of research was pursued by Kang, Perkins and Spencer~\cite{KPS,KPSE} and Drmota, Kang and Panagiotou~\cite{DKP}, who obtained some partial results for bounded-size rules, using PDE-theory and an auxiliary result from~\cite{RWapcont}.
However, they only recovered the exponential rate of decay (i.e., that $\log(\rho_k(\tc \pm \eps)) \approx -(a+o(1))\eps^2 k$ for small~$\eps$ and large~$k$) for a restricted class of rules which are Bohman--Frieze-like. 
We sidestep both shortcomings by directly relating $\rho_k(t)$ with an associated branching process, see Remark~\ref{rem:rhok:BP2}. 
Indeed, the next theorem completely resolves the asymptotic behaviour of $\rho_k(t)$ for all bounded-size $\ell$-vertex rules.
Note that below we have $\psi(\tc \pm \eps)=a\eps^2 + O(\eps^3)$ and $\theta(\tc \pm \eps) = A + O(\eps)$ for rule-dependent constants $a,A>0$, so \eqref{eq:rhok} qualitatively recovers the full Erd{\H o}s--R{\'e}nyi tree-like behaviour of~\eqref{eq:Nk:Er}.  
A more quantitative informal summary of \eqref{eq:Nk:pto} and \eqref{eq:rhok} is that whp $N_k(\tc n \pm \eps n) \approx A k^{-3/2}e^{-(a+o(1))\eps^2 k}n$ for large~$k$ and small~$\eps$ (ignoring technicalities). Note that the similar behaviour of small components above and below $\tc$ is a version of the `duality' phenomenon seen in $G(n,p)$, related to conditioning a branching process on extinction.
\begin{theorem}[Differential equation asymptotics]
\label{thm:rhok}
Let $\cR$ be a bounded-size $\ell$-vertex rule. 
Let the critical time $\tc>0$ and the functions $(\rho_k)_{k \ge 1}$ be as in~\eqref{eq:def:tc} and~\eqref{eq:Nk:pto},
and the set $\cSR$ of reachable component sizes as in~\eqref{eq:cSr}. 
There exist a constant $\eps_0>0$ and non-negative analytic functions $\theta(t)$ and $\psi(t)$ on $I=[\tc-\eps_0,\tc+\eps_0]$ such that 
\begin{equation}\label{eq:rhok}
 \rho_k(t) = (1+O(1/k)) \indic{k \in \cSR} k^{-3/2} \theta(t) e^{-\psi(t) k},
\end{equation}
uniformly in $k\ge 1$ and $t\in I$, with $\theta(\tc),\psi''(\tc)>0$ 
and $\psi(\tc)=\psi'(\tc)=0$. 
Furthermore, $\rho_k(t) \in [0,1]$ and $\sum_{k \ge 1}\rho_k(t) + \rho(t)=1$ for $t \in I$ and $\rho$ as in
~\eqref{eq:L1:pto} and Theorem~\ref{thm:L1rho}.
\end{theorem}
\begin{remark}\label{rem:rhok}
It follows immediately from~\eqref{eq:rhok} that there is a constant $B >0$ such that ${\sum_{j \ge k} \rho_j(\tc)} = {(1+O(1/k)) B k^{-1/2}}$ for all $k \ge 1$.
\end{remark}
\begin{remark}\label{rem:rhok:BP2}
The proof shows that for $t \in I$ we have $\rho_k(t)=\Pr(|\bp_t|=k)$ for a certain branching process~$\bp_t$ defined in Section~\ref{sec:BPI}. 
\end{remark}
The above multiplicative error $1+O(1/k)$ is best possible for Erd{\H o}s--R{\'e}nyi (where $\tc=1/2$, $\psi(t)=-\log(2te^{1-2t})=2t-1-\log(2t)$ and $\theta(\tc)=1/\sqrt{2\pi}$, so that $\psi(\tc\pm\eps) \sim 2 \eps^2$ and $\theta(\tc\pm\eps) \sim 1/\sqrt{2\pi}$ as~$\eps \to 0$).
Moreover, the detailed asymptotics of~\eqref{eq:rhok} resolves conjectures of Kang, Perkins and Spencer~\cite{KPSPC} and Drmota, Kang and Panagiotou~\cite{DKP}.
The indicator $\indic{k \in \cSR}$ may look somewhat puzzling; its presence is due to the generality of $\ell$-vertex rules -- see Remark~\ref{rem:defs} and Section~\ref{sec:period}. In the Achlioptas process case we have $\cSR=\NN^+$, 
i.e., all component sizes are possible, and so the indicator may be omitted.

Although~\eqref{eq:rhok} is very satisfactory for the `idealized' component size distribution $(\rho_k)_{k \ge 1}$, 
we cannot simply combine it with~\eqref{eq:Nk:pto} to obtain the results we would like for the component size distribution $(N_k)_{k \ge 1}$ of the random graph process $(G^{\cR}_{n,i})_{i \ge 0}$, which is of course our main object of interest. 
The problem is that~\eqref{eq:Nk:pto} only applies for $k=O(1)$ and fixed $t = \tc \pm \eps$, whereas we would like to consider~$k \to \infty$ and~$\eps \to 0$. 
In other words, we would like variants of~\eqref{eq:Nk:pto} which allow us (i)~to study large component sizes with $k = k(n) \to \infty$, and (ii)~to `zoom in' on the critical $\tc$, i.e., study $t=\tc\pm \eps$ with $\eps=\eps(n) \to 0$. 
The next theorem 
accommodates both features: it shows that $N_k(i) \sim \rho_k(i/n) n$ holds for a wide range of sizes~$k$ and steps~$i$.
Note that 
there is some  $\gamma = \gamma(\beta,a, \eps_0)>0$ such that the assumptions on $k$ below, and hence
\eqref{eq:Nk}--\eqref{eq:Ngek}, hold for any $1 \le k \le \gamma \log n$, 
with the allowed range
of $k$ increasing as $\eps=\eps(n) \to 0$. (Aiming at simplicity, here we have not tried to optimize the range;
see also Theorem~\ref{thm:NkC}, Corollary~\ref{cor:Nk} and Section~\ref{sec:mom:large}. Note that we allow $\eps=0$.)
\begin{theorem}[Number of vertices in small components]%
\label{thm:Nk}
Let $\cR$ be a bounded-size $\ell$-vertex rule. 
Let the critical time $\tc>0$ be as in~\eqref{eq:def:tc}, and the functions $\rho,(\rho_k)_{k \ge 1}$ as in Theorem~\ref{thm:L1rho} and~\eqref{eq:Nk:pto}, 
and define $a:=\psi''(\tc)>0$ where~$\psi$ is as in Theorem~\ref{thm:rhok}. 
There is a constant $\eps_0>0$ such that, with probability $1-O(n^{-99})$, the following inequalities hold for all steps $(\tc-\eps_0) n \le i \le (\tc+\eps_0) n$ and sizes $1 \le k \le n^{1/10}$ such that
$\eps=i/n-\tc$ satisfies $10a\eps^2 k \le \log n$: 
\begin{align}
\label{eq:Nk}
\frac{N_k(i)}{n} & \in \bigl[ \bigl(1 \pm n^{-1/30}/k\bigr) \cdot \rho_k(i/n) \bigr], \\
\label{eq:Ngek}
\frac{N_{\ge k}(i)}{n} & \in \Bigl[ \bigl(1 \pm n^{-1/30}/k\bigr) \cdot \Bigl(\sum_{j \ge k}\rho_j(i/n) + \rho(i/n)\Bigr) \Bigr],
\end{align}
where $N_{\ge k}(i) := \sum_{k'\ge k} N_{k'}(i)$, and $[a\pm b]$ denotes the interval $[a-b,a+b]$.
\end{theorem}
\begin{remark}\label{rem:thm:Nk}
For comparison with \eqref{eq:intro:smallcpt} in Theorem~\ref{thm:intro} 
note that if we set $t=i/n$ and $\eps=|\tc-t|$, then $\eps^3 k = o(1)$ implies $\psi(t)k = a\eps^2k + o(1)$ in~\eqref{eq:rhok}.
Similarly, $\eps^2 k = o(1)$ implies $\psi(t)k = o(1)$. 
\end{remark}
The multiplicative $1+o(1/k)$
error term in~\eqref{eq:Nk} allows for very precise asymptotic results in combination with~\eqref{eq:rhok}.
Indeed, whp, for all steps $(\tc-\eps_0) n \le i \le (\tc+\eps_0) n$ and sizes $1 \le k \le n^{1/10}$ 
satisfying $10a(\tc -i/n)^2k \le \log n$, we 
have
\[
N_k(i) = (1+O(1/k)) \indic{k \in \cSR} k^{-3/2} \theta(i/n) e^{-\psi(i/n) k} n ,
\]
where in this and all similar formulae, the implicit constant is uniform over the choice of $i=i(n)$ and $k=k(n)$.
Furthermore, combining~\eqref{eq:Ngek} with Remark~\ref{rem:rhok},  we see that, whp, for all $1 \le k \le n^{1/10}$ we have
\begin{equation}\label{eq:Ngek:tc}
N_{\ge k}(\tc n) = (1+O(1/k)) B k^{-1/2} n .
\end{equation}
Thus, at criticalilty we have \emph{polynomial decay} of the tail of the component size distribution, which is a prominent hallmark of the critical window. 
For bounded-size rules the $Bk^{-1/2}n$ asymptotics of~\eqref{eq:Ngek:tc} answers a question of Spencer and Wormald from 2001, see~\cite{JSP}.

\subsection{Susceptibility}\label{sec:intro:Sj}
The susceptibility~$S_2(tn)$ is a key statistic of the phase transition, which has been widely studied in a range of random graph models (see, e.g.~\cite{BCvdHSS2005,JL,J2010,JR2012,JW2016}).
For example, in classical percolation theory the critical density coincides with the point where (the infinite analogue of) the susceptibility diverges, and in the Erd{\H o}s--R{\'e}nyi process it is folklore that 
for $t < \tc^{\mathrm{ER}}=1/2$ we have 
\begin{equation}\label{eq:intro:s2ER}
S_2(G^{\mathrm{ER}}_{n,tn}) \pto \frac{1}{1-2t}.
\end{equation}
More importantly, in the context of bounded-size Achlioptas processes
the location~$\tc$ of the phase transition is \emph{determined} by the critical time where the susceptibility diverges, see~\cite{SW,BK,RWapsubcr}. 
This characterization is somewhat intuitive, since $S_2(tn)=S_2(G^{\cR}_{n,tn})$ is the expected size of the component containing a randomly chosen vertex from $G^{\cR}_{n,tn}$, see~\eqref{eq:def:Sr}. 
Of course, since $L_1(G)^2/n \le S_2(G)\le L_1(G)$, bounds on one of~$L_1(i)$ and~$S_2(i)$ imply bounds on the other.
(For example, $S_2(i) = O(1)$ implies $L_1(i) = O(\sqrt{n}) = o(n)$.) However, one only obtains weak results this way;
proving that whp $L_1(tn)=\Omega(n)$ after the point at which $S_2(tn)$ blows up is far from trivial.

Turning to the susceptibility in 
bounded-size rules, using the differential equation method~\cite{DEM,DEM99,DEMLW} and ideas from~\cite{SW,JS,RWapsubcr} it is nowadays routine to prove that for each $t \in [0,\tc)$ and $r \ge 2$ we have
\begin{equation}\label{eq:Sj:pto}
S_r(tn) \pto s_r(t) ,
\end{equation}
where the functions $s_r=s_r^{\cR}:[0,\tc) \to [1,\infty)$ are the unique solution of a certain system of differential equations (involving also $\rho_1,\ldots,\rho_K$), 
with $\lim_{t \nearrow \tc} s_r(t) = \infty$.
(Recall that $S_{r+1}$ denotes the $r$th moment of the size of the component containing a random vertex.)
Motivated by `critical exponents' from percolation theory and statical physics, the focus has thus shifted towards the finer behaviour of the susceptibility, i.e., the question at what rate $s_r(\tc-\eps)$ blows up as $\eps \searrow 0$ (in the Erd{\H o}s--R{\'e}nyi case we have $s_2(\tc-\eps) \sim (2\eps)^{-1}$, see~\eqref{eq:intro:s2ER} and~\cite{RGD,JL}). 
Using asymptotic analysis of differential equations, Janson and Spencer~\cite{JS} determined the scaling behaviour of $s_2, s_3$ and $s_4$ for the Bohman--Frieze rule. 
For $s_2$ and $s_3$ their argument was generalized by Bhamidi, Budhiraja and Wang~\cite{BBW12b} to all bounded-size rules. 
Based on branching process arguments, 
the next theorem establishes the asymptotic behaviour of $s_r$ for \emph{any}~$r \ge 2$, for the larger class of bounded-size $\ell$-vertex rules.
To avoid clutter below, we adopt the convention that the double factorial $x!!=\prod_{0 \le j < \ceil{x/2}}(x-2j)$
is equal to $1$ for $x \le 0$.
Recall from Remark~\ref{rem:defs} that for an Achlioptas processes~$\per=1$.
\begin{theorem}[Idealized susceptibility asymptotics]%
\label{thm:sjfkt}
Let $\cR$ be a bounded-size $\ell$-vertex rule.
Let the critical time $\tc>0$ and the functions $(s_r)_{r \ge 2}$ be as in~\eqref{eq:def:tc} and~\eqref{eq:Sj:pto}. 
Let 
\begin{equation}\label{eq:sjfkt:Dr}
B_r := (2r-5)!! \cdot \sqrt{2\pi}\theta(\tc)/\per \cdot \psi''(\tc)^{-r+3/2} , 
\end{equation}
where $\per \ge 1$ is defined in Section~\ref{sec:period}, and the functions $\theta(t)$ and $\psi(t)$ are as in Theorem~\ref{thm:rhok}.
Then there exists a constant $\eps_0>0$ such that, for all $r \ge 2$ and $\eps \in (0,\eps_0)$, we have $B_r>0$ and
\begin{equation}\label{eq:sjfkt}
s_r(\tc-\eps) = (1+O(\eps))  B_r \eps^{-2r+3}.
\end{equation}
\end{theorem}
In the language of mathematical physics~\eqref{eq:Sj:pto} and~\eqref{eq:sjfkt} loosely say that, as $\eps \searrow 0$, all bounded-size rules have the same susceptibility-related `critical exponents' as the Erd{\H o}s--R{\'e}nyi process
(where the constant is $B_r=(2r-5)!!2^{-2r+3}$, since $\theta(\tc) = 1/\sqrt{2\pi}$ and $\psi''(\tc)= 4$ by folklore results).
\begin{remark}\label{rem:sjfkt:BP}
The proof shows that for $t \in [\tc-\eps_0,\tc)$ we have $s_r(t)=\E |\bp_t|^{r-1}$ for a certain branching process~$\bp_t$ defined in Section~\ref{sec:BPI}.
\end{remark}

Next we `zoom in' on the critical point~$\tc$, i.e., discuss the behaviour of the susceptibility $S_r(\tc n-\eps n)$ when $\eps=\eps(n) \to 0$. 
Here the subcritical phase in the Erd{\H o}s--R{\'e}nyi case was resolved by Janson and Luczak~\cite{JL}, using martingale arguments, differential equations and correlation inequalities.  
For bounded-size rules Bhamidi, Budhiraja and Wang~\cite{BBW12b,BBW12a} used martingales arguments and the differential equation method
to prove results covering only part of the subcritical phase. 
In particular, for $i=\tc n-\eps n$ with $\eps=\eps(n) \to 0$ their results apply only to $S_2(i)$ and $S_3(i)$
and only
in the restricted range $\eps \ge n^{-1/5}$. 
Using very different methods, 
the next theorem resolves the scaling behaviour of the susceptibility~$S_r(i)$ in the entire subcritical phase. 
In particular, our result applies for any $r \ge 2$ all the way up to the critical window, i.e., we only assume $\eps^3 n \to \infty$. 
Note that $\gamma_{r,n,\eps} = o(1)$ when $\eps=o(1)$, so~\eqref{eq:Sjsub:uniform} intuitively states that whp $S_r(\tc n-\eps n) \approx B_r \eps^{-2r+3}$. 
%
\begin{theorem}[Subcritical susceptibility]%
\label{thm:Sj}
Let $\cR$ be a bounded-size $\ell$-vertex rule with critical time $\tc>0$ as in~\eqref{eq:def:tc},
and define $B_r>0$ as in~\eqref{eq:sjfkt:Dr}.
There are positive constants $\eps_0>0$ and $(A_r)_{r \ge 2}$ such that the following holds for any function $\omega=\omega(n)$ with $\omega \to \infty$ as $n \to \infty$. 
For any integer $r \ge 2$, whp
\begin{equation}\label{eq:Sjsub:uniform}
S_r(i) \in [(1 \pm \gamma_{r,n,\eps}) B_r \eps^{-2r+3}]
\end{equation}
holds in all steps $i=i(n) \ge 0$ such that $\eps=\tc-i/n$ satisfies $\eps^3 n \ge \omega$ and $\eps \le \eps_0$,
where $\gamma_{r,n,\eps}:=A_r\bigl(\eps + (\eps^3n)^{-1/4}\bigr)$ 
and where $[a\pm b]$ denotes the interval $[a-b,a+b]$. 
\end{theorem}
The assumption $r \ge 2$ cannot be relaxed, since $S_1(i) =1$ holds deterministically, cf.~\eqref{eq:def:Sr}.
In~\eqref{eq:Sjsub:uniform} we have not tried to optimize the error term for $\eps=\Theta(1)$, since our main interest concerns the $\eps \searrow 0$ behaviour. 
The supercritical scaling of the susceptibility is less informative and interesting, since $S_2(i)$ is typically dominated by the contribution from the largest component. 
In particular, for $i=\tc n + \eps n$ with $\eps^3n \to \infty$ we believe 
that whp $S_2(i) \sim L_1(i)^2/n$ for any bounded-size rule (for fixed $\eps \in (0,\eps_0)$ this follows from Theorem~\ref{thm:L1}), but we have not investigated this.

\newoddpage

\section{Proof overview}\label{sec:overview}
In this section we give an overview of the proof, with an emphasis on the structure of the argument.
Loosely speaking, one of the key difficulties is that there are non-trivial dependencies between the choices in different rounds. 
To illustrate this, let us do the following thought experiment. Suppose that we change the vertices offered to the rule at one step, and as a consequence, the rule adds a different edge to the graph.
This results in a graph with different component sizes. Hence,
whenever the process samples vertices from these components in subsequent steps, the rule is presented with different component sizes. 
This may alter the decision of the rule, and hence the edge added, 
which can change further subsequent decisions, and so on. 
In other words, changes can propagate throughout the evolution of the process, which makes the analysis challenging.

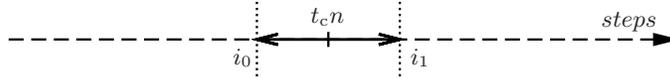
\begin{figure}[t]
\centering
  \setlength{\unitlength}{1bp}%
  \begin{picture}(262.54, 39.69)(0,0)
  \put(0,0){\includegraphics{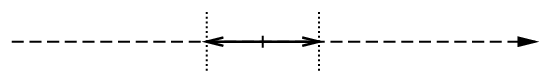}}
  \put(239.37,24.66){\fontsize{8.54}{10.24}\selectfont \makebox[0pt]{$steps$}}
  \put(125.86,26.08){\fontsize{8.54}{10.24}\selectfont \makebox[0pt]{$t_{\mathrm c} n$}}
  \put(97.51,10.49){\fontsize{8.54}{10.24}\selectfont \makebox[0pt][r]{$i_0$}}
  \put(157.04,10.49){\fontsize{8.54}{10.24}\selectfont $i_1$}
  \end{picture}%
	\caption{\label{fig:po1} The basic proof setup. 
We condition on the component structure of the graph after $i_0=(\tc-\sigma)n$ steps, and then reveal information about steps $i_0, \ldots, i_1=(\tc+\eps)n$ via a two-round exposure. 
The crux will be that the distribution of the second exposure round is extremely well behaved (consisting of many independent random choices), which eventually allows us to analyze the component size distribution of the resulting graph~$G_i$ in step $i_0 < i \le i_1$ via branching process methods.}
\end{figure}

For bounded-size rules we overcome this difficulty via the following high-level proof strategy. 
First, we track the evolution of the entire component size distribution during the initial $i_0=(\tc-\sigma)n$ steps, where~$\sigma > 0$ is a small constant. 
Second, using the graph after $i_0$ steps as an anchor, for $i_1=(\tc+\eps)n$ we reveal information about the steps $i_0, \ldots, i_1$ via a two-round exposure argument (not the classical multi-round exposure used in random graph theory). 
We engineer this two-round exposure in a way that eventually allows us to analyze the component size distribution in step $i_0 \le i \le i_1$ via a neighbourhood exploration process which closely mimics a branching process.
Intuitively, this allows us to reduce most questions about the component size distribution to questions about certain branching processes. 
These branching processes are not of a standard form, but we are nevertheless able to analyze them (with some technical effort). 
This close coupling with a branching process is what allows us to obtain such precise results. In this argument
the restriction to bounded-size rules is crucial, see Sections~\ref{sec:exp} and~\ref{sec:VS}.

In the following subsections we further expand on the above ideas, still ignoring a number of technical details and difficulties. 
In Section~\ref{sec:po:exp} we outline our setup and the two-round exposure argument. 
Next, in Section~\ref{sec:po:nebp} we explain the analysis of the component size distribution via exploration and branching processes. 
Finally, in Section~\ref{sec:po:main} we turn to the key statistics $L_1(i)$, $N_k(i)$ and $S_r(i)$, and briefly discuss how we eventually adapt approaches used to study the Erd\H os--R\'enyi model to bounded-size Achlioptas processes.

\subsection{Setup and two-round exposure}\label{sec:po:exp}
In this subsection we discuss the main ideas used in our two-round exposure; see Section~\ref{sec:exp} for the technical details. 
Throughout we fix a bounded-size $\ell$-vertex rule $\cR$ with cut-off~$K$ (as defined in Section~\ref{sec:intro:def}).
Using the methods of~\cite{RWapsubcr}, we start by tracking the evolution of the entire component size distribution up to step~${i_0=(\tc-\sigma)n}$.  
More precisely, we show that the numbers~$N_{k}(i_0)$ of vertices in components of size~$k$ can be approximated by deterministic functions (see Theorem~\ref{thm:init} and Lemma~\ref{lem:Nk2:t}).

\begin{figure}[t]
\centering
  \setlength{\unitlength}{1bp}%
  \begin{picture}(278.80, 115.05)(0,0)
  \put(0,0){\includegraphics{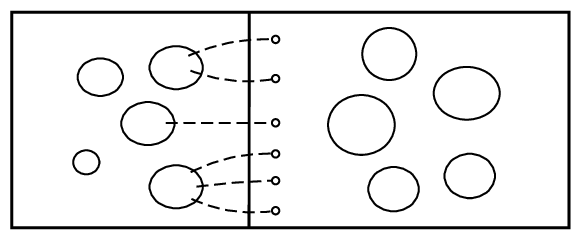}}
  \put(19.27,96.65){\fontsize{11.38}{13.66}\selectfont \makebox[0pt]{$V_S$}}
  \put(258.92,96.65){\fontsize{11.38}{13.66}\selectfont \makebox[0pt]{$V_L$}}
  \end{picture}%
	\caption{\label{fig:po2} Example of the `partial graph' used for the second exposure round of the graph~$G_i$ in step~$i_0 \le i \le i_1$. 
After the first exposure round we have revealed a certain subgraph $H_i$ of $G_i$, together with the $V_S$--endvertices of all $V_S$--$V_L$ edges in $G_i\setminus H_i$ (the endvertices in~$V_L$ are still uniformly random), and the number of $V_L$--$V_L$ edges of $G_i\setminus H_i$ (these edges are still uniform). 
We obtain the graph $G_i$ by (a)~connecting the undetermined $V_S$--$V_L$ edges to random vertices in~$V_L$, and (b)~adding the correct number of random $V_L$--$V_L$ edges. 
The point of this description is that it involves many independent random choices, so it allows us to analyze the component size distribution of the resulting graph~$G_i$ using branching process techniques.} 
\end{figure}

Conditioning on the graph $G_{i_0}=G^{\cR}_{n,i_0}$ after $i_0$ steps, i.e., regarding it as given, 
we shall reveal information about steps $i_0+1, \ldots, i_1=(\tc+\sigma)n$ in two rounds.
We assume (as we may, since the variables~$N_{k}(i_0)$ are concentrated) that each $N_{k}(i_0)$ is close to its expectation. 
We partition the vertex set of~$G_{i_0}$ into~$V_S \cup V_L$, where~$V_S$ contains all vertices that \emph{in the graph $G_{i_0}$} are in components of size at most~$K$ (the labels~$S$ and~$L$ refer to `small' and `large' component sizes).
Note that in any later step $i \ge i_0$, since $G_i\supseteq G_{i_0}$, every vertex~$v \in V_L$ is in a component of $G_i=G^{\cR}_{n,i}$ with size larger than~$K$, i.e., with size~$\omega$ as far as the rule~$\cR$ is concerned.
Hence, when a vertex $v$ in $V_L$ is offered to $\cR$, in order to know the decision made by~$\cR$ we do
not need to know \emph{which} vertex in $V_L$ we are considering -- as far as~$\cR$ is concerned, all 
such vertices have the same component size $\omega$.
In our first exposure round we reveal everything about the vertices offered to~$\cR$ in all steps $i_0< i\le i_1$
except that whenever a vertex in $V_L$ is chosen, we do not reveal which vertex it is; as just observed, this
information tells us what decisions $\cR$ will make.
This allows us to track: (i)~the edges added inside $V_S$, (ii)~the $V_S$--endvertices of the edges added connecting $V_S$ to $V_L$, and (iii)~the \emph{number} of edges added inside~$V_L$. 
(Formally this can be done via the differential equation method~\cite{DEM,DEM99,DEMLW} and branching process techniques, see Section~\ref{sec:DEM}--\ref{sec:BPA}; note that (i)--(ii) track the evolution of the `$V_S$-graph' beyond the critical~$\tc$.)
After this first exposure round we have revealed a subgraph $H_i$ of $G_i$ (called the `partial graph' in Figure~\ref{fig:po2}), consisting of all edges in~$G_{i_0}$, together with all edges
in steps between~$i_0$ and~$i$ with both ends in~$V_S$. Furthermore, we know that~$G_i$ consists of~$H_i$ with
certain edges added: a known number of $V_S$--$V_L$ edges whose endpoints in $V_S$ are known, and a known number of $V_L$--$V_L$ edges.

In the second exposure round the vertices in $V_L$ (corresponding to~(ii) and~(iii) above) are now chosen independently and uniformly at random from~$V_L$; see the proof of Lemma~\ref{lem:exp} for the full details. 
Hence, after conditioning on the outcome of the first exposure round, the construction of $G_i$ from the `partial graph'~$H_i$ described above has a very simple form (see Figure~\ref{fig:po2} and Lemma~\ref{lem:cond}). 
Indeed, for each $V_S$--$V_L$ edge the so-far unknown $V_L$--endpoint is replaced with a uniformly chosen random vertex from~$V_L$. 
Furthermore, we add a known number of uniformly chosen random edges to~$V_L$. 
This setup, consisting of many independent uniform random choices, is ideal for neighbourhood exploration and branching process techniques.

\subsection{Component size distribution}\label{sec:po:nebp} 
To get a handle on the component size distribution of the graph~$G_i=G^{\cR}_{n,i}$ after $i_0 < i \le i_1$ steps, we use neighbourhood exploration arguments to analyze the second exposure round described above.
As usual, we start with a random vertex $v \in V_S \cup V_L$, and iteratively explore its neighbourhoods.
Suppose for the moment that $v\in V_L$.
Recalling the construction of $G_i$ from the partial graph $H_i$, any vertex $w \in V_L$ has neighbours in~$V_L$ and~$V_S$, which arise (a)~via random $V_L$--$V_L$ edges and (b)~via $V_S$--$V_L$ edges with random $V_L$--endpoints. 
Furthermore, each of the adjacent $V_S$--components found in~(b) potentially yields further $V_L$--neighbours via $V_S$--$V_L$ edges.
Repeating this exploration iteratively, we eventually uncover 
the entire component of $G_i$ which contains the initial vertex~$v$. 
Treating (a) and (b) together as a single step, each time we `explore' a vertex in $V_L$ we
reach a random number of new vertices in $V_L$, picking up a random number of vertices in $V_S$ along the way.
As long as we have not used up too many vertices, the sequence of pairs $(\tY_j,\tZ_j)$ giving the number 
of $V_L$ and $V_S$ vertices found in the $j$th step will be close to a sequence of independent copies of some distribution~$(Y_t,Z_t)$
that depends on the `time'~$t=i/n$.
We thus expect the neighbourhood exploration process to closely resemble a two-type branching process~$\bp_t$
with offspring distribution~$(Y_t,Z_t)$, corresponding to $V_L$ and $V_S$ vertices. 
In this branching process,
vertices in $V_S$ have no children (they are counted `in the middle' of a step). Of course, we need to modify the start
of the process to account for the possibility that the initial vertex is in $V_S$.
Writing $\bp_t$ for the (final modified) branching processes, it should seem plausible that
the expected numbers of vertices in components of size~$k$ and in components of size at least~$k$ approximately satisfy
\begin{equation}\label{eq:po:Nk:Ngek}
\E N_{k}(tn) \approx \Pr(|\bp_t|=k) n \qquad \text{and} \qquad \E N_{\ge k}(tn) \approx \Pr(|\bp_t| \ge k) n,
\end{equation}
ignoring technicalities (see Sections~\ref{sec:nep}--\ref{sec:dom} and~\ref{sec:mom:large} for the details).

In view of~\eqref{eq:po:Nk:Ngek}, we need to understand the behaviour of the branching process $\bp_t$. 
Here one difficulty is that we only have very limited explicit knowledge about the offspring distribution $(Y_t,Z_t)$. 
To partially remedy this, we prove that several key variables determined by the first exposure round have exponential tails (see, e.g., inequalities \eqref{eq:Nk:t0:tail}--\eqref{eq:rhok:t0:tail} and \eqref{eq:Qkr:tail}--\eqref{eq:qkr:tail} of Theorems~\ref{thm:init} and~\ref{thm:Qkr}).  
Combining calculus with ODE and PDE techniques (the Cauchy--Kovalevskaya Theorem; see Appendix~\ref{apx:CK}),
this allows us to eventually show that the probability generating function 
\begin{equation}\label{eq:po:MGF}
 \gf(t,\alpha,\beta) := \E\bigl(\alpha^{Y_t}\beta^{Z_t}\bigr)
\end{equation}
is extremely well-behaved, i.e., \emph{(real) analytic} in a neighbourhood of $(\tc,1,1)$, say (see Sections~\ref{sec:DEM}--\ref{sec:AP} and~\ref{sec:BPO}). 
In a companion paper~\cite{BPpaper} written with Svante Janson 
(see also Section~\ref{sec:bpresults} and Appendix~\ref{sec:BP}), we show that the probability of~$\bp_t$ generating~$k$ particles is roughly of the form
\begin{equation}\label{eq:po:bpk:tail}
\Pr(|\bp_{t}|=k) \approx A k^{-3/2} e^{-\psi(t) k} \qquad \text{with} \qquad \psi(\tc \pm \eps) \approx a \eps^2 .
\end{equation}
Turning to the survival probability $\Pr(|\bp_t|=\infty)$, for this the $V_S$--vertices counted by $Z_t$ are irrelevant (since these do not have children, the only possible exception being the first vertex). 
Combining a detailed analysis of~$Y_t$
with standard methods for single-type branching processes, we eventually show that $\E Y_{\tc}=1$, and (in~\cite{BPpaper}) that the survival probability of $\bp_t$ is roughly of the form 
\begin{equation}\label{eq:po:bp:survival}
\Pr(|\bp_t|=\infty) \approx \begin{cases}
	0 , & ~~\text{if $t \le \tc$}, \\
		c \eps  , & ~~\text{if $t=\tc+\eps$},
	\end{cases}
\end{equation}
for small $\eps$ (see Sections~\ref{sec:BPO}--\ref{sec:bpresults}, Appendix~\ref{sec:BP} and~\cite{BPpaper} for the details).

In the above discussion we have ignored a number of technical issues. 
For example, in certain parts of the analysis we need to incorporate various approximation errors: simple coupling arguments would, e.g., break down for large component sizes. (Such errors are not an artifact of our analysis. For example, the number of isolated vertices changes with probability $\Theta(1)$ in each step, so after $\Theta(n)$ steps we indeed expect random fluctuations of order $\Theta(\sqrt{n})$.) 
To deal with such errors we shall use (somewhat involved) 
domination arguments, exploiting that the exploration process usually finds `typical' subsets of the underlying graph (see Section~\ref{sec:dom:dom}). 
Perhaps surprisingly, this allows us to employ dominating distributions $(Y^{\pm}_{t},Z^{\pm}_{t})$
that have probability generating functions which are extremely close to the `ideal' one in~\eqref{eq:po:MGF}:
the dominating branching processes are effectively indistinguishable from the actual exploration process.
In this context one of our main technical contributions is that we are able to carry out (with \emph{uniform} error bounds) 
the point probability analysis~\eqref{eq:po:bpk:tail}
and the survival probability analysis~\eqref{eq:po:bp:survival} despite having only some `approximate information' about the underlying (family of) offspring distributions.
This is key for determining the asymptotic size of the largest component in the entire subcritical and supercritical phases.

\subsection{Outline proofs of the main results}\label{sec:po:main}
Using the setup (and technical preparation) outlined above, we prove our main results for $L_1(i)$, $N_k(i)$ and~$S_r(i)$ by adapting approaches that work for the classical Erd{\H o}s--R{\'e}nyi random graph. Of course, in this more complicated setup many technical details become more involved.
In this subsection we briefly outline the main high-level
ideas that are spread across 
Sections~\ref{sec:setup}--\ref{sec:proof}
(the actual arguments are complicated, for example, by the fact that parts of the branching process analysis rely on Poissonized variants of~$G_i$).

We start with the number $N_k(i)$ of vertices in components of size~$k$. 
After conditioning on the outcome of the first exposure round, we first use McDiarmid's bounded differences inequality~\cite{McDiarmid1989} to show that whp $N_k(i)$ is close to its expected value (here we exploit that the second exposure rounds consists of many independent random choices), and then approximate $\E N_k(i)$ via the branching process results~\eqref{eq:po:Nk:Ngek} and~\eqref{eq:po:bpk:tail}.
The full details of this approach are given in Sections~\ref{sec:BPI} and~\ref{sec:small}, 
and here we just mention one technical point: conditioning allows us to bring concentration inequalities into play, but we must then show that (except for unlikely `atypical' outcomes) conditioning on the first exposure round does not substantially shift the expected value of $N_k(i)$.

Next we turn to the size $L_1(i)$ of the largest component in the subcritical and supercritical phases, i.e., where the step $i=\tc n \pm \eps n$ satisfies $\eps^3 n \to \infty$.  
Intuitively, our arguments hinge on the fact that the expected component size distribution has an exponential cutoff after size $\eps^{-2}=\Theta(\psi(\tc\pm\eps)^{-1})$, see~\eqref{eq:po:Nk:Ngek} and~\eqref{eq:po:bpk:tail}. 
Indeed,~\eqref{eq:po:bpk:tail} and $\int_{k}^{\infty}e^{-a x}\dx = \Theta(a^{-1}e^{-a k})$
suggest that for $k \gg \eps^{-2}$ we roughly have 
\begin{equation}\label{eq:po:bp:uppertail}
\Pr(k \le |\bp_{\tc \pm \eps}| < \infty) = \sum_{j \ge k} \Pr(|\bp_{\tc \pm \eps}|=j) \approx A \sum_{j \ge k} j^{-3/2} e^{-\psi(\tc\pm\eps) j } = \Theta(\eps^{-2} k^{-3/2}) e^{-\psi(\tc\pm\eps) k } .
\end{equation}
In the subcritical phase we have $\Pr(|\bp_{\tc -\eps}| = \infty)=0$ by~\eqref{eq:po:bp:survival}. Using~\eqref{eq:po:Nk:Ngek} we thus expect that for $k \gg \eps^{-2}$ we have
\begin{equation*}
\E N_{\ge k}(\tc n -\eps n) \approx \Pr(|\bp_{\tc -\eps}| \ge k) n = \Pr(k \le |\bp_{\tc - \eps}| < \infty) n \approx \Theta(\eps^{-2} k^{-3/2}) e^{-\psi(\tc-\eps) k } n.
\end{equation*}
By considering which sizes~$k$ satisfy $\E N_{\ge k}(\tc n -\eps n) = \Theta(k)$, this suggests that~whp 
\begin{equation*}
L_1(\tc n-\eps n) \approx \psi(\tc-\eps)^{-1}\log(\eps^3 n).
\end{equation*}
We make this rigorous via the first- and second-moment methods, using a van den Berg--Kesten (BK)-inequality 
like argument for estimating the variance (see Sections~\ref{sec:varsub}, \ref{sec:mom:large} and~\ref{sec:L1:sub} for the details).
Turning to the more interesting supercritical phase, where $i=\tc n + \eps n$, note that the right hand side of~\eqref{eq:po:bp:uppertail} is $o(\eps)$ for $k \gg \eps^{-2} =\Theta(\psi(\tc+\eps)^{-1})$,
and that $\Pr(|\bp_{\tc +\eps}| = \infty) \approx c \eps $ by~\eqref{eq:po:bp:survival}. 
Using~\eqref{eq:po:Nk:Ngek} we thus expect that for $k \gg \eps^{-2}$ we have
\begin{equation*}
\E N_{\ge k}(\tc n +\eps n) \approx \Pr(|\bp_{\tc +\eps}| \ge k) n \approx \Pr(|\bp_{\tc +\eps}| = \infty) n \approx c \eps n.
\end{equation*}
Applying the first- and second-moment methods we then show that whp $N_{\ge \Lambda}(i) \approx \E N_{\ge \Lambda}(i)$ for suitable $\eps^{-2} \ll \Lambda \ll \eps n$, adapting a `typical exploration' argument of Bollob{\'a}s and Riordan~\cite{BR2012} for bounding the variance (see Sections~\ref{sec:SME}, \ref{sec:mom:large} and~\ref{sec:L1:super} for the details).
Mimicking the Erd{\H o}s--R\'enyi sprinkling argument from~\cite{ER1960},
we then show that whp most of these size~$\ge \Lambda$ components quickly join, i.e., form one big component in $o(\eps n)$ steps (see Sections~\ref{sec:SP} and~\ref{sec:L1:super}).  
Using continuity of $\Pr(|\bp_{\tc +\eps}| = \infty)$, this heuristically suggests that~whp 
\begin{equation*}
L_1(\tc n+\eps n) \approx \Pr(|\bp_{\tc +\eps}| = \infty) n \approx c \eps n ,
\end{equation*}
ignoring technicalities (see Section~\ref{sec:L1:super} for the details).

For the subcritical susceptibility $S_r(\tc n-\eps n)$ with $\eps^3 n \to \infty$ we proceed similarly. 
Indeed, substituting the estimates~\eqref{eq:po:Nk:Ngek} and~\eqref{eq:po:bpk:tail} into the definition~\eqref{eq:def:Sr} of $S_r(i)$, since $\psi(\tc-\eps)=\Theta(\eps^2)$ we expect that for~${r \ge 2}$ we have
\begin{equation}\label{eq:po:Sr:E}
\E S_r(\tc n -\eps n) 
\approx A \sum_{k \ge 1} k^{r-5/2} e^{-\psi(\tc-\eps) k } = \Theta\Bigl(\bigl(\psi(\tc-\eps)^{-1}\bigr)^{r-3/2}\Bigr)  = \Theta( \eps^{-2r+3}).
\end{equation}
In fact, comparing the sum with an integral, we eventually find that $\E S_r(\tc n -\eps n) \approx B_r \eps^{-2r+3}$ for small~$\eps$ (see Lemma~\ref{lem:SrE:sub}).
Applying the second-moment method we then show that whp $S_r(i) \approx \E S_r(i)$, using a BK-inequality like argument for bounding the variance (see Sections~\ref{sec:varsub}, \ref{sec:mom:sus} and~\ref{sec:Sj} for the details).

Finally, one non-standard feature of our arguments is that we can prove concentration of the size of the largest component in \emph{every} step outside of the critical window (cf.~Theorem~\ref{thm:L1}).
The idea is to fix a sequence~$(m_j)$ of not-too-many steps that are close enough together that we expect
\begin{equation}\label{eq:po:L1:aux}
L_1(m_j) \approx L_1(m_{j+1}). 
\end{equation}
Since there are not too many steps in the sequence, we can show that whp $L_1(m_j)$ is close to its expected value
for every step $m_j$ in the sequence.
By monotonicity, in all intermediate steps $m_j \le i \le m_{j+1}$ we have 
\begin{equation*}
L_1(m_j) \le L_1(i) \le L_1(m_{j+1}), 
\end{equation*}
which together with~\eqref{eq:po:L1:aux} establishes the desired concentration (related arguments are sometimes implicitly used in the context of the differential equation method).
As we shall see in Section~\ref{sec:L1}, the choice of the step sizes $m_{j+1}-m_{j}$ requires some care, since we need to take a union bound over all auxiliary steps, but this idea can be made to work by proving sufficiently sharp error bounds in various intermediate estimates. 
A~similar proof strategy applies to the susceptibility~$S_r(i)$, which is also monotone (see Section~\ref{sec:Sj} for the~details).

\newoddpage

\section{Preparation and setup}\label{sec:setup}
In this section we formally introduce the proof setup, together with some preparatory results. 
Throughout we fix a bounded-size $\ell$-vertex rule $\cR$ with cut-off $K$, and 
study the graph $G_i=G^{\cR}_{n,i}$ after $i$ steps, where $i=tn$ with $t\approx\tc$. We refer to $t$ (or in general $i/n$) as `time'.
As noted in Remark~\ref{rem:rounding}, whenever we pass from a continuous parameter $t$ to a step number $i$, we round down to the nearest integer, taking $i=\floor{tn}$. We omit this in the notation, since such rounding does not change any of our formulae. This is because all relevant results are insensitive to changing $t$ by $O(1/n)$, or changing $i$ by $O(1)$. This is easy but tedious to check, so we omit the details, noting only that the relevant bounds are all based on the differential equations method, and such insensitivity to rounding is usual (and not always commented on) when it applies, relating to the individual discrete steps begin `small' at the level where the continuous approximation holds.

As discussed in Section~\ref{sec:overview},
we stop the process after the first $i_0 \approx (\tc-\sigma)n$ steps, where $\sigma>0$ is a small constant, and then analyze the evolution of the component structure from step $i_0$ to step $tn$ via a two-round exposure argument. 
The main goals of this section are to formally introduce the two-round exposure, and to relate the second round of the exposure to a random graph model which is easier to analyze. 

Turning to the details, for concreteness let
\begin{equation}\label{def:sigma}
\sigma := \min\left\{\frac{1}{2\ell^2(K+1)}, \: \frac{\tc}{3}\right\} . 
\end{equation}
Set
\begin{equation}\label{def:t0t1t2}
t_0:=\tc-\sigma \qquad \text{and} \qquad t_1 := \tc + \sigma ,
\end{equation}
and 
\begin{equation}\label{def:i0i1i2}
i_0:= t_0 n \qquad \text{and} \qquad i_1 :=  t_1 n ,
\end{equation}
omitting from now on the irrelevant (see above) rounding to integers.
After $i_0$ steps we partition the vertex set into~$V_S$ and~$V_L$, where~$V_S$ contains all vertices in components
of $G_{i_0}$ having size at most $K$. Here the labels $S$ and $L$ correspond to `small' and `large' component sizes. This partition is defined at step~$i_0$, and does not change as our graph evolves.

In Section~\ref{sec:exp} we explain our two-round exposure argument in detail. Then,
in Section~\ref{sec:DEM}, we use the differential equation method to track the number of vertices in small components, as well as parts of the evolution of the graphs induced by~$V_S$ and $V_L$. 
Next, in Section~\ref{sec:BPA} we use branching process techniques to track the evolution of the $V_S$--graph in more detail, which also yields exponential tail bounds for certain key quantities. 
In Section~\ref{sec:AP} we then use PDE theory to show that an associated generating function is analytic.
In Section~\ref{sec:SP} we introduce a convenient form of the Erd{\H o}s--R\'enyi sprinkling argument.
Finally, in Section~\ref{sec:period} we define and study the set $\cSR$ of component sizes that the $\ell$-vertex rule~$\cR$
can produce, and the `period' $\per$ of the rule; for `edge-based' rules such as Achlioptas processes 
these technicalities are not needed.

\subsection{Two-round exposure and conditioning}\label{sec:exp}
Recall that we first condition on $G_{i_0}$. Our aim now is to analyze the steps~$i$ with $i_0 < i \le i_1$. 
Recall that $\vv_i = (v_{i,1}, \ldots, v_{i,\ell})$ denotes the uniformly random $\ell$-tuple of vertices offered to the rule
in step $i$.
Given $G_{i_0}$ we expose the information about steps $i_0<i\le i_1$ in two rounds. 
In the \emph{first exposure round} $\fE_1(i_0,i_1)$, for every step $i_0 < i \le i_1$ we (i) reveal which vertices of $\vv_i = (v_{i,1}, \ldots, v_{i,\ell})$ are in~$V_S$ and which in~$V_L$, and (ii) for those vertices $v_{i,j}$ in~$V_S$, we also reveal precisely which vertex $v_{i,j}$ is. 
In the \emph{second exposure round} $\fE_2(i_0,i_1)$, for every step $i_0 < i \le i_1$ we reveal the choices of all so-far unrevealed vertices in~$V_L$. 

The `added edges', i.e., edges of $G_i\setminus G_{i_0}$, are of three types:
$V_S$--$V_S$ edges (where both endvertices are in $V_S$), $V_L$--$V_L$ edges (where both endvertices are in $V_L$, but still unrevealed after the first exposure round) and $V_S$--$V_L$ edges (where the endvertex in $V_L$ is still unrevealed). 
To be pedantic, we formally mean pairs of vertices, allowing for loops and multiple edges; the term `edge' allows for a more natural and intuitive discussion of the arguments. 
The following lemma encapsulates the key properties of the two-round exposure discussed informally in Section~\ref{sec:po:exp}.
\begin{lemma}
\label{lem:exp}
Given $G_{i_0}$, the information revealed by the first exposure round $\fE_1(i_0,i_1)$ is enough to make all decisions of $\cR$, i.e., to determine for every $i_0 < i \le i_1$ the indices $j_1=j_1(i)$ and $j_2=j_2(i)$ such that $v_{i,j_1}$ and $v_{i,j_2}$ are joined by the rule~$\cR$. 
Furthermore, conditional on $G_{i_0}$ and on the first exposure round, all vertices revealed in the second exposure round $\fE_2(i_0,i_1)$ are chosen independently and uniformly at random from $V_L$. 
\end{lemma}
\begin{proof}
The claim concerning the second exposure round is immediate, since in each step the vertices $\vv_i=(v_{i,1}, \ldots, v_{i,\ell})$ are chosen independently and uniformly random.

Turning to the first exposure round, 
we now make the heuristic arguments of Section~\ref{sec:po:exp} rigorous.
For $i_0\le i\le i_1$ let $E_i$ be the set of edges of $G_i\setminus G_{i_0}$ with
both ends in $V_S$ (the edges added inside $V_S$), and let~$V_i$ be the (multi-)set of vertices of~$V_S$ in at least
one $V_S$--$V_L$ edge in $G_i\setminus G_{i_0}$ (the set of~$V_S$ endvertices of the added $V_S$--$V_L$ edges).
We claim that the information revealed in the first exposure round
determines~$E_i$ and~$V_i$ for each $i_0\le i\le i_1$. The proof is by induction on $i$;
of course, $E_{i_0}=V_{i_0}=\emptyset$. 

Suppose then that $i_0 < i \le i_1$ and that the claim holds for $i-1$. The information revealed
in the first exposure round determines which of the vertices $v_{i,1},\ldots,v_{i,\ell}$ 
are in $V_S$ as opposed to $V_L$, and precisely which vertices those in $V_S$ are.
Let $\vc_i=(c_{i,1}, \ldots, c_{i,\ell}) \in \{1, \ldots, K, \omega\}^\ell$ 
list the sizes of the components of $G_{i-1}$
containing $v_{i,1},\ldots,v_{i,\ell}$, with all
sizes larger than~$K$ replaced by $\omega$. We shall show that $\vc_i$ is determined
by the information revealed in the first exposure round. By the definition of a bounded-size rule,
$\vc_i$ determines the choice made by the rule $\cR$, i.e., the
indices $j_1=j_1(\vc_i)$ and $j_2=j_2(\vc_i)$ such that $v_{i,j_1}$ and $v_{i,j_2}$ are joined by~$\cR$ in step~$i$,
which is then enough to determine $E_i\setminus E_{i-1}$ and $V_i\setminus V_{i-1}$, completing the proof by induction. 

If $v_{i,j} \in V_L$, then $v_{i,j}$ is in a component of $G_{i-1}\supseteq G_{i_0}$ of size at least $K+1$,
so we know that $c_{i,j}=\omega$, even without knowing the particular choice of $v_{i,j}\in V_L$.
Suppose then that $v_{i,j}\in V_S$.
Since we know $G_{i_0}$ and $E_{i-1}$, we know the entire graph $G_{i-1}[V_S]$.
Furthermore, we know exactly which components of $G_{i-1}[V_S]$ are connected to $V_L$ in $G_{i-1}$, namely
those containing one or more vertices of $V_{i-1}$.
Let $C$ be the component of $G_{i-1}[V_S]$ containing $v_{i,j}$.
If $C$ is not connected to $V_L$ in $G_{i-1}$, then $C$ is also a component $G_{i-1}$, whose
size we know. If $C$ is connected to $V_L$
then in $G_{i-1}$ the component containing $C$ has size at least $K+1$, so~$c_{i,j}=\omega$. 
This shows that $c_{i,j}$ is indeed known in all cases, completing the proof.
\end{proof}
Intuitively speaking, after the first exposure round $\fE_1(i_0,i_1)$, for $i_0\le i\le i_1$ we are left
with a `marked' auxiliary graph~$H_i$, as described in Figure~\ref{fig:exp}. 
More precisely, for $i_0\le i\le i_1$ let $H_i$ be the `marked graph' obtained as follows.
Starting from $G_{i_0}$, (i) insert all $V_S$--$V_S$ edge added in steps $i_0<j\le i$,
and (ii) for each $V_S$--$V_L$ edge added in steps $i_0<j\le i$, add a `stub' or `half-edge' to
its endvertex in $V_S$.
Thus, in the (temporary) notation of the proof above, $H_i$ is formed from $G_i$ by adding
the edges in $E_i$ and stubs corresponding to the multiset $V_i$.
Each mark or stub represents an edge to a so-far unrevealed vertex in $V_L$, and 
a $V_S$--vertex can be incident to multiple stubs. 
For $i_0\le i\le i_1$ let $Q_{0,2}(i)$ denote the number of $V_L$--$V_L$ edges
(including loops and repeated edges)
added in total in steps $i_0 < j \le i$, so by definition
\begin{equation}\label{eq:U:init}
Q_{0,2}(i_0)=0.
\end{equation} 
By Lemma~\ref{lem:exp} the information revealed in the first exposure round $\fE_1(i_0,i_1)$
determines the graphs $(H_i)_{i_0\le i\le i_1}$ and the sequence $\bigl(Q_{0,2}(i)\bigr)_{i_0 \le i \le i_1}$.
Furthermore, in the second exposure round we may generate~$G_i$ from $H_i$ by
replacing each stub associated to a vertex $v\in V_S$ by an edge $vw$ to a vertex $w$
chosen independently and uniformly at random from~$V_L$, and
adding $Q_{0,2}(i)$ random $V_L$--$V_L$ edges to $H_i$, where
the endvertices are chosen independently and uniformly at random from~$V_L$. 
(To clarify: there is a version of this exposure argument which constructs all~$G_i$, $i_0\le i\le i_1$, simultaneously, but we shall not need this; we only use that the described exposure argument gives the correct marginal distribution for a single~$G_i$.)  

\begin{figure}[t]
\centering
  \setlength{\unitlength}{1bp}%
  \begin{picture}(278.80, 115.05)(0,0)
  \put(0,0){\includegraphics{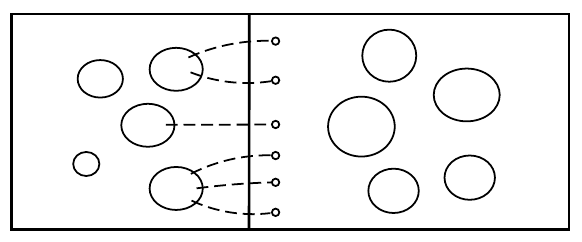}}
  \put(19.27,96.65){\fontsize{11.38}{13.66}\selectfont \makebox[0pt]{$V_S$}}
  \put(258.92,96.65){\fontsize{11.38}{13.66}\selectfont \makebox[0pt]{$V_L$}}
  \end{picture}%
	\caption{\label{fig:exp} An example (the same as in Figure~\ref{fig:po2}) of the auxiliary graph $H_i$, which is determined by the first exposure round. 
A component of type~$(k,r)$ contains~$k$ vertices from $V_S$ and has~$r$ incident $V_S$--$V_L$ edges, whose so-far unrevealed $V_L$--vertices are represented by stubs (here depicted by the circle-shaped endpoints in $V_L$). 
We obtain $G_i$ by (i)~connecting each stub to a randomly chosen vertex from $V_L$, and (ii)~adding $Q_{0,2}(i)$ random $V_L$--$V_L$ edges to $H_i$.} 
\end{figure}

Since our focus is on the component sizes of~$G_i$, the internal structure of the components of~$G_{i_0}$ and~$H_i$
is irrelevant; all we need to know is the size of each component, and how many stubs it contains.
Any component~$C$ of~$H_i$ is either contained in~$V_L$ (in which case~$|C|>K$) or in~$V_S$.
If~$C\subseteq V_S$, then we say that~$C$ has \emph{type~$(k,r)$} if it has size
 $|C|=k$ and  
contains~$r$ stubs, i.e., is incident to $r$ $V_S$--$V_L$ edges in $G_i\setminus H_i$, cf.~Figure~\ref{fig:exp}. (Note that we may have $k=|C|>K$: in $G_{i_0}$ all vertices in $V_S$ are in components of size $k$, but in passing to $H_i$ we in general add $V_S$--$V_S$ edges from steps~$i_0 < j \le i$.)
As usual, for $i \ge 0$ and $k \ge 1$, we write $N_k(i)$ for the number of vertices of~$G_i$ which are in components of size exactly~$k$.  
For $i \ge i_0$, $k \ge 1$ and $r \ge 0$, we write~$Q_{k,r}(i)$ for the number of components of~$H_i$ of type~$(k,r)$.\footnote{Note that $N_k$ counts vertices, and $Q_{k,r}$ counts components; the different normalizations are convenient in different contexts.}
Thus
\begin{equation}
\label{eq:Qkr:init}
Q_{k,r}(i_0) = \indic{r=0, \: 1 \le k\le K} N_{k}(i_0)/k .
\end{equation}
We may think of an added $V_L$--$V_L$ edge as a component of type $(0,2)$: it contains no vertices, but has two stubs
associated to it.
Hence the notation $Q_{0,2}(i)$ above; we let $Q_{0,r}:=0$ for $r\ne 2$.

For $i_0 \le i \le i_1$, let 
\begin{equation}\label{def:fS}
\fS_i := \Bigl(\bigl(N_{k}(i_0)\bigr)_{k > K}, \: \bigl(Q_{k,r}(i)\bigr)_{k,r \ge 0}\Bigr) .
\end{equation}
This \emph{parameter list} contains the essential information about $H_i$.
Given (a possible value of) $\fS_i$, 
treating~$\fS_i$ as deterministic we construct a \emph{random graph $J_i=J(\fS_i)$} as follows:
start with a graph $\tH_i=\tH(\fS_i)$ 
consisting of $Q_{k,r}(i)$ type-$(k,r)$ components for all $k \ge 1$ and $r \ge 0$, and $N_{k}(i_0)/k$ components of size $k$ for all $k > K$.
Let~$V_S$ be the set of vertices in components of the first type, and $V_L$ the set in components of the second type. 
Given $\tH_i$, we then (i)~connect each stub of $\tH_i$ to an independent random vertex in $V_L$, and (ii)~add $Q_{0,2}(i)$ random $V_L$--$V_L$ edges to $\tH_i$. 
By construction and Lemma~\ref{lem:exp} we have the following result.
\begin{lemma}[Conditional equivalence]
\label{lem:cond}
Given $\fS_i$, the random graph~$J_i=J(\fS_i)$ has the same component size distribution as $G_i$ conditioned on the parameter list~$\fS_i$.
\noproof
\end{lemma}
Our strategy for analyzing the component size distribution of~$G_i$ will be as follows. 
In Sections~\ref{sec:DEM}--\ref{sec:BPA} we will show that the random parameter list~$\fS_i$, which is revealed
in the first exposure round, is concentrated, i.e., nearly deterministic.
Then, in the second round (so having conditioned on~$\fS_i$) we use the random model~$J(\fS_i)$
to construct~$G_i$. 
The advantage is that~$J(\fS_i)$ is very well suited to branching process approximation, since
it is defined by a number of independent random choices.

Note for later that, by definition of the discrete variables, for $i \ge i_0$ we have
\begin{equation}\label{eq:NkQkr:sum}
|V_L| = \sum_{k>K} N_k(i_0) ,\qquad |V_S| = \sum_{k \ge 1,\, r \ge 0}k Q_{k,r}(i) \qquad\text{and}\qquad n = |V_L|+|V_S|.
\end{equation}

\subsection{Differential equation approximation}\label{sec:DEM}
In this subsection we study the (random) parameter list~$\fS_i$ defined in Definition~\ref{def:fS}.
We shall track the evolution of several associated random variables
using Wormald's differential equation method~\cite{DEM,DEM99},
which intuitively shows that the trajectories of the random variables stay (after suitable rescaling) 
close to the solution of a corresponding system of ODEs.
This proof method usually works in situations where the expected one-step changes of each random variable 
can approximately be written as a smooth function of the random variables from the collection,
and the worst case one-step changes of each variable are not too big. 
In fact, we shall rely on a variant of the differential equation method due to Warnke~\cite{DEMLW}, 
in order to obtain sufficiently small approximation~errors. 

\subsubsection{Small components}\label{sec:DEM:small}
We start by tracking the number of vertices of $G_i$ which are in components of size $k \in \cC=\{1, \ldots, K,\omega\}$, which we denote by $N_k(i)$. 
Here, as usual, `size~$\omega$' means size at least~$K+1$. 
The following result intuitively shows $N_{k}(i) \approx \rho_k(i/n) n$ for a smooth (infinitely differentiable) function~$\rho_k$, 
whose derivative~$\rho'_k$ is suggested by the expected one-step changes of~$N_k(i)$, cf.~\eqref{eq:der:rhok} and~\eqref{eq:Nk:E:change} below.   
Later we shall show that the $\rho_k$, and the related functions appearing in the next few lemmas, are in fact analytic. 
Recall that we are considering a bounded-size $\ell$-vertex rule~$\cR$. 
\begin{lemma}\label{lem:Nk:t}
With probability at least $1-n^{-\omega(1)}$ we have 
\begin{equation}\label{eq:Nk:t}
\max_{0 \le i \le i_1}\max_{k \in \cC} |N_{k}(i) -\rho_k(i/n) n| \le (\log n) n^{1/2} ,
\end{equation}
where the functions $(\rho_k)_{k \in \cC}$ from $[0,t_1]$ to $[0,1]$ are smooth.
They satisfy $\sum_{k \in \cC}\rho_k(t)=1$ and $\rho'_{\omega}(t)\ge 0$, and are given by the unique solution to  
\begin{equation}\label{eq:der:rhok}
\rho_{k}(0)=\indic{k=1} \qquad \text{and} \qquad \rho'_k(t) = \sum_{\vc \in \cC^{\ell}}\Delta_{\rho}^{\cR}(k,\vc) \prod_{j \in [\ell]} \rho_{c_j}(t) ,
\end{equation}
for certain coefficients $\Delta_{\rho}^{\cR}(k,\vc)\in\ZZ$ with $|\Delta_{\rho}^{\cR}(k,\vc)| \le 2K$.
\end{lemma}
\begin{proof}
This follows from a nowadays standard (see e.g.,~\cite{SW,BK}) application 
of the differential equation method~\cite{DEM,DEM99}, 
with the extra twist that the particularly `small' approximation error term $(\log n) n^{1/2}$ relies on a variant of Warnke~\cite{DEMLW}. 
Let us briefly sketch the details.

To calculate the expectation of~$N_{k}(i+1)-N_{k}(i)$, in step~$i+1$ 
let $\vc_{i+1}={(c_{i+1,1}, \ldots, c_{i+1,\ell})} \in \cC^\ell$ list the sizes of the components of $G_{i}$
containing the random vertices $v_{i+1,1},\ldots,v_{i+1,\ell}$, with all sizes larger than~$K$ replaced by~$\omega$ (as before). 
Noting that edges connecting two vertices in components of size~$\omega$ leave all~$N_k$ with $k \in \cC$ unchanged,
it is easy to check from the definition of a bounded-size rule that if 
all vertices~$v_{i+1,j}$ with~$c_{i+1,j} \neq \omega$ lie in different components, then for each $k \in \cC$ the number of vertices in components of size $k$ changes by a deterministic function $\Delta_{\rho}^{\cR}(k,\vc_{i+1})$ that only depends on~$k$ and~$\vc_{i+1}$, with $|\Delta_{\rho}^{\cR}(k,\vc_{i+1})| \le 2K$.
Furthermore, in step~$i+1$, the probability that at least two of the~$\ell$~randomly chosen vertices lie in the same component of size at most $K$ is at most $\ell^2K/n$. 
So, if~${(\cF_i)_{i\ge 0}}$ denotes the natural filtration associated to our random graph process, using $|N_{k}(i+1)-N_{k}(i)| \le 2K$ it follows as in \cite{SW,RWapunique} that  
\begin{equation}\label{eq:Nk:E:change}
\biggl| \E(N_{k}(i+1)-N_{k}(i) \mid \cF_i) - \sum_{\vc \in \cC^{\ell}}\Delta_{\rho}^{\cR}(k,\vc) \prod_{j \in [\ell]} \frac{N_{c_j}(i)}{n} \biggr| \le \frac{4\ell^2K^2}{n} .
\end{equation}
Since $|\cC^{\ell}| = (K+1)^{\ell} = O(1)$, $|\Delta_{\rho}^{\cR}(k,\vc)| \le 2K$ and $N_{k}(0)=\indic{k=1}n$, 
similar to~\cite{SW,BK,RWapcont} a routine application of the differential equation method variant from~\cite{DEMLW} 
(with the parameter choice~$\lambda := (\log n)^{2/3}n^{-1/2}$, say, so that the approximation error term satisfies~$O(\lambda n) \ll (\log n) n^{1/2}$) 
implies that \eqref{eq:Nk:t} holds with probability at least $1-n^{-\omega(1)}$, where the $(\rho_k(t))_{k \in \cC}$ are the unique solution to~\eqref{eq:der:rhok}. 

Now we turn to properties of the functions $(\rho_k)_{k \in \cC}$. 
By induction on $j$ we see that the $j$th derivatives $\rho^{(j)}_k(t)$ exist for all $k\in \cC$ and $j\ge 0$, i.e., that the $(\rho_k)_{k \in \cC}$ are smooth. 
Since $N_{k}(i) \in [0,n]$ and $\sum_{k \in \cC}N_k(i)=n$, it follows from \eqref{eq:Nk:t} that
$\rho_k(t) \in [0,1]$ and $\sum_{k \in \cC}\rho_k(t)=1$. (This also follows directly from the differential equations, similar to Theorem~2.1 in~\cite{SW}.)  
Finally, $\Delta_{\rho}^{\cR}(\omega,\vc) \ge 0$ and $\rho_k(t) \ge 0$ imply $\rho'_{\omega}(t)\ge 0$. 
\end{proof}

For later reference we now extend the results of Lemma~\ref{lem:Nk:t} to any \emph{fixed} component size $k'$.
One way to do this is to note that any bounded-size rule with cut-off~$K$ can be interpreted as a bounded-size rule with cut-off~$\max\{k',K\}$, and apply Lemma~\ref{lem:Nk:t} to this rule. 
This approach has the minor drawback that as $k'$ varies, the resulting form of the formula for~$\rho_k'$ changes, even though the the function~$\rho_k$ stays the same, of course. 
In the next lemma we take a different approach which avoids this, leading to conceptually simpler differential equations. 
The key point is that the functions $(\rho_{k})_{k\ge 1}$ in~\eqref{eq:Nk2:t}--\eqref{eq:Nk3:t} below are the unique solution of a system of~ODEs. 
Recall that $N_{\ge k}(i) = \sum_{k'\ge k} N_{k'}(i)$.
\begin{lemma}\label{lem:Nk2:t}
Given $k' \ge 1$, with probability at least $1-n^{-\omega(1)}$ we have 
\begin{align}
\label{eq:Nk2:t}
\max_{0 \le i \le i_1} \max_{1 \le k \le k'}|N_{k}(i) - \rho_{k}(i/n)n| & \le (\log n) n^{1/2} , \\
\label{eq:Nk3:t}
\max_{0 \le i \le i_1} \max_{1 \le k \le k'}|N_{\ge k}(i) - \rho_{\ge k}(i/n)n| & \le (\log n) n^{1/2} 
\end{align}
where the functions $\rho_k:[0,t_1]\to [0,1]$ are given by the unique solution to 
the system of differential equations~\eqref{eq:der:rhok} for $1\le k\le K$ and~\eqref{eq:der:rhok3} below for $k>K$,
and we write
\begin{equation}\label{eq:der:rhok2}
\rho_{\ge k}(t) = 1- \sum_{1 \le j < k}\rho_j(t)
\end{equation}
and interpret $\rho_\omega$ as $\rho_{\ge K+1}$. 
Furthermore, the functions $(\rho_k)_{k \ge 1}$ are smooth on~$[0,t_1]$,
with $\rho_{k}(t),\rho_{\ge k}(t) \in [0,1]$. 
\end{lemma}
\begin{proof}
The proof is a minor generalization of that of~\eqref{eq:Nk:t}, so let us omit the details and only outline
how the differential equations are obtained. 
For $1\le k\le K$ the equation \eqref{eq:der:rhok} remains valid; here we may either 
interpret $\rho_\omega$ as $\rho_{\ge K+1}=1-\sum_{k\le K}\rho_k$, or include an equation for $\rho_\omega$
itself; this makes no difference. For~${k>K}$, arguing as for~\eqref{eq:der:rhok} but now with `size~$\ge k+1$'
playing the role of size~$\omega$ we have
\begin{equation}\label{eq:der:rhok3}
\rho_{k}(0)=0 
\qquad \text{and} \qquad 
\rho'_k(t) = \sum_{\vc \in \{1, \ldots, k, \ge k+1\}^{\ell}}\Delta_{\rho}^{\cR}(k,\vc) \prod_{j \in [\ell]} \rho_{c_j}(t) , 
\end{equation}
where the $\Delta_{\rho}^{\cR}(k,\vc)$ are constants with $|\Delta_{\rho}^{\cR}(k,\vc)| \le 2k$. 
Recalling~\eqref{eq:der:rhok} and~\eqref{eq:der:rhok2}, the key observation is that each $\rho'_{k}$ depends only on $\rho_{j}$ with $1 \le j \le \max\{k,K\}$.  
Hence standard results imply that the infinite system of differential equations~\eqref{eq:der:rhok} and~\eqref{eq:der:rhok2}--\eqref{eq:der:rhok3} has a unique solution on~$[0, t_1]$. 
Mimicking the proof of Lemma~\ref{lem:Nk:t}, it then follows that the functions $(\rho_{k})_{k \ge 1}$ are smooth, with ${\rho_{k}(t) \in [0,1]}$ and~${\sum_{1 \le j < k}\rho_j(t) \le 1}$.
\end{proof}

Recall that after $i_0$ steps we partition the set of vertices into $V_S \cup V_L$, where $V_S$ contains all vertices in components of size at most~$K$.
Our later arguments require that whp $|V_S|,|V_L|=\Theta(n)$; in the light of Lemma~\ref{lem:Nk:t}, to show this
it is enough to show that $\min\{\rho_1(t_0),\rho_{\omega}(t_0)\}>0$. 
This is straightforward for $\rho_1(t_0)$; for $\rho_{\omega}(t_0)$ the key observation is that a new component of size $2r$ is certainly formed in any step~$i$ where all vertices $v_{i,1}, \ldots, v_{i,\ell}$ lie in distinct components of size~$r$. 
Hence, via successive doublings, by time~$t_0$ we create many components of size $2^j > K$; 
Lemma~\ref{lem:rhokder} makes this idea rigorous. 

\begin{lemma}\label{lem:rhokder}
Define the functions $(\rho_k)_{k \in \cC}$ as in Lemma~\ref{lem:Nk:t}.
For all $t \in (0,t_1]$ we have $\min\{\rho_1(t),\rho_\omega(t)\}>0$. 
\end{lemma} 

\begin{proof}
As noted above, if $k$ is even then $\Delta_{\rho}^{\cR}(k,(k/2, \ldots, k/2)) =k\ge 1$. 
Furthermore, $\Delta_{\rho}^{\cR}(k,\vc) \ge 0$ if $\vc$ does not contain~$k$.  
Since $\rho_j(t) \ge 0$, $|\Delta_{\rho}^{\cR}(k,\vc)| \le 2 k$ and $\sum_j \rho_j(t)=1$, by the 
form of~$\rho'_k$ in~\eqref{eq:der:rhok} and~\eqref{eq:der:rhok3} 
it readily follows for any integer~$k \ge 1$ that 
\begin{equation}\label{eq:rhokder}
\rho'_k(t) \ge \indic{2 | k} \bigl(\rho_{k/2}(t)\bigr)^{\ell} - 2 k \cdot \ell \rho_k(t) \ge - 2 \ell k \rho_k(t).
\end{equation}

We claim that, for every $j \in \NN$ and $t \in (0,t_1]$, we have $\rho_{2^{j}}(t)>0$;
the proof is by induction on~$j$.
For the base case~$j=0$, from~\eqref{eq:rhokder} we have $(\rho_1(t)e^{2\ell t})'=(\rho_1'(t)+2\ell\rho_1(t))e^{2\ell t}\ge 0$.
Hence $\rho_1(t) \ge \rho_1(0)e^{-2\ell t}=e^{-2\ell t}$.
For the induction step~$j \ge 1$, we write $k=2^j$ to avoid clutter. 
It follows from~\eqref{eq:rhokder} that for $t'\ge t$ we have $\rho_{k/2}(t') \ge \rho_{k/2}(t/2)e^{-2\ell k (t'-t/2) }$.
Since $\rho_{k/2}(t/2)>0$ by induction, we deduce that there is a
$\delta=\delta(k,t)>0$ such that $\rho_{k/2}(t') \ge \delta$ for all $t' \in [t/2,t]$.  
The first inequality in~\eqref{eq:rhokder} implies that $(\rho_{k}(t')e^{2\ell kt})' \ge \delta^{\ell}$ in~$[t/2,t]$, 
which readily implies $\rho_{k}(t) \ge e^{-2 \ell k t} \cdot \delta^\ell t/2 >0$ for~$k=2^j$. 

This completes the proof by noting that $\rho_\omega(t)  \ge \rho_{k}(t)$ whenever~$k>K$. 
\end{proof}%
As we shall discuss in Section~\ref{sec:period}, for
$\ell$-vertex rules it is not true in general 
that $\min_{k \in \cC}\rho_k(t)>0$ for~$t>0$ 
(in contrast to the usual `edge-based' Achlioptas processes
considered in~\cite{SW,BK,BBW11,BBW12b,BBW12a}).

\subsubsection{Random $V_L$--$V_L$ edges}\label{sec:DEM:VL}
Next we focus on the evolution of $Q_{0,2}(i)$, which counts the number of $V_L$--$V_L$ edges added in steps $i_0 < j \le i$. 
\begin{lemma}\label{lem:U:t}
With probability at least $1-n^{-\omega(1)}$ we have 
\begin{equation}\label{eq:U:t}
\max_{i_0 \le i \le i_1} |Q_{0,2}(i) -q_{0,2}(i/n) n| \le (\log n)^2 n^{1/2} ,
\end{equation}
where the function $q_{0,2}:[t_0,t_1] \to [0,1]$ is smooth, with $q_{0,2}(t_0)=0$ and $q_{0,2}'(t) > 0$.
It is given by the unique solution to the differential equation~\eqref{eq:der:u}. 
\end{lemma}
\begin{proof}
This follows again by a routine application of the differential equation method~\cite{DEMLW}, so we only outline the argument. Formally, we use Lemma~\ref{lem:Nk:t} to obtain bounds at step $i_0$, and then track $Q_{0,2}(i)$ and $(N_k(i))_{k \in \cC}$ from step~$i_0$ onwards. We obtain a worse approximation error term than in Lemma~\ref{lem:Nk:t} because here there is already an `initial' error term at step~$i_0$ (from Lemma~\ref{lem:Nk:t}). With some work we could tighten the bounds, but this would not affect our key results.

Analogous to \eqref{eq:Nk:E:change} we consider $\E(Q_{0,2}(i+1)-Q_{0,2}(i) \mid \cF_i)$, i.e., the conditional one-step expected change in~$Q_{0,2}(i)$. This time we need to consider vertices in components of size $\omega$ that are in $V_S$
separately from those in $V_L$.
Let 
\begin{align}
\label{eq:Nk:vL}
\vartheta_{L}(t) := \rho_{\omega}(t_0) ,
\end{align}
which corresponds to the idealized rescaled number of vertices in $V_L$.
Noting that for $i\ge i_0$ there are $N_\omega(i)-|V_L|=N_\omega(i)-N_\omega(i_0)$ vertices in $V_S$ that are in components
of $G_i$ of size $\omega$ (i.e., size at least $K+1$), let
\begin{align}
\label{eq:Nk:vk}
\vartheta_k(t) := \begin{cases}
\rho_k(t) & \text{if $1 \le k \le K$}, \\
\rho_{\omega}(t)-\rho_{\omega}(t_0) & \text{if $k = \omega$}, 
\end{cases}
\end{align}
corresponding to the idealized rescaled number 
of vertices in $V_S$ which are in components of size $k \in \cC=\{1,\ldots, K, \omega\}$. 
Since $\cR$ is a bounded-size rule, $Q_{0,2}(i+1)-Q_{0,2}(i)$ is determined by the following information: 
the sizes $c_{i,j}\in \cC=\{1,\ldots,K,\omega\}$ of the components containing the vertices $v_{i,1}, \ldots, v_{i,\ell}$
and, where $c_{i,j}=\omega$, the information whether $v_{i,j}$ is in $V_L$ or not.
(It does not matter whether any of these vertices lie in the same component or not). 
So, with $|V_L| = N_{\omega}(i_0)$ and~\eqref{eq:U:init} in mind, it is straightforward to see that $q_{0,2}(t)$ is given by the unique solution to
\begin{equation}\label{eq:der:u}
q_{0,2}(t_0)=0 \qquad \text{and} \qquad q_{0,2}'(t) = \sum_{\vc=(c_1, \ldots, c_{\ell}) \in (\cC \cup\{L\})^{\ell}} \Delta^{\cR}(\vc) \prod_{j \in [\ell]} \vartheta_{c_j}(t) ,
\end{equation}
where $\Delta^{\cR}(\vc)=1$ if we have $\vc_{j_1}=\vc_{j_2}=L$ for the indices $\{j_1,j_2\}=\cR(\vc)$ selected by the
rule, and $\Delta^{\cR}(\vc)=0$ otherwise. 

Now we turn to properties of $q_{0,2}(t)$. 
By Lemma~\ref{lem:Nk:t}, all $\vartheta_{k}(t)$ are smooth, so $q_{0,2}(t)$ is smooth by~\eqref{eq:der:u}. 
Similarly, recalling $\rho'_{\omega}(t) \ge 0$ and $\vartheta_{\omega}(t_0) = 0$, we see that $\vartheta_{k}(t) \in [0,1]$ and $\sum_{k \in \cC \cup\{L\}} \vartheta_{k}(t) = 1$. 
Now, if $\ell$ distinct vertices from $V_L$ are chosen, then a $V_L$--$V_L$ edge is added. Hence $\Delta^{\cR}((L, \ldots, L))=1$, which implies $q_{0,2}'(t) \ge \bigl(\rho_{\omega}(t)\bigr)^\ell > 0$ for all $t \in [t_0,t_1]$, see Lemma~\ref{lem:rhokder}.  
Finally, using $\Delta^{\cR}(\vc) \le 1$ and $\sum_{k \in \cC \cup\{L\}} \vartheta_{k}(t) = 1$ we deduce that $q_{0,2}'(t) \le 1$, so by~\cref{def:t0t1t2,def:sigma} we have $q_{0,2}(t) \le q_{0,2}(t_0) + t-t_0 \le 2\sigma \le 1$ for all $t \in [t_0,t_1]$.
\end{proof}

\subsubsection{Components in $V_S$}\label{sec:VS:finite}
We now study the `marked graph'~$H_i$ defined in Section~\ref{sec:exp}, see also Figure~\ref{fig:exp}. 
For $k \ge 1$ and $r \ge 0$, recall that $Q_{k,r}(i)$ counts the number of type-$(k,r)$ components in $H_i$,
i.e., components of $H_i$ which contain~$k$ vertices from $V_S$ and have~$r$ stubs (and so are incident to $r$
$V_S$--$V_L$ edges in $G_i\setminus H_i$).
As usual, we expect that $Q_{k,r}(tn)/n$ can be approximated by a smooth function $q_{k,r}(t)$, 
and our next goal is to derive a system of differential equations that these $q_{k,r}$ must satisfy (again based on the expected one-step changes). 
Note that~\eqref{eq:Qkr:t} below only implies $Q_{k,r}(i)/n \approx q_{k,r}(i/n)$ for \emph{fixed}~$k$ and~$r$ 
(see Section~\ref{sec:VS} for an extension to all $k \ge 1$ and $r \ge 0$). 
\begin{lemma}\label{lem:Qkr:t}
The system of differential equations~\eqref{eq:der:qkr:init} and~\eqref{eq:der:qkr} below has a unique solution $(q_{k,r})_{k \ge 1, r \ge 0}$ on~$[t_0,t_1]$, with each $q_{k,r}:[t_0,t_1] \to [0,1]$ a smooth function. 
Given $k' \ge 1$ and $r' \ge 0$, with probability at least $1-n^{-\omega(1)}$ we have 
\begin{equation}\label{eq:Qkr:t}
\max_{i_0 \le i \le i_1} \max_{\substack{1 \le k \le k' \\ 0 \le r \le r'}}|Q_{k,r}(i) - q_{k,r}(i/n)n| \le (\log n)^{2} n^{1/2} .
\end{equation}
\end{lemma}
\begin{proof}
As in the proof of Lemma~\ref{lem:U:t} we only sketch the differential equation method~\cite{DEMLW} argument. Again, we use Lemma~\ref{lem:Nk:t} to obtain bounds at step $i_0$, and then track $(N_k(i))_{k \in \cC}$ and $(Q_{k,r}(i))_{1 \le k \le k',\,1 \le r \le r'}$
from step~$i_0$ onwards; here, as usual, $\cC=\{1,\ldots,K,\omega\}$.

Since $|Q_{k,r}(i+1)-Q_{k,r}(i)| \le 2$, the `exceptional event' that two of the $\ell$ random vertices
lie in the same $(k,r)$--component of~$H_i$ with $k \le k'$ contributes at most, say, $4\ell^2 k'/n = O(1/n)$ to $\E(Q_{k,r}(i+1)-Q_{k,r}(i) \mid \cF_i)$. 
Hence, recalling the definition of $\rho_k(t)$, by considering the expected one-step changes of $Q_{k,r}(i)$, it is not difficult to see that $q'_{k,r}(t)$ is a polynomial function of $\rho_{\omega}(t_0)$, the $\rho_{\tilde{k}}(t)$ with $\tilde{k} \in \cC$, and the $q_{\tilde{k},\tilde{r}}(t)$ with $1 \le \tilde{k} \le k$ and $0 \le \tilde{r} \le r$ (edges connecting two vertices from $V_L$ or two vertices in $(\tilde{k},\tilde{r})$--components with $\tilde{k} > k$ or $\tilde{r} > r$ leave $Q_{k,r}(i)$ unchanged). 
For later reference, we now spell out these differential equations explicitly. 
By \eqref{eq:Qkr:init}, the initial conditions are 
\begin{equation}\label{eq:der:qkr:init}
q_{k,r}(t_0) = \indic{r=0, \: 1 \le k\le K} \rho_k(t_0)/k .
\end{equation}
Turning to $q'_{k,r}$, set
\begin{equation}\label{def:skr}
s(k,r) := \begin{cases}
	\omega , & ~~\text{if $k \ge K+1$ or $r \ge 1$}, \\
		k  , & ~~\text{otherwise}.\\	
	\end{cases}
\end{equation}
From the relationship between $H_i$ and $G_i$ established in Section~\ref{sec:exp} (see Figure~\ref{fig:exp}),
a vertex $v\in V_S$ in a type-$(k,r)$ component of~$H_i$ is in a component of $G_i$ with size~$s(k,r)$,
where size $\omega$ means size $\ge K+1$. 
Recall that in step $i$ the rule $\cR$ connects $v_{i,j_1}$ with $v_{i,j_2}$, where $\{j_1,j_2\}=\cR(\vc_i)$.
In the following formulae we sum over all possibilities $\vc\in \cC^\ell$ for $\vc_i$, and always tacitly define
\begin{equation*}
\{j_1,j_2\}=\cR(\vc).
\end{equation*}
Bearing in mind that $|V_L| = N_{\omega}(i_0)$, similar arguments to those leading to~\eqref{eq:der:rhok} and~\eqref{eq:Nk:E:change} show that 
\begin{equation}\label{eq:der:qkr}
q'_{k,r}(t)  = \sum_{\vc \in \cC^{\ell}} \Bigl[\prod_{j \in [\ell] \setminus \{j_1,j_2\}} \rho_{c_j}(t)\Bigr] \cdot \Bigl[\sum_{1 \le h \le 3} F_h(k,r,\vc)\Bigr] ,
\end{equation}
where 
\begin{equation}\label{eq:der:qkr:F1}
F_1(k,r,\vc) := \sum_{\substack{k_1+k_2=k: \: k_1,k_2 \ge 1\\r_1+r_2=r: \: r_1,r_2 \ge 0}} k_1 q_{k_1,r_1}(t) k_2 q_{k_2,r_2}(t) \indic{c_{j_1}=s(k_1,r_1), \: c_{j_2}=s(k_2,r_2)} ,
\end{equation}
corresponding to creating a new $(k,r)$--component by adding an edge
within $V_S$, 
\begin{equation}\label{eq:der:qkr:F2}
F_2(k,r,\vc) := \indic{r \ge 1}k q_{k,r-1}(t) \rho_{\omega}(t_0) \bigl[ \indic{c_{j_1}=s(k,r-1), \: c_{j_2}=\omega} + \indic{c_{j_1}=\omega, \: c_{j_2}=s(k,r-1)}\bigr] ,
\end{equation}
corresponding to adding a $V_S$--$V_L$ edge to a $(k,r-1)$--component, and 
\begin{equation}\label{eq:der:qkr:F3}
F_3(k,r,\vc) := -k q_{k,r}(t)\bigl[ \indic{c_{j_1}=s(k,r)} \rho_{c_{j_2}}(t)  + \rho_{c_{j_1}}(t)\indic{c_{j_2}=s(k,r)}\bigr] ,
\end{equation}
corresponding to destroying a $(k,r)$--component by connecting one of its vertices in $V_S$ to something else.
(The normalization is different for $q_{k,r}$ and $\rho_k$ since $Q_{k,r}$ counts components, whereas $N_k$ counts vertices.) 

Turning to properties of the $q_{k,r}(t)$, recall that the $\rho_k(t)$ are smooth on $[0,t_1]$. 
The key observation is that~$q'_{k,r}$ depends only on $(q_{\tk,\tr})_{1 \le \tk \le k,\, 0 \le \tr \le r}$ and $(\rho_j)_{j \in \cC}$, see~\eqref{eq:der:qkr}--\eqref{eq:der:qkr:F3}. 
So, using $|\cC^{\ell}| = (K+1)^{\ell} = O(1)$, standard results imply that the infinite system of differential equations~\eqref{eq:der:qkr:init}--\eqref{eq:der:qkr} has a unique solution on~$[t_0, t_1]$. 
Furthermore, by induction on $j \ge 0$ (and $k+r \ge 1$) we see that all the $q_{k,r}(t)$ are $j$~times differentiable; thus the $(q_{k,r})_{k\ge 1,\,r \ge 0}$ are smooth on $[t_0,t_1]$. 
Finally, since $0 \le Q_{k,r}(i) \le n/k$, standard comparison arguments yield $q_{k,r}(t) \in [0,1]$, say.  
\end{proof}

\subsection{Exploration tree approximation}\label{sec:BPA}
In this subsection we continue studying the (random) parameter list~$\fS_i$ defined in Definition~\ref{def:fS}.
More concretely, we shall track the evolution of several associated random variables using 
the exploration tree method developed in~\cite{RWapsubcr}, which intuitively shows that these variables (i)~are concentrated, and (ii)~have exponential tails. 
This proof method is based on branching process approximation techniques, and it usually works in situations where the quantities in question can be determined by a subcritical neighbourhood exploration process.

\subsubsection{Component size distribution}\label{sec:BPA:init}
We first revisit the number $N_k(i)$ of vertices of $G_i$ in components of size~$k$ in the subcritical phase,
which we studied in~\cite{RWapsubcr} for size rules 
(for bounded-size rules the quantity~$\tcx$ appearing in Theorem~1 of~\cite{RWapsubcr} is equal to~$\tc$ by Theorem~15 of~\cite{RWapsubcr}).
Since $t_0 < \tc$, Theorem~1 of~\cite{RWapsubcr} implies the following result, which applies to \emph{all} component sizes $k \ge 1$ 
and steps $i\le i_0=t_0 n$, showing concentration and exponential tail bounds. We write $D_N$ for the power of $\log n$ in the error term, since we will prove a similar result for the $Q_{k,r}(i)$ later with a different power $D_Q$. 
\begin{theorem}\label{thm:init}
Let $(\rho_k)_{k \ge 1}$ be the functions defined in Lemma~\ref{lem:Nk2:t}. 
There are constants $a,A,D_N,n_0 > 0$ with $A \ge 1$ 
such that for $n \ge n_0$, with probability at least $1-n^{-99}$, the following holds for all $k \ge 1$:
\begin{gather}
\label{eq:Nk:t0}
\max_{0 \le i \le i_0}|N_k(i)-\rho_k(i/n)n| \le (\log n)^{D_N} n^{1/2} ,\\
\label{eq:Nk:t0:tail}
\max_{0 \le i \le i_0}N_{\ge k}(i) \le A e^{-ak}n,\\
\label{eq:rhok:t0:tail}
\sup_{t \in [0,t_0]} \rho_k(t) \le A e^{-ak} .
\end{gather}
Furthermore, we have $\sum_{k \ge 1} \rho_k(t)=1$ for all $t \in [0,t_0]$. \noproof
\end{theorem}
\begin{remark}\label{rem:Nk:t}
To be pedantic, the functions $(\rho_k)_{k \ge 1}$ of Theorem~1 of~\cite{RWapsubcr} could potentially differ from those considered in Lemma~\ref{lem:Nk2:t}. 
However, since both are defined \emph{without} reference to $n$, by~\eqref{eq:Nk2:t} and~\eqref{eq:Nk:t0} these must be equal. 
This justifies (with hindsight) our slight abuse of notation. 
Furthermore, from Lemmas~\ref{lem:Nk2:t} and~\ref{lem:rhokder} and the fact that $\sum_{k \ge 1} \rho_k(t_0)=1$, we see that 
\[
\rho_{\omega}(t_0)=1-\sum_{1 \le k \le K} \rho_k(t_0)=\sum_{k > K}\rho_k(t_0) > 0 .
\]
\end{remark}
Let us briefly outline the high-level proof strategy from Section~2 of~\cite{RWapsubcr}, which we will adapt to~$H_i$ (as defined in Section~\ref{sec:exp}) in a moment. 
The basic idea is to generalize slightly, and establish concentration starting from an initial graph~$F$. 
Using induction, it then suffices to prove concentration during an interval consisting of a small (linear) number of steps. 
For this purpose we use a two-phase\footnote{We use the word phase rather than round to avoid any confusion with the main two-round exposure argument described in Section~\ref{sec:exp}.} 
exposure argument:
we first reveal which $\ell$-tuples appear in the entire interval, and then expose their order (in which they are presented to the rule $\cR$). 
Given a vertex $v$, via the first exposure phase we can severely restrict the set of components (of $F$) and tuples (that appear in the interval) which can influence the size of the component containing $v$ under the evolution of any size rule. 
Indeed, the only components/tuples which can possibly be relevant are those which can be reached from $v$ after adding all $\binom{\ell}{2}$ edges in each $\ell$-tuple appearing in the first exposure phase; we include \emph{all} of these to form an `upper bound' for the component containing $v$, which we will `thin' in the second phase.
Of course, all these tuples and components can be determined by a neighbourhood exploration process. 
If the interval has length $\delta n = \Theta(n)$, then (since at most $\ell n^{\ell-1}$ of the $n^{\ell}$ total $\ell$-tuples contain any given vertex, and each contains at most $\ell-1$ new vertices) it seems plausible that the expected size of the associated offspring distribution is at most roughly  
\begin{equation}\label{eq:init:EOD}
\delta \cdot \ell(\ell-1) \cdot \sum_{k \ge 1}k N_{k}(F)/n = \delta \ell(\ell-1)S_2(F),
\end{equation}
where, as usual, $S_2(F)$ denotes the susceptibility of the graph $F$,
i.e., the expected size of the component containing a random vertex. 

In~\cite{RWapsubcr} our inductive argument hinges on the fact that the associated branching process remains subcritical (i.e., quickly dies out) as long as $\delta \ell(\ell-1)S_2(F) < 1$. 
In the first exposure phase this allows us to couple the neighbourhood exploration process giving the `upper bound' component with an `idealized' branching process that is defined without reference to~$n$. 
In particular, this gives rise to a so-called \emph{exploration tree}~$\cT_{v,\delta}$
(see page~187 in~\cite{RWapsubcr}), which itself contains enough information to reconstruct all relevant tuples and components. 
In the second exposure phase we then reveal the order of the relevant tuples, using the rule $\cR$ to construct the actual component containing~$v$ (see Section~2.4.3 in~\cite{RWapsubcr}). 
We can eventually establish tight concentration since, by the subcritical first phase, $\cT_{v,\delta}$ typically contains rather few components and tuples (see Lemma~14 in~\cite{RWapsubcr}). 
Finally, the above discussion also explains why the inductive argument breaks down around $\tc$, since then a giant component emerges (in which case $S_2(F) = \omega(1)$, so for any $\delta=\Theta(1)$ the branching process just described will be supercritical).

\subsubsection{Distribution of the $V_S$--components}\label{sec:VS}
We now turn, for $k \ge 1$ and $r \ge 0$, to the number $Q_{k,r}(i)$ of components of $H_i$ of type $(k,r)$.
We shall prove that, starting from $F=G_{i_0}$, these random variables remain tightly concentrated for all $i_0 < i \le i_1$ (with exponential tails).
The basic idea is to apply the argument outlined in Section~\ref{sec:BPA:init} for one interval of length $\delta n = i_1-i_0$, see \eqref{def:t0t1t2}--\eqref{def:i0i1i2},
using a minor twist to ensure that the corresponding exploration process remains subcritical even beyond~$\tc$, exploiting the fact that we are restricting to bounded-size rules. 
Recall that in the definition of~$Q_{k,r}(i)$ we do not care about the endpoints in~$V_L$ of the incident $V_S$--$V_L$ edges, see also Figure~\ref{fig:exp}. 
With this in mind, the key observation is that the evolution of the components in~$V_L$ is irrelevant for the evolution of~$Q_{k,r}(i)$: it suffices to know that these have size~$\omega$ (i.e., size $>K$).  
So, starting with a vertex $v \in V_S$, in the exploration process associated with the first exposure phase (which finds all relevant tuples and components) we do not further test reached vertices $w \in V_L$, since we already know that these vertices are in components of size~$\omega$. 
To keep the differences to~\cite{RWapsubcr} minimal, we shall simply pretend that all vertices in $V_L$ lie in distinct `dummy' components of size $K+1$, say, although it would be more elegant to mark reached $V_L$--vertices, by introducing a new vertex type in Section~2.4.2 of~\cite{RWapsubcr}. 
Since $\delta =t_1-t_0=2\sigma \le [\ell^2(K+1)]^{-1}$, see~\eqref{def:sigma}--\eqref{def:t0t1t2},
the branching-out rate of \eqref{eq:init:EOD} thus changes to at most 
\begin{equation*}
\delta \cdot \ell(\ell-1) \cdot \Bigl(\sum_{1 \le k \le K} k N_{k}(F)/n + (K+1)N_{\omega}(F)/n\Bigr) \le \delta \ell(\ell-1)(K+1) < 1 , 
\end{equation*}
suggesting that the exploration process indeed remains subcritical. 
This makes it plausible that, by a minor variant of the proof used in~\cite{RWapsubcr}, we can track the evolution of the $\bigl(Q_{k,r}(i)\bigr)_{k \ge 1, r \ge 0}$ for all $i_0 \le i \le i_1$, i.e., show that they are tightly concentrated around deterministic trajectories, see~\eqref{eq:Qkr} below. 
Furthermore, since the associated exploration process is subcritical, we also expect that these have exponential tails, see~\eqref{eq:Qkr:tail} and~\eqref{eq:qkr:tail} below. 
\begin{theorem}\label{thm:Qkr}
Let $(q_{k,r})_{k \ge 1,r \ge 0}$ be the functions defined in Lemma~\ref{lem:Qkr:t}. 
There are constants $b,B,D_Q,n_0 > 0$ with $B \ge 1$ 
such that for $n \ge n_0$, with probability at least $1-n^{-99}$, the following hold for all $k \ge 1$ and $r \ge 0$:
\begin{gather}
\label{eq:Qkr}
\max_{i_0 \le i \le i_1}|Q_{k,r}(i)-q_{k,r}(i/n)n| \le (\log n)^{D_Q} n^{1/2} ,\\
\label{eq:Qkr:tail}
\max_{i_0 \le i \le i_1}Q_{\ge k, \ge r}(i) \le B e^{-b(k+r)}n,\\
\label{eq:qkr:tail}
\sup_{t \in [t_0,t_1]} q_{k,r}(t) \le B e^{-b(k+r)} ,
\end{gather}
where $Q_{\ge k, \ge r}(i) := \sum_{k' \ge k,\, r' \ge r} Q_{k',r'}(i)$. 
\end{theorem}
\begin{remark}\label{rem:Qkr:sum}
Recall that, by definition of the discrete variables, for $i \ge i_0$ we have  $\sum_{k \ge 1,\, r \ge 0}k Q_{k,r}(i) = |V_S|=n-|V_L| = \sum_{1 \le k \le K} N_k(i_0)$; see \eqref{eq:NkQkr:sum}.
Hence a standard comparison argument (using Lemma~\ref{lem:Nk:t} and exponential tails) shows that for $t \in [t_0,t_1]$ we have 
\[
\sum_{k \ge 1,\, r \ge 0} k q_{k,r}(t) = \sum_{1 \le k \le K} \rho_k(t_0) .
\]
\end{remark}
\begin{proof}[Outline proof of Theorem~\ref{thm:Qkr}]
Let $\delta_t = t - t_0$. 
Starting with $F=G_{i_0}$ satisfying the conclusions of Lemma~\ref{lem:Nk:t}, the core argument is a minor modification of the proof of Theorem~3 in~\cite{RWapsubcr}, with $Q=K+1$.  
As discussed, the main idea is to pretend that all vertices of $V_L$ are in distinct (dummy) components of size~$K+1$. 
In particular, when constructing the exploration tree $\cT_{v,\delta_t}$, in the case $|C_{u_i}(F)| > K$ at the top of page~188 
in~\cite{RWapsubcr}, we simply add $K+1$ new `dummy' vertex nodes as the children of~$u_i$ (from the tree structure of~$\cT_{v,\delta_t}$ it then is clear that~$u_i$ lies in a component of size~$>K$, which is all we need in Section~2.4.3 of~\cite{RWapsubcr} to make the decisions of~$\cR$). 
Of course, later on in the exploration argument these dummy vertices need not be further tested for neighbours, but in the domination arguments of Lemma~7 and~9 of~\cite{RWapsubcr} we shall pretend that they are tested (this only generates more fictitious vertices, which is safe for upper bounds). 
Let $(k',r')$ be the `worst case' type of the component containing~$v$, which results after adding all $\binom{\ell}{2}$ edges in each $\ell$-tuple appearing (the actual type $(k,r)$ satisfies $k \le k'$ and $r \le r'$). 
One important observation is that the number of vertex nodes of~$\cT_{v,\delta_t}$ dominates $k'+(K+1)r' \ge k'+r'$.  
With this in mind, the corresponding variants of Lemma~7 and~9 of~\cite{RWapsubcr} yield exponential tails in~$k+r$. 
All other parts of the proof of Theorem~3 in~\cite{RWapsubcr} carry over with only obvious minor changes; we leave the details to the interested reader. 
\end{proof}
We close this section by noting a simple bound on the derivatives $q_{k,r}'(t)$.
\begin{lemma}\label{lem:Qkrd}
Define $b>0$ as in Theorem~\ref{thm:Qkr}. There is a constant $B'$ such that for all $k\ge 1$ and $r\ge 0$ we have $\sup_{t\in [t_0,t_1]}|q_{k,r}'(t)| \le B' k^3(r+1) e^{-b(k+r)}$. 
\end{lemma}
\begin{proof}
By \eqref{eq:qkr:tail}, each of the quantities~$F_h$ appearing in~\eqref{eq:der:qkr} and defined in
\eqref{eq:der:qkr:F1}--\eqref{eq:der:qkr:F3} is (crudely) at most a constant times $k^3(r+1)e^{-b(k+r)}$.
Recalling that $|\rho_k(t)|\le 1$ for all $k\in\cC=\{1,2,\ldots,K,\omega\}$, the claimed bound follows from~\eqref{eq:der:qkr}. 
\end{proof}

\subsection{Analyticity}\label{sec:AP}
In this subsection we use PDE theory to establish analytic properties of an idealized version of the parameter list~$\fS_i$, which will later be important for our branching processes analysis (see Section~\ref{sec:BPO} and Appendix~\ref{ss:BPsetup}).
In fact, we believe that our fairly general approach for establishing analyticity may be of independent interest. 
Before turning to the main result, we give a simple preparatory lemma. For the definition of (real) analytic, see Definition~\ref{def:analytic}.
\begin{lemma}\label{lem:ODE}
The functions $(\rho_k)_{k \in \cC}$, $q_{0,2}$ and $(q_{k,r})_{k\ge 1,r\ge 0}$ 
defined in Lemmas~\ref{lem:Nk:t}, \ref{lem:U:t} and~\ref{lem:Qkr:t} are analytic on~$(t_0,t_1)$.  
\end{lemma}
\begin{proof}
For $t \in [0,t_1]$ the functions $(\rho_k)_{k \in \cC}$ are the unique solution to a finite system~\eqref{eq:der:rhok} of ODEs.
The Cauchy--Kovalevskaya Theorem for ODEs (see Theorem~\ref{thm:CK:ODE} in Appendix~\ref{apx:CK}) thus implies that the $(\rho_k)_{k \in \cC}$ are analytic on $(0,t_1)$. 
Recalling~\eqref{eq:Nk:vL}--\eqref{eq:Nk:vk}, equation~\eqref{eq:der:u} shows that for $t \in [t_0,t_1)$ the derivative $q_{0,2}'$ is a polynomial 
function of known analytic functions. It follows that~$q_{0,2}$ is analytic on $(t_0,t_1)$. 
Finally, for $t \in [t_0,t_1)$, $k\ge 1$ and $r\ge 0$, the derivative $q'_{k,r}$ is a polynomial function of 
$(q_{\tk,\tr})_{1 \le \tk \le k,\, 0 \le \tr \le r}$ and known analytic functions, see~\eqref{eq:der:qkr}--\eqref{eq:der:qkr:F3}.  
By induction on $k+r$, the Cauchy--Kovalevskaya Theorem (see Theorem~\ref{thm:CK:ODE})
thus implies that each $q_{k,r}$ is
analytic on $(t_0,t_1)$, completing the proof. 
\end{proof}

With the functions $(q_{k,r})_{k \ge 1, r \ge 0}$ and $q_{0,2}$ as defined in Lemmas~\ref{lem:U:t} and~\ref{lem:Qkr:t},
our main aim in this section is to study the generating function
\begin{equation}\label{def:P}
P(t,x,y) := \sum_{k,r \ge 0} x^{k}y^{r} q_{k,r}(t) ,
\end{equation}
where (for notational convenience) we set
\begin{equation}\label{def:qkr:conv}
q_{0,r}(t) :\equiv 0 \quad \text{for all $r \neq 2$}. 
\end{equation}
Recall that $t_0=\tc-\sigma$ and $t_1=\tc+\sigma$. 
With~$b$ as in the exponential tail bound~\eqref{eq:qkr:tail}, let
\begin{equation}\label{def:D}
\cD := (t_0,t_1) \times \bigl(-e^{b/3},e^{b/3}\bigr)^2  \subset \RR^3.
\end{equation}
From~\eqref{eq:qkr:tail} it is easy to see that~$P=P(t,x,y)$ converges absolutely for $(t,x,y) \in \cD$. 
We shall now show that in fact $P$ is (real) analytic in this domain.%
\footnote{Substituting the exponential tails~\eqref{eq:qkr:tail} into the equation \eqref{eq:der:qkr} for $q'_{k,r}$, by analyzing the combinatorial structure of the derivatives $q^{(j)}_{k,r}$ it is possible to prove directly that $\sup_{t \in [t_0,t_1]}|q^{(j)}_{k,r}(t)| \le B_j e^{-b(k+r)/2}$, say. It is then not difficult to show smoothness (infinite differentiability) of~$P=P(t,x,y)$. What we prove in Theorem~\ref{thm:PDE} is stronger.}
In our approach the complementary conclusions of the differential equation method (the equations for $q'_{k,r}$) and the exploration tree approach (the exponential decay of $q_{k,r}$) work hand-in-hand with PDE theory (the Cauchy--Kovalevskaya Theorem). 
\begin{theorem}\label{thm:PDE}
The function $P(t,x,y)$ is analytic in the domain~$\cD$ defined in~\eqref{def:D}. 
More precisely, for each $\ttt_0\in (t_0,t_1)$ there is a $\delta>0$ such that the function $P(t,x,y)$
has an analytic extension to the complex domain $\fD_{\delta}(\ttt_0) := \{ (t,x,z) \in \CC^3: \: |t-\ttt_0| < \delta \text{ and } |x|,|y|<e^{b/3}\}$.
\end{theorem}
Our proof strategy is roughly as follows. 
Using the `nice' form of the differential equations $q'_{k,r}$ (given by Lemma~\ref{lem:Qkr:t}), we show that a minor modification of~$P$ satisfies a first-order~PDE of the form $P_t = F(t,x,y,P_x)$. 
A general result from the theory of partial differential equations (the Cauchy--Kovalevskaya Theorem, see Appendix~\ref{apx:CK}) then allows us to deduce that this PDE has an analytic local solution, say $\tP=\tP(t,x,y)$. 
Here the exponential tail of the $q_{k,r}$ (given by Theorem~\ref{thm:Qkr}) will be a crucial input, ensuring that the boundary data of the corresponding PDE are analytic. 
To show that~$\tP$ and~$P$ coincide for real~$t$ (first-order PDEs can, in general, have additional non-analytic solutions), we substitute the Taylor series $\tP(t,x,y)=\sum_{k,r,s}c_{k,r,s}x^ky^r(t-\ttt_0)^s$  back into both sides of the PDE, and essentially show that the functions 
$\tq_{k,r}(t) = \sum_{s}c_{k,r,s}(t-\ttt_0)^s$ satisfy the same system of ODEs as the functions $q_{k,r}(t)$. 
Exploiting that `nice' systems of ODEs have unique solutions we obtain $\tq_{k,r}(t)=q_{k,r}(t)$, which by~\eqref{def:P} establishes $\tP(t,x,y)=\sum_{k,r}x^ky^r\tq_{k,r}(t) = P(t,x,y)$ for real~$t$, as desired. 
\begin{proof}[Proof of Theorem~\ref{thm:PDE}]
With the definition~\eqref{def:P} of~$P$ and the convention~\eqref{def:qkr:conv} in mind, let 
\begin{equation}\label{def:R}
R(t,x,y) := \sum_{k \ge 1,\, r \ge 0} x^{k}y^{r} q_{k,r}(t) = P(t,x,y)- y^2q_{0,2}(t).
\end{equation}
We shall show that, for each real $\ttt_0 \in (t_0,t_1)$, the function $R=R(t,x,y)$ has an analytic extension to a complex domain $\cD_\delta(\ttt_0)$
as in the statement of the theorem.
By Lemma~\ref{lem:ODE} it follows that $P=P(t,x,y)$ has such an extension (after decreasing $\delta>0$, if necessary),
and in particular that~$P$ is analytic in~$\cD$.

We start with some basic properties of $R$ in the slightly larger domain
\[
 \cD^+ := \{ (t,x,z) \in \RR \times \CC^2: \: t \in (t_0,t_1) \text{ and } |x|,|y|<e^{b/2}\} ,
\] 
where $b$ is as in~\eqref{eq:qkr:tail} and~\eqref{def:D}. 
Since $|x^ky^r|\le e^{b(k+r)/2}$ and $|q_{k,r}(t)|\le B e^{-b(k+r)}$, the sum in~\eqref{def:R} converges
uniformly in $\cD^+$. For $(t,x,y)\in \cD^+$ we claim that the partial derivative $R_t$ satisfies
\begin{equation}\label{eq:Rt}
R_t := \frac{\partial}{\partial t} R(t,x,y) = \sum_{k \ge 1,\, r \ge 0} x^{k}y^{r} q'_{k,r}(t) .
\end{equation} 
By basic analysis, to see this it suffices to show uniform convergence of the sum on the right of~\eqref{eq:Rt}.
But this follows from bound $|q'_{k,r}(t)|\le B' k^3(r+1)e^{-b(k+r)}$ from Lemma~\ref{lem:Qkrd}. 
For $(t,x,y)\in \cD^+$ we similarly see~that 
\begin{equation}\label{eq:Rx}
R_x := \frac{\partial}{\partial x}R(t,x,y) = \sum_{k \ge 1,\, r \ge 0} kx^{k-1}y^{r} q_{k,r}(t) .
\end{equation}

The plan now is to substitute our formulae for the derivatives $q'_{k,r}$ into~\eqref{eq:Rt}
and then rewrite the resulting expression in terms of known expressions and functions (in order to eventually obtain a PDE for $R$).  
Turning to the details, for $\vc \in \cC^\ell$ with $\cC=\{1, \ldots, K, \omega\}$ we define for brevity
\begin{equation}\label{eq:der:Psi}
\Psi_{\vc} :=\prod_{j \in [\ell] \setminus \{j_1,j_2\}} \rho_{c_j}(t),
\end{equation}
where, as usual in this section, $\{j_1,j_2\}=\cR(\vc)$. Thus we may write~\eqref{eq:der:qkr} as
\begin{equation}\label{eq:der:qkr2}
q'_{k,r}(t)  = \sum_{\vc \in \cC^{\ell}} \Psi_{\vc} \sum_{1 \le h \le 3} F_h(k,r,\vc),
\end{equation}
with the $F_h$ defined in~\eqref{eq:der:qkr:F1}--\eqref{eq:der:qkr:F3}.
Substituting the formula for $q'_{k,r}$ given by~\eqref{eq:der:qkr2} into~\eqref{eq:Rt}, we see that
\begin{equation}\label{eq:Pt}
R_t = \sum_{\vc \in \cC^{\ell}} \Psi_{\vc}  \sum_{k \ge 1,\, r \ge 0} x^{k}y^{r} \: \sum_{1 \le h \le 3} F_h(k,r,\vc).
\end{equation}

Recalling that components with more than~$K$ vertices are formally assigned size~$\omega$,
we expect that in~\eqref{eq:Pt} almost all terms in the multiple sum 
come from the case $(c_{j_1},c_{j_2})=(\omega,\omega)$. 
With this in mind, let~$\tF_1$--$\tF_3$ be modified versions of~$F_1$--$F_3$ where all conditions $c_{j_x}=s(\cdot,\cdot)$ in~\eqref{eq:der:qkr:F1}--\eqref{eq:der:qkr:F3}  are replaced by $c_{j_x}=\omega$.
Thus
\begin{equation}\label{eq:tF1}
 \tF_1(k,r,\vc) = \sum_{\substack{k_1+k_2=k: \: k_1,k_2 \ge 1\\r_1+r_2=r: \: r_1,r_2 \ge 0}} k_1 q_{k_1,r_1} k_2 q_{k_2,r_2} \indic{c_{j_1}=\omega, \: c_{j_2}=\omega},
\end{equation}
recalling that $\{j_1,j_2\}=\cR(\vc)$.
Since the indicator function can be moved outside the sum, using $x^{k-2}y^rk_1k_2 = k_1x^{k_1-1}y^{r_1}k_2x^{k_2-1}y^{r_2}$ and then~\eqref{eq:Rx} we see that 
\begin{equation}\label{eq:Pt:F1:0}
\begin{split}
\sum_{k \ge 1,\, r \ge 0} x^{k}y^{r} \tF_1(k,r,\vc) &= \indic{c_{j_1}=c_{j_2}=\omega}x^2\sum_{k \ge 1,\, r \ge 0} \ \sum_{\substack{k_1+k_2=k: \: k_1, k_2 \ge 1\\ r_1+r_2=r: \: r_1, r_2 \ge 0}} x^{k-2}y^r k_1q_{k_1,r_1} k_2q_{k_2,r_2} \\
& = \indic{c_{j_1}=c_{j_2}=\omega} x^2 (R_x)^2 .
\end{split}
\end{equation}
Proceeding analogously, since
\[
 \tF_2(k,r,\vc) = 2 \indic{r \ge 1}k q_{k,r-1} \rho_{\omega}(t_0) \indic{c_{j_1}=\omega, \: c_{j_2}=\omega},
\]
using $kx^ky^r = xy k x^{k-1}y^{r-1}$ and~\eqref{eq:Rx} we deduce 
\begin{equation}\label{eq:Pt:F2:0}
\begin{split}
\sum_{k \ge 1,\, r \ge 0} x^{k}y^{r} \tF_2(k,r,\vc) &= 2 \rho_{\omega}(t_0) \indic{c_{j_1}=c_{j_2}=\omega} xy \sum_{k,r\ge 1}  kx^{k-1}y^{r-1}q_{k,r-1} \\ 
& = 2 \rho_{\omega}(t_0) \indic{c_{j_1}=c_{j_2}=\omega} xy R_x .
\end{split}
\end{equation}
Similarly, but more simply, 
\begin{equation}\label{eq:Pt:F3:0}
 \sum_{k \ge 1,\, r \ge 0} x^{k}y^{r} \tF_3(k,r,\vc) = -xR_x \bigl[ \indic{c_{j_1}=\omega} \rho_{c_{j_2}}(t)  + \rho_{c_{j_1}}(t)\indic{c_{j_2}=\omega}\bigr] .
\end{equation}
Since $\Psi_{\vc}$, defined in~\eqref{eq:der:Psi}, is a polynomial in the functions $(\rho_k)_{k \in \cC}$, multiplying \eqref{eq:Pt:F1:0}--\eqref{eq:Pt:F3:0} by $\Psi_{\vc}$ and summing
over the finite set $\cC^\ell$, we see that
\begin{equation}\label{eq:Pt:F1F2F3}
\sum_{\vc \in \cC^{\ell}} \Psi_{\vc} \sum_{k \ge 1,\, r \ge 0} x^{k}y^{r} \: \sum_{1 \le h \le 3} \tF_h(k,r,\vc) = x^2 (R_x)^2f_1(t)  +  xy R_x f_2(t)  +  x R_x f_3(t),
\end{equation}
where each $f_h(t)$ is a polynomial in the functions $(\rho_j)_{j \in \cC}$.

We now turn to the differences $F_h-\tF_h$ for $h \in \{2,3\}$. 
Recalling~\eqref{def:skr}, for~$F_2$ defined in~\eqref{eq:der:qkr:F2} we note that $s(k,r-1)\ne\omega$ if and only if $k \le K$ and $r=1$. Hence $F_2(k,r,\vc)=\tF_2(k,r,\vc)$ whenever $k > K$ or $r \neq 1$. 
For~$F_3$ defined in~\eqref{eq:der:qkr:F3} we note that $s(k,r)\ne\omega$ 
if and only if $k \le K$ and $r=0$. Hence $F_3(k,r,\vc)=\tF_3(k,r,\vc)$ whenever $k > K$ or $r >0$. 
Considering the finite number of cases with $F_h \neq \tF_h$, it follows that 
\begin{equation}\label{eq:Pt:F2F3}
\sum_{\vc \in \cC^{\ell}} \Psi_{\vc} \sum_{k \ge 1,\, r \ge 0} x^{k}y^{r} \sum_{h \in \{2,3\}}\bigl[F_h(k,r,\vc)-\tF_h(k,r,\vc)\bigr]= \sum_{1 \le k \le K}\sum_{0 \le r \le 1} x^ky^r f_{k,r}(t) ,
\end{equation}
where each $f_{k,r}(t)$ is a polynomial in the functions $(\rho_j)_{j \in \cC}$ and $(q_{j,0})_{1 \le j \le K}$.

For the difference $F_1-\tF_1$ more care is needed. 
Subtracting~\eqref{eq:tF1} from~\eqref{eq:der:qkr:F1} we have
\begin{equation}\label{eq:F1d}
F_1(k,r,\vc) -\tF_1(k,r,\vc) 
 = \sum_{\substack{k_1+k_2=k: \: k_1,k_2 \ge 1\\r_1+r_2=r: \: r_1,r_2 \ge 0}} k_1 q_{k_1,r_1} k_2 q_{k_2,r_2} \bigl[\indic{c_{j_1}=s(k_1,r_1), \: c_{j_2}=s(k_2,r_2)} - \indic{c_{j_1}=c_{j_2}=\omega}\bigr], 
\end{equation}
where $\{j_1,j_2\}=\cR(\vc)$, as usual.
Since $s(k_h,r_h)\ne\omega$ implies $k_h \le K$ and $r_h=0$, 
for $(k,r)=(k_1+k_2, r_1+r_2)$ with $k > 2K$ or $r \ge 1$ \emph{at most one} of the two events $s(k_1,r_1) \neq \omega$ and $s(k_2,r_2) \neq \omega$ can occur. 
Thus, considering these events in turn,
when $(k,r)$ satisfies $k>2K$ or $r\ge 1$ we have
\[ 
 F_1(k,r,\vc) -\tF_1(k,r,\vc)  = \Delta_1(k,r,\vc)+\Delta_2(k,r,\vc)
\]
where 
\[
 \Delta_1(k,r,\vc) := \sum_{1 \le k_1 \le K} \indic{k > k_1} k_1 q_{k_1,0} (k-k_1) q_{k-k_1,r} \bigl[\indic{c_{j_1}=k_1}-\indic{c_{j_1}=\omega}\bigr] \indic{c_{j_2}=\omega},
\]
and $\Delta_2= \Delta_2(k,r,\vc)$ is defined similarly, swapping the roles of $(k_1,r_1,j_1)$ and $(k_2,r_2,j_2)$.
Relabeling $k_2$ in the sum in $\Delta_2$ as $k_1$, and changing the order of the product, we can write
$\Delta_2$ in the same form as $\Delta_1$ but with different indicator functions. It follows that
when $(k,r)$ satisfies $k>2K$ or $r\ge 1$, then 
\begin{equation}\label{eq:F1tF1}
 F_1(k,r,\vc) -\tF_1(k,r,\vc) = \sum_{1 \le k_1 \le K} \indic{k > k_1} k_1 q_{k_1,0} (k-k_1) q_{k-k_1,r} I(\vc,k_1) 
\end{equation}
for some coefficients $I(\vc,k_1)\in \{-2,-1,0,1,2\}$.
Writing $x^ky^r(k-k_1)=x^{k_1+1} (k-k_1)x^{k-k_1-1}y^r$, we have
\begin{equation}\label{eq:F1tF1m}
 \sum_{k \ge 1,\,r\ge 0} x^ky^r \indic{k > k_1} k_1 q_{k_1,0} (k-k_1) q_{k-k_1,r} I(\vc,k_1)= 
I(\vc,k_1) x^{k_1+1} k_1q_{k_1,0}  R_x.
\end{equation}
Now the formula~\eqref{eq:F1tF1} does not apply to $(k,r)$ with $k\le 2K$ and $r=0$. 
However, there are only finitely many such terms in~\eqref{eq:F1d}. 
Using~\eqref{eq:F1tF1}--\eqref{eq:F1tF1m} it thus follows that
\begin{equation}\label{eq:Pt:F1}
\begin{split}
\sum_{\vc \in \cC^{\ell}} \Psi_{\vc} \sum_{k \ge 1,\, r \ge 0} x^{k}y^{r} \bigl[F_1(k,r,\vc)-\tF_1(k,r,\vc)\bigr] &= \sum_{1 \le k_1 \le K} x^{k_1+1}R_x g_{k_1}(t)  + \sum_{1 \le k \le 2K}x^k g_{k,0}(t), 
\end{split}
\end{equation}
where each $g_{k_1}(t)$ and $g_{k,0}(t)$ is a polynomial in the functions $(\rho_j)_{j \in \cC}$ and $(q_{j,0})_{1 \le j \le 2K}$, say. 
 
To sum up, writing $F_h = \tF_h + (F_h-\tF_h)$ and substituting~\eqref{eq:Pt:F1F2F3}, \eqref{eq:Pt:F2F3} and~\eqref{eq:Pt:F1} and into~\eqref{eq:Pt}, for $(t,x,y) \in \cD^+$ we arrive at a first-order PDE of the form 
\begin{equation}\label{eq:Pt:PDE}
R_t = \sum_{1 \le k \le 2K}\sum_{0 \le r \le 1}\sum_{0 \le s \le 2} f_{k,r,s}(t)x^k y^r (R_x)^s,
\end{equation}
where each function $f_{k,r,s}(t)$ is a polynomial in the functions $(\rho_j)_{j \in \cC}$ and $(q_{j,0})_{1 \le j \le 2K}$. 
Since finite sums and products of analytic functions are analytic, by Lemma~\ref{lem:ODE} it follows that the functions $f_{k,r,s}(t)$ are analytic in $(t_0,t_1)$. 
Given $\ttt_0\in (t_0,t_1)$, we shall formally define the initial data of the PDE via
\begin{equation}\label{eq:P:u0}
 R(\ttt_0,x,y)=R^{\ttt_0}(x,y) ,
\end{equation}
where $R^{\ttt_0}(x,y):=R(\ttt_0,x,y)$ is a two-variable function. 
For each $t \in (t_0,t_1)$ we claim that $R^t(x,y):=R(t,x,y)$ is analytic in the \emph{complex} domain 
\[
\cX = \cX(b/2) := \bigl\{(x,y) \in \CC^2 : \: |x|,|y|<e^{b/2}\bigr\} .
\]
This is routine: with~$t \in (t_0,t_1)$ fixed, \eqref{def:R} defines $R^t(x,y)$ 
as a power series which, as noted earlier, converges absolutely if $|x|,|y| \le e^{b/2}$. 
By standard results, $R^t$ is thus analytic in~$\cX$.

The plan now is to fix $\ttt_0\in (t_0,t_1)$ and construct an
analytic local solution $\tR=\tR(t,x,y)$ to the PDE~\eqref{eq:Pt:PDE} 
with initial data \eqref{eq:P:u0}. Then we shall show that for real $t$ near $\ttt_0$ and complex $x,y$ with $|x|,|y| \le e^{b/3}$,
this solution coincides with $R=R(t,x,y)$ as defined in~\eqref{def:R}. 
This shows that $\tR$ is the required analytic extension of $R$. 
Turning to the details, fix $\ttt_0 \in (t_0,t_1)$. 
Since the functions $(\rho_j)_{j \in \cC}$ and $(q_{j,0})_{1 \le j \le 2K}$ are real analytic, each
has a complex analytic extension to a neighbourhood of $\ttt_0$. Thus we may extend these functions
simultaneously to a complex domain of the form 
\[
\cT = \cT(\eps) : = \{t \in \CC: \: |t-\ttt_0| < \eps\} ,
\]
where $\eps=\eps(\ttt_0)>0$. Recalling that functions $f_{k,r,s}(t)$ appearing in~\eqref{eq:Pt:PDE}
are polynomials of $(\rho_j)_{j \in \cC}$ and $(q_{j,0})_{1 \le j \le 2K}$, the $f_{k,r,s}(t)$ also extend to $\cT$, and from \eqref{eq:Pt:PDE} we see that there is thus a (complex)
analytic function $F:\cT \times \cX \times \CC \to \CC$ such that for $(t,x,y)\in\cT\times \cX$ the first-order PDE~\eqref{eq:Pt:PDE}--\eqref{eq:P:u0} may be written as $R_t = F(t,x,y,R_x)$, with analytic initial data~\eqref{eq:P:u0} for $(x,y) \in \cX$. 
Applying a convenient version of the Cauchy--Kovalevskaya Theorem for first-order PDEs (see Theorem~\ref{thm:CK:PDE} in Appendix~\ref{apx:CK}), there exists $0 < \delta < \eps$ such that with
\[
\fD = \fD_{\delta}(\ttt_0) := \cT(\delta) \times \cX(b/3)\subset \CC^3
\] the following holds: there is a function 
\begin{equation*}
\tR(t,x,y) = \sum_{k,r,s \ge 0}c_{k,r,s}x^ky^r(t-\ttt_0)^s
\end{equation*}
which is analytic in the complex domain~$\fD$, 
and which satisfies~\eqref{eq:Pt:PDE}--\eqref{eq:P:u0} for all $(t,x,y) \in \fD$ (with~$R$ replaced by~$\tR$). 
Define
\[
\tq_{k,r}(t) := \sum_{s \ge 0}c_{k,r,s}(t-\ttt_0)^s ,
\] 
so that~$\tR$ can be written as 
\begin{equation}\label{eq:P:Q1}
\tR(t,x,y) = \sum_{k,r \ge 0}x^ky^r\tq_{k,r}(t) .
\end{equation}
Now $\tR(\ttt_0,x,y)$ and $R^{\ttt_0}(x,y)=R(\ttt_0,x,y)$ are both analytic for $(x,y) \in \cX(b/3)$,
and they agree on this domain. By the uniqueness of Taylor series it follows that 
\begin{equation}\label{eq:P:pq0}
\tq_{k,r}(\ttt_0) = 
\begin{cases} 
q_{k,r}(\ttt_0) & \text{if $k \ge 1$ and $r \ge 0$,}\\ 
0  & \text{otherwise.}
\end{cases}
\end{equation}
Note that all terms on the right hand side of the partial time-derivative~\eqref{eq:Pt:PDE} contain some power $x^j$ with~${j \ge 1}$. 
Since $\tR=\tR(t,x,y)$ satisfies~\eqref{eq:Pt:PDE} (with~$R$ replaced by~$\tR$), using $\frac{\partial}{\partial t}\tR(t,x,y) = \sum_{k,r \ge 0}x^ky^r\tq'_{k,r}(t)$ we readily infer $\tq'_{0,r}(t)=0$, which by~\eqref{eq:P:pq0} implies $\tq_{0,r}(t) \equiv 0$. 
Note that~$\tR$ from~\eqref{eq:P:Q1} can thus be simplified~to  
\begin{equation}\label{eq:P:Q2}
\tR(t,x,y) = \sum_{k \ge 1,\,r \ge 0}x^ky^r\tq_{k,r}(t) ,
\end{equation}
which closely mimics the form of~$R$ defined in~\eqref{def:R}. 
For brevity, we shall henceforth always tacitly assume~${k \ge 1}$ and~${r \ge 0}$. 
Substituting~$\tR=\tR(t,x,y)$ as written in~\eqref{eq:P:Q2} into both sides of~\eqref{eq:Pt:PDE} (with~$R$ replaced by~$\tR$), we now compare the coefficients of the $x^ky^r$ terms on both sides. 
By tracing back how we arrived at~\eqref{eq:Pt:F1F2F3}--\eqref{eq:Pt:F2F3} and \eqref{eq:Pt:F1}--\eqref{eq:Pt:PDE}, it is not difficult to see (by effectively doing all our calculations `in reverse') that we obtain differential equations of the form\footnote{Note that $\tq'_{k,r}$ in~\eqref{eq:P:pDE} contains some $q_{j,0}$ terms, which arise due to the functions $f_{k,r,s}$ in~\eqref{eq:Pt:PDE}.} 
\begin{equation}\label{eq:P:pDE}
\tq'_{k,r}(t) = J_{k,r}\Bigl(\bigl(\rho_j(t)\bigr)_{j \in \cC}, \: \bigl(\tq_{i,j}(t)\bigr)_{\substack{1 \le i \le k\\0 \le j \le r}}, \: \bigl(q_{j,0}(t)\bigr)_{1 \le j \le 2K}\Bigr) ,
\end{equation}
where the structure of the polynomial functions $J_{k,r}$ coincides with the derivatives of the $q_{k,r}$ in the sense that these satisfy  
\begin{equation}\label{eq:P:qDE}
q'_{k,r}(t) = J_{k,r}\Bigl(\bigl(\rho_j(t)\bigr)_{j \in \cC}, \: \bigl(q_{i,j}(t)\bigr)_{\substack{1 \le i \le k\\0 \le j \le r}}, \: \bigl(q_{j,0}(t)\bigr)_{1 \le j \le 2K}\Bigr) .
\end{equation}
Analogous to Lemma~\ref{lem:Qkr:t}, the form of the infinite system of real 
differential equations~\eqref{eq:P:pq0} and~\eqref{eq:P:pDE} ensures that it has a unique solution $(\tq_{k,r})_{k \ge 1,r\ge 0}$ in~$(\ttt_0-\delta,\ttt_0+\delta)$.
From~\eqref{eq:P:qDE}, the functions $q_{k,r}(t)$ satisfy this system of differential equations (the boundary condition holds
trivially), so $q_{k,r}(t)=\tq_{k,r}(t)$ in this (real) interval, 
and hence $R(t,x,y)=\tR(t,x,y)$ for all $(t,x,y) \in \fD=\fD_{\delta}(\ttt_0)$ with real~$t \in \RR$.
Since~$\tR$ is analytic in~$\fD_{\delta}(\ttt_0)$, this shows that~$R$ has the required analytic extension, completing the proof.
\end{proof}
Note that if the PDE~\eqref{eq:Pt:PDE}--\eqref{eq:P:u0} had a unique solution, then the last part of the above proof would be redundant (where we show that analytic local solutions extend~$R$).
For the interested reader we mention that $S:=xR_x$ satisfies a quasi-linear PDE, where $S(t,1,1)=\sum_{k \ge 1,\,r\ge 0} kq_{k,r}(t)$ is constant by Remark~\ref{rem:Qkr:sum}.

Keeping the notational convention~\eqref{def:qkr:conv}, we now derive some basic properties of 
\begin{equation}\label{def:W}
u(t) := \sum_{k,r \ge 0}r (r-1) q_{k,r}(t) =P_{yy}(t,1,1) .
\end{equation}
\begin{lemma}\label{lem:w}
The function $u(t)$ is (real) analytic for $t \in (t_0,t_1)$, with $u(t)\ge u(t_0)=0$.
Furthermore, $u'(t) > 0$ for all $t \in (t_0,t_1)$.  
\end{lemma}
That $u'(t) > 0$ is fairly intuitive by (i) noting that we have $q'_{0,2}(t) > 0$ by Lemma~\ref{lem:U:t}, and (ii) observing that the discrete random variable $W(i)=\sum_{k \ge 1,\,r \ge 0} r(r-1)Q_{k,r}(i)$ is non-decreasing. 
\begin{proof}[Proof of Lemma~\ref{lem:w}]
Observing that $u(t)=P_{yy}(t,1,1)$, Theorem~\ref{thm:PDE} immediately implies that $u(t)$ is analytic for $t \in (t_0,t_1)$. 
Furthermore, by Lemmas~\ref{lem:U:t} and~\ref{lem:Qkr:t} we have $q_{k,r}(t) \ge 0$, which implies $u(t) \ge 0$. 
By Lemma~\ref{lem:U:t} and~\eqref{eq:der:qkr:init} of Lemma~\ref{lem:Qkr:t} we also have $q_{k,r}(t_0)=0$ for all $r \ge 1$, which readily yields $u(t_0) = 0$.  

It remains to establish that $u'(t)>0$. Recalling the convention~\eqref{def:qkr:conv} we have $u(t) = 2 q_{0,2}(t) + w(t)$ where $w(t) = \sum_{k \ge 1,\,r \ge 2}r (r-1) q_{k,r}(t)$.
Since $q'_{0,2}(t) > 0$ by Lemma~\ref{lem:U:t}, it suffices to prove $w'(t) \ge 0$ for $t \in (t_0,t_1)$. 
The basic strategy is to compare $w(t)$ with $W(i)=\sum_{k \ge 1,\,r \ge 2} r(r-1)Q_{k,r}(i)$. 
Combining the inequalities \eqref{eq:Qkr}--\eqref{eq:qkr:tail} of Theorem~\ref{thm:Qkr} (which imply $Q_{k,r}(i) = 0$ if $k+r \ge (\log n)^2$, say), it follows that~whp 
\begin{equation}\label{eq:W:approx}
\max_{i_0 \le i \le i_1}\bigl|W(i) - w(i/n) n\bigr| \le (\log n)^{D_Q + 9} n^{1/2}  .
\end{equation} 
To prove $w'(t) \ge 0$ for $t \in (t_0,t_1)$, it suffices to show $w(\tau_2) \ge w(\tau_1)$ for all $t_0 < \tau_1 \le \tau_2 < t_1$. 
Here the key observation is that $W(i)$ cannot decrease in any step (after adding a $V_S$--$V_L$ or $V_S$--$V_S$ edge we have $W(i+1) \ge W(i)$, and $V_L$--$V_L$ edges are irrelevant). 
Indeed, recalling $i_j = t_j n$, using~\eqref{eq:W:approx} it follows for all $t_0 < \tau_1 \le \tau_2 < t_1$ that 
\[
w(\tau_2)-w(\tau_1) \ge - 2(\log n)^{D_Q + 9} n^{-1/2} .
\]
Since $w(t)$ does not depend on $n$, we thus have $w(\tau_2) -w(\tau_1) \ge 0$, completing the proof. 
\end{proof}

\subsection{Sprinkling}\label{sec:SP}
In this subsection we introduce a dynamic variant of the classical Erd{\H o}s--R{\'e}nyi sprinkling argument from~\cite{ER1960} (see Lemmas~4 and~6 of~\cite{RWapcont} for related arguments), which will later be key for studying the size of the largest component of~$G_i$ (see Section~\ref{sec:L1:super}). 
Intuitively, sprinkling quantifies the following idea: if there are many vertices in large components, then most of these components should quickly merge to form a `giant' component as the process evolves.
Later we shall apply Lemma~\ref{lem:SP} below with $\Lambda=\omega(\eps^{-2})$, $x=\Theta(\eps n)$ and $\xi= o(1)$ chosen such that $\Delta_{\Lambda,x,\xi} = o(\eps n)$ and $x/\Lambda = \omega(1)$. 
\begin{lemma}
\label{lem:SP}
For any bounded size $\ell$-vertex rule with cut-off $K$ there are constants $\lambda,\eta>0$ such that the following holds. 
Let $\cS_{i,\Lambda,x,\xi}$ denote the event that $N_{\ge \Lambda}(i) \ge x$ implies $L_1(i+\Delta_{\Lambda,x,\xi}) \ge (1-\xi) N_{\ge \Lambda}(i)$, where $\Delta_{\Lambda,x,\xi} = \lambda n^2/(\xi \Lambda x)$. 
Then for all $i \ge i_0$, $x \ge \Lambda > K$ and $\xi >0$ we have $\Pr(\neg\cS_{i,\Lambda,x,\xi}) \le \exp(-\eta x/\Lambda) + n^{-\omega(1)}$. 
\end{lemma}
\begin{proof}
We may assume $\xi \in (0,1)$, since the claim is trivial otherwise. 
As $N_{\omega}(i)$ is monotone increasing in~$i$, 
by~\eqref{eq:Nk:t} and Lemma~\ref{lem:rhokder} there is a constant 
$\alpha>0$ such that, with probability $1-n^{-\omega(1)}$, we have 
\begin{equation*}
\min_{i \ge i_0}N_{\omega}(i) \ge N_{\omega}(i_0) \ge \alpha n . 
\end{equation*}
Let $\lambda = 4/\alpha^{\ell-2}$ and $\eta=1/9$. 
Note that, conditional on $G_i$ satisfying $N_{\omega}(i) \ge \alpha n$ and $N_{\ge \Lambda}(i) \ge x$, 
it suffices to show that $\cS_{i,\Lambda,x,\xi}$ fails with (conditional) probability at most $\exp(-\eta x/\Lambda)$. 
Turning to the details, let~$W$ denote the union of all components of~$G_i$ with size at least~$\Lambda$.  
Clearly, the number of components of~$G_{j}$ meeting~$W$ is (a)~at most $|W|/\Lambda$ in step~$j=i$, and (b)~monotone decreasing as~$j$ increases.
Moreover, until there is a component containing at least $(1-\xi)|W|$ vertices, in each step we have probability at least
\begin{equation}\label{eq:SP:q}
\min_{i' \ge i}\left(\frac{N_{\omega}(i')}{n}\right)^{\ell-2} \frac{|W|}{n} \frac{\xi|W|}{n} \ge \alpha^{\ell-2} \xi \left(\frac{|W|}{n}\right)^2 =: q
\end{equation}
of joining two vertices from~$W$ that are in distinct components (equation~\eqref{eq:SP:q} exploits that the bounded size $\ell$-vertex rule has cut-off $K < \Lambda$), 
in which case the number of components meeting~$W$ reduces by one; for later reference we call such steps \emph{joining}. 
Recalling $|W|= N_{\ge \Lambda}(i) \ge x$ and $x \ge \Lambda$, define 
\begin{equation*}
M := \ceilL{\frac{2}{q} \cdot \frac{|W|}{\Lambda}} \le \frac{4 n^2}{\alpha^{\ell-2}\xi \Lambda |W|} \le \frac{4 n^2}{\alpha^{\ell-2}\xi \Lambda x} = \Delta_{\Lambda,x,\xi},
\end{equation*}
and note that $qM \ge 2x/\Lambda$. 
Starting from~$G_i$, using standard Chernoff bounds (and stochastic domination) it follows that, with probability at least $1-\exp\bigl(- \eta x/\Lambda\bigr)$, say,   
after at most $M$ additional steps either 
(i)~at least $qM/2 \ge |W|/\Lambda$ joining steps occurred, which is impossible, 
or
(ii)~there is a component containing at least $(1-\xi)|W|$ vertices. 
In case (ii) we have $L_1(i+M) \ge (1-\xi)|W| = (1-\xi)N_{\ge \Lambda}(i)$, which completes the~proof.  
\end{proof}
%


\subsection{Periodicity and reachable component sizes}\label{sec:period}
In this subsection we study the component sizes that can appear  
in the random graph process~$(G^{\cR}_{n,i})_{i \ge 0}$.
In a standard Achlioptas process (i.e., an `edge rule'), all component sizes
are possible, since if the rule is presented with two potential edges each
of which would join an isolated vertex to a component of size $k$, then it
must form a component of size $k+1$. Indeed, this observation easily leads
to an inductive lower bound on the rate of formation of components of size $k$
as a function of $k$ and $t$ (see below). For $\ell$-vertex rules, this need not be the case.
For example, there are $3$-vertex rules that never join components of size~$1$ and~$2$,
and so never form a component of size~$3$. 
Indeed, the rule may always choose 
a larger component to join to something else, if available, and if presented
with three components of sizes $1$ and~$2$, may join two of the same size.

This phenomenon affects not only small component sizes: there are bounded-size rules which never join an
isolated vertex to any component other than another isolated vertex, for example, and so
only create components (other than isolated vertices) of even size.  This effect needs to be
taken into account when considering the large-$k$ asymptotics of $\rho_k(t)$ and $N_k(i)$.
Fortunately, for the asymptotics, there is only one relevant parameter, the `period' $\per$ of the rule, defined below.
In an Achlioptas process we have $\per=1$; handling the general case requires no major
new ideas, but is a little fiddly. The reader may thus wish to skip the rest of this section,
and to focus on the (most important) case $\per=1$ throughout the paper.

Recall that an $\ell$-vertex size rule $\cR$ is defined by a function $\cR:(\NN^+)^\ell \to \binom{[\ell]}{2}$
giving the (distinct) indices of the vertices that the rule will join when
presented with vertices in components of size $(c_1,\ldots,c_\ell)$. For each such size-vector
$\vc$, let $s(\vc)$ denote the size of the resulting component, in the case that the edge added
does join two components. (If not, no new component size results.) Thus $s(\vc)=c_{j_1}+c_{j_2}$,
where $\cR(\vc)=\{j_1,j_2\}$. Define the set $\cS=\cSR$ of \emph{reachable component sizes}
to be the smallest subset of the positive integers such~that
\[ 
 1 \in \cS  \qquad \text{and} \qquad \vc\in \cS^\ell \implies s(\vc)\in \cS.
\]
In other words, $\cS=\bigcup_{r\ge 0} \cS_r$, where
\begin{equation}\label{eq:cSr}
 \cS_0=\{1\} \qquad \text{and} \qquad \cS_{r+1}=\cS_r \cup \{s(\vc):\vc\in \cS_r^\ell\}.
\end{equation}

If $\cR$ is an Achlioptas process, then we have $\cSR=\NN^+$.
Indeed, an Achlioptas process is a 4-vertex
rule in which $\{j_1,j_2\}$ is always either $\{1,2\}$ or $\{3,4\}$. For such a rule
$s(k,1,k,1)=k+1$, and it follows by induction that $k\in \cSR$ for all $k\ge 1$. 
For bounded-size $\ell$-vertex rules we now record that, beyond the cut-off~$K$, all elements of~$\cSR$ are multiples of~$\per$, 
and that beyond some perhaps larger integer, all multiples of~$\per$ are in~$\cSR$.

\begin{lemma}\label{lem:allowed}
Let $\cR$ be a bounded-size $\ell$-vertex rule with cut-off $K$.
Then there are integers $\per,\kR \ge 1$, 
with~$\per$ a power of two,
such that for $k\ge \kR$ we have $k\in \cSR$ if and only if $k$ is a multiple of $\per$.
Furthermore, if $k\in \cSR$ and $k>K$, then $k$ is a multiple of $\per$. 
\end{lemma}

\begin{proof}
We shall write $\cS$ for $\cSR$ to avoid clutter. 
Let $a$ be the smallest integer such that $2^a > K$, 
and define $\cS^+=\{i\in \cS: i\ge 2^a\}$.

For any rule, we have $s(i,i,\ldots,i)=2i$, so $i\in \cS$ implies $2i\in \cS$ and $\cS$
contains all powers of two (as used earlier in the proof of Lemma~\ref{lem:rhokder}). 
If a size-vector $\vc$ has all $c_i\ge 2^a$, then the bounded-size rule `sees' only large components (size $>K$), 
and so makes some fixed choice, say (relabeling if needed)
$\{j_1,j_2\}=\{1,2\}$. It follows that if $i\in \cS^+$ then $i+2^a\in \cS^+$,
since $s(i,2^a,\ldots,2^a)=i+2^a$. 

Let $\bSp$ be the set of residue classes modulo~$2^a$ that appear
in $\cS^+$. Within each residue class in $\bSp$, if $k$ is the smallest element of the class included
in $\cS^+$, then we have $k,k+2^a,k+2\cdot 2^a,\ldots \in \cS^+$.
It follows that beyond some constant $\IR$ (which may be significantly
larger than~$2^a$) we have $i\in \cS^+$ if and only if $i$ is in one of the classes in $\bSp$.

Considering the case when one component has size $i$ and the others have size $j$, we see
(as above) that $i,j\in \cS^+$ implies $i+j\in \cS^+$. Hence $\bSp$ is closed
under addition, and is thus a subgroup of $\ZZ/2^a\ZZ$. Hence
there is some number $\per$, a divisor of~$2^a$, such that
$\bSp$ consists of all multiples of $\per$. Hence, 
for $i\ge \kR := \max\{\IR,2^a\}$, 
we have $i\in \cS$ (which is equivalent to $i\in \cS^+$) if and only if $i$
is a multiple of~$\per$. 

For the final statement, if $k > K$, then we have $k+m2^a\in \cS$ for all
$m\ge 0$, by induction on $m$ (using $s(k+i2^a,2^a,\ldots,2^a)=k+(i+1)2^a$, as above). 
Since $k+m2^a\ge \kR$ for $m$ large enough, 
we have that $\per$ divides $k+m2^a$. 
Since $\per$ is a divisor of $2^a$, we deduce that~$\per$ divides~$k$. 
\end{proof}

We call $\per$, which is uniquely defined by the properties given in Lemma~\ref{lem:allowed}, the \emph{period} of the rule.
The constant $\kR$ is not uniquely defined; for definiteness, we may take it to be the minimal integer
with the given property. 
For Achlioptas process we have $\per=1$ and $\kR=1$ since $\cSR=\NN^+$ (as noted above). 
We state this observation as a lemma for ease of reference.

\begin{lemma}\label{lem:allowed:AP}
If $\cR$ corresponds to a bounded-size Achlioptas process, then $\cSR=\NN^+$, and $\per=\kR=1$. \noproof
\end{lemma}

The component sizes in $\cSR$ are all those than can possibly appear in $G^{\cR}_{n,i}$; we next note
for any $k\in \cSR$, we will see many components of this size -- for $i=\Theta(n)$ and $n$ large enough, on average
a constant fraction of the vertices will be in components of size $k$. 
Recall that $N_{k}(tn) \approx \rho_k(t) n$ (see Lemma~\ref{lem:Nk2:t} and Theorem~\ref{thm:init}).

\begin{lemma}\label{lem:rhokS1}
Let $\cR$ be a bounded-size rule.
Let $(\rho_k)_{k \ge 1}$ be the functions defined in Lemma~\ref{lem:Nk2:t}.
If $t \in (0,t_1]$ and $k \ge 1$, then $\rho_k(t)>0$ if and only if $k \in \cSR$.
\end{lemma}
\begin{proof}
From Lemma~\ref{lem:Nk2:t} it follows that, for each fixed $k \ge 1$, whp we have
\begin{equation}\label{eq:Nk:rhokS1}
\max_{0 \le i \le t_1 n}|N_k(i)/n - \rho_k(i/n)| \le (\log n) n^{-1/2} = o(1) . 
\end{equation}
In the case $k \not\in \cSR$, by construction we have $N_k(i)=0$ for all $i \ge 0$ (with probability one), 
so $\rho_k(t)=0$ for all $t \in [0,t_1]$. 

We now turn to the case $k \in \cSR$.
Instead of adapting the differential inequality~\eqref{eq:rhokder} based proof of Lemma~\ref{lem:rhokder}, 
we here give a perhaps more intuitive alternative argument, which extends more easily to the functions $q_{k,r}$. 
If $0 \le i'<i$ and $C$ is a component of $G^{\cR}_{n,i'}$ with $k$ vertices, then $C$ is also a component of $G^{\cR}_{n,i}$
with (conditional) probability at least $(1-k/n)^{\ell(i-i')}$, simply by considering the event that none of the $\ell$ random
vertices in any of steps $i'+1,\ldots,i$ falls in $C$ (see Lemma~5 in~\cite{RWapcont} for similar reasoning). 
For~$k$ constant and $i-i'=O(n)$, this probability is $\exp(-k\ell (i-i')/n +o(1))$, 
so $\E(N_k(i) \mid G^{\cR}_{n,i'}) \ge N_k(i') \cdot \exp(-k\ell (i-i')/n +o(1))$. 
Together with~\eqref{eq:Nk:rhokS1} it follows easily that for $0 \le t' \le t$ we have
\begin{equation}\label{eq:decay}
 \rho_k(t)\ge \rho_k(t')e^{-k\ell(t-t')}.
\end{equation}

Define $\cS_r$ as in \eqref{eq:cSr}. We now show that for each $r\ge 0$, for every $k\in \cS_r$ and $t\in (0,t_1]$
we have $\rho_k(t)>0$. We prove this by induction on~$r$. The base case $r=0$ is immediate, since $\rho_1(0)=1$
and so $\rho_1(t)\ge e^{-\ell t}>0$ by~\eqref{eq:decay}. For the induction step, let $k\in \cS_{r+1}$ and $t\in (0,t_1]$.
Then there are $c_1,\ldots,c_\ell\in \cS_r$ with $s(\vc)=k$. By induction we have $\rho_{c_j}(t/2)>0$ for $j=1,\ldots,\ell$.
Hence, using \eqref{eq:decay}, there is some $\delta=\delta(k,t)>0$ such that $\rho_{c_j}(t')\ge \delta$ for all
$t'\in [t/2,t]$. For $n$ large enough,
in each step $i$ with $tn/2\le i\le tn$, by~\eqref{eq:Nk:rhokS1} we thus have probability at 
least, say, $\delta^\ell/2$ of selecting vertices in distinct components of sizes $c_1,\ldots,c_\ell$, and thus 
forming a component with $k$ vertices. Such a component, once formed, has probability at least $e^{-k \ell t/2+o(1)}$
of surviving to step $tn$, as above. It follows that $\E N_k(tn) \ge e^{-k\ell t/2+o(1)}\delta^\ell/2 \cdot t n/2=\Omega(n)$. 
By~\eqref{eq:Nk:rhokS1} this implies~$\rho_k(t) >0$. 
\end{proof}

As well as the possible component sizes, we need to consider the possible `sizes' of $(k,r)$-components that can
appear in the marked graph $H_i$ defined in Section~\ref{sec:exp} (see Figure~\ref{fig:exp}). 
Recall that $Q_{k,r}(tn) \approx q_{k,r}(t) n$ (see Lemmas~\ref{lem:U:t}--\ref{lem:Qkr:t} and Theorem~\ref{thm:Qkr}).

\begin{lemma}\label{lem:allowed:kr}
Let $\cR$ be a bounded-size rule with cut-off $K$. Define $(k,r)$-components as in Section~\ref{sec:exp},
and let $(q_{k,r})_{k\ge 1, r\ge 0}$ be the functions defined in Lemma~\ref{lem:Qkr:t}.
Then there is a set $\cS^*_{\cR}\subset \NN^2$ with the following properties:
\begin{enumerate}[label=\upshape(\roman*)] 
\item\label{allowed:1}  the marked graph $H_i$ can only contain $(k,r)$-components with $(k,r)\in \cSkr$,

\item\label{allowed:2} $(0,r)\in \cS^*_{\cR}$ if and only if $r=2$,

\item\label{allowed:pos} for $t\in (t_0,t_1]$, $k\ge 1$ and $r\ge 0$ we have $q_{k,r}(t)>0$ if and only if $(k,r)\in \cS^*_{\cR}$, 

\item\label{allowed:r0} $(k,0)\in \cSkr$ if and only if $k\in \cSR$,

\item\label{allowed:r1} if $k>K$, $r\ge 0$ and $k\in \cSR$, then $(k,r)\in \cSkr$, and 

\item\label{allowed:per} if $(k,r)\in \cSkr$ and $r\ge 1$, then $k$ is a multiple of $\per$.

\end{enumerate}
\end{lemma}

\begin{proof}
We define $\cSkr$ to be the set of all pairs $(k,r)$ such that it is possible (for some $n$ and $i$)
for a $(k,r)$-component to appear in the marked graph $H_i$, together with the exceptional pair $(0,2)$, 
whose inclusion is convenient later. Thus properties \ref{allowed:1} and \ref{allowed:2} hold by definition;
pairs $(0,r)$ play no role in the other properties.

By Lemma~\ref{lem:Qkr:t} we have $q_{1,0}(t_0)=\rho_{1}(t_0)/1>0$, 
and whp $\max_{i_0 \le i \le i_1}|Q_{k,r}(i)/n-q_{k,r}(i/n)|=o(1)$ for fixed $k \ge 1$ and $r \ge 0$. 
It is clearly possible to give a recursive description of $\cSkr \setminus \{(0,2)\}$ similar to that for~$\cSR$, 
starting with~$\{(1,0)\}$.
Property \ref{allowed:pos} thus follows by an argument analogous to the proof of Lemma~\ref{lem:rhokS1}; 
we omit the details. 

For property \ref{allowed:r0}, note that a $(k,0)$ component in $H_i$ is a $k$-vertex component in $G_i$,
so $(k,0)\in \cSkr$ implies $k\in \cSR$.
In the reverse direction, if $k\in \cSR$ then there is a finite sequence of steps (corresponding
to the recursive description of $\cSR$) by which a $k$-vertex component may be `built' from isolated
vertices. This sequence is also possible starting after step $i_0$, using isolated vertices in $V_S$.
This shows that $(k,0)\in \cSkr$. 

A slight extension of the previous argument proves \ref{allowed:r1}. Indeed,
if $k\in \cSR$ and $k>K$, then it is possible for a $(k,0)$ component $C$ to form in $V_S$.
Since $k>K$ (so this component is large) it is possible in a later step for an edge to be added
joining $C$ to $V_L$; this can happen any number of times.

Finally, for \ref{allowed:per}, note that $V_L$ may contain components of any size $k\in \cSR$ with $k>K$. This includes
all large enough multiples of $\per$. Since a $(k,r)$-component may join to $r$ such components, which
may happen not to be joined to any other $(k',r')$-components,
we see that if $(k,r)\in \cSkr$, then $k+rm\per\in \cSR$ for all large
enough $m$. By Lemma~\ref{lem:allowed}, it follows that $k$ is itself a multiple of $\per$.
\end{proof}

\newoddpage

\section{Component size distribution: coupling arguments}\label{sec:cpl}
In this section we study a variant of the random graph~$J_i=J(\fS_i)$ introduced in Section~\ref{sec:exp}. 
Our goal is to relate the component size distribution of $J_i$ to a `well-behaved' branching process~$\bp_{i/n}$; our analysis hinges on a step-by-step neighbourhood exploration process.  
The general idea of comparing such exploration processes with a branching process is nowadays standard, although the details are more involved than usual. 
In this section we take a `static' viewpoint, considering a single value of $i=i(n)$, which we assume throughout to lie in the range
\[
 i_0\le i\le i_1
\]
with $i_0$ and $i_1$ defined as in \eqref{def:i0i1i2}; the associated `time' $t=i/n$ satisfies $t \in [t_0,t_1]$.
In terms of the overall plan, in the previous section we studied properties of the marked graph $H_i$ illustrated in Figure~\ref{fig:exp}, and here we study what happens in phase two of our exploration, when we connect each $(k,r)$-component to $r$ random vertices in $V_L$. It will be convenient to think of each $(k,r)$-component as similar to a hyperedge in $V_L$ with $r$ vertices; we will `Poissonize' so that we can work in a hypergraph with independence between edges, rather than a given number of edges.

Although in the end we wish to analyze $J(\fS_i)$, we will also consider 
random graphs constructed from other parameter lists $\fS$, motivating the following definition.
\begin{definition}\label{def:paramlist}
A \emph{parameter list} $\fS$ is an ordered pair
\begin{equation}\label{def:fS:rep}
 \fS=\Bigl(\bigl(N_{k}\bigr)_{k > K}, \: \bigl(Q_{k,r}\bigr)_{k,r \ge 0} \Bigr) 
\end{equation}
where each $N_k$ is an integer multiple of $k$ and each $Q_{k,r}$ is a non-negative real number.
We always assume, without further comment,
that only finitely many of the $N_k$ and $Q_{k,r}$ are non-zero.
\end{definition}

An important example is the random parameter list $\fS_i$ defined in \eqref{def:fS}, arising 
from the random graph process process $G^{\cR}_{n,i}$ after $i$ steps.
In this case, each $Q_{k,r}\in \NN$, but
it will be useful to allow non-integer values for the $Q_{k,r}$ later.
Given a parameter list as above, we 
denote the individual parameters in~$\fS$ by~$N_k=N_k(\fS)$ and~$Q_{k,r}=Q_{k,r}(\fS)$.

\begin{definition}[Initial graph~$H$ and random graph~$J$]\label{def:J}
Let $\fS=\bigl((N_{k})_{k > K}, \: (Q_{k,r})_{k,r \ge 0} \bigr)$ be a parameter list with $N_k,Q_{k,r} \in \NN$.
We define the \emph{initial graph} $H=H(\fS)$ as follows.
For each $k \ge 1$ and $r\ge 0$ take $Q_{k,r}$ type-$(k,r)$ components (i.e., components with~$k$ vertices and~$r$ `stubs');
the union of their vertex sets is $V_S=V_S(\fS)$. In addition, take~$N_{k}/k$
components of size~$k$ for each $k > K$;
the union of their vertex sets is $V_L=V_L(\fS)$. The \emph{order} of~$\fS$ is
\[
 |\fS| :=  |V_S| + |V_L| = \sum_{k\ge 1,\,r\ge 0} k Q_{k,r} + \sum_{k>K} N_k.
\]
We construct the \emph{random graph $J=J(\fS)$} by (i)~connecting each stub of each $(k,r)$-component in~$H$ to an independent random vertex in~$V_L$, and (ii)~for each $r \ge2$ adding $Q_{0,r}$ random hyperedges $(x_1, \ldots, x_r) \in (V_L)^r$ to $H$, where the vertices $x_j$ are all chosen independently and uniformly at random from~$V_L$.
\end{definition}

When $Q_{0,r}=0$ for $r\ne 2$, as is the case for the random parameter list $\fS_i$ defined in \eqref{def:fS},
the construction above is exactly the construction described in Section~\ref{sec:exp} (see also Figure~\ref{fig:exp});
the slightly more general form here avoids unnecessary case distinctions later.
Hence, we may restate Lemma~\ref{lem:cond} as follows.
\begin{lemma}[Conditional equivalence]\label{lem:cond'}
Let $i=i(n)$ satisfy $i_0\le i\le i_1$, and let $\fS_i$ be the random parameter list
generated by the random graph process $G^{\cR}_{n,i}$.
Then, conditional on $\fS_i$, the random graphs
$J(\fS_i)$ and $G_i=G^{\cR}_{n,i}$ have the same component size distribution. \noproof
\end{lemma}

For technical reasons it will often be convenient to work with a `Poissonized' version of~$J(\fS)$. 
\begin{definition}[Poissonized random graph~$\JP$]\label{def:F}
Let $\fS=\bigl((N_{k})_{k > K}, \: (Q_{k,r})_{k,r \ge 0} \bigr)$ be a parameter list with $N_k \in \NN$ and $Q_{k,r} \in [0,\infty)$.
The \emph{Poissonized random graph $\JP=\JP(\fS)$} 
is defined exactly as in Definition~\ref{def:J}, except that the numbers of type-$(k,r)$ components
are now independent Poisson random variables with mean~$Q_{k,r}$. 
\end{definition}
It will be useful to think of the $(k,r)$--components (together with their~$r$ adjacent edges) as distinct $r$-uniform hyperedges with weight~$k$, i.e., with $k$~attached $V_S$--vertices. 
Indeed, this point of view unifies~(i) and~(ii) from our construction: both then correspond to hyperedges $g=(x_1, \ldots, x_r) \in (V_L)^r$ of weight~$k$, with independent $x_j \in V_L$. 
Using standard splitting properties of Poisson processes, it now is easy to see that, for all $k \ge 0$ and $r \ge 0$, hyperedges $g=(x_1, \ldots, x_r) \in (V_L)^r$ of weight~$k$, henceforth referred to as \emph{$(k,r)$--hyperedges}, appear in $\JP(\fS)$ according to independent Poisson processes with rate
\begin{equation}\label{eq:rates:qkr}
\hlambda_{k,r} = \hlambda_{k,r}(\fS) := Q_{k,r}/|V_L|^r = Q_{k,r}(\fS)/|V_L(\fS)|^r ,
\end{equation}
where $Q_{k,r}$ and $|V_L|=\sum_{k > K}N_{k}$ are determined by the parameter list $\fS$. This is essentially a version of the inhomogeneous
random hypergraph model of Bollob\'as, Janson and Riordan~\cite{BJR}, though with the extra
feature of weights on the hyperedges, and built on top of the (deterministic, when $\fS$ is given) initial graph 
\begin{equation}\label{HLdef}
 H_L = H_L(\fS) := H[V_L] = H(\fS)[V_L(\fS)] ,
\end{equation}
i.e., the graph on $V_L$ consisting of $N_k/k$ components each size $k>K$, see also Figure~\ref{fig:initial}.
It is this model that we shall work with much of the time. 

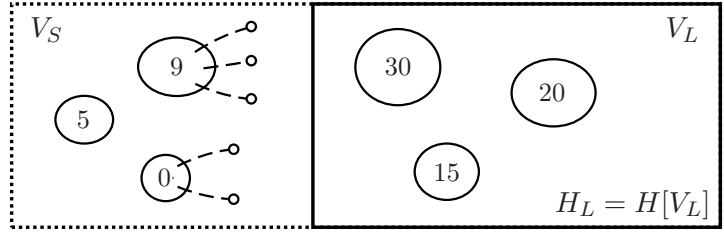
\begin{figure}[t]
\centering
  \setlength{\unitlength}{1bp}%
  \begin{picture}(278.42, 95.58)(0,0)
  \put(0,0){\includegraphics{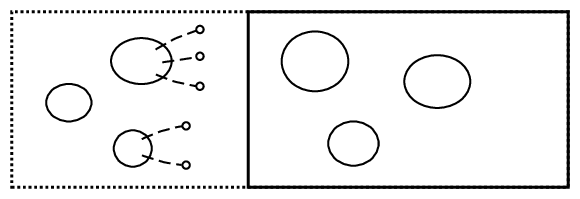}}
  \put(258.55,77.24){\fontsize{11.38}{13.66}\selectfont \makebox[0pt]{$V_L$}}
  \put(151.33,62.98){\fontsize{9.96}{11.95}\selectfont \makebox[0pt]{$30$}}
  \put(63.20,21.24){\fontsize{9.96}{11.95}\selectfont \makebox[0pt]{$0$}}
  \put(67.92,62.97){\fontsize{9.96}{11.95}\selectfont \makebox[0pt]{$9$}}
  \put(18.89,77.24){\fontsize{11.38}{13.66}\selectfont \makebox[0pt]{$V_S$}}
  \put(169.63,23.07){\fontsize{9.96}{11.95}\selectfont \makebox[0pt]{$15$}}
  \put(209.38,53.11){\fontsize{9.96}{11.95}\selectfont \makebox[0pt]{$20$}}
  \put(32.69,43.14){\fontsize{9.96}{11.95}\selectfont \makebox[0pt]{$5$}}
  \put(239.77,11.54){\fontsize{11.38}{13.66}\selectfont \makebox[0pt]{$H_L=H[V_L]$}}
  \end{picture}%
	\caption{\label{fig:initial} 
	Example of the initial (marked) graph~$H=H(\fS)$ and the induced subgraph~$H_L=H[V_L]$. $H_L$ is a standard (unmarked)
graph, consisting of $N_k(\fS)/k$ components of size $k$ for each $k>K$. In $H$, for each $k,r \ge 0$ there are also $Q_{k,r}$ (or $\Po(Q_{k,r})$ in the Poissonized
model) $(k,r)$-components: $k$-vertex components with $r$ stubs.}
\end{figure}

In this section our main focus is on parameter lists which satisfy the typical properties of those arising from
$(G^{\cR}_{i})_{i_0 \le i \le i_1}$ derived in Section~\ref{sec:setup}, which we now formalize.
\begin{definition}[$t$-nice parameter lists]\label{def:nice}
For $t\in [t_0,t_1]$ we say that a parameter list 
$\fS=\bigl((N_{k})_{k > K}, \: (Q_{k,r})_{k,r \ge 0} \bigr)$
is \emph{$t$-nice} if it satisfies $N_k,Q_{k,r} \in \NN$ and the following conditions,
where the constants $D_N, A, a$ and $D_Q, B, b$
are as in Theorems~\ref{thm:init} and~\ref{thm:Qkr}:
\begin{align}
\label{eq:Nk:t0'}
|N_k-\rho_k(t_0)n| &\le (\log n)^{D_N} n^{1/2} \hspace{-5.0em}&\hspace{-5.0em}&\hspace{-5.0em} \forall \: k>K,\\
\label{eq:Qkr'}
|Q_{k,r}-q_{k,r}(t)n| &\le (\log n)^{D_Q} n^{1/2} \hspace{-5.0em}&\hspace{-5.0em}&\hspace{-5.0em} \forall \: k,r\ge 0,\\
\label{eq:Nk:t0:tail'}
N_{\ge k} &\le A e^{-ak}n \hspace{-5.0em}&\hspace{-5.0em}&\hspace{-5.0em} \forall \: k>K,\\
\label{eq:Qkr:tail'}
Q_{k,r} &\le B e^{-b(k+r)}n \hspace{-5.0em}&\hspace{-5.0em}&\hspace{-5.0em} \forall \: k,r\ge 0,\\
\label{eq:nice:n}
|\fS| &= n,
\end{align}
\begin{equation}
\label{eq:nice:allowed}
 N_k>0 \implies k\in \cSR \quad\text{and}\quad Q_{k,r}>0 \implies (k,r)\in \cSkr.
\end{equation}
\end{definition}
Note that the definition involves~$n$, so formally we should write \emph{$(n,t)$-nice}. However, the value of~$n$ should always be clear from context.
When the value of~$t$ is not so important, we sometimes write \emph{nice} rather than $t$-nice. 
The final condition above is technical, and simply expresses that the component sizes/types in a $t$-nice 
parameter list are ones that could conceivably arise in our random graph process (see Section~\ref{sec:period}).
\begin{lemma}[Parameter lists of $G_i$ are typically nice]\label{lem:nice}
Let $\cN_i$ denote the event that the random parameter list $\fS_i=\fS(G^{\cR}_{n,i})$ defined in \eqref{def:fS} is $t$-nice, where~$t=i/n$,
and let $\cN = \bigcap_{i_0 \le i \le i_1} \cN_i$. Then
\begin{equation*}
 \Pr(\neg\cN) = O(n^{-99}) .
\end{equation*}
\end{lemma}
\begin{proof}
By Theorem~\ref{thm:init}, relations  \eqref{eq:Nk:t0'} and \eqref{eq:Nk:t0:tail'} hold
with the required probability. Similarly, Theorem~\ref{thm:Qkr} gives
\eqref{eq:Qkr'} and \eqref{eq:Qkr:tail'} for $k\ge 1$. For $k=0$, we have $Q_{k,r}(i)=0$
unless $r=2$. When $k=0$ and $r=2$, the bound \eqref{eq:Qkr'} holds with the required probability by
Lemma~\ref{lem:U:t} (increasing $D_Q$ if necessary), and \eqref{eq:Qkr:tail'} holds trivially
if $B\ge e^{2b}$, which we may assume. Finally, \eqref{eq:nice:n} follows from~\eqref{eq:NkQkr:sum},
and \eqref{eq:nice:allowed} from the definitions of~$\cSR$ and~$\cSkr$ in Section~\ref{sec:period}.
\end{proof}

For later reference, we collect some convenient properties of nice parameter lists. 
First, since there exists $k^*\in \cSR \setminus [K]$ with $\rho_{k^*}(t)>0$ (see Lemmas~\ref{lem:allowed} and~\ref{lem:rhokS1}),
using~\eqref{eq:Nk:t0'} and~\eqref{eq:nice:n} we deduce that there is an absolute constant $\zeta > 0$ such that
for $n$ large enough, every $t$-nice parameter list satisfies
\begin{equation}\label{eq:nice:VL}
 \zeta n \le |V_L| \le n .
\end{equation}
Second, letting $B_0=2/b$ where $b$ is the constant in~\eqref{eq:Qkr:tail'}, since $Q_{k,r}$ is an integer,
for $n$ large enough~\eqref{eq:Qkr:tail'} implies 
\begin{equation}\label{eq:Qkr:stop}
 Q_{k,r} =0 \quad\text{whenever}\quad k+r\ge B_0\log n,
\end{equation}
say. Arguing similarly for $N_k$ (using \eqref{eq:Nk:t0:tail'}), and stating only a crude bound to avoid dealing with the constants,
for~$n \ge n_0(a,A,b,B,\zeta)$ any $t$-nice parameter list satisfies
\begin{equation}\label{eq:small:Psi}
 \max_{k \ge \Psi}N_{k}=0, \quad \max_{k+r \ge \Psi}Q_{k,r}=0 \quad\text{and}
\quad \max_{k,r} Q_{k,r}\le \Psi |V_L|,
\end{equation}
where $\Psi := (\log n)^2$; for the final bound we simply use $Q_{k,r}\le Bn$ from \eqref{eq:Qkr:tail'} and \eqref{eq:nice:VL}. 
Throughout this section we shall always, without further comment,
assume that $n$ is large enough such that every $t$-nice parameter list 
satisfies~\eqref{eq:nice:VL}--\eqref{eq:small:Psi}.

The remainder of this section is organized as follows; throughout we consider $\fS$ which is $t$-nice, for some $t\in [t_0,t_1]$.
First, in Section~\ref{sec:nep} we introduce a neighbourhood exploration process for $\JP(\fS)$, which we couple with an `idealized' branching process~$\bp_t$ in Section~\ref{sec:BPI}. 
Next, in Section~\ref{sec:dom} we show that $J(\fS)$ can be `sandwiched' between two instances of~$\JP(\cdot)$, say $\JP(\fS^-_t) \subseteq J(\fS) \subseteq \JP(\fS^+_t)$, which we are able to study via associated `dominating' branching processes~$\bp^{\pm}_t$. 
Later, in Section~\ref{sec:proof}, we will use Lemmas~\ref{lem:cond'} and~\ref{lem:nice} to transfer properties of $J(\fS)$ back to the original random graph
process.

\subsection{Neighbourhood exploration process}\label{sec:nep} 
In this subsection we introduce a neighbourhood exploration process which initially may be coupled \emph{exactly} with a certain branching process (defined in Section~\ref{sec:BPI}).
With an eye on our later arguments we shall consider an arbitrary parameter list $\fS=\bigl((N_{k})_{k > K}, \: (Q_{k,r})_{k,r \ge 0} \bigr)$.
We intuitively start the exploration of $\JP=\JP(\fS)$ with a random vertex from $V_S \cup V_L$, but in the Poissonized model this requires some care,
since $V_S$ is a random set. 
For this reason we first discuss the exploration starting from a \emph{set} $W$ of vertices from $V_L$, where $W$ is a union of components of~$H_L=H_L(\fS)$, deferring the details of the initial generation to Section~\ref{sec:nep:fg}.

\subsubsection{Exploring from an initial set}\label{sec:nep:set}

Writing, as usual, $C_v(G)$ for the (vertex set of) the component of the graph $G$ containing a given vertex $v$, 
let 
\[
 C_W(\JP) := \bigcup_{v\in W} C_v(\JP).
\]
The basic idea is that each vertex $v \in V_L$ has neighbours in~$V_S$ and in~$V_L$. 
Indeed, via each $(k,r)$--hyperedge $(x_1, \ldots, x_r)$ containing~$v$ we reach $k$~new $V_S$--vertices and find up to $r-1$ new neighbours $\{x_1, \ldots, x_r\} \setminus \{v\}$ in~$V_L$ (there could be fewer if there are clashes). 
Repeating this exploration iteratively, we eventually find $C_W(\JP)$, see Figure~\ref{fig:nep} (see also Figure~\ref{fig:exp} for the related graph~$H_i$).

Turning to the details of the exploration process, we shall maintain sets of \emph{active} and \emph{explored} vertices in $V_L=V_L(\fS)$, as well as the \emph{number} of reached vertices from~$V_S=V_S(\fS)$.
After step $j$ of the exploration we denote these by
$\cA_j$, $\cE_j$ and $S_j$, respectively; note that $\cA_j\cup \cE_j$, the set of `reached' vertices in $V_L$, 
will always be a union of components of $H_L$.
Initially, given a union $W \subseteq V_L$ of components of~$H_L$, we start with the active set $\cA_0=W$, the explored set $\cE_0=\emptyset$, and some initial number $S_0 \in \NN$ (it will later be convenient to allow $S_0>0$). 
In step $j \ge 1$, we pick an active vertex $v_j \in \cA_{j-1}$.
For each $k \ge 0$ and $r \ge 1$ we then proceed as follows, see also Figure~\ref{fig:nep}. 
We sequentially test the presence and multiplicity of each so-far untested $(k,r)$-hyperedge $g\in (V_L)^r$ of the form $(v_j,w_1, \ldots, w_{r-1})$, \ldots, $(w_1, \ldots, w_{r-1},v_j)$, and denote the resulting multiset of `newly found' hyperedges by $\cH_{j,k,r}$.
Now, for each hyperedge $g \in \cH_{j,k,r}$ we increase the number of $V_S$--vertices reached by~$k$, and
mark all `newly found' vertices, i.e., all vertices in $\bigcup_{1 \le h \le r-1}C_{w_h}(H_L) \setminus (\cA_{j-1} \cup \cE_{j-1})$, as active.
Finally, we move the vertex~$v_j$ from the active set to the explored set. 
As usual, we stop the above exploration process if $|\cA_{j}|=0$, in which case 
\begin{gather}
\label{eq:inv:final}
\cE_j = C_W(\JP) \cap V_L \quad \text{ and } \quad S_j = |C_W(\JP) \cap V_S|+S_0 .
\end{gather}

\begin{figure}[t]
\centering
  \setlength{\unitlength}{1bp}%
  \begin{picture}(187.29, 95.60)(0,0)
  \put(0,0){\includegraphics{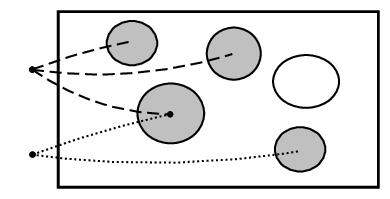}}
  \put(167.42,77.29){\fontsize{11.38}{13.66}\selectfont \makebox[0pt]{$V_L$}}
  \put(84.98,43.97){\fontsize{9.96}{11.95}\selectfont \makebox[0pt]{$v$}}
  \put(8.15,59.85){\fontsize{9.96}{11.95}\selectfont \makebox[0pt]{$k$}}
  \put(8.15,19.06){\fontsize{9.96}{11.95}\selectfont \makebox[0pt]{$0$}}
  \end{picture}%
	\caption{\label{fig:nep} Example of the neighbourhood exploration process in $\JP=\JP(\fS)$ starting with $W=C_v(H_L)$, i.e., with $W$ the vertex set of the component of~$H_L$ containing~$v$.
Assuming that $v=v_1$, in step one we find the weighted hyperedges containing $v$; the weight corresponds to the size of the associated $V_S$--component. 
Here the vertex~$v$ is connected to three $V_L$--components via two hyperedges, which have weight~$k$ and~$0$, respectively (the weight~$0$ hyperedge simply corresponds to a $V_L$--$V_L$ edge). 
The exploration process then marks the vertices of the newly discovered components as active, and increases the number of reached $V_S$--vertices by $k+0=k$. 
Afterwards it marks $v=v_1$ as explored, and repeats the same procedure for the next active vertex $v_2$, and so on.}
\end{figure}
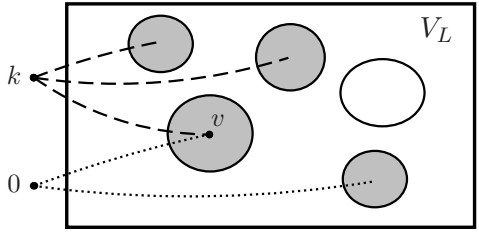

It will be convenient to extend the definitions of $\cA_j$, $\cE_j$ and $S_j$ to all $j \ge 0$. 
Namely, if $|\cA_j|=0$ then we set $X_{j'}=X_j$ for all $j' > j$ and $X \in \{\cA,\cE,S\}$.
Note that, by construction, the following properties hold for all $j \ge 0$:
\begin{gather}
\label{eq:inv:disj}
\cA_j \cup \cE_j  \subseteq  \cA_{j+1} \cup \cE_{j+1} \subseteq V_L, \quad S_j \le S_{j+1}, \quad |\cA_j \cup \cE_j| + S_j \le |C_W(\JP)| + S_0 .
\end{gather}
Furthermore, since the vertex $v_j$ is moved from the active to the explored set in step $j \ge 1$ whenever $|\cA_{j-1}| \ge 1$, it is not difficult to see that for all $j \ge 0$ we have 
\begin{equation}\label{eq:inv:A}
|\cA_j| = \max\bigl\{|\cA_j \cup \cE_j|-j,0\bigr\}.
\end{equation}
Note that, in view of~\cref{eq:inv:final,eq:inv:disj,eq:inv:A},
to study the size of $C_W(\JP)$ it is enough to track the evolution of $(|\cA_j \cup \cE_j|,S_j)_{j \ge 0}$. 
For this reason we shall study
\begin{equation}
\label{def:Mj}
M_j := |\cA_j \cup \cE_j|,
\end{equation}
the number of vertices of $V_L$ reached by the exploration process after~$j$ steps. 

In Section~\ref{sec:nep:fg} we shall specify the initial distribution of $S_0$ and $W$, which then in turn defines the process 
\begin{equation}\label{def:Ti}
 \cT = \cT(\fS) := (M_j,S_j)_{j \ge 0} .
\end{equation}
Intuitively, $\cT$ corresponds to a random walk which counts the number of $V_S$ and $V_L$ vertices reached by the exploration process.
For convenience we also define
\begin{equation}\label{def:Ti:size}
|\cT| := |C_W(\JP)|+S_0 = |C_W(\JP) \cap V_L| + \bigl(|C_W(\JP) \cap V_S|+S_0\bigr)  ,
\end{equation}
the total number of reached vertices, including those that we started with (see Section~\ref{sec:nep:fg}).
Note that~$\cT$ determines $|\cT|$, cf.~\eqref{eq:inv:final}--\eqref{def:Ti} above. 
One main goal of this section is to show that the distribution of the random variable $|\cT|$ determines the expected component size distribution of~$\JP=\JP(\fS)$, see~\eqref{eq:def:NjT:E}--\eqref{eq:def:NgejR:E} below.

\subsubsection{Initial generation and first moment formulae}\label{sec:nep:fg}
We now turn to the `initial generation', which yields the input $W$ and $S_0$ for our exploration process.
Recall that the vertex set of the Poissonized model $\JP=\JP(\fS)$ is \emph{not} deterministic (in contrast to that of~$J(\fS)$).  
We are eventually interested in $N_j(\JP)$ with $j \ge 1$, which denotes the number of vertices in components of size~$j$ in~$\JP$. 
Writing $\cQ_{k,r}$ for the (random) set
of all $(k,r)$--hyperedges in $\JP$, note that
\begin{equation}\label{eq:def:Nj}
 N_j(\JP) =  \sum_{v \in V_L}\indic{|C_v(\JP)|=j} + \sum_{k \ge 1,\, r \ge 0} \sum_{g \in \cQ_{k,r}} k \indic{|C_{g}(\JP)|=j} ,
\end{equation}
where $C_{\cX}(\JP)$ denotes the component of $\JP$ which contains $\cX \in \{v,g\}$. 
Using standard results from the theory of point processes (see, e.g., Lemma~\ref{lem:palm} in Appendix~\ref{apx:palm}) we have 
\begin{equation}\label{eq:Qkr:palm}
\E\Bigl( \sum_{\tC \in \cQ_{k,r}} \indic{|C_{\tC}(\JP)|=j} \Bigr) = Q_{k,r} \cdot \Pr(|C(\JP_{k,r})|=j) ,
\end{equation}
where $C(\JP_{k,r})$ is defined as follows.
We add an `extra' $V_S$--component $\tC$ with $|\tC|=k$, and connect $\tC$ to $\JP=\JP(\fS)$ via~$r$ random vertices in~$V_L$, i.e., we add an extra $(k,r)$--hyperedge~$g$. 
We then write $C(\JP_{k,r})$ for the component of the resulting graph which contains~$\tC$.  
It follows that 
\begin{equation}\label{eq:def:Nj:E}
\E N_j(\JP) = \sum_{k > K}\sum_{\substack{v \in V_L:\\ |C_{v}(H_L)|=k}}\Pr(|C_v(\JP)|=j) + \sum_{k \ge 1,\,r \ge 0} k Q_{k,r}\Pr(|C(\JP_{k,r})|=j) .
\end{equation}
By definition,
\begin{equation}\label{eq:def:Nj:sum}
 \sum_{k >K} N_{k}(H_L) + \sum_{k \ge 1,\, r \ge 0} k Q_{k,r} = |V_L| + |V_S| = |\fS| .
\end{equation}
Thus~\eqref{eq:def:Nj:E} implicitly defines the desired initial distribution of $\cT=(M_j,S_j)_{j \ge 0}$.
Namely, for any $v \in V_L$, with probability~$1/|\fS|$ we start the exploration process with 
\begin{equation}\label{eq:def:SW:Y}
S_0:=0 \quad \text{and} \quad W:= C_{v}(H_L) ,
\end{equation}
and for any $k \ge 1, r \ge 0$, with probability $k Q_{k,r}/|\fS|$ we select $w_1, \ldots, w_r \in V_L$ independently and uniformly at random, and then start the exploration process with
\begin{equation}\label{eq:def:SW:ZR}
S_0:=k \quad \text{and} \quad W:= \bigcup_{1 \le h \le r} C_{w_h}(H_L) .
\end{equation}
In the first case, from~\eqref{def:Ti:size} we have
\begin{equation}\label{eq:def:Ti:Y}
|\cT|=|C_W(\JP)|+S_0=|C_v(\JP)|,
\end{equation}
while in the second case~\eqref{def:Ti:size} and the construction of $C(\JP_{k,r})$ yield 
\begin{equation}\label{eq:def:Ti:ZR}
|\cT|=|C_W(\JP)|+S_0 = |C_W(\JP)|+k=|C(\JP_{k,r})|.
\end{equation}
Combining this discussion with~\eqref{eq:def:Nj:E}--\eqref{eq:def:Nj:sum}, we thus obtain
\begin{equation}\label{eq:def:NjT:E}
\begin{split}
\E N_j(\JP) & = \biggl[\sum_{k > K}\sum_{\substack{v \in V_L:\\ |C_{v}(H_L)|=k}}\Bigl(\frac{1}{|\fS|} \cdot \Pr(|C_v(\JP)|=j)\Bigr) + \sum_{k \ge 1,\,r \ge 0} \Bigl(\frac{k Q_{k,r}}{|\fS|} \cdot \Pr(|C(\JP_{k,r})|=j)\Bigr)\biggr] \cdot |\fS| \\
& = \Pr(|\cT| = j) |\fS| .
\end{split}
\end{equation}
Finally, defining $N_{\ge j}(\JP)$ in the obvious way, for future reference we similarly have
\begin{equation}\label{eq:def:NgejR:E}
\E N_{\ge j}(\JP) = \Pr(|\cT| \ge j) |\fS| .
\end{equation}

\subsubsection{Variance estimates}\label{sec:varsub}

One application of the exploration process described above is the following bound on the 
variance of the number of vertices in a range of component sizes.
It will turn out later that, for the parameters~$\Lambda_j$ we are interested in, 
in the subcritical case
the upper bound proved below is small compared to $(\E X)^2$, allowing us to apply Chebyshev's inequality
to establish concentration (see Sections~\ref{sec:mom:large} and~\ref{sec:L1:sub}).
\begin{lemma}[Truncated variance of $N_{\ge \Lambda}$]
\label{lem:Nkvar:sub}
Let $\fS$ be an arbitrary parameter list, and define $\JP=\JP(\fS)$ as in Definition~\ref{def:F}.
For all $0 \le \Lambda_1 \le \Lambda_2$, setting $X=N_{\ge \Lambda_1}(\JP)-N_{\ge \Lambda_2}(\JP)$ we have
\begin{equation}\label{eq:Nkvar:sub}
\Var X \le \E X \bigl(\E N_{\ge \Lambda_2}(\JP) + \Lambda_2\bigr) .
\end{equation}
\end{lemma}
\begin{proof}
The key step in the proof is the following van den Berg--Kesten-type estimate: we claim 
that for all $R_1,R_2 \subseteq V_L$ and $\cI_1,\cI_2\subseteq \NN$ with $I_2\subseteq [\Lambda,\infty)$ we have
\begin{equation}\label{eq:pr:var}
\Pr\bigl(|C_{R_1}(\JP)| \in \cI_1, \: |C_{R_2}(\JP)| \in \cI_2, \:  C_{R_1}(\JP) \cap C_{R_2}(\JP) =\emptyset \bigr)
    \le \Pr(|C_{R_1}(\JP)| \in \cI_1) \Pr(|C_{R_2}(\JP)| \ge \Lambda) .
\end{equation}
To prove this claim, given $U \subseteq V_L$ we define $\JPU$ as the subgraph of $\JP$ obtained by deleting 
all vertices and hyperedges involving vertices from $U$. 
Clearly, if $C_{R_1}(\JP) \cap V_L = U$ and $C_{R_1}(\JP)$ and $C_{R_2}(\JP)$ are disjoint, then $C_{R_2}(\JP) = C_{R_2}(\JPU)$. 
Exploring as in Section~\ref{sec:nep:set}, starting from $W_1=C_{R_1}(H_L)$,
we can determine $C_{R_1}(\JP)=C_{W_1}(\JP)$ while only revealing information about hyperedges involving vertices from  $U=C_{R_1}(\JP) \cap V_L$. 
By construction, the remaining weighted hyperedges $g=(x_1, \ldots, x_r) \in (V_L \setminus U)^r$ have not yet been tested, i.e., still appear according to independent Poisson processes.
So, since $\JPU$ does \emph{not} depend on the status of the revealed hyperedges (which each involve at least one vertex from~$U$), it follows 
(by conditioning on all possible sets~$U$) 
that the left hand side of~\eqref{eq:pr:var} is at most
\begin{equation}\label{eq:pr:var:2}
 \Pr\bb{|C_{R_1}(\JP)| \in \cI_1} \cdot \max_{R_1 \subseteq U \subseteq V_L \setminus R_2} \Pr\bb{|C_{R_2}(\JPU)| \in \cI_2}.
\end{equation}
Since $\JPU \subseteq \JP$ and $\cI_2 \subseteq [\Lambda, \infty)$, this implies~\eqref{eq:pr:var}. 

Having established the claim, we turn to the bound on $\Var X$; the main complication involves dealing with the first
step of our exploration, i.e., with the possibility of starting in the random set $V_S$.
Analogous to~\eqref{eq:def:Nj} we write $X=N_{\ge \Lambda_1}(\JP)-N_{\ge \Lambda_2}(\JP)$ as a sum of terms of the form $\indic{|C_v(\JP)| \in \cI}$ or $k\indic{|C_{g}(\JP)| \in \cI}$, where
\[
 \cI :=[\Lambda_1,\Lambda_2).
\]  
Note that $X^2$ involves pairs $\bigl(C_1(\JP),C_2(\JP)\bigr)$ of components of the form $C_j(\JP) \in \{C_{v}(\JP),C_{g}(\JP)\}$. 
Let
\begin{equation}\label{eq:lem:Nkvar:X2}
X^2=Y+Z, 
\end{equation} 
where $Y$ contains all summands with pairs of equal components, i.e., $C_1(\JP)=C_2(\JP)$, and $Z$ all summands with pairs of distinct components, i.e., $C_1(\JP) \neq C_2(\JP)$. 
Since each relevant component contains at most~$\Lambda_2$ vertices, we have~$Y \le X \cdot \Lambda_2$, so
\begin{equation*}
 \E Y \le \E X \cdot \Lambda_2.
\end{equation*} 
For $Z$ we proceed similarly as for~\eqref{eq:def:Nj:E}. In particular, using standard results from the theory of point processes (see, e.g., Lemmas~\ref{lem:palm}--\ref{lem:palm:2} in Appendix~\ref{apx:palm}), we have 
\begin{equation}\label{eq:Qkr2:palm}
\begin{split}
& \E\Bigl( \sum_{f\in \cQ_{k_1,r_1}}\sum_{g\in \cQ_{k_2,r_2}} k_1 k_2 \indic{|C_{f}(\JP)| \in \cI, \: |C_{g}(\JP)|\in \cI , \: C_{f}(\JP) \neq C_{g}(\JP)} \Bigr) \\
& \qquad = k_1Q_{k_1,r_1}k_2Q_{k_2,r_2} \Pr\big(|C_1(\JP_{++})| \in \cI, \: |C_2(\JP_{++})| \in \cI , \: C_1(\JP_{++}) \neq C_2(\JP_{++}) \bigr) ,
\end{split}
\end{equation}
where $C_1(\JP_{++})$ and $C_2(\JP_{++})$ arise analogous to Section~\ref{sec:nep:fg}, 
by adding an extra $(k_j,r_j)$-hyperedge for $j=1,2$. (Here it is important that $C_{f}(\JP) \neq C_{g}(\JP)$ implies $f \neq g$, so
we add two distinct extra hyperedges.) More precisely, we form $\JP_{++}$ by adding, for each $j=1,2$, an `extra'
component $\tC_j$ with $k_j$ vertices, joined to a random set $R_j=\{w_{j,1}, \ldots, w_{j,r_j}\}$ of vertices of $\JP$,
where the $w_{j,h}$ are chosen independently and uniformly from $V_L$. Then $C_j(\JP_{++})$ is the component of
$\JP_{++}$ containing $\tC_j$. (Since the definition of $\JP_{++}$ depends on $(k_1,r_1)$ and $(k_2,r_2$),
the notation $C_j(\JP_{k_1,r_1,k_2,r_2})$ analogous to that used in Section~\ref{sec:nep:fg} would also be appropriate;
we avoid this as being too cumbersome.)

In the case we are interested in, the components $C_j(\JP_{++})$ in the augmented graph $\JP_{++}$ are distinct and thus disjoint.
Thus, for each $j$, 
$C_j(\JP_{++})$ consists of the $k_j$ vertices in $\tC_j$ together with all vertices in $C_{R_j}(\JP)$. Hence, in this case,
$|C_j(\JP_{++})| \in \cI$ if and only if $|C_{R_j}(\JP)| \in \cI_j = [\Lambda_1-k_j,\Lambda_2-k_j)$.
Using~\eqref{eq:pr:var} it follows that, conditional on $R_1$ and $R_2$, we have
\[
 \Pr\bigl(|C_1(\JP_{++})| \in \cI, \: |C_2(\JP_{++})| \in \cI, \:  C_1(\JP_{++}) \neq C_2(\JP_{++})\bigr) \le 
  \Pr(|C_{R_1}(\JP)| \in \cI_1) \Pr(|C_{R_2}(\JP)| \ge \Lambda_1-k_2) .
\]
Define $C(\JP_{k_j,r_j})$ as in the previous subsection, adding only one extra component with $k_j$ vertices
joined to a set $R_j$ consisting of $r_j$ random vertices from $V_L$, so $|C(\JP_{k_j,r_j})|=|C_{R_j}(\JP)|+k_j$.
Then, taking the expectation over the independent random sets $R_j$ and applying \eqref{eq:Qkr:palm} twice,
we deduce that the right hand side of~\eqref{eq:Qkr2:palm} is at most 
\begin{equation*}
 k_1 Q_{k_1,r_1} \Pr(|C(\JP_{k_1,r_1})| \in \cI) \cdot k_2Q_{k_2,r_2} \Pr(|C(\JP_{k_2,r_2})| \ge \Lambda_1 ) .
\end{equation*}
The estimates for the other terms of $\E Z$ involving $|C_{v}(\JP)|,|C_{w}(\JP)|$ and $|C_{v}(\JP)|,|C_{g}(\JP)|$ are similar, but much simpler,
and we conclude that
\begin{equation}\label{eq:lem:Nkvar:Z}
 \E Z \le \E X \cdot \E N_{\ge \Lambda_1}(\JP) = (\E X)^2 + \E X \cdot \E N_{\ge \Lambda_2}(\JP).
\end{equation} 
Hence
\[ 
 \E X^2 = \E Z + \E Y \le (\E X)^2 + \E X \bb{ \E N_{\ge \Lambda_2}(\JP) + \Lambda_2 },
\]
completing the proof. 
\end{proof}

Using related arguments, we next prove an upper bound on the ($r$th order) susceptibility of $\JP(\fS_i)$. 
Since $\JP=\JP(\fS_i)$ has approximately (rather than exactly)~$n$ vertices,
for any graph~$G$ it will be convenient to define the `modified' susceptibility 
\begin{equation}\label{eq:def:SrG}
 S_{r,n}(G) := \sum_{C}|C|^r/n = \sum_{k \ge 1} k^{r-1} N_k(G)/n . 
\end{equation}
Note that we divide by $n$, rather than by the actual number of vertices of $G$.
For later reference we collect the following basic properties of this parameter.
\begin{remark}\label{rem:Snr}
Let $r \ge 1$. For any $n$-vertex graph~$G$ we have $S_r(G)=S_{r,n}(G)$. 
For any two graphs $F \subseteq H$ we have $S_{r,n}(F) \le S_{r,n}(H)$; to see this,
it suffices to check the case where~$F$ and~$H$ differ by a single edge or isolated vertex. 
\end{remark}
In the subcritical case
it will turn out that the bound~\eqref{eq:Srvar:sub} below is small enough to establish concentration via Chebyshev's inequality (see Sections~\ref{sec:mom:sus} and~\ref{sec:Sj}).
\begin{lemma}[Variance of $S_{r,n}$]
\label{lem:Srvar:sub}
Let $\fS$ be an arbitrary parameter list, and define $\JP=\JP(\fS)$ as in Definition~\ref{def:F}. 
For $r \ge 2$ we have 
\begin{equation}\label{eq:Srvar:sub}
\Var S_{r,n}(\JP) \le n^{-1} \E S_{2r,n}(\JP).
\end{equation} 
\end{lemma}
\begin{proof}
We shall mimic the basic proof strategy of Lemma~\ref{lem:Nkvar:sub}, but treat pairs of equal components with more care. 
Turning to the details, analogous to~\eqref{eq:pr:var} we first claim that for all $R_1,R_2 \subseteq V_L$ we have 
\begin{equation}\label{eq:E:var}
\E\bigl(|C_{R_1}(\JP)|^{r-1} |C_{R_2}(\JP)|^{r-1} \indic{C_{R_1}(\JP) \cap C_{R_2}(\JP) = \emptyset}\bigr) \le \E|C_{R_1}(\JP)|^{r-1} \E|C_{R_2}(\JP)|^{r-1}.
\end{equation}
Indeed, using the conditioning argument leading to~\eqref{eq:pr:var:2} we see that the left hand side of~\eqref{eq:E:var} is at most 
\begin{equation*}
\E|C_{R_1}(\JP)|^{r-1} \cdot \max_{R_1 \subseteq U \subseteq V_L \setminus R_2} \E|C_{R_2}(\JPU)|^{r-1} ,
\end{equation*}
where $\JPU$ is the subgraph of $\JP$ obtained by deleting all vertices and hyperedges involving vertices from $U$. 
Since $\JPU \subseteq \JP$ and so $|C_{R_2}(\JPU)| \le |C_{R_2}(\JP)|$, this establishes inequality~\eqref{eq:E:var}. 

Next we focus on $\Var S_{r,n}(\JP)$. Inspired by~\eqref{eq:def:Nj}, using $|C|^r=\sum_{v \in C}|C|^{r-1}$ we rewrite $nS_{r,n}(\JP)$ as 
\begin{equation*}
X: = nS_{r,n}(\JP) = \sum_{v \in V_L}|C_v(\JP)|^{r-1} + \sum_{k \ge 1,\, r \ge 0} \sum_{g \in \cQ_{k,r}} k |C_g(\JP)|^{r-1} ,
\end{equation*}
where $C_{\cX}(\JP)$ denotes the component of $\JP$ which contains $\cX \in \{v,g\}$. 
Since $X^2$ involves pairs $\bigl(C_1(\JP),C_2(\JP)\bigr)$ of components, analogous to~\eqref{eq:lem:Nkvar:X2} we may write 
\begin{equation}\label{eq:lem:Srvar:sub:X2}
X^2=Y+Z, 
\end{equation} 
where $Y$ contains all summands with pairs of equal components, i.e., $C_1(\JP)=C_2(\JP)$, and $Z$ all summands with pairs of distinct components, i.e., $C_1(\JP) \neq C_2(\JP)$. 
Now, in any graph $G$, the sum corresponding to~$Y$ is $\sum_{v,w\in V(G)}|C_v|^{r-1}|C_w|^{r-1}\indic{C_v=C_w}$, which counts
$|C|^{2r}$ for each component of $G$. Thus, specializing to $G=\JP$, we have $Y=nS_{2r,n}(\JP)$.

For $Z$ we proceed analogously to~\eqref{eq:Qkr2:palm}--\eqref{eq:lem:Nkvar:Z} in the proof of Lemma~\ref{lem:Nkvar:sub}. 
Indeed, combining standard results from the theory of point processes with inequality~\eqref{eq:E:var}, as in that proof we see that
\begin{equation*}
\E Z \le \E X \cdot \E X = (\E X)^2.
\end{equation*} 
From the bounds above we have $\E X^2=\E Y+\E Z \le (\E X)^2+n\E(S_{2r,n}(\JP))$. Hence $\Var X\le n\E(S_{2r,n}(\JP))$.
Since $X=n S_{r,n}(\JP)$, this completes the proof.
\end{proof}

\subsubsection{Some technical properties}\label{sec:nep:ed}
To facilitate the coupling arguments to come in Sections~\ref{sec:BPI:cpl} and~\ref{sec:dom:dom}, we next derive some technical properties of the random walk $\cT=(M_j,S_j)_{j \ge 0}$ associated to the exploration process.\footnote{We apologize for the clash of notation between $S_j$, the number of vertices picked up in $V_S$ after $j$ steps of our exploration, and $S_r(G)$, the $r$-th order susceptibility of a graph $G$.}
We expect that the neighbourhoods in $\JP=\JP(\fS)$ are initially `tree-like', which suggests that in the exploration process we can initially replace the sets $\cH_{j,k,r}$  of reached hyperedges by multisets which do \emph{not} depend on the so-far tested tuples.
The technical lemma below formalizes this intuition via the random multisets $\fH_{k,r}$. 
Recalling~\eqref{eq:inv:A}, note that $M_{j-1} \ge j$ if and only if $|\cA_{j-1}| \ge 1$, i.e., if the exploration has not yet finished after step~$j-1$. 
\begin{lemma}\label{lem:dom}
Let $\fS=\bigl((N_{k})_{k > K}, \: (Q_{k,r})_{k,r \ge 0} \bigr)$ be a parameter list.
Define $\cT(\fS) = (M_j,S_j)_{j \ge 0}$ as in~\eqref{def:Ti}. 
Independently for each $k,r$, 
let $\fH_{k,r}=\fH_{k,r}(\fS)$ be a random multiset where tuples $g=(w_{1}, \ldots, w_{r-1}) \in (V_L)^{r-1}$ appear according to independent Poisson processes with rate $r\hlambda_{k,r}$, where $\hlambda_{k,r}=\hlambda_{k,r}(\fS)$ is defined as in~\eqref{eq:rates:qkr}. 
Given $j \ge 1$, condition on $(M_i,S_i)_{0 \le i \le j-1}$. 
Then there is a coupling (of the conditional distribution) of 
$(M_j,S_j)$ with the $\fH_{k,r}$ such that, with probability one, we have  
\begin{align}
\label{eq:dom:Mj}
M_{j}-M_{j-1} & \le \sum_{k \ge 0,\, r \ge 2}
 \ \ \sum_{(w_1, \ldots, w_{r-1}) \in \fH_{k,r}}
 \ \ \sum_{1 \le h \le r-1}|C_{w_h}(H_L(\fS))|, \\
\label{eq:dom:Sj}
S_{j}-S_{j-1} & \le \sum_{k,r \ge 1}k|\fH_{k,r}| .
\end{align}
If, in addition, $\fS$ satisfies~\eqref{eq:small:Psi} 
and $M_{j-1} \ge j$ holds, then
we have equality in~\eqref{eq:dom:Mj} and~\eqref{eq:dom:Sj}
with probability at least $1-O((\log n)^{22}M_{j-1}/|V_L|)$, 
where the implicit constant is absolute.
\end{lemma}
For later reference we remark that, using standard splitting properties of Poisson processes, $\fH_{k,r}$ may be generated by a more tractable two-stage process. 
Namely, we first determine 
\begin{equation}\label{def:tlkr}
|\fH_{k,r}| \sim \Po(\tlambda_{k,r}) \quad \text{ with } \quad \tlambda_{k,r} = \tlambda_{k,r}(\fS) :=|V_L|^{r-1} \cdot r\hlambda_{k,r} = rQ_{k,r}/|V_L|.
\end{equation}
Then, given $|\fH_{k,r}|=y_{k,r}$, we set 
\begin{equation}\label{def:hkr}
\fH_{k,r}=\fH_{k,r}(\fS):=\bigl\{(w_{k,r,1,1}, \ldots, w_{k,r,1,r-1}), \ldots, (w_{k,r,y_{k,r},1}, \ldots, w_{k,r,y_{k,r},r-1})\bigr\}, 
\end{equation}
where all the vertices $w_{k,r,y,h} \in V_L$ are chosen independently and uniformly at random. 
\begin{proof}[Proof of Lemma~\ref{lem:dom}]
Throughout we condition on the `history' $(M_i,S_i)_{0 \le i \le j-1}$. 
If $M_{j-1} < j$ then (by \eqref{eq:inv:A}) we have $|\cA_{j-1}| =0$ and thus $(M_j,S_j)=(M_{j-1},S_{j-1})$, so all claimed bounds hold trivially. 
We may thus assume $M_{j-1} \ge j$. 

We start with the upper bounds~\eqref{eq:dom:Mj}--\eqref{eq:dom:Sj}. The plan is to embed the found hyperedges~$\cH_{j,k,r}$ into the potentially larger multiset~$\fH_{k,r}$. 
To this end we map each tuple $g \in \cH_{j,k,r}$ of the form $(v_j,w_1, \ldots, w_{r-1})$, \ldots, $(w_1, \ldots, w_{r-1},v_j)$ to~$(w_1, \ldots, w_{r-1})$; this can be done in a unique way by deleting the first coordinate which is equal to~$v_j$, say. 
Let~$\chH_{j,k,r}$ denote the resulting multiset of tuples $(w_1, \ldots, w_{r-1}) \in (V_L)^{r-1}$. 
Since~$\fH_{k,r}$ uses rate $r \cdot \hlambda_{k,r}$, by standard superposition properties of Poisson processes 
there is a natural coupling such~that 
\begin{equation}\label{eq:cpl:PjHj}
\chH_{j,k,r} \subseteq \fH_{k,r} .
\end{equation}
By definition of the exploration process we have
\begin{align}
\label{eq:cpl0:M}
M_{j}-M_{j-1} & = \Bigl|\Bigl(\bigcup_{k \ge 0,r \ge 2\ }\bigcup_{(w_1, \ldots, w_{r-1}) \in \chH_{j,k,r}}\ \bigcup_{1 \le h \le r-1} C_{w_h}(H_L) \Bigr) \setminus \Bigl(\cA_{j-1} \cup \cE_{j-1}\Bigr)\Bigr| , \\
\label{eq:cpl0:S}
S_{j}-S_{j-1} & = \sum_{k,r \ge 1}k|\chH_{j,k,r}| .
\end{align}
Together~\cref{eq:cpl:PjHj,eq:cpl0:M,eq:cpl0:S} imply~\eqref{eq:dom:Mj}--\eqref{eq:dom:Sj}.

Turning to the question of equality in~\eqref{eq:dom:Mj}--\eqref{eq:dom:Sj}, assume from now on that \eqref{eq:small:Psi} holds.
Recall that in each step $j' \le j$ the exploration process inspects all so-far untested tuples containing~$v_{j'}$. 
Hence, in the natural coupling, equality holds in~\eqref{eq:cpl:PjHj} whenever $\fH_{k,r}$ contains no vertices from 
$\cE_{j}=\{v_1, \ldots, v_{j}\} \subseteq \cA_{j-1} \cup \cE_{j-1}$. 
Set 
\begin{equation*}
 X: = \sum_{k \ge 0,\, r \ge 2}|\fH_{k,r}| 
 \qquad \text{and} \qquad
 X': = \sum_{k \ge 0,\, r \ge 2} (r-1) |\fH_{k,r}|.
\end{equation*}
Let~$\cG$ be the `good' event that, as $w_\tau$ runs over the $X'$ vertices appearing in the random multiset~$\bigcup_{k,r}\fH_{k,r}$ (cf.~\eqref{def:hkr}), 
we have
(i)~each $C_{w_\tau}(H_L)$ is disjoint from $\cA_{j-1} \cup \cE_{j-1}$ (which is itself a union of components of~$H_L$), and 
(ii)~the $X'$ components $C_{w_\tau}(H_L)$ of~$H_L$ are pairwise distinct (and so pairwise disjoint).
As noted above, property~(i) of~$\cG$ implies $\chH_{j,k,r} = \fH_{k,r}$. 
If $\cG$ holds, then the union in~\eqref{eq:cpl0:M} is disjoint, and it follows that we have equality in~\eqref{eq:dom:Mj} and~\eqref{eq:dom:Sj}.

Let $\Psi=(\log n)^2$.
To estimate the probability that $\cG$ fails to hold, we use the two stage construction of~$\fH_{k,r}$ discussed around~\eqref{def:tlkr}--\eqref{def:hkr}: we first reveal all the sizes~$|\fH_{k,r}|$, and then sequentially reveal the $X' \le (\Psi-1) X  \le \Psi X$ random vertices $w_{\tau} \in V_L$ appearing in all of the sets~$\fH_{k,r}$.
Since \eqref{eq:small:Psi} implies that all components of $H_L$ have size at most $\Psi$,
for each random vertex~$w_\tau \in V_L$ 
it then is enough to consider the event that
(i)~$w_\tau$ equals one of the $M_{j-1}$ vertices in $\cA_{j-1} \cup \cE_{j-1}$, 
or (ii)~$w_\tau$ equals one of the at most $(\tau-1)\Psi \le X' \Psi \le \Psi^2 X$ so far `discovered' vertices in $\bigcup_{1 \le x < \tau} C_{w_x}(H_L)$. 
Using conditional expectations it follows that 
\begin{equation*}
\begin{split}
\Pr(\neg\cG) \le \E\biggl( \E\Bigl( \Psi X \cdot \Bigl(\frac{M_{j-1}}{|V_L|} + \frac{\Psi^2 X}{|V_L|}\Bigr) \; \Big| \; (|\fH_{k,r}|)_{k,r \ge 0}\Bigr)\biggr) \le \Psi^3 \bigl(M_{j-1}\E X + \E X^2\bigr) / |V_L| . 
\end{split}
\end{equation*}
Noting that $X$ is a Poisson random variable (by standard superposition properties), using $\E|\fH_{k,r}| = r Q_{k,r}/|V_L| \le\Psi^2$ we see that $\Var X = \E X \le (\Psi+1)^2 \cdot \Psi^2 = O(\Psi^4)$ and $\E X^2 = \Var X + (\E X)^2 = O(\Psi^8)$. 
Recalling $M_{j-1} \ge j \ge 1$ we infer $\Pr(\neg\cG) = O(\Psi^{11} M_{j-1}/|V_L|)$, completing the proof. 
\end{proof}

We now turn to properties of the random initial values $(M_0,S_0)$. 
Since the exploration process starts with $\cA_0 = W$ and $\cE_0=\emptyset$, by~\eqref{def:Mj} we have $M_0 = |W|$. 
The next lemma is an immediate consequence of the constructions~\cref{eq:def:SW:Y,eq:def:Ti:Y,eq:def:SW:ZR,eq:def:Ti:ZR} in Section~\ref{sec:nep:fg}. The interpretation of \eqref{def:MSR} is that for $y>K$ we start at a component in $V_L$ with size $y$ (see \eqref{eq:def:SW:Y}), while for $y=0$ we essentially start from a $(z,r)$ component -- more precisely, as in \eqref{eq:def:SW:ZR} we start with $z$ vertices in $V_S$ and the union of the components of $V_L$ containing $r$ random vertices $w_1,\ldots,w_r$. For the final estimate below we use \eqref{eq:small:Psi} to bound both the number of random vertices 
and the component sizes by $\Psi=(\log n)^2$.
\begin{lemma}\label{lem:dom:init}
Let $\fS=\bigl((N_{k})_{k > K}, \: (Q_{k,r})_{k,r \ge 0} \bigr)$ be any parameter list, and
define $\cT(\fS) = (M_j,S_j)_{j \ge 0}$ as in~\eqref{def:Ti}. Let 
$(Y_{0,\fS},Z^0_{\fS},R_{\fS})$ be the probability distribution on $\NN^3$ given by 
\begin{equation}\label{def:MSR}
\Pr\bigl((Y_{0,\fS},Z^0_{\fS},R_{\fS})=(y,z,r)\bigr) = \frac{N_{y}(\fS) \indic{y>K, \: z=0, \: r=0}  + z Q_{z,r}(\fS) \indic{y=0, \: z \ge 1}}{|\fS|} ,
\end{equation}
and let $w_1,w_2,\ldots \in V_L=V_L(\fS)$ be chosen uniformly at random, independently of each other and of $(Y_{0,\fS},Z^0_{\fS},R_{\fS})$.
Then there is a coupling between $\cT(\fS)$ and $(Y_{0,\fS},Z^0_{\fS},R_{\fS},w_1,w_2,\ldots )$ such that with probability one we have 
\begin{align}
\label{eq:dom:M0}
M_0 &\le Y_{0,\fS} + \sum_{1 \le h \le R_{\fS}} |C_{w_h}(H_L(\fS))| ,\\
\label{eq:dom:S0}
S_0 & = Z^0_{\fS}.
\end{align}
Furthermore, if $\fS$ satisfies~\eqref{eq:small:Psi}, then we have equality in~\eqref{eq:dom:M0} with probability at least $1-\Psi^3/|V_L|$. \noproof 
\end{lemma}

\subsection{Idealized process}\label{sec:BPI}
Let $\fS$ be a $t$-nice parameter list (see Definition~\ref{def:nice}), which we recall roughly corresponds to the data describing a marked graph $H_i$, $i=tn$, from which we will construct $G_i$. In this subsection we compare the random walk $\cT=\cT(\fS)=(M_j,S_j)_{j \ge 0}$ of the exploration process with a closely related `idealized' branching process~$\bp_t$ that is defined \emph{without} reference to $\fS$, or indeed to~$n$.
The precise definitions (given below) are rather involved. However, given how our exploration process
treats vertices in $V_S$ and $V_L$, and that the first step of the exploration process is special,
it is not surprising that $\bp_t$ will be a special case of the following general class of branching processes.
\begin{definition}\label{def:bp:new}
Let $(Y,Z)$ and $(Y^0,Z^0)$ be probability distributions on $\NN^2$. We write
$\bp^1=\bp^1_{Y,Z}$ 
for the Galton--Watson branching process started with a single particle of type~$L$, in which each particle of type~$L$ has $Y$ children of type~$L$ and $Z$
of type~$S$. Particles of type~$S$ have no children, and the children of different particles are independent. 
We write
$\bp=\bp_{Y,Z,Y^0,Z^0}$ for the branching process defined as follows: 
start in generation one with $Y^0$ particles of type~$L$ and $Z^0$ of type~$S$.
Those of type~$L$ have children according to~$\bp^1_{Y,Z}$, independently of each other 
and of the first generation. Those of type~$S$ have no children.
We write~$|\bp|$ ($|\bp^1|$) for the total number particles in~$\bp$ $(\bp^1)$. 
\end{definition}

\subsubsection{Two probability distributions}\label{sec:BPI:distr}
For $t \in [t_0,t_1]$, to define our branching process $\bp_t$, we shall define a distribution $(Y_t,Z_t)$ on $\NN^2$ that 
gives the `idealized' (limiting) behaviour of the numbers of $V_L$ and $V_S$ vertices found in one step of our exploration
process, and a corresponding distribution $(Y_t^0,Z_t^0)$ for the first step.
Recall from Definition~\ref{def:nice}, or from Theorems~\ref{thm:init} and~\ref{thm:Qkr}, that the quantities 
$\rho_k(t_0)$ and $q_{k,r}(t)$ defined in Lemmas~\ref{lem:Nk2:t} and~\ref{lem:Qkr:t}, are (informally speaking),
the idealized versions of $N_k/n$ and $Q_{k,r}/n$ arising in $t$-nice parameter lists. Recall
also (from Remark~\ref{rem:Nk:t}) that $\rho_\omega(t_0)=\sum_{k>K} \rho_k(t_0)$.

Let~$N$ be the probability distribution on~$\NN$ with
\begin{equation}\label{def:N}
\Pr(N=k) = \indic{k > K} \rho_{k}(t_0)/\rho_\omega(t_0) .
\end{equation}
Intuitively, this corresponds to an idealized version of the distribution of $|C_w(H_L(\fS_{tn}))|$, 
where $\fS_{tn}=\fS(G^{\cR}_{n,tn})$ is the random parameter list defined in \eqref{def:fS} and
the vertex~$w$ is chosen uniformly at random from~$V_L$ (see~\eqref{HLdef} for the definition of
the `initial graph' $H_L$).
We henceforth write~$N_h$ and $N_{k,r,y,h}$ for independent copies of~$N$. 

We also define
\begin{equation}\label{def:Hkt:l}
H_{k,r,t} \sim \Po(\lambda_{k,r}(t)) \quad \text{ with } \quad \lambda_{k,r}(t) := rq_{k,r}(t)/\rho_{\omega}(t_0) ,
\end{equation}
which corresponds to an idealized version of $|\fH_{k,r}|$, see~\eqref{def:tlkr}. 
Of course, we take the random variables $H_{k,r,t}$ to be independent. 
We define 
\begin{align}
\label{def:Yt2Zt}
(Y_t,Z_t) & := \Bigl(\sum_{k \ge 0,\,r \ge 2}\  \sum_{1 \le j \le H_{k,r,t}}\  \sum_{1 \le h \le r-1} N_{k,r,j,h}, \: \sum_{k,r \ge 1} k H_{k,r,t} \Bigr) ,
\end{align}
in analogy with the quantities appearing on the right-hand side in~\eqref{eq:dom:Mj}--\eqref{eq:dom:Sj}.  
Turning to $(Y^0_t,Z^0_t)$, we
define $(Y_{0,t},Z^0_t,R_t)$ as the probability distribution on~$\NN^3$ with 
\begin{equation}\label{def:YZR}
\Pr\bigl((Y_{0,t},Z^0_t,R_t)=(y,z,r)\bigr) = \rho_y(t_0) \indic{y>K, \: z=0, \: r=0}  + z q_{z,r}(t) \indic{y=0, \: z \ge 1} ,
\end{equation}
and set
\begin{equation}\label{def:Yt0}
Y^0_t := Y_{0,t} + \sum_{1 \le h \le R_t} N_h ,
\end{equation}
in analogy with~\eqref{def:MSR}--\eqref{eq:dom:M0}.
That \eqref{def:YZR} indeed defines a probability distribution follows from Remark~\ref{rem:Qkr:sum} and $\sum_{k \ge 1}\rho_k(t_0)=1$ of Theorem~\ref{thm:init}.

For $t\in [t_0,t_1]$ we can now formally define the `idealized' branching process $\bp_t$: it is simply
\begin{equation}\label{def:bp:t}
 \bp_t := \bp_{Y_t,Z_t,Y_t^0,Z_t^0}.
\end{equation}
Of course, we define $\bp_t^1=\bp^1_{Y_t,Z_t}$ also.

Our main goal in the rest of this subsection is to prove the following result, showing that we can 
approximate the expected number of vertices in small components of $J=J(\fS)$ via~$\bp_t$. 
\begin{theorem}[Approximating the expectations of~$N_j$ and~$N_{\ge j}$]\label{thm:Nk:E} 
Let $t\in [t_0,t_1]$, and let~$\fS$ be a $t$-nice parameter list. 
Set $D_{\cT}:=\max\{D_N,D_Q\}+25$, where~$D_N, D_Q>0$ are as in Definition~\ref{def:nice}. 
Define~$J=J(\fS)$ as in Definition~\ref{def:J}.
Then
\begin{align}
\label{eq:ENkC}
\bigl|\E N_j(J) - \Pr(|\bp_t|=j) n\bigr|  & = O(j (\log n)^{D_{\cT}}n^{1/2})\hbox{ and}\\
\label{eq:ENgekC}
\bigl|\E N_{\ge j}(J) - \Pr(|\bp_t| \ge j) n\bigr|  & = O(j (\log n)^{D_{\cT}}n^{1/2}),
\end{align}
uniformly over all $j\ge 1$, $t\in [t_0,t_1]$ and all $t$-nice parameter lists $\fS$.
\end{theorem}

Before embarking on the proof, we establish the `parity constraints' that the distributions of $(Y_t,Z_t)$
and $(Y_t^0,Z_t^0)$ satisfy. 
Recall (see Lemma~\ref{lem:allowed}) that $\cSR$ denotes the set of all component sizes that the rule
$\cR$ can possibly produce, and $\per$ the period of the rule. 
Below, the precise form of the finite `exceptional set' $\{0\}\times [K]$ is irrelevant for our later argument;
it arises only due to the generality of $\ell$-vertex rules -- for Achlioptas processes $\per=1$ by Lemma~\ref{lem:allowed:AP}
and so the next lemma holds trivially.

\begin{lemma}\label{lem:inlattice}
For any $t \in (t_0,t_1]$, the following hold always: 
\[
 (Y_t,Z_t)\in (\per\NN)^2  \qquad\text{and}\qquad
 (Y_t^0,Z_t^0)\in (\per\NN)^2 \ \cup\  (\{0\}\times [K]).
\]
\end{lemma}

\begin{proof}
We first consider the distribution $N$ defined in \eqref{def:N}.
If $k\notin \cSR$ then $\rho_k(t)=0$ and so $\Pr(N=k)=0$.
By Lemma~\ref{lem:allowed}, if $k\in \cSR$ and $k>K$, then $k$
is a multiple of $\per$. It follows that $N$ can only take values in $\per\NN$.
The set of values $(k,r)$ for which $q_{k,r}(t)>0$ is described in Lemma~\ref{lem:allowed:kr}.
In \eqref{def:Yt2Zt}, we can have a contribution from a particular pair of values
$(k,r)$ only if $r\ge 1$ and $\lambda_{k,r}(t)>0$, which implies $q_{k,r}(t)>0$, i.e., $(k,r)\in \cSkr$.
By Lemma~\ref{lem:allowed:kr} this implies that $k$ is a multiple of $\per$. Since each random
variable $N$ is always a multiple of $\per$ (this follows from Lemma~\ref{lem:allowed} since $N>K$ always holds by the definition of $N$), 
it follows that $(Y_t,Z_t)\in (\per\NN)^2$ holds always.

We now turn to $(Y_t^0,Z_t^0)$. Using again that for $r\ge 1$ we can only have $q_{k,r}(t)>0$ if $k$
is a multiple of~$\per$, we see from \eqref{def:YZR} that both $Y_{0,t}$ (and hence, from \eqref{def:Yt0}, $Y_t^0$)
and $Z_t^0$ will be multiples of~$\per$ unless $(Y_{0,t},Z_t^0,R_t)=(0,z,0)$ for some~$z$
with $q_{z,0}(t)>0$. But in this case, by Lemma~\ref{lem:allowed:kr} we have $(z,0)\in \cSkr$ and 
so $z\in \cSR$. And since $R_t=0$, we have $Y^0_t=0$. 
To sum up, $(Y_t^0,Z_t^0)\in (\per\NN)^2 \ \cup\  (\{0\}\times \cSR)$ holds always. 
This completes the proof since Lemma~\ref{lem:allowed} implies $\cSR \setminus \per\NN \subseteq [K]$. 
\end{proof}

\subsubsection{Coupling}\label{sec:BPI:cpl}
In this section we prove the key coupling result relating our exploration process to $\bp_t$. We start with a technical lemma. 
The constants $b$, $B$, $D_N$ and $D_Q$ here (and throughout the section) are those in Definition~\ref{def:nice}.
Recall that $H_{k,r,t}$ is defined in \eqref{def:Hkt:l} and $\fH_{k,r}$ in Lemma~\ref{lem:dom}. 
\begin{lemma}\label{lem:dtv}
Let $t\in [t_0,t_1]$, and let~$\fS$ be a $t$-nice parameter list. 
Define a probability distribution $N'$ on $\NN$ by $N' \sim |C_{w}(H_L(\fS))|$, where $w \in V_L$ is chosen uniformly at random. 
Then, writing $\Psi = (\log n)^2$, for $n \ge n_0(b,B)$ we have
\begin{align}
\label{eq:dtvN}
\dtv{N}{N'} &= O((\log n)^{D_N+2}n^{-1/2}), \\
\label{eq:dtvUH}
\sum_{k,r \ge 0}\dtv{H_{k,r,t}}{|\fH_{k,r}|} & = O((\log n)^{D_Q+4}n^{-1/2}) , \\
\label{dtvYZR}
\dtv{(Y_{0,t},Z^0_{t},R_{t})}{(Y_{0,\fS},Z^0_\fS,R_\fS)} & = O((\log n)^{D_Q+4}n^{-1/2}) ,\\
\label{eq:UH:Ch}
\sum_{k,r \ge 0 \::\: k+r \le \Psi}\Pr(H_{k,r,t} \ge \Psi) & \le n^{-\omega(1)},\\
\label{eq:UH:tr}
 |\fH_{k,r}| &=0 \hbox{ whenever }k+r\ge \Psi.
\end{align}
\end{lemma}
\begin{proof}
For $k>K$, by the definition \eqref{def:N} of $N$ we have $\Pr(N=k)=\rho_k(t_0)/\rho_\omega(t_0)$, while 
by definition $\Pr(N'=k)=N_k(\fS)/|V_L(\fS)|$. Note also that $\sum_{k>K}\rho_k(t_0)=\rho_\omega(t_0)>0$
by Remark~\ref{rem:Nk:t}.
By the condition~\eqref{eq:Nk:t0'} of $\fS$ being $t$-nice, each $N_k(\fS)$ is within $(\log n)^{D_N} n^{1/2}$
of $\rho_k(t_0)n$. For $k>\Psi$ we have $N_k(\fS)=0$ by \eqref{eq:small:Psi}, while $\sum_{k>\Psi} \rho_k(t_0)=n^{-\omega(1)}$
by~\eqref{eq:rhok:t0:tail}. Inequality~\eqref{eq:dtvN} follows easily from these bounds.

In preparation for inequality~\eqref{eq:dtvUH}, note that $\dtv{\Po(x)}{\Po(y)} \le |x-y|$. 
Since $|\fH_{k,r}| \sim \Po(\tlambda_{k,r})$ for $\tlambda_{k,r}=rQ_{k,r}/|V_L|$ as in~\eqref{def:tlkr}, 
now~\eqref{eq:dtvUH} follows by combining~\eqref{eq:Nk:t0'}, \eqref{eq:Qkr'}, \eqref{eq:Qkr:tail'} and \eqref{eq:qkr:tail}.
The proof of~\eqref{dtvYZR} is analogous.

Since $\E H_{k,r,t} = r q_{k,r}(t)/\rho_\omega(t) = O(r e^{-b(k+r)}) = O(1)$ by~\eqref{eq:qkr:tail},
the bound~\eqref{eq:UH:Ch} follows from standard Chernoff bounds (for Poisson random variables).

The final inequality~\eqref{eq:UH:tr} is a simple consequence of $\E |\fH_{k,r}| = rQ_{k,r}/|V_L| = 0$, cf.~\eqref{def:tlkr} and~\eqref{eq:small:Psi}.  
\end{proof}

In preparation for the proof of Theorem~\ref{thm:Nk:E}, we now show that the number $|\cT|$ of vertices reached by the exploration
process defined in \eqref{def:Ti} is comparable to the number of particles~$|\bp_t|$, unless both are fairly big. 
In the light of~\eqref{eq:def:NjT:E}--\eqref{eq:def:NgejR:E}, this will be key for understanding the number of vertices 
in components of a given size in $\JP=\JP(\fS)$.
\begin{theorem}[Coupling of the exploration process and the branching process~$\bp_t$]
\label{thm:cpl:btT}
Let $t\in [t_0,t_1]$, and let~$\fS$ be a $t$-nice parameter list. 
Set $D_{\cT}:=\max\{D_N,D_Q\}+25$, where~$D_N, D_Q>0$ are as in Definition~\ref{def:nice}. 
Then there is a coupling of $\cT=\cT(\fS)$ and $\bp_t$ such that for every $\Lambda=\Lambda(n) \in \NN$,
with probability at least $1-O(\Lambda(\log n)^{D_{\cT}}n^{-1/2})$ we have $|\cT|=|\bp_t|$ or $\min\{|\cT|,|\bp_t|\}>\Lambda$. Here the implicit constant is uniform over the choice of $t$, $\fS$ and $\Lambda$.
\end{theorem}
The proof is based on a standard inductive coupling argument, exploiting that a natural one-by-one breadth-first search exploration of~$\bp_t$ induces a random walk (with respect to the number of reached vertices). 
The idea is to show that the numbers of vertices from~$V_S$ and~$V_L$ found by the exploration process are equal to the numbers of type~$S$ and~$L$ particles generated by $\bp_t$.
More formally, starting with $(M_0,S_0)$ and $(Y^0_t,Z^0_t)$, the plan is to step-by-step couple $(M_j-M_{j-1},S_{j}-S_{j-1})$ with a new independent copy of $(Y_t,Z_t)$ at each step.
With the proof of Lemma~\ref{lem:dom} in mind, the basic line of reasoning is roughly as follows: we can construct each $(M_j-M_{j-1},S_{j}-S_{j-1})$ by sampling at most $(\log n)^{O(1)}$ random vertices $\tw \in V_L$, and by~\eqref{eq:dtvN} the corresponding component sizes $|C_{\tw}(H_L)|$ can each be coupled with~$N$ up to $(\log n)^{O(1)}n^{-1/2}$ errors; a similar remark applies to the other variables, see Lemma~\ref{lem:dtv}. 
So, we expect that the coupling fails during the first $O(\Lambda)$ steps with probability at most $O(\Lambda (\log n)^{O(1)}n^{-1/2})$.

\begin{proof}[Proof of Theorem~\ref{thm:cpl:btT}]
As the statement is trivial for $\Lambda = 0$ or $\Lambda = \Omega(n^{1/2})$, we assume throughout that $1 \le \Lambda=\Lambda(n) = O(n^{1/2})$. 
The basic idea is to construct the coupling inductively, revealing $\bp_{t}$ and $\cT=\cT(\fS)=(M_j,S_j)_{j \ge 0}$ step-by-step.
We shall in fact couple the numbers of type~$L$ and type~$S$ particles of $\bp_{t}$ with the number of vertices from~$V_L$ and~$V_S$ found by the exploration process.
We shall consider only steps $0 \le j \le \Lambda$. This suffices since, 
if the coupling succeeds this far, then after $\Lambda$ steps, either the exploration process has stopped, and we have $|\cT|=|\bp_t|$,
or it has not, in which case $|\cT|>\Lambda$ and $|\bp_t|>\Lambda$.

For the base case $j=0$ we claim that there is a coupling such that, with probability $1-O((\log n)^{D_\cT}n^{-1/2})$, we have $M_0=Y^0_t$ and $S_0 = Z^0_t$. 
Writing $\Psi=(\log n)^2$ for brevity as usual,
by the final statement of Lemma~\ref{lem:dom:init} we see that, with probability $1-O(\Psi^3/n)$, equality holds in~\eqref{eq:dom:M0} and~\eqref{eq:dom:S0}.
Now the desired coupling of $(M_0,S_0)$ with $(Y^0_{t},Z^0_{t})$ is straightforward. 
Indeed, we first couple 
$(Y_{0,\fS},Z^0_\fS,R_\fS)$ with $(Y_{0,t},Z^0_{t},R_{t})$, and then for $1 \le h \le R_\fS \le \Psi$ we sequentially couple $|C_{w_h}(H_L)|$ with independent copies of~$N$. 
Using Lemma~\ref{lem:dtv} it follows that the described coupling
fails with probability at most  
\[
O(\Psi^3/n) + \dtv{(Y_{0,t},Z^0_{t},R_{t})}{(Y_{0,\fS},Z^0_\fS,R_\fS)} + \Psi \cdot \dtv{N}{N'} = O((\log n)^{D_{\cT}}n^{-1/2}) .
\]

Turning to the inductive step, consider $0<j\le \Lambda$.
We may assume that $M_{j-1} \ge j$, since otherwise the exploration has stopped already, and
that $M_{j-1} \le \Lambda$, since otherwise $|\cT|,|\bp_t|>\Lambda$. 
It suffices to show that, conditioning on the first $j-1$ steps of the exploration,
there is a coupling such that, with probability $1-O((\log n)^{D_{\cT}} n^{-1/2})$, we have $(M_j-M_{j-1},S_{j}-S_{j-1})=(Y_{t},Z_{t})$. 
By Lemma~\ref{lem:dom}, with probability $1-O(\Psi^{11} M_{j-1}/n) = 1- O(\Psi^{11}n^{-1/2})$, 
we have equality in~\eqref{eq:dom:Mj} and~\eqref{eq:dom:Sj}.
Recalling the two-stage process generating $\fH_{k,r}$ described around~\eqref{def:tlkr}--\eqref{def:hkr}, the desired coupling is then straightforward in view of Lemma~\ref{lem:dtv}.
Indeed, we first couple each $|\fH_{k,r}|$ with $H_{k,r,t}$. 
Of course, we abandon our coupling whenever $\max_{k +r \le \Psi}H_{k,r,t} \ge \Psi$, 
which by Lemma~\ref{lem:dtv} occurs with probability at most $n^{-\omega(1)}$. 
After this first step, in order to couple $(M_j-M_{j-1},S_{j}-S_{j-1})$ with $(Y_{t},Z_{t})$, by the two-stage definition of $\fH_{k,r}$ it remains to sequentially couple at most $(\Psi+1)^2 \cdot \Psi \le \Psi^4$ independent copies of $N' \sim |C_{w}(H_L)|$ with independent copies of~$N$.
To sum up, using Lemma~\ref{lem:dtv} it follows that the described coupling 
fails with probability at most  
\[
O(\Psi^{11}n^{-1/2}) + \sum_{k \ge 0,\,r \ge 0}\dtv{H_{k,r,t}}{|\fH_{k,r}|} + n^{-\omega(1)} + \Psi^4 \cdot \dtv{N}{N'} = O((\log n)^{D_{\cT}}n^{-1/2}) .
\]
Since we only consider $\Lambda+1=O(\Lambda)$ 
steps $j$ in total, this completes the proof of Theorem~\ref{thm:cpl:btT}.
\end{proof}

\subsubsection{Proof, and consequences, of Theorem~\ref{thm:Nk:E}}\label{sec:proofs:Nk}
The proof of Theorem~\ref{thm:Nk:E} hinges on the basic observation that, for any graph, adding or deleting an edge changes the number of vertices in components of size $j$ (at least $j$) by at most $2j$. 
Since $\JP=\JP(\fS)$ is a Poissonized version of~$J=J(\fS)$ we thus expect that $\E N_{k}(J) \approx \E N_{k}(\JP)$. 
The approximation $\E N_{k}(\JP) \approx \Pr(|\bp_{t}|=k)n$ then follows from~\eqref{eq:def:NjT:E} and the coupling of Theorem~\ref{thm:cpl:btT}. 
\begin{proof}[Proof of Theorem~\ref{thm:Nk:E}]
Let $\JP=\JP(\fS)$ and $\cT=\cT(\fS)$. First, using~\eqref{eq:def:NjT:E} and Theorem~\ref{thm:cpl:btT} we have 
\begin{equation}\label{eq:Nk:E:0}
\bigl|\E N_{j}(\JP) - \Pr(|\bp_t|=j)n \bigr| = \bigl|\Pr(|\cT|=j) - \Pr(|\bp_t|=j)\bigr| \cdot n = O(j (\log n)^{D_{\cT}}n^{1/2}). 
\end{equation}
Next, we relate $\E N_{j}(J)$ and $\E N_{j}(\JP)$ by de-Poissonization, defining $Y_{k,r} \sim \Po(Q_{k,r}(\fS))$ for convenience (as usual, all these random variables are independent). 
From the definition of $J=J(\fS)$, 
we see that the effect of increasing of $Q_{k,r}$ by $1$ on the random 
graph $J=J(\fS)$ may be thought of as follows: we first add a new component of size $k$ (which changes any $N_j(J)$ by at most $k$), and
then we add $r$ edges. From the definition of
$\JP=\JP(\fS)$, using a natural coupling and our basic Lipschitz observation it follows that we may couple $J$ and $\JP$ so that
\begin{equation}\label{eq:Nk:E:1}
\begin{split}
|N_{j}(J) - N_{j}(\JP)| &\le \sum_{k,r \ge 0} \bigl(2jr+k\bigr) \cdot |Y_{k,r}-Q_{k,r}(\fS)| . 
\end{split}
\end{equation}
For any random variable $Z \sim \Po(\mu)$ we have $\E |Z-\mu| \le \sqrt{\Var Z} = \sqrt{\mu}$ by Jensen's inequality.  
Since $\E Y_{k,r} = Q_{k,r}(\fS) \le B e^{-b(k+r)}n$ by~\eqref{eq:Qkr:tail'}, using~\eqref{eq:Nk:E:1} it follows that $|\E N_{j}(J) - \E N_{j}(\JP)| = O(jn^{1/2})$, which together with~\eqref{eq:Nk:E:0} and $D_{\cT} \ge 1$ completes the proof of inequality~\eqref{eq:ENkC} for~$\E N_{j}(J)$.

Turning to $\E N_{\ge j}(J)$, from \eqref{eq:def:NgejR:E} and Theorem~\ref{thm:cpl:btT} we have 
\begin{equation*}
 \bigl|\E N_{\ge j}(\JP) - \Pr(|\bp_t| \ge j)n\bigr| = \bigl|\Pr(|\cT| \ge j) -\Pr(|\bp_t| \ge j)\bigr| \cdot n = O(j (\log n)^{D_{\cT}}n^{1/2}). 
\end{equation*}
Now, changing $N_{j}(\cdot)$ to $N_{\ge j}(\cdot)$, the rest of the argument for~\eqref{eq:ENkC} carries over to prove~\eqref{eq:ENgekC}.
\end{proof}
We next prove two corollaries to Theorem~\ref{thm:Nk:E}, relating~$\bp_t$ to (i)~the solutions to certain differential equations from Section~\ref{sec:DEM:small} and (ii)~the functions $\rho$ and $s_r$ from Sections~\ref{sec:intro:L1} and~\ref{sec:intro:Sj}. 
\begin{corollary}
\label{cor:Nk}
Let $(\rho_{k})_{k \ge 1}$ and $(\rho_{\ge k})_{k \ge 1}$ be the functions defined in Lemma~\ref{lem:Nk2:t}.
Then $\rho_{k}(t) = \Pr(|\bp_t|=k)$ and $\rho_{\ge k}(t) = \Pr(|\bp_t| \ge k)$ for all $t \in [t_0,t_1]$ and $k \ge 1$. 
\end{corollary}

\begin{proof}
Fix $t \in [t_0,t_1]$ and $k \ge 1$. 
Using the fact, proved in Lemma~\ref{lem:nice}, that the random parameter list $\fS_{tn}=\fS(G^\cR_{n,tn})$ is almost always nice, 
and the conditioning lemma, Lemma~\ref{lem:cond'}, it is easy to deduce from Theorem~\ref{thm:Nk:E}
that $|\E N_k(tn) - \Pr(|\bp_t|=k) n| = O(k (\log n)^{D_{\cT}}n^{1/2})$. 
Also, from Lemma~\ref{lem:Nk2:t} it is easy to see that $|\E N_k(tn) - \rho_k(t) n| = O((\log n)^{2}n^{1/2})$. 
As $n \to \infty$, it follows that 
\[
\bigl|\Pr(|\bp_t|=k) - \rho_k(t)\bigr| = O\bigl(k (\log n)^{D_{\cT}}n^{-1/2}+(\log n)^{2}n^{-1/2}\bigr) = o(1) .
\]
Since~$\bp_t$ and~$\rho_k(t)$ are both defined without reference to~$n$, it follows that $\Pr(|\bp_t|=k)=\rho_k(t)$. 
The same argument (with obvious notational changes) gives $\rho_{\ge k}(t) = \Pr(|\bp_t| \ge k)$.
\end{proof} 

\begin{corollary}
\label{cor:rhosr}
Let the functions $\rho$ and $(s_r)_{r \ge 2}$ be as in~\eqref{eq:L1:pto} and~\eqref{eq:Sj:pto}. 
Then $\rho(t) = \Pr(|\bp_t|=\infty)$ for $t \in [t_0,t_1]$, and $s_r(t) = \E |\bp_t|^{r-1} \in [1,\infty)$ for all $t \in [t_0,\tc)$ and $r \ge 2$.  
\end{corollary}

\begin{proof}
For $t \in [0,\infty)$, Theorem~3 and Section~5 of~\cite{RWapcont} imply $\rho(t)=1-\sum_{k \ge 1} \rho_k(t)$.  
For $t \in [t_0,t_1]$, by Corollary~\ref{cor:Nk} we conclude 
\begin{equation}\label{eq:rho:sum}
\rho(t)=1-\sum_{k \ge 1} \rho_k(t) = 1-\sum_{k \ge 1} \Pr(|\bp_t|=k)= \Pr(|\bp_t|=\infty) .
\end{equation}

For $t \in [0,\tc)$, the main result of~\cite{RWapsubcr} implies $\sum_{k \ge 1}\rho_k(t) = 1$ and $S_r(tn) \pto \sum_{k \ge 1}k^{r-1}\rho_k(t) \in [1,\infty)$ for $r \ge 2$. 
For $t \in [t_0,\tc)$, using~\eqref{eq:rho:sum} and Corollary~\ref{cor:Nk} we infer $\Pr(|\bp_t|=\infty) = 0$ and $\sum_{k \ge 1}k^{r-1}\rho_k(t)=\E |\bp_t|^{r-1}$, so that~\eqref{eq:Sj:pto} implies $s_r(t) = \E |\bp_t|^{r-1} \in [1,\infty)$. 
\end{proof}

\subsection{Dominating processes}\label{sec:dom} %
In this subsection we relate the random graph~$J=J(\fS)$ more directly to the Poissonized model~$\JP=\JP(\fS)$. 
Loosely speaking, for any `time' $t \in [t_0,t_1]$ the plan is to slightly adjust (decrease or increase) the parameters of~$\fS$, see~\eqref{def:fS:rep}, and `sandwich' $J=J(\fS)$ between $\JPpm=\JP(\fS^{\pm}_t)$ such that \emph{typically} $\JPm \subseteq J \subseteq \JPp$. 
Using stochastic domination this will allow us to avoid one major drawback of the coupling arguments from Section~\ref{sec:BPI}, namely, that the `coupling errors' deteriorate for moderately large component sizes (see Theorems~\ref{thm:Nk:E} and~\ref{thm:cpl:btT}).
The basic idea is that the number of vertices of $\JPpm$ found by the exploration process is bounded from above and below by perturbed variants~$\bp^{\pm}_t$ of the idealized branching process~$\bp_{t}$, formalized in Definition~\ref{def:bp2pm} below.

\subsubsection{Sandwiching}\label{sec:dom:sw}
Given parameter lists $\fS$ and $\fS'$, we write $\fS\preceq\fS'$
if $\fS=\bigl((N_{k})_{k > K}, \: (Q_{k,r})_{k,r \ge 0} \bigr)$
and $\fS'=\bigl((N_{k})_{k > K}, \: (Q_{k,r}')_{k,r \ge 0} \bigr)$ with $Q_{k,r}\le Q'_{k,r}$
for all $k,r\ge 0$. From the definition (Definition~\ref{def:J}) of the random graph $J(\fS)$, if $\fS\preceq\fS'$
then there is a coupling such that
\begin{equation}\label{eq:J:mon}
 J(\fS) \subseteq J(\fS').
\end{equation}
Similarly, 
if $\fS\preceq\fS'$
then (see Definition~\ref{def:F}) 
there is a coupling such that
\begin{equation}\label{eq:F:mon}
 \JP(\fS) \subseteq \JP(\fS').
\end{equation}
Our next lemma states that, for any $t$-nice~$\fS$,
we can whp sandwich~$J(\fS)$ between the more tractable Poissonized random graphs~$\JPpm = \JP(\fS^{\pm}_t)$, for parameter lists $\fS^\pm_t$ which
we now define. 
Recall that in this case $Q_{k,r}=Q_{k,r}(\fS) \approx q_{k,r}(t) n$, with $q_{k,r}(t)=0$ when $(k,r)\not\in \cSkr$ (see~\eqref{eq:Qkr'} and Lemma~\ref{lem:allowed:kr}).

\begin{definition}\label{def:cpl}
Set $B_0=2/b$ and $b_0=b/400$, where~$b$ is as in Definition~\ref{def:nice}. 
Given $t\in [t_0,t_1]$ and a $t$-nice parameter list~$\fS$, let
\begin{gather}
\label{eq:qkr:pm}
q^{\pm}_{k,r,n}(t) := \indic{k,r \ge 0, \: k+r \le B_0 \log n}\indic{(k,r)\in \cSkr} \max\left\{q_{k,r}(t) \pm e^{-b_0(k+r)}n^{-0.49},0\right\} ,\\
\label{eq:Qkr:pm}
Q^{\pm}_{k,r}(t) := q^{\pm}_{k,r,n}(t)n, 
\end{gather}
and define the parameter lists $\fS^+_t$ and $\fS^-_t$ by
\begin{equation}\label{eq:fS:pm}
 \fS^{\pm}_t := \Bigl(\bigl(N_{k}\bigr)_{k > K}, \: \bigl(Q^{\pm}_{k,r}(t)\bigr)_{k,r \ge 0} \Bigr) .
\end{equation}
\end{definition}
Note that the definition of $\fS^\pm_t$ depends not only on $\fS$ (via the $N_k$), but also on $t$ and on $n$.
The parameters~$Q^{\pm}_{k,r}(t)$ defined in~\eqref{eq:Qkr:pm} do \emph{not} depend on~$\fS$.
Of course, in~\eqref{eq:qkr:pm} the precise numerical value~$0.49$ is irrelevant for our later arguments (any $\gamma \in (1/3,1/2)$ suffices). 
The second indicator function in \eqref{eq:qkr:pm} simply restricts the types of $(k,r)$-components to ones that can possibly 
appear in $G^{\cR}_{n,i}$; see Lemma~\ref{lem:allowed:kr}.

\begin{lemma}[Sandwiching between Poissonized random graphs]
\label{lem:cpl}
Let $t\in [t_0,t_1]$, and let $\fS=\bigl((N_{k})_{k > K}, \: (Q_{k,r})_{k,r \ge 0} \bigr)$ be a $t$-nice parameter list. 
Define $\fS^\pm_t$ as in Definition~\ref{def:cpl}. Then
there is a coupling such that we have
\begin{equation}\label{eq:cpl}
 \JP(\fS^{-}_t) \subseteq J(\fS) \subseteq \JP(\fS^{+}_t) 
\end{equation}
with probability $1-n^{-\omega(1)}$.
\end{lemma}
\begin{proof}
We can construct $\JP(\fS_t^\pm)$ by first exposing the associated `Poissonized' parameters $Y_{k,r} \sim \Po(Q_{k,r}^\pm)$, and then setting
\begin{equation}\label{eq:lem:cpl:Fi}
\JP(\fS_t^\pm) = J(\fS^{*}) \quad \text{ with } \quad \fS^{*}=\Bigl(\bigl(N_{k}\bigr)_{k > K}, \: \bigl(Y_{k,r}\bigr)_{k, r \ge 0}\Bigr).
\end{equation} 
By \eqref{eq:J:mon} it thus suffices to show that the associated random parameters used in typical realizations of $\JP(\fS^{\pm}_t)$ sandwich
the parameter list $\fS$ from above and below.
We shall prove this using standard Chernoff bounds, which imply that any Poisson random variable~$Y$ with mean $\mu = O(n)$ satisfies $\Pr(|Y-\mu| \ge (\log n)^2n^{1/2}) \le n^{-\omega(1)}$, say. 

Turning to the details, let $b,B>0$ be as in~\eqref{eq:Qkr:tail'}, and $B_0=2/b$, $b_0=b/400$ as in Definition~\ref{def:cpl}.
Define $Y^{\pm}_{k,r} \sim \Po(Q^{\pm}_{k,r}(t))$. 
We henceforth consider only $(k,r)\in \cSkr$, since otherwise $Y_{k,r}^\pm=0$ by definition and $Q_{k,r}=0$ by~\eqref{eq:nice:allowed}.
Note that, with $D_Q>0$ as in~\eqref{eq:Qkr'}, for $n \ge n_0(D_Q)$ we have 
\begin{equation}\label{eq:cpl:lb}
\min_{k+r \le B_0 \log n}  e^{-b_0(k+r)}n^{-0.49} \ge n^{-0.495} \ge 4 (\log n)^{\max\{D_Q,2\}}n^{-1/2}. 
\end{equation}
So, using the discussed Chernoff bounds, with probability $1-n^{-\omega(1)}$ we thus obtain 
\begin{gather}
\label{eq:qkr:pm:ub}
Y^{+}_{k,r} \ge Q^{+}_{k,r}(t) -(\log n)^2n^{1/2} \ge q_{k,r}(t)n + 2 (\log n)^{D_Q}n^{1/2} ,\\
\label{eq:qkr:pm:lb}
Y^{-}_{k,r} \le \indic{Q^{-}_{k,r}(t) > 0}\bigl(Q^{-}_{k,r}(t) +(\log n)^2n^{1/2}\bigr) \le \indic{Q^{-}_{k,r}(t) > 0}\bigl(q_{k,r}(t) n -2 (\log n)^{D_Q}n^{1/2}\bigr) 
\end{gather}
simultaneously for all $k,r \ge 0$ with $k+r \le B_0 \log n$. 
Comparing \eqref{eq:qkr:pm:ub}--\eqref{eq:qkr:pm:lb} with~\eqref{eq:Qkr'}, using $Q_{k,r} \ge 0$ it follows that with probability at least $1-n^{-\omega(1)}$ we have
\begin{equation}\label{eq:Y-QY+}
 Y^{-}_{k,r} \le Q_{k,r} \le Y^{+}_{k,r} 
\end{equation}
for all $k,r$ with $k+r\le B_0\log n$.

Turning to $k+r> B_0\log n$, 
if $n$ is large enough ($n \ge n_0(b_0,B_0,B)$), then we have $Q_{k,r}=0$ by \eqref{eq:Qkr:stop}, while
$q_{k,r,n}^\pm=0$ by definition.
Hence the inequalities \eqref{eq:Y-QY+} hold trivially, completing the proof.
\end{proof}
\begin{remark}\label{rem:cpl2}
In view of~\eqref{eq:qkr:tail} and \eqref{eq:qkr:pm}, setting $B_1 = B + 1$ and $b_0 = b/400$, for $n \ge 1$ we have
\begin{equation}
\label{eq:qkr:pm:tail}
\sup_{t \in [t_0,t_1]}q^{\pm}_{k,r,n}(t) \le B_1 e^{-b_0(k+r)}.
\end{equation} 
\end{remark}

\subsubsection{Stochastic domination}\label{sec:dom:dom}
Let~$\fS$ be a $t$-nice parameter list, and define $\fS^{\pm}_t$ as in Definition~\ref{def:cpl}. 
Aiming at studying the component size distribution of $\JPpm=\JP(\fS^\pm_t)$,
we shall next define branching processes~$\bp_t^{1,\pm}$ and~$\bp_t^{\pm}$. 
Since~$\fS^\pm_t$ depends on~$\fS$ and on~$n$, these branching processes will also depend on~$\fS$ and~$n$,
in contrast to the idealized processes~$\bp^1_t$ and~$\bp_t$ defined in Section~\ref{sec:BPI}.
To define our branching processes, we need to define the corresponding offspring distributions. 
For $\bp_t^-$ this is a little involved, 
since we need to deal with `clashes' in the exploration process.
The main idea is that the exploration process finds subsets of the vertices which resemble `typical' subgraphs of~$\JP$. For example, in view of~\eqref{eq:Nk:t0:tail'} we expect that only a $e^{-\Omega(k)}$-fraction of the discovered vertices originate from size-$k$ components of $V_L$ (for any~$k \ge 1$), 
which intuitively explains the definition of~$N^{-}=N^{-}(\fS)$ given in~\eqref{eq:N:LB} below, bearing in mind that we
will only consider up to $O(n^{2/3})$ exploration steps. 

First, let us fix some constants.
Given $a,A>0$ as in~\eqref{eq:Nk:t0:tail'}, let
\begin{equation}\label{eq:def:D0a0}
 D_0 := 2/a \qquad \text{and} \qquad a_0:=\min\{a/2,1/(6 D_0)\}.
\end{equation}
Given~$\zeta>0$ as in~\eqref{eq:nice:VL}, let
\begin{equation}\label{def:A0}
A_0 := 16/\zeta \cdot \sum_{k \ge 1}kA e^{-a_0 k} \qquad \text{and} \qquad A_1 := A_0/\zeta .
\end{equation}

\begin{definition} \label{def:cpl:UB}
Given $t\in [t_0,t_1]$ and a $t$-nice parameter list~$\fS$, define a probability distribution $N^+$ on $\NN$ by
\begin{equation}\label{eq:N:UB}
 \Pr(N^{+}=k) = \indic{k > K}N_k(\fS)/|V_L(\fS)|,
\end{equation}
and set
\begin{equation}\label{eq:lambda:kr:UB}
 \lambda^{+}_{k,r}(t) := r Q^+_{k,r}(t)/|V_L(\fS)| ,
\end{equation}
where $Q^{\pm}_{z,r}(t)$ is defined in Definition~\ref{def:cpl}.
Similarly, define a probability distribution $N^-$ on $\NN$ by
\begin{equation}\label{eq:N:LB}
\Pr(N^{-}=k) = \begin{cases}
	\indic{k > K}\max\bigl\{N_k(\fS)-A_0 e^{-a_0 k} n^{2/3},0\bigr\}/|V_L(\fS)|, & ~~\text{if $k \ge 1$}, \\ 
	1-\Pr(N^{-} \ge 1) , & ~~\text{if $k = 0$},
	\end{cases}
\end{equation}
and set
\begin{equation}\label{eq:lambda:kr:LB}
 \lambda^{-}_{k,r}(t) := \max\bigl\{\bigl(1- A_1 r n^{-1/3}\bigr)rQ^-_{k,r}(t)/|V_L(\fS)|,0\bigr\}.
\end{equation}
Define two probability distributions $(Y^{\pm}_{0,t},Z^{0,\pm}_t,R^{\pm}_t)$ on $\NN^3$ by
\begin{equation}\label{eq:YZR:dom}
 \Pr\bigl((Y^{\pm}_{0,t},Z^{0,\pm}_t,R^{\pm}_t)=(y,z,r)\bigr) = \frac{N_y(\fS)\indic{y > K, \: z=0, \: r=0}  + z Q^{\pm}_{z,r}(t)\indic{y=0, \: z \ge 1} }{|\fS^{\pm}_t|} ,
\end{equation}%
where $Q^{\pm}_{z,r}(t)$ and $\fS^\pm_t$ are defined in Definition~\ref{def:cpl}.
Finally, define $(Y_t^{\pm},Z_t^{\pm})$ and $(Y_t^{0,\pm},Z_t^{0,\pm})$
analogous to $(Y_t,Z_t)$ and $(Y_t^{0},Z_t^{0})$ in
\eqref{def:Hkt:l}--\eqref{def:Yt2Zt} and \eqref{def:YZR}--\eqref{def:Yt0}, but with $N$, $\lambda_{k,r}(t)$ and $(Y_{0,t},Z^0_t,R_t)$
replaced by $N^{\pm}$, $\lambda^{\pm}_{k,r}(t)$ and $(Y^{\pm}_{0,t},Z^{0,\pm}_t,R^{\pm}_t)$.
\end{definition}
\begin{definition}\label{def:bp2pm}
Given $t\in [t_0,t_1]$ and a $t$-nice parameter list~$\fS$, we define $\bp_t^{\pm}=\bp_t^{\pm}(\fS)$ as $\bp_{Y_t^\pm,Z_t^\pm,Y^{0,\pm}_t,Z^{0,\pm}_t}$,
where the general branching process definition is given in Definition~\ref{def:bp:new},
and the offspring distributions $(Y_t^{\pm},Z_t^{\pm})$ and $(Y_t^{0,\pm},Z_t^{0,\pm})$
are defined as above. Similarly, we define $\bp_t^{1,\pm}=\bp_t^{1,\pm}(\fS)$ as $\bp^1_{Y_t^\pm,Z_t^\pm}$.
\end{definition}

One of our main goals is to show that we can approximate the expected number of vertices in components of at least a given 
(large) size in $\JPpm=\JP(\fS^{\pm}_t)$ via the branching processes~$\bp_t^{\pm}$; see Theorem~\ref{thm:ENgekD} below.
To do this, we need to relate, via stochastic domination, the exploration processes in $\JPpm$ to the branching 
processes $\bp_t^{\pm}$.

Before turning to the domination arguments, we first record two simple observations.
Firstly, $|\fS^{\pm}_t| \approx n$ for $t$-nice~$\fS$. 
More precisely, using~\eqref{eq:nice:n}, \eqref{eq:Qkr'}, \eqref{eq:qkr:pm}--\eqref{eq:Qkr:pm}, \eqref{eq:Qkr:stop}, with
$B_0 = 2/b$ we have 
\begin{equation}\label{def:n:pm:diff}
\begin{split}
& \Bigl||\fS^\pm_t|-n\Bigr| = \Bigl||\fS^\pm_t|-|\fS|\Bigr|  = \Bigl| |V_S(\fS^{\pm}_t)| - |V_S(\fS)| \Bigr| \le \sum_{k \ge 1,\, r \ge 0}k |Q^\pm_{k,r}(t)-Q_{k,r}(\fS)| \\
& \qquad \le \sum_{\substack{k \ge 1, r \ge 0:\\k+r \le B_0 \log n}} k \bigl[e^{-b_0(k+r)}n^{0.51} + (\log n)^{D_Q} n^{1/2}\bigr] = O(n^{0.51}) .
\end{split}
\end{equation}

The next observation concerns parity constraints.
\begin{lemma}\label{lem:inlattice:dom}
For any $t$-nice parameter list $\fS$ with $t \in (t_0,t_1]$, the random variables defined above satisfy
\[
 (Y_t^{\pm},Z_t^{\pm}) \in (\per\NN)^2 \quad\text{and}\quad
 (Y_t^{0,\pm},Z_t^{0,\pm}) \in (\per\NN)^2 \ \cup\  (\{0\}\times [K])
\]
with probability~$1$.
\end{lemma}
\begin{proof}
For $k > 0$, note that $\Pr(N^{\pm}=k)>0$ implies~$k>K$ and $N_k(\fS)>0$, so $k\in \cSR \setminus [K]$ by~\eqref{eq:nice:allowed}. 
By Lemma~\ref{lem:allowed} it follows that $N^{\pm}$ can only take values in $\per \NN$. 
Furthermore, $\lambda_{k,r}^\pm(t)>0$ implies $Q_{k,r}^\pm(t)>0$, and hence, by \eqref{eq:qkr:pm}, that $(k,r)\in \cSkr$.
Now the remaining argument of Lemma~\ref{lem:inlattice} carries over. 
\end{proof}

The following theorem states that the number of vertices found by the exploration process $\cT^+=\cT(\fS_t^+)$ is dominated from above by the total size of the branching process $\bp^{+}_{t}$. 

\begin{theorem}[Stochastic domination of the exploration process `from above']
\label{thm:cpl:UB}
Let $t\in [t_0,t_1]$, and let $\fS$ be a $t$-nice parameter list. Define $\fS^+_t$ as in Definition~\ref{def:cpl},
and the branching process $\bp^+_t=\bp_t^+(\fS)$ as in Definition~\ref{def:bp2pm}.
Define the exploration process $\cT^+=\cT(\fS^+_t)$ as in~\eqref{def:Ti}--\eqref{def:Ti:size} of Section~\ref{sec:nep}. 
Then there is a coupling such that $|\cT^+| \le |\bp^{+}_{t}|$. 
\end{theorem}
\begin{proof}
By definition of $\cT^+$ and $\bp^{+}_{t}$ it suffices to show that there is a coupling satisfying the following properties with probability one: 
(i)~for $j=0$ the initial values satisfy $M_0 \le Y^{0,+}_{t}$ and $S_0 \le Z^{0,+}_{t}$, 
and (ii)~for every $j \ge 1$ the step-wise differences satisfy $M_{j}-M_{j-1} \le Y^{+}_{t}$ and $S_{j}-S_{j-1} \le Z^{+}_{t}$. 
This claim is an immediate consequence of Lemmas~\ref{lem:dom} and~\ref{lem:dom:init}. 
Indeed, the case~$j=0$ follows from Lemma~\ref{lem:dom:init} (noting that $(Y^{+}_{0,t},Z_t^{0,+},R_t^{+})=(Y_{0,\fS^+_t},Z^0_{\fS^+_t},R_{\fS^+_t})$ and $N_k(\fS^+_t) = N_k(\fS)$ hold), and the case $j \ge 1$ follows by applying Lemma~\ref{lem:dom} inductively.
\end{proof}

The next theorem states that the exploration process $\cT^-$ is dominated from below by the branching process $\bp^{-}_{t}$ until both have found many vertices, namely at least~$n^{2/3}$ (this cutoff aims at simplicity rather than the best bounds).  
\begin{theorem}[Stochastic domination of the exploration process `from below']
\label{thm:cpl:LB}
Let $t\in [t_0,t_1]$, and let $\fS$ be a $t$-nice parameter list.
Define $\fS^-_t$ as in Definition~\ref{def:cpl},
and the branching process $\bp^-_t=\bp_t^-(\fS)$ as in Definition~\ref{def:bp2pm}.
Define the exploration process $\cT^-=\cT(\fS^-_t)$ as in~\eqref{def:Ti}--\eqref{def:Ti:size} of Section~\ref{sec:nep}. 
Then there is a coupling such that, with probability $1-n^{-\omega(1)}$, we have $|\cT^-| \ge |\bp^-_{t}|$ or $\min\{|\cT^-|,|\bp^-_{t}|\} > n^{2/3}$. 
\end{theorem}
\begin{proof}
We follow the approach of Theorem~\ref{thm:cpl:btT}, and think of the exploration process as sequence of random vertex sampling steps. 
Intuitively, to achieve domination we (i)~only sample from a subset of all vertices, and (ii)~give up as soon as certain unlikely events occur (corresponding to `atypical' explorations). 

Turning to the details, for brevity, define 
\[
 \Lambda := n^{2/3} .
\]
Given $\zeta>0$ as in~\eqref{eq:nice:VL}, note that for $n \ge n_0(\zeta,A_0)$ large enough we have
\begin{equation}\label{def:Nk:xi}
 |V_L|-A_0\Lambda \ge \zeta n - A_0n^{2/3} \ge \zeta/2 \cdot n .
\end{equation}
Note that, in view of~\eqref{eq:Nk:t0:tail'} and the definition \eqref{eq:def:D0a0} of $a_0,D_0$,
for $n \ge n_0(a,A)$  we have
\begin{gather}
\label{eq:cpl:B0}
N_k(\fS) = 0 \quad\text{for all}\quad k \ge D_0 \log n , \\
\label{eq:cpl:a0}
\min_{k \le D_0 \log n} e^{-2a_0k}\Lambda  \ge n^{1/3} .
\end{gather}
We now introduce~$\cT^*=\cT^*(\fS^-_t)$, which is a slight modification of the exploration process for $\JPm=\JP(\fS^-_t)$ described in Section~\ref{sec:nep}.
The initial set $W \subseteq V_L$ and initial value $S_0$ are chosen exactly as for $\cT(\fS^-_t)$. 
Given these, $\cT^*$ finds a subset 
\begin{equation}
\label{eq:cpl:CW}
C^-_{W}(\JPm) \subseteq C_{W}(\JPm) 
\end{equation}
of the set $C_{W}(\JPm)$ of vertices of $\JPm$ reachable from $W$. 
The only difference from $\cT(\fS^-_t)$ is that in each step with $j \ge 1$, when we explore $v_j \in \cA_{j-1}$, we only test for new $V_L$--vertices from 
\[
 V_{L,j} := V_{L} \setminus \bigl(\cE_{j-1} \cup \{v_{j}\}\bigr), 
\] 
i.e., only consider a \emph{subset} of the hyperedges tested by the original process. 
More precisely, we consider (test for their presence in $\JPm=\JP(\fS^-_t)$) all $k$-weighted hyperedges $g\in (V_L)^r$ of the form $(v_j,w_1, \ldots, w_{r-1})$, \ldots, $(w_1, \ldots, w_{r-1},v_j)$ with all $w_{h} \in V_{L,j}$. 
A key point is that (since $v_j$ is added to the `explored' set after step~$j$) a given hyperedge is tested at most once in this process. 
Comparing the numbers of vertices found by $\cT^-=\cT(\fS^-)$ and $\cT^*$, see~\eqref{def:Ti:size} and~\eqref{eq:cpl:CW},  we infer that 
\begin{equation}
\label{eq:cpl:T}
|\cT^-| = |C_{W}(\JPm)| + S_0 \ge |C^-_{W}(\JPm)| + S_0 = |\cT^*| .
\end{equation}

Analogous to the proof of Theorem~\ref{thm:cpl:btT}, in view of~\eqref{eq:cpl:T} the basic idea is to inductively couple $\cT^*=(M_j,S_j)_{j \ge 0}$ with $\bp^-_{t}$, such that $M_j$ and $S_j$ dominate (from above) the corresponding number of type~$L$ and~$S$ particles found in the exploration of $\bp^-_{t}$.
Within each step $j$ of this coupling we shall perform a (random) number of \emph{vertex sampling steps}. 
Recalling that $\cE_j\cup \cA_j$ is the set of vertices reached after $j$ steps, in each vertex sampling step we shall
(i) reveal a random vertex and add either zero or one components of $H_L$ to the set of reached vertices, and (ii) reveal
a new independent random variable with distribution $N^-$ (for the details see below). 
We stop constructing the coupling as soon as, after \emph{any} vertex sampling step, either of the following properties holds: 
\begin{itemize}\itemsep1pt \parskip0pt \parsep0pt
	\item[(P1)] we can already witness that $|\bp_t^-| > \Lambda$, or 
	\item[(P2)] the set of reached vertices contains, for some $k \ge 1$, more than $A_0 e^{-a_0 k} \Lambda$ vertices from $V_L$--components of size at least~$k$.
\end{itemize}%
Of course, we also stop constructing the coupling if after the end of some step $j$ we have $\cA_j=0$, i.e., the exploration has finished.
Note that (P1) says that, in our coupled exploration of the branching process $\bp_t^-$, we have already reached more than $\Lambda$
particles. If we either complete the coupling, or stop due to (P1), then we say the coupling \emph{succeeds}; if we stop
due to (P2), the coupling \emph{fails}. From the way the coupling is defined (below), if
we reach the end of the exploration we have $|\cT^*|\ge |\bp_t^-|$, while if we stop due to (P1) then, because the coupling succeeded
up to this point (so we have reached at least as many vertices in $\cT^*$ as particles in $\bp_t^-$),
we have $|\cT^*|,|\bp_t^-|>\Lambda$. We shall show that the probability of failure is $n^{-\omega(1)}$.

We start with step~$j=0$, considering the initial set~$W$ arising in the definition of~$\cT^*$,
chosen exactly as in $\cT(\fS_t^-)$. 
Recall from~\eqref{eq:def:SW:Y}--\eqref{eq:def:SW:ZR} that $W$ is the union of the components of $H_L$ containing
a certain number (either $1$ or $r$ in the two cases) of vertices $w_h$ chosen uniformly at random from 
$V_L$. Let
\begin{equation}\label{eq:dom:C}
 C^*_{w_h} := C_{w_h}(H_L) \setminus \bigl(\bigcup_{1 \le s < h} C_{w_s}(H_L) \bigr) .
\end{equation}
Since components of $H_L$ are disjoint, this definition simply says that $C^*_{w_h}=C_{w_h}(H_L)$ if this is a `new' component, 
and $C^*_{w_h}=\emptyset$ if it is a `repeated' component.
In particular, $|\bigcup_{1 \le h \le r} C_{w_h}(H_L)|=\sum_{1 \le h \le r} |C^*_{w_h}|$. 
As long as (P2) does not hold, for all $k\ge 1$ we have  
\begin{equation}\label{eq:dom:C:Pr}
 \Pr\bb{|C^*_{w_h}| =k \mid w_1, \ldots, w_{h-1}} \ge \frac{\max\bigl\{N_k(\fS)-A_0 e^{-a_0 k} \Lambda,0\bigr\}}{|V_L|} = \Pr(N^-=k) ,
\end{equation}
so that each random variable $|C^*_{w_h}|$ stochastically dominates~$N^-$. 
Consequently, there is a coupling such that either~(P2) occurs, or $M_0 =|W|\ge Y^-_{0,t} + \sum_{1 \le h \le R^-_t} N^-_h=Y^{0,-}_t$ and $S_0 \ge Z^{0,-}_t$ both hold, establishing the base case. 

Next we turn to step $j \ge 1$. We may assume that the exploration has not finished (i.e., $j \le M_{j-1}$), and that
(P2) does not hold; otherwise the coupling has already either succeeded or failed. 
Since (P2) does not hold, taking $k=1$, we have reached at most $A_0\Lambda$ vertices in $V_L$,
and hence $j\le A_0\Lambda$.
As in Section~\ref{sec:nep:ed}, we analyze the modified exploration process using auxiliary variables that may be constructed via a two-stage process.
Turning to the details, let~$\fH^-_{k,r,j}$
denote the multi-set of $(r-1)$-tuples $(w_1,\ldots,w_{r-1})$ of vertices in~$V_{L,j}$
corresponding to $k$-weighted hyperedges $(v_j,w_1, \ldots, w_{r-1})$, \ldots, $(w_1, \ldots, w_{r-1},v_j)$ found in step~$j$
of our exploration. By standard properties of Poisson random variables, we have
$|\fH^-_{k,r,j}| \sim \Po(\tlambda^-_{k,r,j})$ with
\[
 \tlambda^-_{k,r,j} := \frac{|V_{L,j}|^{r-1}rQ_{k,r}^-(t)}{|V_L|^r} = \frac{(|V_L|-j)^{r-1}rQ_{k,r}^-(t)}{|V_L|^r} \ge \lambda^-_{k,r}(t),
\]
where we use the bound $j/|V_L| \le A_0 \Lambda/|V_L| \le A_0/\zeta \cdot n^{-1/3}= A_1 n^{-1/3}$ (see~\eqref{eq:nice:VL} and~\eqref{def:A0}) to establish the final inequality.
Moreover, conditional on $|\fH^-_{k,r,j}|=y_{k,r}$, we have
\[
\fH^-_{k,r,j}=\bigl\{(w_{k,r,1,1}, \ldots, w_{k,r,1,r-1}), \ldots, (w_{k,r,y_{k,r},1}, \ldots, w_{k,r,y_{k,r},r-1})\bigr\},
\]
where each $w_{k,r,y,h} \in V_{L,j}$ is chosen independently and uniformly at random. 
Recalling the definition of the modified exploration process, it is not difficult to see that 
\begin{align}
\label{eq:dom:M}
 M_{j}-M_{j-1} &= \Bigl|\Bigl(\bigcup_{k \ge 0, r \ge 2}\ \bigcup_{1 \le y \le |\fH^-_{k,r,j}|} \ \bigcup_{1 \le h \le r-1} C_{w_{k,r,y,h}}(H_L) \Bigr) \setminus \Bigl(\cA_{j-1} \cup \cE_{j-1}\Bigr)\Bigr| , \\ 
\label{eq:dom:S}
S_{j}-S_{j-1} &= \sum_{k,r \ge 1} k |\fH^-_{k,r,j}| .
\end{align}
To bound $M_j-M_{j-1}$ from below, we write the right-hand side of~\eqref{eq:dom:M} as a disjoint union of sets $C^*_{w_{\tau}}$, similar to~\eqref{eq:dom:C}. Indeed, $\cA_{j-1}\cup\cE_{j-1}$ is a union of components of $H_L$, so each $C_{w_{\tau}}(H_L)$ is either
a `new' component (disjoint from $\cA_{j-1}\cup\cE_{j-1}$ and from those appearing before), or a `repeat'. In the latter
case we set $C^*_{w_\tau}=\emptyset$.
Since $|V_{L,j}| \le |V_L|$, arguing as for \eqref{eq:dom:C:Pr} we see that, as long as (P2) does not occur, each $|C^*_{w_{\tau}}|$ stochastically dominates~$N^-$. 
There is thus a coupling such that either (P1) or (P2) occurs (causing us to stop partway through step $j$),
or $M_{j}-M_{j-1} \ge Y^-_t$ and $S_{j}-S_{j-1} \ge Z^-_t$ both hold, completing the induction~step. 

If the coupling above succeeds (i.e., stops due to finishing, or due to (P1)), then it has the required properties. Thus it 
only remains to show that it is unlikely to fail, i.e., stop due to~(P2) occurring. Let $\cB_1$ be the `bad' event that
we stop due to (P2) after more than $2\Lambda$ vertex sampling steps, and let $\cB_2$ be the event that we stop
due to (P2) after at most $2\Lambda$ vertex sampling steps; to complete the proof we must show that $\Pr(\cB_i)=n^{-\omega(1)}$
for $i=1,2$; as we shall now see, this is straightforward.

From the definition \eqref{eq:N:LB} of $N^-$ and the fact that $|V_L(\fS)|=\sum_{k>K} N_k(\fS)$, we have
\[
 \Pr(N^-=0) \le \sum_{k>K} A_0 e^{-a_0 k} n^{2/3}/|V_L| = O(n^{2/3})/|V_L| = o(1),
\]
recalling \eqref{eq:nice:VL}. In particular, if $n$ is large enough then $\Pr(N^-\ge 1)\ge 3/4$.
In each vertex sampling step we `reach' $N^-$ new vertices in $\bp_t^-$; if $\cB_1$ holds then we do not stop (for any reason)
within $2\Lambda$ vertex sampling steps, and in particular (considering (P1)) we carry out $2\Lambda$ such steps
reaching at most $\Lambda$ particles in $\bp_t^-$. Hence
\[
 \Pr(\cB_1) \le \Pr\bb{ \Bin(2\Lambda,3/4) \le \Lambda }
\]
which, by a standard Chernoff bound, is $e^{-\Omega(\Lambda)} = n^{-\omega(1)}$.

Turning to $\cB_2$, let
\[
 \pi_k := \frac{2A e^{-2 a_0k}}{\zeta}.
\]
By the choice \eqref{def:A0} of $A_0$, for any $k'$ we have
\[
 \sum_{k\ge k'} 8 k \pi_{k} \le\sum_{k\ge k'} \frac{16k A e^{-a_0k}}{\zeta} \cdot e^{-a_0 k'}  \le A_0 e^{-a_0 k'}.
\]
Hence, if we stop due to (P2), there is some~$k$ such that we have reached more than~$8k\pi_k\Lambda$ vertices in components
of size exactly~$k$, and in particular, have reached a new component of size exactly~$k$ more than~$8\pi_k\Lambda$ times.

As noted above, if (P2) does not hold, then $j\le A_0\Lambda$.
Since $a \ge 2 a_0$ and $N_k(\fS)\le N_{\ge k}(\fS)\le A e^{-ak}n$ (from \eqref{eq:Nk:t0:tail'}),
as long as (P2) does not hold, then (for $n$ large enough) in each vertex sampling step we have
\begin{equation*}
 \Pr(|C_{\tilde{w}}(H_L)| = k \mid  \cdots) \le \frac{N_{k}(\fS)}{|V_{L,j}|} \le \frac{Ae^{-ak}n}{|V_L| -A_0\Lambda } \le \frac{2A e^{-2 a_0k}}{\zeta} = \pi_k ,
\end{equation*}
recalling \eqref{def:Nk:xi}. By~\eqref{eq:cpl:B0}, for $n$ large enough (which we always assume), all components of $H_L$ have size at most $D_0\log n$.
From the discussion above,
if $\cB_2$ holds then, considering when (P2) first holds, there is some $k\le D_0\log n$ such that within the first at most
$2\Lambda$ vertex sampling steps there are at least $8\pi_k\Lambda$ steps in which we choose a component of size~$k$, 
with (P2) not holding at the start of any of these steps. 
The probability of this (for a given $k$) is at most
\[
 \Pr\bb{ \Bin(2\Lambda,\pi_k) \ge 8\Lambda\pi_k } = e^{-\Omega(\Lambda\pi_k)},
\]
using a standard Chernoff bound.
From~\eqref{eq:cpl:a0} we have $\Lambda \pi_k = \Omega(n^{1/3})$, so this probability is $n^{-\omega(1)}$,
and summing over $k\le D_0\log n$ we conclude that $\Pr(\cB_2)=n^{-\omega(1)}$, completing the proof.
\end{proof}
We are now ready to bound the expected number of vertices of $\JPpm$ in components of at least a certain size.
The lower bound in~\eqref{eq:ENgekD} below is only non-trivial for $k \le n^{2/3}$, but this suffices for our purposes. 
\begin{theorem}[Sandwiching the expectation of~$N_{\ge k}$]
\label{thm:ENgekD}
Let $t\in [t_0,t_1]$, and let $\fS$ be a $t$-nice parameter list. 
Define~$\fS_t^{\pm}$ as in Definition~\ref{def:cpl}, $\JPpm=\JP(\fS^{\pm}_t)$ as in Definition~\ref{def:F}, and $\bp_t^{\pm}=\bp_t^{\pm}(\fS)$ as in Definition~\ref{def:bp2pm}.
Then for all $k \ge 1$ we have 
\begin{equation}\label{eq:ENgekD}
\indic{k \le n^{2/3}}\Bigl(\Pr(|\bp^-_t| \ge k) |\fS^-_t|- n^{-\omega(1)}\Bigr) \le \E N_{\ge k}(\JPm) \le \E N_{\ge k}(\JPp) \le \Pr(|\bp^+_t| \ge k) |\fS^+_t|.
\end{equation}
\end{theorem} 
\begin{proof}
Since $\fS^-_t \preceq\fS^+_t$, from~\eqref{eq:F:mon} we have $\E N_{\ge k}(\JPm) \le \E N_{\ge k}(\JPp)$. 
Furthermore,
we have $\E N_{\ge k}(\JP(\fS^\pm_t)) = \Pr(|\cT(\fS^\pm_t)| \ge k) |\fS^\pm_t|$ by~\eqref{eq:def:NgejR:E}. 
Noting that $|\fS^\pm_t| = O(n)$, the result follows from
Theorems~\ref{thm:cpl:UB}--\ref{thm:cpl:LB}.
\end{proof}

\subsubsection{Second moment estimate}\label{sec:SME}
In this subsection we use domination arguments (as in the previous section) to prove an upper bound on the second moment of 
the number $N_{\ge\Lambda}(\JPp)$ of vertices in `large' components, where $\JPp=\JP(\fS_t^+)$ with $\fS$ a $t$-nice parameter list, and `large' means containing at least $\Lambda$ vertices for some cut-off value $\Lambda$ (often, but not always, $n^{2/3}$). This will later be key for analyzing the size of the giant 
component in the supercritical case (see Sections~\ref{sec:mom:large} and~\ref{sec:L1:super}).

Before turning to the formal details, we shall outline the argument considering the simpler quantity $X_L$, the number of vertices in $V_L=V_L(\fS)$ that,
in $\JPp$, are in components of size at least $\Lambda$. 
Note that we may write the second moment of $X_L$ as 
\[
  \E X_L^2 = \sum_{v_1\in V_L}\sum_{v_2\in V_L} \Pr\bb{|C_{v_1}|\ge \Lambda,\, |C_{v_2}|\ge \Lambda}.
\]
In turn, we can express this as $|V_L|^2$ times
\[
 p := \Pr\bb{|C_{v_1}|\ge \Lambda,\, |C_{v_2}|\ge \Lambda},
\]
where the `starting vertices' $v_1$ and $v_2$ are chosen independently and uniformly from $V_L$.
To estimate $p$, the basic plan is to explore $\JPp$ outwards from the vertices $v_1$ and $v_2$ as usual,
comparing each exploration to a (dominating) branching process $\bp_{t,i}^+$, $i=1,2$. Here, as in previous sections, we view 
$\JPp$ as a (weighted) hypergraph on $H_L$, consisting of the original edges inside $H_L$ plus some $(k,r)$-hyperedges,
each of which connects $r$ random vertices of $V_L$ and contributes $k$ extra vertices (in $V_S$) of its own.

The main difficulty we encounter is dependence between the two explorations:
we seek a moment bound applicable in the supercritical case, when the main contribution is from $v_1$ and $v_2$
lying in the same component, the giant component. So, unchecked, the explorations are very likely to interfere
with each other. To deal with this, we first explore the graph from $v_1$, but stopping the exploration early if we reach $\Lambda$
vertices (so far, this is how we estimated $\E X_L$, or rather $\E N_{\ge \Lambda}$). Let $U_1$ be the set of vertices reached by
the first exploration, and $A_1\subseteq U_1$ the `boundary', i.e., those vertices reached but not yet fully explored (tested
for new neighbours).
If the exploration `succeeds' (reaches $\Lambda$ vertices), then we
start exploring from $v_2$. It may be that $v_2\in U_1$, in which case $|C_{v_2}|=|C_{v_1}|\ge \Lambda$. Otherwise, to avoid dependence,
we explore from $v_2$ but only within $V_L\setminus U_1$. The tricky
case is when the second exploration stops, revealing the component $C'$ of $v_2$ in the subgraph on $V_L\setminus U_1$,
and it turns out that $|C'|<\Lambda$.
In this case it might still be true that $|C_{v_2}|\ge \Lambda$; this happens if and only if there is a hyperedge
joining $A_1$ to $C'$. (We have not yet tested these hyperedges.) We can bound the probability of this by an estimate
proportional to $|A_1||C'|/n$. It turns out that this can be too large, so we use an idea of Bollob\'as and Riordan~\cite{BR2012}: we introduce a second early stopping condition for the first exploration, if the boundary at any point becomes
too large (see also Figure~\ref{fig:Var1}). As usual, we couple our explorations step-by-step with a dominating branching process, $\bp^+_t$,
using bounds on the probability that~$\bp^+_t$ is large to bound the probability we are looking for.
We end up with an overestimate, but since we only claim an upper
bound (and the estimate turns out to be tight enough), this is no problem. 

Let us now turn to the details. 
We present here the combinatorial part of the argument, leading to the rather complicated bound~\eqref{eq:Nkvar:sup1} below.
In the supercritical case ($t=\tc+\eps$ with $\eps^3n \to \infty$) we shall later use it together with results about the
branching processes $\bp_t^{+}$ and $\bp_t^{1,+}$ to derive that $X:=N_{\ge \Lambda}(\JPp)$ satisfies $\Var X = o( (\E X)^2 )$ for $\Lambda = \omega(\eps^{-2})$, see Lemma~\ref{lem:NkE:sup}. 
In particular, the quantities $\tau$, $\nu$ and $\rho_1$ appearing below (which depend on $\fS$) 
will then satisfy $\tau\sim \Pr\bb{|\bp_t^+|\ge \Lambda} \sim \Pr\bb{|\bp^+_t|=\infty} =\Theta(\eps)$, $\nu=O(\eps^{-1})$,  
and~$\rho_1 = \Theta(\eps)$. 
The assumption~$\Lambda \ge \eps^{-2}$ is not optimal here, but will be natural in the later probabilistic~arguments. 
%
\begin{lemma}[Second moment of $N_{\ge \Lambda}$]
\label{lem:Nkvar:sup}
Let $t \in [t_0,t_1]$ with $\eps=t-\tc>0$, and let $\fS$ be any $t$-nice parameter list.
Define~$\fS_t^{+}$ as in Definition~\ref{def:cpl}, $\JPp=\JP(\fS^{+}_t)$ as in Definition~\ref{def:F}, and $\bp_t^{+}=\bp_t^{+}(\fS)$ and $\bp^{1,+}_t=\bp^{1,+}(\fS)$
as in Definition~\ref{def:bp2pm}. 
Then, setting $X:=N_{\ge \Lambda}(\JPp)$, for all $\eps^{-2} \le \Lambda \le n^{2/3}$ we have
\begin{equation}\label{eq:Nkvar:sup1}
 \E X^2  \le 
(\log n)^2  \E X +
 |\fS_t^+|^2\Bigl(
  \tau \Pr\bb{|\bp_t^+|\ge \Lambda} + O\bb{ \tau \eps\nu n^{-1/3} + \nu n^{-1} + n^{-1/3} \Pr(|\bp_t^+|=\infty)} \Bigr),
\end{equation}
where
\[
 \nu := \E(|\bp^+_{t}|  \indic{ |\bp^+_{t}|<\infty }) 
\]
and $\tau$ is a certain quantity related to the branching process which, whenever $\rho_1 =\Pr(|\bp^{1,+}_t| = \infty)$ is positive, satisfies
\begin{equation}\label{eq:Nkvar:sup:sigma}
 \tau \le  \bb{1-e^{-\rho_1\ceil{\eps\Lambda}}}^{-1} \cdot \Pr\bb{|\bp_t^+| \ge \Lambda} .
\end{equation}
\end{lemma}
Note that the branching process interpretation of $\tau$, in \eqref{eq:SMM:key:2} below, is not essential to the statement of the lemma, whose content is the combination of the bounds \eqref{eq:Nkvar:sup1} and \eqref{eq:Nkvar:sup:sigma}.
\begin{proof}
We follow the strategy outlined above the statement of the theorem, with the main additional complication being that we must
consider also vertices in $V_S$, which is a random set; as in previous sections this leads to exploration and branching 
processes with a more complicated start.

We write $V_L=V_L(\fS^+_t)$, $N_k = N_k(\fS_t^+)=N_k(\fS)$ and $Q_{k,r}=Q_{k,r}^+(t)$ to avoid clutter in the notation.
Let $X=N_{\ge \Lambda}(\JPp)$; our goal is to approximate the second moment~$\E X^2$. 
Analogous to~\eqref{eq:def:Nj} we write
\[
 X  =  \sum_{v \in V_L}\indic{|C_v(\JPp)|\ge\Lambda} 
 + \sum_{k \ge 1,\, r \ge 0}\, \sum_{g \in \cQ_{k,r}} k \indic{|C_{g}(\JPp)|\ge\Lambda},
\]
where $\cQ_{k,r}$ is the (random) set of $(k,r)$-hyperedges in $\JPp$.
We shall expand $X^2$, 
putting all summands with terms from the same $g\in \cQ_{k,r}$ into the sum $Y$ below.
More precisely, we write 
\[
 X^2 = Y + Z \quad\text{where}\quad Z=Z_1 + Z_2 + Z_3 +Z_4,
\]
with
\begin{align*}
 Y &:= \sum_{k,r} \sum_{g \in \cQ_{k,r}} k^2 \indic{|C_{g}(\JPp)|\ge\Lambda}, \\
 Z_1 &:=  \sum_{v_1 \in V_L}\ \sum_{v_2 \in V_L} \indic{|C_{v_1}(\JPp)|\ge\Lambda} \indic{|C_{v_2}(\JPp)|\ge\Lambda}, \\
 Z_2 &:=  \sum_{v_1 \in V_L}\  \sum_{k_2,r_2} \sum_{g_2 \in \cQ_{k_2,r_2}} k_2 \indic{|C_{v_1}(\JPp)|\ge\Lambda}  \indic{|C_{g_2}(\JPp)|\ge\Lambda}, \\
 Z_3 &:=  \sum_{k_1,r_1} \sum_{g_1 \in \cQ_{k_1,r_1}}\  \sum_{v_2 \in V_L} k_1 \indic{|C_{g_1}(\JPp)|\ge\Lambda} \indic{|C_{v_2}(\JPp)|\ge\Lambda} , \\
 Z_4 &:=  \sum_{k_1,r_1} \sum_{g_1 \in \cQ_{k_1,r_1}}\ \sum_{k_2,r_2}  \sum_{g_2 \in \cQ_{k,r}} k_1k_2 \indic{|C_{g_1}(\JPp)|\ge\Lambda}  \indic{|C_{g_2}(\JPp)|\ge\Lambda}\indic{g_1\ne g_2}.
\end{align*}
Note that the decomposition above is very different from that used in \eqref{eq:lem:Nkvar:X2} and \eqref{eq:lem:Srvar:sub:X2}, where $Y$ contained
all terms with the two $g_i$ (or $v_i$) in the same component; here we only put the $g_1$, $g_2$ term into $Y$ if $g_1=g_2$.
The reason for using a different decomposition here is that the previous decomposition is only useful when we have
an upper bound on the size of the relevant components. The term $Y$ will be easy to handle (see below); we first discuss
how to evaluate the expectations of the $Z_i$.

For $Z_1$ we of course have
\[
 \E Z_1 = |V_L|^2\ \Pr\bb{ |C_{v_1}(\JPp)| \ge \Lambda,\,|C_{v_2}(\JPp)| \ge \Lambda },
\]
where $v_1$ and $v_2$ are independently chosen uniformly at random from $V_L$. For the remaining terms we have
the complication that the sets $\cQ_{k,r}$ being summed over are random sets, with Poisson sizes. For $Z_2$, arguing
exactly as for \eqref{eq:Qkr:palm}, by standard results on point processes we have
\[
 \E Z_2 = |V_L| \sum_{k_2,r_2} k_2 Q_{k_2,r_2}  \Pr\bb{ |C_{v_1}(\JPp_2)|\ge\Lambda,\, |C_{g_2}(\JPp_2)|\ge\Lambda}, \\
\]
where as before $v_1$ is a random vertex from $V_L$, but now $\JPp_2$ is formed from $\JPp$ by adding an `extra' $(k_2,r_2)$-hyperedge
$g_2$, joined (as always) to $r_2$ vertices of $V_L$ chosen independently at random. There is of course a similar formula for $\E Z_3$.
Finally, arguing as for \eqref{eq:Qkr2:palm}, we have
\[
 \E Z_4 =  \sum_{k_1,r_1}  \sum_{k_2,r_2} k_1 Q_{k_1,r_1} k_2 Q_{k_2,r_2} 
   \Pr\bb{ |C_{g_1}(\JPp_{12})|\ge\Lambda,\, |C_{g_2}(\JPp_{12})|\ge\Lambda}, \\
\]
where $\JPp_{12}$ is formed from $\JPp$ by adding an extra $(k_1,r_1)$-hyperedge $g_1$ and a \emph{distinct} extra $(k_2,r_2)$-hyperedge~$g_2$, 
both joined to $\JPp$ in the usual way.
(Here we used the condition $g_1\ne g_2$ in the definition of $Z_4$.)

At this point it seems that we might have several cases to consider; fortunately, we can group them back together, using the argument
for \eqref{eq:def:NjT:E} in Section~\ref{sec:nep:fg}, but considering the (random) starting points $x_1$, $x_2$ for two explorations;
each $x_i$ will play the role of either $v_i$ or $g_i$.

To define $x_1$, for any $v\in V_L$, with probability $1/|\fS_t^+|$ we set $x_1=v_1=v$,
while for each $k_1\ge 1$ and $r_1\ge 0$, with probability $k_1Q_{k_1,r_1}/|\fS_t^+|$ we set $x_1=g_1$ with $g_1$ of
type $(k_1,r_1)$. Since $\fS_t^+$ is a parameter list, this defines a probability distribution.
Define $x_2$ in the same way, but independently from $x_1$. Then it is straightforward to see from the formulae
for $\E Z_i$ above that
\[
 \E Z = |\fS_t^+|^2\ \Pr\bb{ |C_{x_1}(\JPp_*)|\ge\Lambda,\, |C_{x_2}(\JPp_*)|\ge\Lambda},
\]
where $\JPp_*$ is obtained from $\JPp$ as follows: if $x_1=g_1$ is of type $(k_1,r_1)$, we add $g_1$ as an extra $(k_1,r_1)$ hyperedge,
and similarly if $x_2=g_2$. If both hold, the added extra hyperedges are distinct (and independent).

As before, we deal with starting our exploration from an extra hyperedge by passing to the set $W$ of vertices reachable
in one step. More precisely, for each $i$, if $x_i=v_i$ we set
\[
  S_{0,i}:=0 \quad\text{and}\quad W_i := C_{v_i}(H_L),
\]
while if $x_i$ is a $(k_i,r_i)$-hyperedge $g_i$, we set 
\[
  S_{0,i}:=k_i \quad\text{and}\quad W_i := \bigcup_{1\le h\le r_i} C_{w_{h,i}}(H_L),
\]
with all $w_{h,i}$ chosen independently and uniformly from $V_L$. For each $i$, this is exactly the starting rule
used in~\eqref{eq:def:SW:Y}--\eqref{eq:def:SW:ZR}, with the two starts independent.

As in Section~\ref{sec:nep:fg}, we shall explore \emph{within $\JPp=\JP(\fS^{+}_t)$}, from each starting set~$W_i$. 
Let us write~$\cT_i$ for this exploration, defined exactly as $\cT=\cT(\fS_t^+)$ is defined in Section~\ref{sec:nep:fg},
with $(W_i,S_{0,i})$ as the initial values. Of course, the two explorations are far from independent,
since they explore the \emph{same} random hypergraph~$\JPp$. There is also another problem: in the exploration
$\cT_1$, say, the initial values account for the possibility of an added $(k_1,r_1)$-hyperedge $g_1$ -- if we
do add such a hyperedge in forming $\JPp_*$, then in $\cT_1$ we already count its~$k_1$ extra vertices from
the start, and account for the fact that~$g_1$ connects the vertices in~$W_1$ by marking these as `reached' 
right from the start. Unfortunately, $\cT_1$ does not account for any connections formed by possibly adding $g_2$.
Still, if the exploration $\cT_1$ does not reach any vertex in~$W_2$, then we do have $|\cT_1|=|C_{x_1}(\JPp_*)|$.
To formalize this, let
\[
 \cP := \bigl\{ W_1\text{ and }W_2\text{ are connected in }\JPp \}
\]
be the event that there is a path from some vertex in $W_1$ to some vertex in $W_2$ within the hypergraph $\JPp$.
This is exactly the event that the complete exploration $\cT_1$ meets $W_2$, or vice versa.
When $\cP$ does not hold, then $|\cT_i| = |C_{x_i}(\JPp_*)|$. Hence, setting
\[
 \cB : = \bigl\{ |\cT_1|\ge\Lambda \text{ and } |\cT_2|\ge \Lambda \bigr\}
\]
we have
\[
 \bigl\{ |C_{x_1}(\JPp_*)|\ge\Lambda \text{ and } |C_{x_2}(\JPp_*)|\ge\Lambda \bigl\} \: \subseteq \: \cB \cup \cP,
\]
and so
\begin{equation}\label{eq:EX2:BP}
 \E X^2 = \E Y  + \E Z \le \E Y + |\fS_t^+|^2 \cdot \Pr(\cB\cup\cP).
\end{equation}

The easiest term to bound on the right-hand side above is $\E Y$. As in previous sections, set
\[
 \Psi := (\log n)^2,
\]
which is simply a convenient `cut-off' for the various distributions involved.
Indeed, recalling our conventions $N_k=N_k(\fS)$ and $Q_{k,r}=Q^+_{k,r}(t)$, by~\eqref{eq:small:Psi}, \eqref{eq:Qkr:pm} and~\eqref{eq:qkr:pm:tail}
we see that  
\begin{equation}\label{eq:small:Psi'} 
 \max_{k \ge \Psi}N_{k}=0 \quad \text{and} \quad \max_{k+r \ge \Psi}Q_{k,r}=0.
\end{equation}
In other words, all components of our initial marked graph $H$ have size at most $\Psi$, where in a $(k,r)$-hyperedge we count
the $k$ vertices and the $r$ stubs in determining its size. It follows immediately that $Y \le \Psi \cdot X$, so 
\begin{equation}\label{eq:SMM:Ya}
 \E Y \le \Psi \cdot \E X . 
\end{equation}

\begin{figure}[t]
\centering
  \setlength{\unitlength}{1bp}%
  \begin{picture}(285.71, 106.98)(0,0)
  \put(0,0){\includegraphics{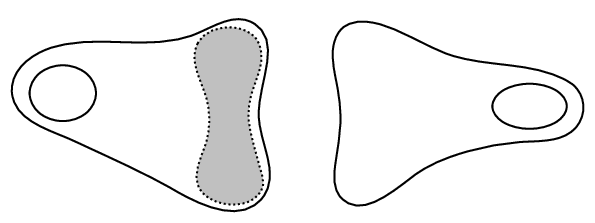}}
  \put(29.83,59.19){\fontsize{11.38}{13.66}\selectfont \makebox[0pt]{$W_1$}}
  \put(109.59,76.92){\fontsize{11.38}{13.66}\selectfont \makebox[0pt]{$A_1$}}
  \put(40.91,92.43){\fontsize{11.38}{13.66}\selectfont \makebox[0pt]{$U_1$}}
  \put(254.04,52.77){\fontsize{11.38}{13.66}\selectfont \makebox[0pt]{$W_2$}}
  \put(244.73,82.47){\fontsize{11.38}{13.66}\selectfont \makebox[0pt]{$U_2$}}
  \end{picture}%
	\caption{\label{fig:Var1}%
A rough sketch of the different cases which underlie the second moment analysis of Lemma~\ref{lem:Nkvar:sup}, ignoring $V_S$--vertices for simplicity (and a number of other technical details). 
By construction, $|C_{W_1}(\JPp)| \ge \Lambda$ implies the `stopping event' $\cS$. 
When~$\cS$ holds, then the exploration process for the first component $C_{W_1}(\JPp)$ has found the $V_L$--vertices in $U_1 \subseteq V_L$, but the vertices in $A_1 \subseteq U_1$ have not yet been (fully) explored (i.e., have not yet been (fully) tested for neighbours). 
In particular, $\cS$ implies that either $|U_1| \approx \Lambda$ or $|A_1| \approx 2\ceil{\eps \Lambda}$. 
For the initial generation of the second component $C_{W_2}(\JPp)$, there are then two possibilities: $W_2$ is either (i) disjoint from $U_1$, or (ii) contains at least one vertex from $U_1$. 
In the more interesting case (i) we then explore $C_{W_2}(\JPp)$ \emph{without} testing any hyperedges containing vertices from~$U_1$. 
If $|C_{W_2}(\JPp)| \ge \Lambda$, then there are two cases for the `restricted' subcomponent $C^-_{W_2}(\JPp) \subseteq C_{W_2}(\JPp)$.
Either $C^-_{W_2}(\JPp)$ already contains at least $\Lambda$ vertices, or there must be a connection from $C^-_{W_2}(\JPp)$ to~$U_1$. 
By construction all vertices in $U_1 \setminus A_1$ have already been explored (i.e., tested for neighbours), so in the second case at least one so-far untested hyperedge must connect~$A_1$ to~$U_2$.}
\end{figure}
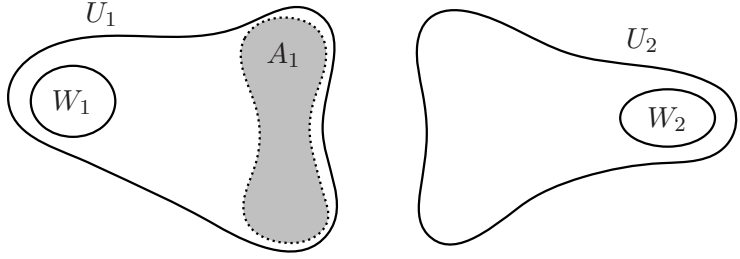

It remains to bound $\Pr(\cB\cup \cP)$, which we do by considering the explorations
$\cT_1$ and $\cT_2$ defined within $\JPp$ (not $\JPp_*$);
our aim is to compare the explorations with two independent copies $\bp^+_{t,i}$, $i=1,2$, of 
the branching process~$\bp^+_t=\bp^+_t(\fS)$ given by Definition~\ref{def:bp2pm}. 
In order to retain sufficient independence in the analysis, we will need to consider restricted
explorations $\cT_i^-\subseteq \cT_i$, defined in different ways for~$i=1$ and~$i=2$.

Since each exploration $\cT_i$ has the same distribution as the exploration $\cT(\fS_t^+)$ defined
in previous sections, by Theorem~\ref{thm:cpl:UB} we may couple $\cT_1$ with $\bp_{t,1}^+$ so that
the latter dominates the former. As outlined above, we may wish to abandon the exploration/coupling
part way through; in fact, it will be convenient to construct $\cT_1^-\subseteq \cT_1$, and couple it with $\bp_{t,1}^+$, in an `edge-by-edge' way,
to allow for stopping part way through step~$j$.
Thus, in step $j \ge 1$ (where we process $v_j \in \cA_{j-1}$) we \emph{sequentially} consider, for all $k \ge 0$ and $r \ge 1$ with $k+r \le \Psi$, each so-far untested $k$-weighted hyperedge $g\in (V_L)^r$ of the form $(v_j,w_1, \ldots, w_{r-1})$, \ldots, $(w_1, \ldots, w_{r-1},v_j)$. 
For each such hyperedge~$g$ we test the presence and multiplicity~$m_g$; if $m_g \ge 1$ we (i) mark the vertices $\bigcup_{1 \le h \le r-1}C_{w_h}(H_L) \setminus (\cA_{j-1} \cup \cE_{j-1})$ as active, and (ii) increase the number of found $V_S$--vertices by $k m_g$. Finally, at the end of step~$j$ we move~$v_j$ from the set of active vertices to the set of explored vertices.
We stop the exploration/coupling if either (i) at the beginning (when the reached and active sets are both equal to $W_1$),
or (ii) \emph{after} completely processing any particular hyperedge~$g$, one of the following two conditions holds: 
\begin{itemize}\itemsep1pt \parskip0pt \parsep0pt
	\item[(P1)] the exploration has reached at least $\Lambda$ vertices in $V_L$, or
	\item[(P2)] there are currently at least $2\ceil{\eps \Lambda}$ active vertices (vertices in $V_L$ that have been reached but not yet fully explored).
\end{itemize}
For later reference we define the event
\[
\cS := \{\text{we stop $\cT_1$ due to (P1) or (P2)}\} .
\]
If $\cS$ does not hold then we complete the coupling,
exploring the entirety of $\cT_1$. Thus
\[
 \neg\cS \quad\text{implies}\quad \cT_1^-=\cT_1.
\]
If $|\cT_1|\ge \Lambda$ then $\cS$ holds: at the latest we stop when we have reached
$\Lambda$ vertices in the exploration of $\cT_1$. In other words,
\begin{equation}\label{eq:SMM:key}
 |\cT_1| \ge \Lambda \quad \text{implies} \quad \cS.
\end{equation}

Define $|\cT_1^-|$ to be the total number of vertices reached by the possibly truncated exploration $\cT_1^-$,
including the $|W_1|+S_{0,1}$ initial vertices. 
From the domination in the coupling we certainly have $|\bp_{t,1}|\ge |\cT_1^-|$.
Let $w(\bp^+_{t,1})$ denote the maximum (supremum) of the number of $L$--particles in any generation of the branching process~$\bp^+_{t,1}$.
Since we constructed our coupling in breadth-first search order, if~(P2) holds then the at least $2\ceil{\eps \Lambda}$ active vertices
at this point correspond to particles in $\bp^+_{t,1}$ that are contained in two consecutive generations of~$\bp^+_{t,1}$. 
Allowing for the fact that we might stop due to (P2) partway through the exploration of some vertex~$v_j$, it follows that 
\begin{equation}\label{eq:SMM:width}
\text{(P2)} \quad \text{implies} \quad w(\bp^+_{t,1}) \ge \ceilL{{\frac{2\ceil{\eps \Lambda}-1}{2}}} \ge \ceil{\eps \Lambda} .
\end{equation}
By the domination in the coupling, (P1) implies $|\bp^+_{t,1}|\ge \Lambda$. Thus if $\cS$ holds,
either $|\bp^+_{t,1}|\ge \Lambda$ or $w(\bp^+_{t,1}) \ge \ceil{\eps \Lambda}$.
Since $|\bp^+_{t,1}| < \Lambda$ implies $|\bp^+_{t,1}| < \infty$, 
we thus obtain 
\begin{equation}\label{eq:SMM:key:2}
 \Pr(\cS) \: \le \: \Pr(|\bp^+_{t}| \ge \Lambda) + \Pr\bb{|\bp^+_{t}| < \infty, \: w(\bp^+_{t}) \ge \ceil{\eps \Lambda}} =: \tau,
\end{equation}
where we think of $\Pr(|\bp^+_{t}| \ge \Lambda)$ as the `main term'. Here we have dropped the subscripted $1$'s, since
$\bp_{t,1}^+$ has the same distribution as $\bp_t^+$.

Let $U_1\subseteq V_L$ denote the set of $V_L$-vertices reached by the first exploration $\cT_1^-$,
the possibly truncated version of $\cT_1$. Also, let $A_1\subseteq U_1$ denote the set of active (not yet fully explored)
vertices at the end of $\cT_1^-$, so $A_1\ne\emptyset$ only if we stopped the exploration early.
Since we only stop \emph{after} completely processing any hyperedge, we can `overshoot' our stopping criteria somewhat,
but, from~\eqref{eq:small:Psi'}, only by at most $(\Psi-1) \cdot \Psi \le \Psi^2$, say.
It follows that
\begin{align}
\label{eq:W1L}
 |U_1| & \le \Lambda + \Psi^2  \le 2 n^{2/3},\\
\label{eq:W1A}
 |A_1| & \le 2\ceil{\eps \Lambda} + \Psi^2 \le 5 \eps n^{2/3},
\end{align} 
where we tacitly used $\max\{\Lambda,1,\Psi^2\} \le n^{2/3}$ and $\eps^{-1} \cdot \Psi^2 \le \sqrt{\Lambda} \cdot \Psi^2 \le n^{2/3}$.

After the first exploration, by definition all potential hyperedges meeting $U_1\setminus A_1$ have already been tested
for their presence in $\JPp$. Furthermore, no such hyperedges containing any vertices outside $U_1$ are present in $\JPp$. (Such vertices
would have been reached by the exploration.) The remaining untested potential hyperedges are of two types: those
entirely outside $U_1$, and those meeting $A_1$. Let~$\fE$ denote the set of potential hyperedges of the latter type.

Turning to the second exploration $\cT_2$, we define $\cT_2^-\subseteq \cT_2$ as follows:
we explore as usual starting from the initial set $W_2$ and initial value $S_{0,2}$,
but only testing hyperedges outside $U_1$. More precisely, if~$W_2$ meets~$U_1$ we shall not explore at all (defining $\cT_2^-$ to
be empty and have size 0, say); otherwise,
we run our second exploration in the subgraph $\JPp(U_1^\cc)$ obtained by deleting all vertices in $U_1$ and all incident
hyperedges, as in the proof of Lemma~\ref{lem:Nkvar:sub}. We write $|\cT_2^-|$
for the size of $\cT_2^-$, i.e., the number of vertices reached, counting vertices in $V_S$, and including the initial $|W_2|+S_{0,2}$
vertices. 
Since $\JPp(U_1^\cc)$ is an induced subgraph of $\JPp$, there is a coupling such that $\cT_2^-$ is dominated by $\bp_{t,2}^+$; we construct this coupling exactly
as for $\cT_1$, except that some tests (of edges meeting $A_1$) are simply omitted. This coupling gives
\begin{equation}\label{eq:C2:cpl}
 |\cT_2^-| \le |\bp^+_{t,2}| ,
\end{equation}
where $\bp^+_{t,2}$ has the distribution of $\bp^+_t$ and is independent of $\bp^+_{t,1}$. Indeed, we obtain independence
of the coupled branching processes for the same reason that we can couple with a branching process in each case: in each step 
our arguments show that the \emph{conditional} distribution of the number of new vertices we reach in the hypergraph exploration is dominated
by the distribution arising in the branching process. For~$\cT_2^-$ the conditioning here is on the entire exploration~$\cT_1^-$
as well as all earlier steps of $\cT_2^-$.

Let $U_2$ be the set of vertices in $V_L$ reached by the (restricted) exploration $\cT_2^-$. Let
\[
 \cD_2 := \{ W_2\cap U_1\ne\emptyset \},
\]
and let 
\[
 \cE_2 := \{\text{at least one hyperedge in $\fE$ meeting $U_2$ is present in $\JPp$}\}. 
\]
We claim that
\[
 \cP = \cD_2\cup \cE_2.
\]
To see this, suppose $\cP$ holds but not $\cD_2$. Then the set $U_2$ of vertices in $V_L$ found by $\cT_2^-$ consists of
all vertices in $U_1^\cc$ connected, in $\JPp[U_1^\cc]$, to $W_2$. Since there is a path from $W_2$
to $W_1$ in $\JPp$, there must be a hyperedge in $\JPp$  containing some vertex in $U_2$ and some vertex in $U_1$. But such a hyperedge
must be in $\fE$. Hence $\cP\subseteq \cD_2\cup \cE_2$. The reverse containment is immediate.

If $\cP=\cD_2\cup \cE_2$ does \emph{not} hold, then $\cT_2^-=\cT_2$. Hence, if $\cB\setminus\cP$ holds,
we have $|\cT_2^-|=|\cT_2|\ge\Lambda$. Since $|\bp^+_{t,2}|\ge |\cT_2^-|$ by \eqref{eq:C2:cpl}, using \eqref{eq:SMM:key} we see that
\[
 \cB\setminus \cP \quad\text{implies}\quad \cS \cap \{|\bp^+_{t,2}|\ge \Lambda \}.
\]
If $\cS$ does not hold, then the exploration $\cT_1^-$ runs to completion, and in particular $A_1=\emptyset$ and hence $\fE=\emptyset$,
so~$\cE_2$ cannot hold.
Thus
\[
 \cE_2 = \cS \cap \cE_2.
\]
Recalling that $\cP=\cD_2\cup \cE_2$, we conclude that
\[
 \cB\cup \cP  \subseteq   \bb{ \cS \cap \{|\bp^+_{t,2}|\ge \Lambda \} } \cup (\cS \cap \cE_2) \cup \cD_2.
\]
Hence
\begin{equation}\label{eq:SMM:tot}
 \Pr(\cB\cup\cP) \le \Pr\bb{ \cS,\ |\bp^+_{t,2}|\ge \Lambda }
 + \Pr\bb{\cS,\ \cE_2,\ |\bp^+_{t,2}|<\Lambda }  + \Pr(\cD_2).
\end{equation}

The first (main) term in \eqref{eq:SMM:tot} is easy to bound: the branching process $\bp^+_{t,2}$ has the distribution of $\bp_t^+$
and is independent of our first exploration and hence of $\cS$. Thus
\begin{equation}\label{eq:SMM:term2}
 \Pr\bb{ \cS,\ |\bp^+_{t,2}|\ge \Lambda } = \Pr(\cS) \Pr\bb{|\bp_t^+|\ge\Lambda}.
\end{equation}

We now turn to $\Pr\bb{\cS,\ \cE_2,\ |\bp^+_{t,2}|<\Lambda }$. We will evaluate this by conditioning on the result
of the two explorations $\cT_1^-$ and $\cT_2^-$ as well as the coupled branching process $\bp_{t,2}^+$.
(More formally, we condition on all information revealed during these explorations.)
The first key observation is that, for any two distinct vertices $x_1,x_2 \in V_L$, the probability $\pi$ that they are connected by some so-far untested hyperedge satisfies~$\pi=O(n^{-1})$. 
To see this, note that there are at most $r^2 |V_L|^{r-2} = O(r^2 n^{r-2})$ hyperedges containing $x_1,x_2$, and each so-far untested hyperedge appears independently according to a Poisson process with rate $\hlambda^+_{k,r}= \hlambda_{k,r}(\fS^+_t) =Q_{k,r}^+(t)/|V_L|^r = O(e^{-b_1(k+r)}n^{1-r})$, see~\eqref{eq:rates:qkr}, \eqref{eq:Qkr:pm} and~\eqref{eq:qkr:pm:tail}. 
Using a union bound argument, it follows that
\[
 \pi \le \sum_{k\ge 0,\, r \ge 0} \Bigl[r^2 |V_L|^{r-2} \cdot \hlambda^+_{k,r}\Bigr] = O\Bigl(\sum_{k\ge 0,\, r \ge 0} r^2 e^{-b_1(k+r)}/n\Bigr) = O(n^{-1}) .
\]
Recall that $U_2$ is the set of vertices in $V_L$ reached by $\cT_2^-$. Since $|U_2|\le |\cT_2^-|\le |\bp_{t,2}^+|$,
recalling \eqref{eq:W1A} the total number of pairs of vertices $(x_1,x_2) \in A_1 \times U_2$ is at most
\[
|A_1 \times U_2|= |A_1| \cdot |U_2| \le |A_1|\cdot |\bp^+_{t,2}| \le   5 \eps n^{2/3} |\bp_{t,2}^+| .
\]
If $\cE_2$ holds, at least one of these pairs is connected by some so-far untested hyperedge.
Hence, by a union bound argument, using $\pi=O(1/n)$ we infer 
\[
 \Pr\bb{ \cE_2 \mid \cT_1^-,\cT_2^-,\bp_{t,2}^+ } 
\le |A_1 \times U_2| \cdot \pi = O(\eps n^{-1/3}) \cdot |\bp^+_{t,2}| ,
\]
and so
\[
 \Pr\bb{ \cE_2,\ |\bp^+_{t,2}|<\Lambda \mid \cT_1^-,\cT_2^-,\bp^+_{t,2} } = 
\indic{ |\bp^+_{t,2}|<\Lambda } \cdot \Pr\bb{ \cE_2 \mid \cT_1^-,\cT_2^-,\bp^+_{t,2} } 
= O(\eps n^{-1/3}) \cdot |\bp^+_{t,2}| \indic{ |\bp^+_{t,2}|<\Lambda }.
\]
Taking the expectation over $\cT_2^-$ and $\bp^+_{t,2}$ (which are coupled with each other), 
since $\bp^+_{t,2}$ is independent of~$\cT_1^-$ we conclude that
\[
 \Pr\bb{ \cE_2,\ |\bp^+_{t,2}|<\Lambda \mid \cT_1^- } = O(\eps n^{-1/3}) \cdot \E(|\bp^+_{t,2}|  \indic{ |\bp^+_{t,2}|<\Lambda }) .
\]
Since $\bp_{t,2}^+$ has the distribution of $\bp_t^+$, we have
\[
\E(|\bp^+_{t,2}|  \indic{ |\bp^+_{t,2}|<\Lambda }) = \E\bigl(|\bp^+_{t}|  \indic{ |\bp^+_{t}|<\Lambda }\bigr) \le \E\big(|\bp^+_{t}|  \indic{ |\bp^+_{t}|<\infty }\big) =:\nu,
\]
so $\Pr\bb{ \cE_2,\ |\bp^+_{t,2}|<\Lambda \mid \cT_1^- } = O(\eps\nu n^{-1/3})$.
Since this holds whatever the outcome of $\cT_1^-$, and this outcome determines whether $\cS$ holds, we conclude that
\begin{equation}\label{eq:cE2}
 \Pr\bb{\cS,\ \cE_2,\ |\bp^+_{t,2}|<\Lambda }  = O(\eps\nu n^{-1/3}) \cdot \Pr(\cS).
\end{equation}

Next we bound the probability of the event $\cD_2$ that the random starting set $W_2$ of our second exploration intersects
the set $U_1$ of $V_L$-vertices reached by the first, truncated exploration.
Now $W_2=C_{R_2}(H_L)$ is the union of the components of the initial graph $H_L$ containing the random vertices in $R_2$,
where $R_2$ consists either of a single vertex random vertex of $V_L$, or of $r_2$ independent random vertices of $V_L$.
Since $U_1$ is a union of components of $V_L$,
the event $\cD_2$ holds if and only if $R_2$ contains at least one vertex from $U_1$. Since $R_2$ is independent of $U_1$, 
using conditional expectations we thus infer
\[
 \Pr(\cD_2) \le \Pr(|R_2\cap U_1| \ge 1) \le \E\biggl( \E\biggl( |R_2| \cdot \frac{|U_1|}{|V_L|} \; \bigg| \; |R_2|, U_1 \biggr)\biggr) = \frac{\E |R_2|\cdot \E |U_1| }{|V_L|}.
\]
Recall that $|\fS_t^+|= \Theta(n)$ by~\eqref{def:n:pm:diff}, and that $V_L=V_L(\fS^+_t)$ satisfies $|V_L|=\Theta(n)$ by~\eqref{eq:nice:VL} and the definition of $\fS_t^+$, see~\eqref{eq:fS:pm}.
Since the variables $Q_{k,r}=Q_{k,r}^+(t)$ have exponential tails, see~\eqref{eq:Qkr:pm} and~\eqref{eq:qkr:pm:tail}, we have
\[
 \E |R_2| = 
\sum_{v_2 \in V_L} \frac{1}{|\fS_t^+|}
+ 
\sum_{k_2,r_2} r_2 \frac{k_2Q_{k_2,r_2}}{|\fS_t^+|} = O(1)\cdot \frac{n + \sum_{k,r} kr e^{-b_1(k+r)}n}{n}
 = O(1).
\]
Since $|V_L|=\Theta(n)$, we thus obtain
\[
 \Pr(\cD_2) = O(n^{-1}) \cdot \E |U_1|  = O(n^{-1}) \cdot\E( \min\{|\bp_{t,1}^+|,2n^{2/3} \} ) = O(n^{-1}) \cdot \E( \min\{|\bp_t^+|,2n^{2/3} \} ),
\]
where in second step we used $|U_1|\le |\cT_1^-| \le |\cT_1| \le |\bp_{t,1}^+|$ and the bound \eqref{eq:W1L},
and in the final step we used that $\bp_{t,1}^+$ has the distribution of $\bp_t^+$.
Now
\[
 \E( \min\{|\bp_t^+|,2n^{2/3} \} ) \le \E( |\bp_t^+|\indic{|\bp_t^+|<\infty} ) + \E( 2n^{2/3}\indic{|\bp_t^+|=\infty} )
 = \nu + 2n^{2/3}\Pr(|\bp_t^+|=\infty).
\]
Hence
\begin{equation}\label{eq:cD2}
 \Pr(\cD_2) \le O\bb{\nu n^{-1} + n^{-1/3}\Pr(|\bp_t^+|=\infty)}.
\end{equation}

Combining \cref{eq:SMM:tot,eq:SMM:term2,eq:cE2,eq:cD2} yields
\[
 \Pr(\cB\cup\cP) \le \Pr(\cS) \cdot \bb{ \Pr(|\bp_t^+|\ge \Lambda) + O(\eps\nu n^{-1/3})}
 + O\bb{\nu n^{-1}+n^{-1/3} \Pr(|\bp_t^+| = \infty)}.
\]
Together with \cref{eq:EX2:BP,eq:SMM:Ya,eq:SMM:key:2}, this completes the proof of inequality~\eqref{eq:Nkvar:sup1}. 

It remains to prove the claimed upper bound~\eqref{eq:Nkvar:sup:sigma} for $\tau$ defined in~\eqref{eq:SMM:key:2}. 
Recall that the width~$w(\bp)$ of a branching process is defined as the supremum of the number of particles in any generation.
We first estimate the probability of the event involving $w(\bp^+_{t}) \ge \ceil{\eps \Lambda}$ in~\eqref{eq:SMM:key:2}.
Analogous to Section~2 of~\cite{BR2012}, by sequentially exploring~$\bp_{t}^{+}$ generation-by-generation, we can stop at the first generation with at least~$\ceil{\eps \Lambda}$ particles of type~$L$. 
The children of each of these particles form independent copies of the branching process~$\bp^{1,+}_t=\bp^{1,+}(\fS)$
 defined in Definition~\ref{def:bp2pm} (note that this process differs from $\bp_{t}^{+}$), so the conditional probability of dying out is at most $(1-\rho_1)^{\ceil{\eps \Lambda}}$ for $\rho_1 =\Pr(|\bp^{1,+}_t| = \infty)$.
Since $\rho_1>0$, it follows that
\begin{equation*}
\Pr\bb{|\bp^+_{t}| < \infty \: | \: w(\bp^+_{t}) \ge \ceil{\eps \Lambda}} \le (1-\rho_1)^{\ceil{\eps \Lambda}} \le e^{-\rho_1 \ceil{\eps \Lambda}} < 1 .
\end{equation*}
Note that, for any two events $\cX,\cY$ with $\Pr(\neg \cX \mid \cY)>0$, we have 
\[
 \Pr(\cX, \: \cY) = \Pr(\cY) \cdot \Pr(\cX \mid \cY)  = \frac{\Pr(\neg \cX , \: \cY)}{\Pr(\neg \cX \mid \cY)} \cdot \Pr(\cX \mid \cY) \le \Pr(\neg \cX) \cdot \frac{\Pr(\cX \mid \cY)}{1-\Pr(\cX \mid \cY)} .
\] 
Since $x/(1-x)$ is monotone increasing for $x < 1$, we thus obtain  
\begin{equation*}
\Pr\bb{|\bp^+_{t}| < \infty, \: w(\bp^+_{t}) \ge \ceil{\eps \Lambda}} \le \Pr(|\bp^+_{t}| = \infty) \cdot \frac{e^{-\rho_1 \ceil{\eps \Lambda}}}{1-e^{-\rho_1 \ceil{\eps \Lambda}}} ,
\end{equation*}
which together with $\Pr(|\bp^+_{t}| = \infty) \le \Pr(|\bp^+_{t}| \ge \Lambda)$ and $1+x/(1-x)=1/(1-x)$ completes the proof of inequality~\eqref{eq:Nkvar:sup:sigma}. 
\end{proof}

\newoddpage

\section{Component size distribution: qualitative behaviour}\label{sec:cpl2}
In this section we study the Poissonized random graphs $\JPpm=\JP(\fS^{\pm}_t)$ introduced in Section~\ref{sec:dom}. 
Our goal is to use properties of the closely related branching processes~$\bp_t$ and~$\bp^{\pm}_t$, together
with results from the previous section, to estimate various moments of the component size distribution of~$\JPpm$.

In Section~\ref{sec:BPO}
we establish several technical properties of the offspring distributions of~$\bp_t$ and~$\bp^{\pm}_t$.
In Section~\ref{sec:bpresults} we state results for the survival and point probabilities of these branching processes, 
which 
in Section~\ref{sec:mom} are then used to estimate the first moment and variance of (a)~the number of vertices of~$\JPpm=\JP(\fS^{\pm}_t)$ in components of at least certain sizes and (b)~the $r$th order susceptibility of~$\JPpm$.  
As a by-product, we also establish several results (Theorems~\ref{thm:L1rho}, \ref{thm:rhok} and~\ref{thm:sjfkt}) describing the qualitative behavior of various limiting functions appearing in Section~\ref{sec:results}. 
Finally, as mentioned earlier, in Section~\ref{sec:proof} we will use Lemmas~\ref{lem:cond'}, \ref{lem:nice} and~\ref{lem:cpl} to transfer properties of $\JP(\fS^{\pm}_t)$ back to the original random graph process~$G^{\cR}_{n,tn}$.

\subsection{Properties of the offspring distributions}\label{sec:BPO}
In this subsection we revisit the branching processes~$\bp_t,\bp_t^1$ and~$\bp_t^\pm,\bp_t^{1,\pm}$ defined in Sections~\ref{sec:BPI} and~\ref{sec:dom:dom},
and derive properties of their offspring distributions.
We start with the `idealized' offspring distributions $(Y_t,Z_t)$ and $(Y^0_t,Z^0_t)$ defined in Section~\ref{sec:BPI:distr},
studying the probability generating functions 
\begin{equation}\label{eq:f:mgf}
\gf(t,\alpha,\beta) := \E\bigl( \alpha^{Y_{t}} \beta^{Z_{t}}\bigr) \qquad \text{and} \qquad \gf^0(t,\alpha,\beta) := \E\bigl( \alpha^{Y^0_{t}}\beta^{Z^0_{t}}\bigr).
\end{equation}
These expectations (infinite sums) make sense for complex $\alpha$ and $\beta$ whenever the corresponding sum converges
absolutely. A priori, they make sense only for real $t$; however, we shall show that both probability generating functions
extend to analytic functions in a certain complex domain.

\begin{theorem}\label{thm:f:mgf}
There exist $\delta>0$ and~$R>1$ such that the functions $\gf(t,\alpha,\beta)$ and $\gf^0(t,\alpha,\beta)$ 
are defined for all real $t$ with $|t-\tc|<\delta$ and complex $\alpha,\beta$ with $|\alpha|,|\beta|<R$. 
Furthermore, each of these functions has an analytic extension
to the complex domain $\fD_{\delta,R}:=\{(t,\alpha,\beta) \in \CC^3: \: |t-\tc|<\delta \text{ and } |\alpha|,|\beta|<R\}$. 
\end{theorem}

\begin{proof}
The proof hinges on the following two facts: (i)~that the probability generating function of the distribution~$N$ defined in~\eqref{def:N} is analytic due to the exponential tails of Theorem~\ref{thm:init}, 
and (ii)~that the generating function $P(t,x,y)$ defined in~\eqref{def:P} is analytic by Theorem~\ref{thm:PDE}. 

Turning to the details, we first study 
\begin{equation}\label{eq:fN}
\Phi(\alpha):=\E \alpha^N = \sum_{k > K}\alpha^k \rho_{k}(t_0)/\rho_\omega(t_0).
\end{equation}
Let $\beta_0:=e^{b/3}>1$, where $b>0$ is the constant in~\eqref{eq:Qkr:tail'}. 
Recalling the exponential tail bound $|\rho_{k}(t_0)| \le A e^{-ak}$ of~\eqref{eq:rhok:t0:tail},
standard results for power series yield that $\Phi(\alpha)=\E \alpha^N$ is analytic for all
$\alpha \in \CC$ with~$|\alpha| < e^{a}$.
Since $\Phi(1)=1$, we may pick $\alpha_0 \in (1,e^a)$ such that $\Phi(\alpha_0) < \beta_0$. Since $\Phi$ is a power series
with non-negative coefficients, it follows that $|\Phi(\alpha)| < \beta_0$ for all $\alpha \in \CC$ with $|\alpha|\le \alpha_0$. 
We shall prove the result with~${R:=\min\{\alpha_0,\beta_0\}>1}$.

Recalling the definition of $(Y_t,Z_t)$, see~\eqref{def:Yt2Zt}, by independence and 
using that $H_{k,r,t} \sim \Po(\lambda_{k,r}(t))$ 
it follows that  
\begin{equation}\label{eq:ff}
\gf(t,\alpha,\beta) 
= \prod_{k \ge 0,\, r \ge 1} \E\Bigl(\Bigl[\bigl(\E \alpha^N\bigr)^{r-1}\beta^k\Bigr]^{H_{k,r,t}}\Bigr) 
= \exp\Bigl\{\sum_{k \ge 0,\, r \ge 1}\lambda_{k,r}(t)\Bigl(\bigl(\Phi(\alpha)\bigr)^{r-1}\beta^{k}-1\Bigr)\Bigr\} .
\end{equation}
Recalling $\lambda_{k,r}(t)= r q_{k,r}(t)/\rho_{\omega}(t_0)$ and the definition of $P(t,x,y)$, see~\eqref{def:P}, we see that
\begin{equation*}
 \gf(t,\alpha,\beta) = \exp\Bigl\{ \Bigl(P_y(t,\beta,\Phi(\alpha))-P_y(t,1,1)\Bigr) / \rho_{\omega}(t_0)\Bigr\}. 
\end{equation*}
By Theorem~\ref{thm:PDE} there is some $\delta>0$ such that
$P(t,x,y)$ has an analytic extension to the complex domain~$\fD_{\delta,\beta_0}$.
Replacing $P$ by this extension in the formula above gives the required analytic extension of~$\gf$,
since derivatives, compositions and products of analytic functions are analytic.

Finally we consider $\gf^0(t,\alpha,\beta)= \E ( \alpha^{Y^0_{t}} \beta^{Z^0_{t}})$, which from \eqref{def:YZR}--\eqref{def:Yt0} satisfies
\begin{equation}\label{eq:f0}
\gf^0(t,\alpha,\beta)  = \sum_{k > K} \rho_k(t_0) \alpha^k + \sum_{z \ge 1,\, r \ge 0} z q_{z,r}(t) \bigl(\E \alpha^N\bigr)^{r} \beta^z = \rho_{\omega}(t_0)\Phi(\alpha) + \beta P_x(t,\beta,\Phi(\alpha)).
\end{equation}
Using again that derivatives, products and compositions of analytic functions are analytic, we see that $\gf^0(t,\alpha,\beta)$ also has
an analytic extension of the claimed form.
\end{proof}

Since $\gf_{\alpha}(t,1,1)=\E Y_t$ and $\gf_{\alpha\alpha}(t,1,1)=\E Y_t(Y_t-1)$, Theorem~\ref{thm:f:mgf} implies that
$\E Y_t$, $\E Y_t^2$ and thus $\Var Y_t$ are analytic for $t \in (\tc-\eps,\tc+\eps)$. 
A similar argument applies to $Z_{t}$, $Y^0_t$ and $Z^0_t$. 
Intuitively, we now show that $\tc$~is the `critical point' of the 
branching process $\bp_t= \bp_{Y_t,Z_t,Y_t^0,Z_t^0}$ defined in Section~\ref{sec:BPI:distr} 
(as expected, since a linear size giant component appears after time~$\tc$ in the random graph process). 

\begin{lemma}\label{lem:Zt:E}
We have $\E Y_{\tc}=1$. 
Furthermore, for all $t \in (t_0,t_1)$ we have
\begin{equation}\label{eq:lem:Zt:E:pos}
 \E Y^0_t >0 \qquad\text{and}\qquad \ddt \E Y_t > 0 . 
\end{equation}
\end{lemma} 

\begin{proof}
Fix $t \in (t_0,t_1)$. 
Recalling the definition of $Y_t$ and $u(t)$, see~\eqref{def:Yt2Zt} and~\eqref{def:W}, using independence,
the fact that $H_{k,r,t} \sim \Po(\lambda_{k,r}(t))$ and that $\E H_{k,r,t} = \lambda_{k,r}(t)= r q_{k,r}(t)/\rho_{\omega}(t_0)$, we see that
\begin{equation*}
\E Y_t = \sum_{k \ge 0,\,r \ge 2} \Bigl[ \E H_{k,r,t} \cdot (r-1) \cdot \E N  \Bigr]= \Bigl[\sum_{k,r \ge 0}r(r-1) q_{k,r}(t) \Bigr] \cdot \E N/ \rho_{\omega}(t_0) = u(t) \E N/ \rho_{\omega}(t_0) .
\end{equation*}
Since $\E N > 0$, Lemma~\ref{lem:w} thus entails $\ddt\E Y_t = u'(t) \cdot \E N/ \rho_{\omega}(t_0) > 0$. 
By~\eqref{def:YZR}--\eqref{def:Yt0} we similarly have $\E Y^0_t \ge \E Y_{0,t} = \E N \cdot \rho_{\omega}(t_0)  > 0$.

We next prove $\E Y_{\tc}=1$.
By Corollary~\ref{cor:rhosr} we have $\Pr(|\bp_t|=\infty)=\rho(t)$ for $t \in [t_0,t_1]$, so the discussion below~\eqref{eq:L1:pto} implies 
\begin{equation}\label{eq:lem:tc:bptinfty}
\Pr(|\bp_t|=\infty)=0 \text{ for $t \in [t_0,\tc]$} \quad \text{and} \quad \Pr(|\bp_t|=\infty)>0 \text{ for $t \in (\tc,t_1]$.} 
\end{equation}
Recall that the branching process $\bp_t$ has (except for the initial generation) a two-type offspring distribution~$(Y_t,Z_t)$, which corresponds to particles of type~$L$ and~$S$, respectively.  
Since only type~$L$ particles (which are counted by~$Y_t$) have children, by~\eqref{eq:lem:tc:bptinfty} standard branching process results imply $\E Y_{t} \le 1$ for $t \in [t_0,\tc)$ and~$\E Y_{t} \ge 1$ for $t \in (\tc,t_1]$. 
Now~$\E Y_{\tc} = 1$ follows since $\E Y_t$ is analytic and thus continuous at~$t=\tc$.
\end{proof}

Intuitively speaking, we next show that no linear relation of the form $aY_t + b Z_t=c$ holds.

\begin{lemma}\label{lem:YtZtlb}
Define $\per$ as in Lemma~\ref{lem:allowed}. 
There exists $k_0 > K$ such that, for all $t \in (t_0,t_1)$,
\begin{equation}\label{utiltlower}
 \min\Bigl\{\Pr(Y_t=k_0,Z_t=k_0), \: \Pr(Y_t=k_0+\per,Z_t=k_0), \: \Pr(Y_t=k_0,Z_t=k_0+\per)\Bigr\} > 0 .
\end{equation}
\end{lemma}

\begin{proof}
Fix $t \in (t_0,t_1)$. 
By Lemma~\ref{lem:allowed} there exists $k_0 \in \cSR$ with $k_0 \ge \max\{K+1,\kR\}$ and $k_0+\per \in \cSR$.  
By Lemma~\ref{lem:rhokS1}, $\rho_{k_0}(t_0)$ and $\rho_{k_0+\per}(t_0)$ are positive.
Furthermore, since $k_0>K$, by Lemma~\ref{lem:allowed:kr}\ref{allowed:r1}
we have $(k_0,2)\in \cSkr$ and $(k_0+\per,2)\in \cSkr$,
and hence $q_{k_0,2}(t)$ and $q_{k_0+\per,2}(t)$ are positive.
We consider the cases $\sum_{k,r \ge 0}H_{k,r,t} \in \{H_{k_0,2,t},H_{k_0+\per,2,t}\}$ in the definition~\eqref{def:Yt2Zt} of $(Y_t,Z_t)$. 
For $k^* \in \{k_0,k_0+\per\}$ we then focus on the event $H_{k^*,2,t}=1$, and consider the cases $N_{k^*,1,1,1} \in \{k_0,k_0+\per\}$ in the definition~\eqref{def:Yt2Zt} of $(Y_t,Z_t)$. 
Recalling that $\sum_{k,r \ge 0}\lambda_{k,r}(t) \in (0,\infty)$, it follows that~\eqref{utiltlower} holds. 
\end{proof}

We now turn to the `perturbed' offspring distributions $(Y^{\pm}_t,Z^{\pm}_t)$ and $(Y^{0,\pm}_t,Z^{0,\pm}_t)$ defined in Section~\ref{sec:dom:dom}.
Note that these distributions depend not only on~$t$, but also on the ($t$-nice) parameter list~$\fS$, see Definition~\ref{def:cpl:UB}. 
The next result intuitively states that \emph{all} such probability generating functions $\gf^{\pm}$ and $\gf^{0,\pm}$, defined in~\eqref{eq:fPT:mgf} below, are almost indistinguishable from the corresponding `idealized' $\gf$ and $\gf^0$ defined in~\eqref{eq:f:mgf}. 

\begin{theorem}\label{thm:fPT:mgf}
There exist $C,n_0>0$ and $R>1$ such that the following holds for all $n\ge n_0$,
all $t \in [t_0,t_1]$ and all $t$-nice parameter lists~$\fS$.
Define $(Y^{\pm}_t,Z^{\pm}_t)$ and $(Y^{0,\pm}_t,Z^{0,\pm}_t)$ as in Definition~\ref{def:cpl:UB}, and set
\begin{equation}\label{eq:fPT:mgf}
 \gf^{\pm}(t,\alpha,\beta) := \E\bigl( \alpha^{Y^{\pm}_{t}}\beta^{Z^{\pm}_{t}}\bigr)
 \qquad \text{and} \qquad \gf^{0,\pm}(t,\alpha,\beta) := \E\bigl( \alpha^{Y^{0,\pm}_{t}}\beta^{Z^{0,\pm}_{t}}\bigr) .
\end{equation}
Then, writing $\fR := \{x \in \CC: |x| \le R\}$, we have
\begin{gather}
\label{eq:fPT:mgf:ff0}
\sup_{\alpha,\beta \in \fR} \max\Bigl\{\bigl|\gf(t,\alpha,\beta)\bigr|, \: \bigl|\gf^{\pm}(t,\alpha,\beta)\bigr|, \: \bigl|\gf^0(t,\alpha,\beta)\bigr| , \: \bigl|\gf^{0,\pm}(t,\alpha,\beta)\bigr|\Bigr\} \le C ,\\
\label{eq:fPT:mgf:f}
\sup_{\alpha,\beta \in \fR}\Bigl|\gf(t,\alpha,\beta) -\gf^{\pm}(t,\alpha,\beta) \Bigr| \le C n^{-1/3}, \\
\label{eq:fPT:mgf:f0}
\sup_{\alpha,\beta \in \fR} \Bigl|\gf^{0}(t,\alpha,\beta) - \gf^{0,\pm}(t,\alpha,\beta)\Bigr| \le C n^{-1/3} .
\end{gather}
\end{theorem}

\begin{proof}
We start by showing that $N^{\pm}$ and $\lambda^{\pm}_{k,r}(t)$ are very good approximations to $N$ and $\lambda_{k,r}(t)$. 
Here and throughout the proof, all constants do not depend on $t\in [t_0,t_1]$ or on the choice of $\fS$. 
Combining the  definitions of $N$ and $N^{\pm}$ (see~\eqref{def:N} and Definition~\ref{def:cpl:UB}),
with the exponential tails of $\rho_k(t_0)$ and $N_{\ge k}$ (see~\eqref{eq:rhok:t0:tail} and~\eqref{eq:Nk:t0:tail'}), 
we see that 
there are absolute constants~$d,D,n_0>0$ 
such that, for $n \ge n_0$,  
\begin{align}
\label{eq:f:N:exp}
\max\bigl\{\Pr(N=k), \: \Pr(N^{\pm}=k)\bigr\} &= O(e^{-ak}) \le D e^{-d k}.
\end{align}
Recall that $\Pr(N^{\pm}=k)$ approximates $N_k/|V_L|=N_k(\fS)/|V_L(\fS)|$, and that $N_k$ approximates $\rho_k(t_0)n$ (see Definition~\ref{def:cpl:UB} and~\eqref{eq:Nk:t0'}). 
After decreasing~$d$ and increasing~$D,n_0$ (if necessary), 
using ${\rho_k(t) \le A e^{-ak}}$ and~${a \ge a_0}$ (see~\eqref{eq:rhok:t0:tail} and~\eqref{eq:def:D0a0}) 
together with the upper bound~\eqref{eq:f:N:exp} and \eqref{def:n:pm:diff},
it is routine (but slightly messy) to see that, for~$n \ge n_0$,  
\begin{align}
\label{eq:f:N:diff}
\bigl|\Pr(N=k)-\Pr(N^{\pm}=k)\bigr| &= O\bigl(\min\bigl\{(\log n)^{D_{\bp}}n^{-1/2}+e^{-a_0k}n^{-1/3}, \: e^{-a k}\bigr\}\bigr) \le D e^{-d k} n^{-1/3} .
\end{align}
For $\lambda_{k,r}(t)$ and $\lambda^{\pm}_{k,r}(t)$ as defined in~\eqref{def:Hkt:l} and Definition~\ref{def:cpl:UB}, 
similar reasoning shows that (again after decreasing~$d$ and increasing~$D,n_0$, if necessary), for $n \ge n_0$, 
\begin{align}
\label{eq:f:l:exp}
\max\bigl\{|\lambda_{k,r}(t)|, \: |\lambda^{\pm}_{k,r}(t)|\bigr\} &\le D e^{-d(k+r)}, \\
\label{eq:f:l:diff}
\bigl|\lambda_{k,r}(t) - \lambda^{\pm}_{k,r}(t)\bigr| &\le D e^{-d (k+r)} n^{-1/3} .
\end{align}

With the above estimates in hand, the proof boils down to routine calculations (analogous to those from Theorem~\ref{thm:f:mgf}).
Turning to the details, let
\begin{equation*}
\Phi(\alpha):=\E \alpha^N \qquad \text{and} \qquad \Phi^{\pm}(\alpha):=\E \alpha^{N^{\pm}} . 
\end{equation*}
Using~\eqref{eq:ff} we write 
\begin{equation}\label{eq:f:mod}
 \gf(t,\alpha,\beta) = \exp\Bigl\{\sum_{k \ge 0,\, r \ge 1}\lambda_{k,r}(t)\Bigl(\bigl(\Phi(\alpha)\bigr)^{r-1}\beta^k-1\Bigr)\Bigr\} = : \exp\Bigl\{\Gamma(t,\alpha,\beta)\Bigr\}.
\end{equation}
Recalling the definition of $(Y^{0,\pm}_t,Z^{0,\pm}_t)$, see Definition~\ref{def:cpl:UB}, arguing as for~\eqref{eq:ff} we obtain
\begin{equation}\label{eq:ff:mod}
 \gf^{\pm}(t,\alpha,\beta) = \exp\Bigl\{\sum_{k \ge 0,\, r \ge 1}\lambda^{\pm}_{k,r}(t)\Bigl(\bigl(\Phi^{\pm}(\alpha)\bigr)^{r-1}\beta^k-1\Bigr)\Bigr\} = : \exp\Bigl\{\Gamma^{\pm}(t,\alpha,\beta)\Bigr\}.
\end{equation}
Now $\Phi(1)=\Phi^\pm(1)=1$. Using the (uniform) exponential tail bound~\eqref{eq:f:N:exp} to bound 
the derivatives of $\Phi$ and of $\Phi^\pm$, we may find a constant $1<R<e^{d/2}$ such that
$\Phi(R),\Phi^\pm(R)<e^{d/2}$. Writing $\fR := \{x \in \CC: |x| \le R\}$ as in the statement of the theorem,
since $\Phi$ and $\Phi^\pm$
are power series with non-negative coefficients it follows~that
\begin{equation}\label{eq:fPT:alpha_0}
\sup_{\alpha \in \fR}\max\Bigl\{\bigl|\Phi(\alpha)\bigr|, \: \bigl|\Phi^{\pm}(\alpha)\bigr|\Bigr\} \le e^{d/2} .
\end{equation}
Together with the exponential tail bound~\eqref{eq:f:l:exp} and $R \le e^{d/2}$, it follows that there is a $C_1 \ge 1$ such that 
\begin{equation}\label{eq:fPT:Gamma:abs}
\sup_{\alpha, \beta \in \fR} \max\Bigl\{ \bigl|\Gamma(t,\alpha,\beta)\bigr|, \:  \bigl|\Gamma^{\pm}(t,\alpha,\beta)\bigr| \Bigr\} \le \sum_{k \ge 0,\, r \ge 1} De^{-d (k+r)}\Bigl(e^{d(k+r)/2}+1\Bigr) \le C_1 .
\end{equation}
Furthermore, using~$R \le e^{d/2}$ and the exponential difference estimate~\eqref{eq:f:N:diff}, there is a $C_2 >0$ such that 
\begin{equation}\label{eq:fPT:Phi}
\sup_{\alpha \in \fR}\left|\Phi(\alpha)-\Phi^{\pm}(\alpha)\right| \le \sum_{k \ge 0}e^{dk/2} \cdot De^{-d k}n^{-1/3} \le C_2 n^{-1/3}.
\end{equation}
Note that (as easily seen by induction), for all $I \in \NN$ we have 
\begin{equation}\label{eq:product}
\Bigl|\prod_{h \in [I]} y_h-\prod_{h \in [I]} z_h\Bigr| \le \sum_{j \in [I]} |y_j-z_j| \cdot \prod_{1 \le h < j} |y_h|\prod_{j < h \le I} |z_h|.
\end{equation}
Together with the bound~\eqref{eq:fPT:alpha_0} and the difference estimate~\eqref{eq:fPT:Phi}, it now follows for $r \ge 1$ that 
\begin{equation}\label{eq:fPT:Phi:partial}
\begin{split}
\sup_{\alpha \in \fR} \left|\bigl(\Phi(\alpha)\bigr)^{r-1}-\bigl(\Phi^{\pm}(\alpha)\bigr)^{r-1}\right| & \le r \cdot C_2 n^{-1/3} \cdot (e^{d/2})^{\max\{r-2,0\}}\le C_2 r e^{dr/2} n^{-1/3}. 
\end{split}
\end{equation}
Together with the difference estimates~\eqref{eq:f:l:diff}, the upper bound~\eqref{eq:fPT:alpha_0} and $R \le e^{d/2}$, using~\eqref{eq:product} we also infer that there is a $C_3 \ge 1$ such that, say, 
\begin{equation}\label{eq:fPT:Gamma:diff}
\sup_{\alpha, \beta \in \fR}  \left|\Gamma(t,\alpha,\beta) - \Gamma^{\pm}(t,\alpha,\beta) \right| \le 
\sum_{k \ge 0,\, r \ge 1} D (C_2r+2) e^{-d(k+r)/2} n^{-1/3} \le C_3 n^{-1/3} .
\end{equation}
Together with~\eqref{eq:fPT:Gamma:abs} and~\eqref{eq:f:mod}--\eqref{eq:ff:mod}, 
setting $C :=2 C_3 e^{C_1}$, say, for $n \ge n_0(C)$ large enough this readily establishes~\eqref{eq:fPT:mgf:f} and the upper bounds for~$\gf$ and $\gf^{\pm}$ in~\eqref{eq:fPT:mgf:ff0}.  

Finally, we omit the analogous arguments for $\gf^{0}(t,\alpha,\beta)$ and $\gf^{0,\pm}(t,\alpha,\beta)$.
\end{proof}

\subsection{Branching process results}\label{sec:bpresults}
In this subsection we state a number of results concerning the branching processes~$\bp_t$ and~$\bp_t^\pm$, which 
we are proved in a companion paper~\cite{BPpaper} written with Svante Janson 
(modulo a reduction given in Appendix~\ref{sec:BP}).
As we shall see, their survival and point probabilities are qualitatively similar to standard Galton--Watson branching process arising in the context of classical Erd\H os--R\'enyi random graphs. 
In particular, for $t=\tc+\eps$ the survival probabilities grow linearly in~$\eps$, and for $t=\tc\pm\eps$ the size-$k$ point probabilities decay exponentially in~$\Theta(\eps^2k)$. 

We start with our results for the `idealized' branching process $\bp_t= \bp_{Y_t,Z_t,Y_t^0,Z_t^0}$ defined in Section~\ref{sec:BPI:distr}.
\begin{theorem}[Survival probability of $\bp_t$]\label{thsurv-simple}
There exists $\eps_0>0$ such that the survival probability $\rho(t)=\Pr(|\bp_t|=\infty)$ is zero for $\tc-\eps_0\le t\le \tc$, 
is positive for $\tc < t \le \tc + \eps_0$. 
Furthermore, $\rho(t)$ is analytic on $[\tc,\tc+\eps_0]$; 
more precisely, there are constants $a_i$ with $a_1>0$ such that
\begin{equation*}
 \rho(\tc+\eps) = \sum_{i=1}^\infty a_i\eps^i .
\end{equation*}
for $\eps \in [0,\eps_0]$.
Moreover, an analogous statement holds for $\rho_1(t)=\Pr(|\bp^1_t|=\infty)$, where $\bp_t^1=\bp^1_{Y_t,Z_t}$ is defined as in Section~\ref{sec:BPI:distr}. 
\end{theorem}
Note that this result and Corollary~\ref{cor:rhosr}, which gives~$\rho(t) = \Pr(|\bp_t|=\infty)$ for $t \in [t_0,t_1]$,
immediately imply Theorem~\ref{thm:L1rho}.
Recall from Section~\ref{sec:period} that $\cSR$ is the set of component sizes which can be produced by
the rule $\cR$, and that for $t>0$, $\rho_k(t)>0$ if and only if $k\in \cSR$ (see Lemma~\ref{lem:rhokS1}).

\begin{theorem}[Point probabilities of $\bp_t$]\label{asyfull-simple}
There exists $\eps_0>0$ such that 
\begin{equation}\label{eq:asfs}
 \Pr(|\bp_t|=k) = (1+O(1/k)) \indic{k\in \cSR}k^{-3/2} \theta(t) e^{-\psi(t) k} 
\end{equation}
uniformly over all $k\ge 1$ and $t\in I = [\tc-\eps_0,\tc+\eps_0]$, where the functions $\theta$, $\psi$ are analytic
on $I$ with $\theta(t)>0$, $\psi(t)\ge 0$, $\psi(\tc)=\psi'(\tc)=0$, and $\psi''(\tc)>0$.
\end{theorem}
Note that the last condition implies in particular that $\psi(\tc\pm\eps)=a\eps^2+O(\eps^3)$ where $a>0$. 
Throughout the paper, $\psi(t)$ and $\theta(t)$ refer to the functions $\psi$ and $\theta$ appearing in the result above. 
Note that Theorem~\ref{asyfull-simple} and Corollary~\ref{cor:Nk}, which gives~$\rho_k(t) = \Pr(|\bp_t|=k)$ for $t \in [t_0,t_1]$, 
immediately imply Theorem~\ref{thm:rhok}.

Next, we state our results for the `perturbed' branching processes $\bp^{\pm}_t=\bp^{\pm}_t(\fS)$ defined in Section~\ref{sec:dom:dom}. 
Note that each is actually a family of branching processes, one for each $t$-nice parameter list $\fS$. In the following 
results the conditions ensure that $n$ is at least some constant, which may be made large by choosing~$T$ large and~$\eps_0$ 
small. In other words, particular small values of $n$ play no role.
\begin{theorem}[Survival probability of $\bp^\pm_t$]\label{thsurv-simple-pm}
There exist $\eps_0,C,T>0$ such that, writing $I_n = \{t \in \RR: T n^{-1/3} \le |\tc-t| \le \eps_0\}$,
for any $n\ge 1$, any $t \in I_n$ and any $t$-nice parameter list~$\fS$ the following holds for $\bp^{\pm}_t=\bp^{\pm}_t(\fS)$ as in Definition~\ref{def:bp2pm}. 
The survival probabilities $\Pr(|\bp^\pm_t|=\infty)$ are zero if $t \le \tc$, and if $t > \tc$ they are positive and satisfy 
\[
 \bigl|\Pr(|\bp^\pm_t|=\infty) - \rho(t)\bigr|  \le Cn^{-1/3} ,
\]
where the function $\rho$ is as in Theorem~\ref{thsurv-simple}.
Moreover, an analogous statement holds for $\Pr(|\bp^{1,\pm}_t|=\infty)$, 
where $\bp^{1,\pm}_t=\bp^{1,\pm}_t(\fS)$ is as in Definition~\ref{def:bp2pm}. 
\end{theorem}

Recall that $\per$ is the period of the rule $\cR$, defined in Section~\ref{sec:period}. As usual, $K$
is simply the cut-off size of the bounded-size rule $\cR$.

\begin{theorem}[Point probabilities of $\bp^\pm_t$]\label{asyfull-simple-pm}
There exist $\eps_0,C,T>0$
such that, writing $I_n = \{t \in \RR: T n^{-1/3} \le |\tc-t| \le \eps_0\}$, 
for $\bp^{\pm}_t=\bp^{\pm}_t(\fS)$ as in Definition~\ref{def:bp2pm} we have
\[
 \Pr(|\bp^\pm_t|=k) = (1+O(1/k)+O(n^{-1/3})) \indic{k\equiv 0\mathrm{\ mod\ }\per}k^{-3/2} \theta(t) e^{-\xi(\fS) k}
\]
uniformly over all $n\ge 1$, $k>K$, $t\in I_n$ and $t$-nice parameter lists $\fS$,
where the functions $\theta$ and $\psi$ are as in Theorem~\ref{asyfull-simple}, and
\begin{equation}\label{eq:xifS}
 \bigl|\xi(\fS)  - \psi(t)\bigr| \le C n^{-1/3}|t-\tc| .
\end{equation}
\end{theorem}

We shall later apply these results with $\eps=\eps(n)$ satisfying $\eps^3n\to\infty$,
in which case $\Pr(|\bp^\pm_{t+\eps}|=\infty) \sim \rho(t+\eps) = \Theta(\eps)$ and $\Pr(|\bp^\pm_{t-\eps}|=\infty) = \rho(t-\eps)=0$. 
Furthermore, for $t= \tc \pm \eps$ and $\eps^3n\to\infty$ we also have $\xi(\fS)\sim \psi(t)=\Theta(\eps^2)$. 

\begin{remark}\label{rem:period}
The indicator functions $\indic{k\in\cSR}$ and $\indic{k\equiv 0\mathrm{\ mod\ }\per}$ in Theorems~\ref{asyfull-simple}
and~\ref{asyfull-simple-pm}, and condition $k>K$ in the latter, may seem somewhat mysterious, so let us comment
briefly. Firstly, without the indicator function, for any fixed $k$, the conclusion \eqref{eq:asfs} holds
trivially. Indeed the function $f_k(t) := k^{-3/2} \theta(t) e^{-\psi(t) k}$
is positive at $t=\tc$ and is continuous, so reducing $\eps_0$ if 
necessary, it is bounded and bounded away from zero. 
Since probabilities lie in $[0,1]$, 
by simply taking the implicit constant in the $O(1/k)$
term large enough, for a fixed~$k$ we can thus ensure that \eqref{eq:asfs} holds without the indicator function. 
A similar comment applies to Theorem~\ref{asyfull-simple-pm}. It might thus appear 
that neither result says anything for small (fixed) $k$, but this is not quite true. When the relevant
indicator function is $0$, the result asserts that the corresponding probability is $0$. 
In the context of Theorem~\ref{asyfull-simple}, 
for $k \not\in \cSR$ and $t \in [t_0,t_1]$ we know that $\Pr(|\bp_t|=k)=\rho_k(t)=0$  
by Corollary~\ref{cor:Nk} and Lemma~\ref{lem:rhokS1}.
We could perhaps
define the processes $\bp^\pm_t$ so that their sizes (when finite) always lie in~$\cSR$, but we
have not done so. Hence the slightly different condition in Theorem~\ref{asyfull-simple-pm}.
In any case, the interest is only in~$k$ large, and in this case, from Lemma~\ref{lem:allowed}, 
$k\in \cSR$ if and only if~$k$ is a multiple of~$\per$.
\end{remark}

\begin{remark}\label{rem:bpsimple}
In Theorems~\ref{thsurv-simple}--\ref{asyfull-simple-pm} we may take the same constants $\eps_0,C,T$ in all cases (by choosing the minimum and maximum, respectively). 
Furthermore, by increasing~$T$, in Theorems~\ref{thsurv-simple-pm}--\ref{asyfull-simple-pm} we may assume that $\rho^\pm(t) \ge \rho(t)/2$, $\rho^\pm_1(t) \ge \rho_1(t)/2$ and $\xi(\fS) \ge \psi(t)/2$ hold for $t \in I_n$.
\end{remark}

The proofs of the results above are deferred to Appendix~\ref{sec:BP} and the companion paper~\cite{BPpaper}.
They rely on various technical properties of $\bp_t$ and $\bp_t^\pm$ established in Section~\ref{sec:BPO} (and some basic properties from Sections~\ref{sec:BPI}--\ref{sec:dom}),
but are otherwise independent of, and rather different from, the arguments in the present~paper.

\subsection{Moment estimates}\label{sec:mom}
In this subsection we estimate various moments of the component size distribution of the Poissonized random graphs $\JPpm = \JP(\fS^{\pm}_t)$. 
Firstly, in Section~\ref{sec:mom:large} we estimate the expected number of vertices in `large' components of~$\JPpm$, 
and show that the variance is small. 
Then, in Section~\ref{sec:mom:sus} we establish analogous statements for the expectation and variance of the modified susceptibility $S_{r,n}(\JPpm)$ defined in~\eqref{eq:def:SrG}.

Our proofs combine the domination arguments from Section~\ref{sec:dom:dom} with the branching processes estimates from Section~\ref{sec:bpresults}.
To apply both, we often need to make additional assumptions on the component sizes~$k$ we study. 
In particular, due to the lower bound in the domination result Theorem~\ref{thm:ENgekD} we often restrict our attention to $k \le n^{2/3}$. 
Similarly, for $t=\tc \pm \eps$ we often assume $k \le n^{1/3}/\eps$ since this implies $k \xi(\fS)  = k\psi(t) + O(1)$ in the branching process estimates of Theorem~\ref{asyfull-simple-pm}, see~\eqref{eq:xifS}. 
Furthermore, to take advantage of the fact that the tails decay exponentially in $k\psi(t) = \Theta(\eps^2k)$ for $t=\tc \pm \eps$, we typically also assume~${k \ge \eps^{-2}}$.
These constraints will not severely affect our later applications. 
For example, in Section~\ref{sec:proof} we exploit that when $\eps^3n \to \infty$ we can choose suitable $k = \omega(\eps^{-2}\log(\eps^3n))$ with $\eps^{-2} \ll k \ll \min\{n^{2/3},n^{1/3}/\eps\}$, i.e., which satisfies all the constraints (with room to spare).

For later reference we note the following simple summation result, which will be convenient in a number of technical estimates (see Lemma~\ref{lem:Sj:sum} for a further refinement). 
\begin{lemma}\label{lem:sum}
For all $u \in \RR$ with $u \neq -1$ there exists $C_u > 0$ such that for all $\delta > 0$ and $j_0 \ge 1$ we have
\begin{equation}\label{eq:sum:UB}
\sum_{j \ge j_0}j^{u} e^{-\delta j} \le C_u \bigl(1 + \delta^{-(u+1)}\bigr) e^{-\delta j_0/2}. 
\end{equation}
For all $u \in \RR$ with $u >0$ there exists $D_u > 0$ such that for all $\delta>0$ and $j_0 >0$ we have  
\begin{equation}\label{eq:sum:UB:j0}
\sum_{j \ge j_0}j^{-u} e^{-\delta j} \le D_u \delta^{-1} j_0^{-u} e^{-\delta j_0} . 
\end{equation}
\end{lemma}
\begin{proof}
Inequality~\eqref{eq:sum:UB} is immediate for $u < -1$, taking $C_u=\sum_{j\ge 1} j^u<\infty$.
For $u>-1$ it suffices to show that the sum of the terms $j^{u} e^{-\delta j/2}$ with $j \ge \delta^{-1}$ is at most a constant times the sum of these terms with $1 \le j \le \delta^{-1}$. This follows easily from the bounds $\int_{0}^{\delta^{-1}}x^u\dx = \Theta(\delta^{-(u+1)})$, $x^{u} e^{-\delta x/4} = O(\delta^{-u})$ and $\int_{\delta^{-1}}^{\infty}e^{-\delta x/4}\dx = \Theta(\delta^{-1})$. 
Similarly, inequality~\eqref{eq:sum:UB:j0} follows readily from $j^{-u} \le j_0^{-u}$ and $\int_{z}^{\infty}e^{-\delta x}\dx = O(\delta^{-1}e^{-\delta z})$. 
\end{proof}

\subsubsection{Number of vertices in large components}\label{sec:mom:large}

Our goal is to estimate the expectation and variance of the number $N_{\ge \Lambda}$ of vertices in `large' components of~$\JPpm$. 
We start with the subcritical case $i=(\tc -\eps) n$. 
Since the expectation $\E N_{\ge \Lambda}$ drops exponentially with rate $\psi(t) \Lambda =\Theta(\eps^2 \Lambda)$, where $t=\tc-\eps$, in~\eqref{eq:NkE:sub} below the leading constant is irrelevant for our purposes (with a little care, 
we can also obtain the precise asymptotics when $\eps^{-2} \ll \Lambda \ll \min\{n^{2/3},n^{1/3}/\eps\})$.
\begin{lemma}[Subcritical expectation of $N_{\ge \Lambda}$]
\label{lem:NkE:sub}
There exist constants $\eps_0,d,D,T > 0$ such that the following holds for all $t \in [t_0,t_1]$ with $\eps=\tc-t \in [T n^{-1/3},\eps_0]$, and all $t$-nice parameter lists~$\fS$.
Define~$\psi:[\tc-\eps_0,\tc+\eps_0] \to [0,\infty)$ 
as in Theorem~\ref{asyfull-simple}, $\fS_t^{\pm}$ as in Definition~\ref{def:cpl}, $\JPpm=\JP(\fS^{\pm}_t)$ as in Definition~\ref{def:F}, and $\bp_t^{\pm}=\bp_t^{\pm}(\fS)$ as in Definition~\ref{def:bp2pm}.
If $\max\{\eps^{-2},K\} < \Lambda \le \min\{n^{2/3}, n^{1/3}/\eps\}$, then
\begin{equation}\label{eq:NkE:sub}
d \eps^{-2} \Lambda^{-3/2}e^{-\psi(t)\Lambda}n - n^{-\omega(1)} \le \E N_{\ge \Lambda}(\JPpm ) \le D \eps^{-2} \Lambda^{-3/2}e^{-\psi(t)\Lambda}n .
\end{equation}
\end{lemma}
\begin{proof}
Since $t < \tc$, Theorem~\ref{thsurv-simple-pm} implies $\Pr(|\bp^{\pm}_t| = \infty )=0$. 
So, since $\Lambda \le n^{2/3}$, Theorem~\ref{thm:ENgekD} gives 
\[
 \Pr(\Lambda \le |\bp^{-}_t| < \infty ) |\fS^-_t| - n^{-\omega(1)} \le \E N_{\ge \Lambda}(\JPpm)  \le \Pr(\Lambda \le |\bp^{+}_t| < \infty ) |\fS^+_t| . 
\]
Since $|\fS^{\pm}_t| = \Theta(n)$ by~\eqref{def:n:pm:diff}, it remains to estimate $\Pr(\Lambda \le |\bp^{\pm}_t| < \infty )$. 

By Theorem~\ref{asyfull-simple-pm} and inequality~\eqref{eq:sum:UB:j0}, 
we have
\begin{equation}\label{eq:NkE:sub:UB}
\begin{split}
\Pr(\Lambda \le |\bp^{+}_t| < \infty ) & = \sum_{k \ge \Lambda}\Pr(|\bp^{+}_t| =k)= O\Bigl(\sum_{k\ge \Lambda}k^{-3/2} e^{-\xi(\fS_t^+)k} \Bigr) \\
& = O\bigl(\xi(\fS_t^+)^{-1} \Lambda^{-3/2} e^{-\xi(\fS_t^+)\Lambda}\bigr)  = O\bigl(\psi(t)^{-1} \Lambda^{-3/2} e^{-\psi(t)\Lambda}\bigr),
\end{split}
\end{equation}
using for the last step $\xi(\fS_t^+)\ge \psi(t)/2$ (see Remark~\ref{rem:bpsimple}) and $\Lambda|\xi(\fS_t^+)-\psi(t)|=O(1)$, 
which follows from~\eqref{eq:xifS} and~$\Lambda \le n^{1/3}/\eps$. 
Since $\psi(t) = \Theta(\eps^2)$, this establishes the upper bound in~\eqref{eq:NkE:sub}. 

For the lower bound, we pick $\Lambda \le \Lambda' < \Lambda+\per$ such that $\Lambda' \modp$. 
Applying Theorem~\ref{asyfull-simple-pm} similarly to~\eqref{eq:NkE:sub:UB}, using $\Lambda' \ge \Lambda > K$, $\Lambda'|\xi(\fS_t^-)-\psi(t)|=O(1)$, $\int_{y}^{z}e^{-a x}\dx = a^{-1}e^{-a y}(1-e^{-a(z-y)})$ and $\Lambda' \psi(t) =\Theta(\Lambda \eps^2)=\Omega(1)$ it follows that 
\begin{equation*}
\begin{split}
\Pr(\Lambda \le |\bp^{-}_t| < \infty ) & \ge \sum_{\Lambda ' \le k \le 2\per \Lambda'}\Pr(|\bp^{-}_t| =k) = \Omega\Bigl(\sum_{\Lambda ' \le k \le 2\per \Lambda'}\indic{k \modp}k^{-3/2} e^{-\xi(\fS_t^-)k}\Bigr) \\
& = \Omega\Bigl(\Lambda^{-3/2} \sum_{\Lambda'/\per \le j \le 2\Lambda'}e^{-\psi(t)\per j }\Bigr) = \Omega\bigl(\psi(t)^{-1} \Lambda^{-3/2} e^{-\psi(t)\Lambda'}\bigr). 
\end{split}
\end{equation*}
This establishes the lower bound in~\eqref{eq:NkE:sub} since $\psi(t) = \Theta(\eps^2)$ and $|\Lambda'-\Lambda| = O(1)$. 
\end{proof}

We now turn to the more interesting supercritical case $i=(\tc +\eps) n$ (here our estimates are tailored for our goal of proving concentration in every step, see Section~\ref{sec:L1:super}; otherwise simpler bounds would suffice). 
Recall from Theorem~\ref{thsurv-simple} that $\Pr(|\bp_t|=\infty)=\Theta(\eps)$. 
Assuming $\Lambda = \omega(\eps^{-2})$ and $\eps^3n\to\infty$, the right hand side of~\eqref{eq:NkE:sup} is $o(\eps n)$, so the result below implies $\E N_{\ge \Lambda}(\JPpm) \sim \Pr(|\bp_t|=\infty)n = \Theta(\eps n)$. 
Under the same assumptions we also have small variance, since then $\Var N_{\ge \Lambda}(\JPpm) = o((\eps n)^2)$ by~\eqref{eq:Nkvar:sup2}.  
\begin{lemma}[Supercritical expectation and variance of $N_{\ge \Lambda}$]
\label{lem:NkE:sup} 
There exist constants $\eps_0,d_1,D,T > 0$ such such that the following holds for all $t \in [t_0,t_1]$ with $\eps=t-\tc \in (T n^{-1/3},\eps_0]$, and all $t$-nice parameter lists~$\fS$.  
Define~$\fS_t^{\pm}$ as in Definition~\ref{def:cpl}, $\JPpm=\JP(\fS^{\pm}_t)$ as in Definition~\ref{def:F}, and $\bp_t$ as in Section~\ref{sec:BPI}.  
If~${\max\{\eps^{-2},K\} < \Lambda \le \min\{n^{2/3}, n^{1/3}/\eps\}}$, then 
\begin{align}
\label{eq:NkE:sup}
 \bigl|\E N_{\ge \Lambda}(\JPpm)-\Pr(|\bp_t|=\infty)n\bigr| & \le D\eps n\bigl(e^{-d_1\eps^2 \Lambda} + (\eps^3 n)^{-1/3}\bigr) ,\\
\label{eq:Nkvar:sup2}
 \Var N_{\ge \Lambda}(\JPpm) & \le D (\eps n)^2\bigl(e^{-d_1\eps^2 \Lambda}+ (\eps^3n)^{-1/3}\bigr) .
\end{align}
\end{lemma}
\begin{proof}
By Theorems~\ref{thsurv-simple-pm} and~\ref{asyfull-simple-pm}
and Remark~\ref{rem:bpsimple} there is a constant $d_1>0$ such that $\Pr(|\bp^{1,\pm}_t(\fS)|=\infty) \ge \Pr(|\bp^{1}_t|=\infty)/2 \ge d_1 \eps$ and $\xi(\fS_t^\pm) \ge \psi(t)/2 \ge d_1 \eps^2$. 

We first focus on $\E N_{\ge \Lambda}(\JPpm)$. 
Analogous to the proof of Lemma~\ref{lem:NkE:sub}, using Theorem~\ref{thm:ENgekD} we readily~obtain
\begin{equation}\label{eq:NkE:sup:2}
\Pr(|\bp^{-}_t| = \infty)|\fS^-_t|-n^{-\omega(1)} \le \E N_{\ge \Lambda}(\JPpm) \le \Pr(|\bp^{+}_t| \ge \Lambda)|\fS^+_t| .
\end{equation}
Proceeding similarly to~\eqref{eq:NkE:sub:UB}, using inequality~\eqref{eq:sum:UB:j0} 
together with $\xi(\fS_t^+) \ge d_1 \eps^2$ and $\Lambda^{-3/2} \le \eps^3$, it follows~that 
\begin{equation}\label{eq:NkE:sup:4}
\Pr(\Lambda \le |\bp^{+}_t| < \infty) 
= O\Bigl(\sum_{k \ge \Lambda}k^{-3/2} e^{-\xi(\fS_t^+) k}\Bigr) 
= O\bigl(\eps^{-2}\Lambda^{-3/2} e^{-d_1\eps^2 \Lambda}\bigr) 
= O\bigl(\eps e^{-d_1\eps^2 \Lambda}\bigr) . 
\end{equation}
Note that $|\fS^\pm_t| = n(1+o(n^{-1/3}))$ by~\eqref{def:n:pm:diff}. 
By Theorem~\ref{thsurv-simple-pm} we have $\bigl|\Pr(|\bp^\pm_t|=\infty) - \Pr(|\bp_t|=\infty)\bigr|  \le Cn^{-1/3}$; this
and~\eqref{eq:NkE:sup:2}--\eqref{eq:NkE:sup:4} imply~\eqref{eq:NkE:sup} for suitable $D>0$. 

We now turn to the variance of $X:= N_{\ge \Lambda}(\JPp)$. 
Here the second moment estimate of Lemma~\ref{lem:Nkvar:sup} will be key, which involves the two auxiliary parameters~$\nu$ and~$\tau$. 
Analogous to~\eqref{eq:NkE:sup:4}, using inequality~\eqref{eq:sum:UB} 
together with $\xi(\fS_t^+) \ge d_1 \eps^2$, we obtain 
\begin{equation*}
 \nu = \E( |\bp^+_{t}| \indic{|\bp^+_{t}| < \infty}) = \sum_{k\ge 1} k \Pr(|\bp^+_{t}|=k) = O\Bigl(\sum_{k \ge 1} k^{-1/2} e^{-\xi(\fS_t^+) k}\Bigr)
  = O(\eps^{-1}) .
\end{equation*}
Using inequality~\eqref{eq:Nkvar:sup:sigma} for $\tau$, noting $\Pr(|\bp^{1,\pm}_t(\fS)|=\infty) \ge d_1 \eps$ and $\eps^2 \Lambda \ge 1$, we infer 
\[
\tau \le (1+O(e^{-d_1 \eps^2 \Lambda})) \cdot \Pr(|\bp^+_{t}| \ge \Lambda) .
\]
We now estimate $\E X^2$ by bounding each term on the right hand side of~\eqref{eq:Nkvar:sup1} from Lemma~\ref{lem:Nkvar:sup}.
Using $\Pr(|\bp^{+}_t| \ge \Lambda)| = \Theta(\eps)$ and $(\eps n)^{-1} \cdot \eps^{-2} = (\eps^3 n)^{-1}$
we readily see that 
\[
\nu n^{-1} = O((\eps n)^{-1}) = O((\eps n)^{-1}) \cdot \eps^{-2} \Pr(|\bp^{+}_t| \ge \Lambda)^2 = O((\eps^3n)^{-1}) \cdot \Pr(|\bp^{+}_t| \ge \Lambda)^2 .
\]
Noting $n^{-1/3} \cdot \eps^{-1} = (\eps^3 n)^{-1/3}$, we similarly see that 
\[
\tau \eps \nu n^{-1/3} + n^{-1/3} \Pr(|\bp^{+}_t| = \infty) = O(n^{-1/3}) \cdot \Pr(|\bp^{+}_t| \ge \Lambda) = O((\eps^3 n)^{-1/3}) \cdot \Pr(|\bp^{+}_t| \ge \Lambda)^2. 
\]
Using~\eqref{eq:NkE:sup:2} for the first step and then $\Pr(|\bp^{+}_t| \ge \Lambda)|\fS^+_t| = \Theta(\eps n)$, we also obtain 
\[
(\log n)^2\cdot \E X \le n^{2/3} \cdot \Pr(|\bp^{+}_t| \ge \Lambda)|\fS^+_t| = O((\eps^3n)^{-1/3}) \cdot \bigl[\Pr(|\bp^{+}_t| \ge \Lambda)|\fS^+_t|\bigr]^2.
\]
Substituting the above estimates into~\eqref{eq:Nkvar:sup1}, 
using $\eps^3 n \ge 1$ (which follows from the assumption $\eps^{-2} < \Lambda \le n^{1/3}/\eps$)  
we obtain
\begin{equation*}
 \E X^2  \le \Bigl[\bigl(1+O\bigl(e^{-d_1 \eps^2 \Lambda} + (\eps^3n)^{-1/3}\bigr)\bigr) \cdot \Pr(|\bp^+_{t}| \ge \Lambda) |\fS^+_t| \Bigr]^2 .
\end{equation*}
Recall that $\bigl|\Pr(|\bp^+_t|=\infty) - \Pr(|\bp_t|=\infty)\bigr| = O(n^{-1/3})$ and $|\fS^\pm_t| = n(1+o(n^{-1/3}))$. 
Using \eqref{eq:NkE:sup:4} and $n^{-1/3} = O((\eps^3n)^{-1/3})$ it follows that 
\begin{equation*}
 \E X^2  \le \Bigl[\bigl(1+O\bigl(e^{-d_1 \eps^2 \Lambda} + (\eps^3n)^{-1/3}\bigr)\bigr) \cdot \Pr(|\bp_{t}| =\infty)n \Bigr]^2 .
\end{equation*}
Estimating $\E X = \E N_{\ge \Lambda}(\JPp)$ by~\eqref{eq:NkE:sup} above, using $\Pr(|\bp_{t}| =\infty) = \Theta(\eps)$ and $\Var X = \E X^2 - (\E X)^2$ now inequality~\eqref{eq:Nkvar:sup2} 
follows for~$\JPp=\JP(\fS^+_t)$, increasing the constant~$D$ if necessary.

It remains to bound the variance of $\tilde{X} := N_{\ge \Lambda}(\JPm)$. 
Noting $\fS^-_t \preceq\fS^+_t$, using~\eqref{eq:F:mon} we infer $\E \tilde{X}^2 \le \E X^2$ and thus $\Var \tilde{X} \le \E X^2-(\E \tilde{X})^2$. 
So, since~\eqref{eq:NkE:sup} yields the same qualitative estimates for $\E \tilde{X}$ and $\E X$, inequality~\eqref{eq:Nkvar:sup2} for $\JPm=\JP(\fS^-_t)$ follows analogously to our above estimates for~$\Var X$. 
\end{proof}

\subsubsection{Susceptibility}\label{sec:mom:sus} 
We now turn to the susceptibility in the subcritical case $i=(\tc-\eps)n$. 
Our goal is to approximate the expectation and variance of the (modified $r$th order) susceptibility $S_{r,n}(\JP)=\sum_{k \ge 1}k^{r-1}N_{k}(\JP)/n$ defined in~\eqref{eq:def:SrG}, 
exploiting that we have good control over $\E N_k(\JP) \approx \Pr(|\bp^{\pm}_t| = k)$. 
Similar to~\eqref{eq:po:Sr:E} and~\eqref{eq:sum:UB}, in view of~\eqref{eq:def:NjT:E} and $\psi(\tc-\eps)=\Theta(\eps^2)$ we expect for $r \ge 2$ that
\begin{equation}\label{eq:Sj:sub:heur:0}
\E S_{r,n}(\JP) \approx \sum_{k \ge 1}k^{r-1}\Pr(|\bp^{\pm}_t| = k) = \E |\bp^{\pm}_t|^{r-1} \approx \sum_{k \ge 1} \Theta(k^{r-5/2}) e^{-\psi(\tc-\eps) k} \approx \Theta(\eps^{-2r+3}) . 
\end{equation}
We shall obtain a sharper estimate by comparing the above sum with an integral. 
To avoid clutter, in~\eqref{eq:Sj:sum} below we use the convention that $x!!=\prod_{0 \le j < \ceil{x/2}}(x-2j)$ is equal to $1$ when $x=-1$.
\begin{lemma}\label{lem:Sj:sum}
For all $r \in \NN$ with $r \ge 2$ there exists $C_r > 0$ such that for all $\delta>0$ we have
\begin{equation}\label{eq:Sj:sum}
\biggl|\sum_{j \ge 1}j^{r-5/2} e^{-\delta j} - \frac{(2r-5)!! \sqrt{2\pi}}{(2\delta)^{r-3/2}}\biggr| \le C_r \bigl(1+\delta^{-(r-5/2)}\bigr) . 
\end{equation}
\end{lemma}
\begin{proof}
The basic idea is to compare the sum in~\eqref{eq:Sj:sum} with the integral
\begin{equation*}
f(r) := \int_{0}^{\infty} x^{r-5/2}e^{-\delta j} \dx .
\end{equation*}
Let $g(x) := x^{r-5/2}e^{-\delta x}$. For $r = 2$ the function $g(x)$ is monotone decreasing, and for $r \ge 3$ there is $x_\delta = \Theta(\delta^{-1})$ such that $g(x)$ is increasing for $x \le x_\delta$ and decreasing for $x \ge x_{\delta}$.
It follows that 
\begin{equation*}
\Bigl|\sum_{j \ge 1}j^{r-5/2} e^{-\delta j} - f(r)\Bigr| \le O(1) + \indic{r \ge 3} O(\delta^{-(r-5/2)}) .
\end{equation*}
It remains to evaluate the integral $f(r)$; this is basic calculus.
For $r=2$ the substitution $y^2=\delta x$ allows us to determine $f(2)$ via the Gauss error function:
\begin{equation*}
f(2) = \int_{0}^{\infty} x^{-1/2}e^{-\delta x} \dx = \sqrt{\frac{\pi}{\delta}} \cdot \frac{2}{\sqrt{\pi}}\int_{0}^{\infty} e^{-y^2} \dy = \sqrt{\frac{\pi}{\delta}} .
\end{equation*}
For $r \ge 3$ we use integration by parts to infer 
\begin{equation*}
f(r) = -\frac{x^{r-5/2}e^{-\delta x}}{\delta} \bigg|_0^\infty + \frac{(r-5/2)}{\delta} \cdot \int_{0}^{\infty} x^{r-7/2} e^{-\delta x} \dx =  \frac{(2r-5) f(r-1)}{2\delta} .
\end{equation*}
Solving the above recurrence for $r \ge 2$ completes the proof.
\end{proof}
As a step towards making~\eqref{eq:Sj:sub:heur:0} rigorous, we now estimate the `idealized' moments~$\E |\bp_{\tc-\eps}|^{r-1}$. 
\begin{lemma}\label{lem:SjE:bp}
For $\theta(t)$ and $\psi(t)$ as defined in Theorem~\ref{asyfull-simple}, let 
\begin{equation}\label{eq:SjE:Br}
B_r := \frac{(2r-5)!! \sqrt{2\pi}\theta(\tc)}{\per { [\psi''(\tc)]^{r-3/2}}} .
\end{equation}
Then $B_r > 0$ for $r \ge 2$. Furthermore, there exists $\eps_0>0$ such
that, for all $r \ge 2$ and $\eps \in (0,\eps_0)$, 
\begin{equation}\label{eq:SrE:bp}
\E |\bp_{\tc-\eps}|^{r-1} = (1+O(\eps))B_r\eps^{-2r+3}. 
\end{equation}
\end{lemma}
\begin{proof}
For brevity, let $t:=\tc-\eps$. 
Theorem~\ref{thsurv-simple} gives $\Pr(|\bp_t| = \infty)=0$, so Theorem~\ref{asyfull-simple} and Lemma~\ref{lem:allowed}~imply
\begin{equation}\label{eq:lem:SjE:bp:0}
\E |\bp_{t}|^{r-1} = \sum_{k \ge 1}k^{r-1}\Pr(|\bp_t| = k) = \sum_{k \ge \kR} \indic{k \modp} (1+O(1/k))k^{r-5/2} \theta(t) e^{-\psi(t)k} + O(1),
\end{equation}
where the implicit constants are independent of~$\eps$ and~$k$.
Taking the parity constraint into account, it follows~that 
\begin{equation}\label{eq:lem:SjE:bp:1}
\E |\bp_{t}|^{r-1} = \sum_{j \ge 1} (\per j)^{r-5/2} \theta(t) e^{-\psi(t)\per j} + O\Bigl(\sum_{j \ge 1} j^{r-7/2} e^{-\psi(t)\per j}\Bigr) + O(1).
\end{equation}
Estimating the first sum by~\eqref{eq:Sj:sum}, and the second sum by~\eqref{eq:sum:UB}, it follows that
\begin{equation}\label{eq:lem:SjE:bp:2}
\E |\bp_{t}|^{r-1} = \per^{r-5/2}\theta(t) \cdot \frac{(2r-5)!! \sqrt{2\pi}}{[2 \psi(t)\per]^{r-3/2}} + O\bigl(\psi(t)^{-(r-5/2)}\bigr) + O(1).
\end{equation}
Recalling $t=\tc-\eps$, note that $\psi(\tc-\eps) = \psi''(\tc)\eps^2/2 + O(\eps^3)$ and $\theta(\tc-\eps) = \theta(\tc)+O(\eps)$.
Since $\psi''(\tc),\theta(\tc)>0$, it follows that $\theta(t) = (1+O(\eps)) \theta(\tc)$, $[2 \psi(t)]^{r-3/2} = (1+O(\eps))[\eps^2 \psi''(\tc)]^{r-3/2}$ and $\psi(t)=\Theta(\eps^2)$. 
This completes the proof of~\eqref{eq:SrE:bp} since $O((\eps^2)^{-(r-5/2)}) + O(1) = O(\eps) \cdot \eps^{-2r+3}$.
\end{proof}
Note that, combined with Corollary~\ref{cor:rhosr}, which gives~$s_r(t) = \E |\bp_t|^{r-1}$ for $t \in [t_0,\tc)$, Lemma~\ref{lem:SjE:bp}
implies Theorem~\ref{thm:sjfkt}.
Mimicking the above calculations and using Lemma~\ref{lem:Srvar:sub}, we now approximate the expectation and variance of $S_{r,n}(\JP)$. 
Note that~\eqref{eq:SrE:sub} below yields $\E S_{r,n}(\JP) \sim B_r\eps^{-2r+3}$ whenever $\eps\to 0$ and $\eps^3 n \to \infty$. 
Furthermore, \eqref{eq:SrVar:sub} shows that we have small variance whenever~$\eps^3 n \to \infty$. 
\begin{lemma}[Subcritical expectation and variance of $S_{r,n}$]
\label{lem:SrE:sub}
There exist positive constants
$c,T,\eps_0 > 0$ and $(a_r,b_r)_{r \ge 2}$ such such that the following holds for all $t \in [t_0,t_1]$ with $\eps=\tc-t \in [Tn^{-1/3},\eps_0]$, and all $t$-nice parameter lists $\fS$.
Define~$\fS_t^{\pm}$ as in Definition~\ref{def:cpl}, and $\JPpm=\JP(\fS^{\pm}_t)$ as in Definition~\ref{def:F}. 
If $r \ge 2$ and $\eps^3 n \ge c$, then 
\begin{align}
\label{eq:SrE:sub}
|\E S_{r,n}(\JPpm)-B_r\eps^{-2r+3}| &\le a_r \bigl(\eps + (\eps^3 n)^{-1/3}\bigr) \eps^{-2r+3} ,\\
\label{eq:SrVar:sub}
\Var S_{r,n}(\JPpm) &\le b_r (\eps^3 n)^{-1}\bigl(\E S_{r,n}(\JPpm)\bigr)^2 ,
\end{align}
where $B_r > 0$ is defined as in~\eqref{eq:SjE:Br}.
\end{lemma}
\begin{proof} 
Let $\Lambda := \eps^{-2}(\log \eps^3 n)^2$, which satisfies $2\max\{\eps^{-2},K\} < \Lambda \le n^{2/3}$ for $\eps^3 n$ large enough. 
Aiming at (monotone) coupling arguments, note that 
\begin{equation}\label{eq:def:SrG:2}
S_{r,n}(G) = \sum_{k \ge 1} k^{r-1} N_k(G)/n = \sum_{k \ge 1} \bigl[k^{r-1} - (k-1)^{r-1}\bigr] N_{\ge k}(G)/n .
\end{equation}
Recall that $|\fS^{\pm}_t|=(1+o(n^{-1/3})) n$ by~\eqref{def:n:pm:diff}. 
Estimating $\E N_{\ge k}(\JPpm)$ via Theorem~\ref{thm:ENgekD}, and noting that $\Pr(|\bp^{\pm}_t| = \infty)=0$ by Theorem~\ref{thsurv-simple-pm}, 
using $\Lambda \le n^{2/3}$ and $\E |\bp^{\pm}_t|^{r-1} = \sum_{k \ge 1} k^{r-1}\Pr(|\bp^{\pm}_t| = k)$ 
it follows that 
\begin{equation}\label{eq:SrE:sub:1}
\begin{split}
\E S_{r,n}(\JPpm) =  (1 +o(n^{-1/3})) \E |\bp^{\pm}_t|^{r-1}  + O\Bigl(\sum_{k \ge \Lambda} k^{r-1}\Pr(|\bp^{-}_t| = k)+ n^{-\omega(1)}\Bigr). 
\end{split}
\end{equation}

Proceeding analogously to~\eqref{eq:lem:SjE:bp:0}--\eqref{eq:lem:SjE:bp:2}, 
replacing the $\Pr(|\bp_t|=k)$ estimate of Theorem~\ref{asyfull-simple} with the $\Pr(|\bp^\pm_t|=k)$ estimate of Theorem~\ref{asyfull-simple-pm},  
it follows that 
\[
\E |\bp^{\pm}_t|^{r-1} = \bigl(1+O(n^{-1/3})\bigr) \cdot \per^{r-5/2}\theta(t) \frac{(2r-5)!! \sqrt{2\pi}}{[2 \xi(\fS^\pm_t)\per]^{r-3/2}} + O\bigl(\xi(\fS^\pm_t)^{-(r-5/2)}\bigr) + O(1) .
\]
Since $|\xi(\fS^\pm_t)-\psi(t)| = O(n^{-1/3} \eps)$ and $\psi(t) = \Theta(\eps^2)$, see Theorem~\ref{asyfull-simple-pm} and Remark~\ref{rem:bpsimple}, 
we have $\xi(\fS^\pm_t) = (1+O((\eps^3n)^{-1/3})) \cdot \psi(t)$. 
So, by proceeding analogously to the deduction of~\eqref{eq:SrE:bp} from~\eqref{eq:lem:SjE:bp:2}, using $\eps = \Omega(n^{-1/3})$ it follows that  
\begin{equation}\label{eq:SrE:sub:main}
\E |\bp^{\pm}_t|^{r-1} = \bigl(1+O(\eps) + O((\eps^3n)^{-1/3}) \bigr) \cdot B_r\eps^{-2r+3} .
\end{equation}

Estimating $\Pr(|\bp^{-}_t| = k)$ via Theorem~\ref{asyfull-simple} and recalling that $\xi(\fS^-_t) \ge \psi(t)/2 = \Theta(\eps^2)$ by Remark~\ref{rem:bpsimple}, 
using $\xi(\fS^-_t)\Lambda \ge 2\log(\eps^3 n)$ (for $\eps^3 n$ large enough) and inequality~\eqref{eq:sum:UB},
it follows that
\begin{equation*}
\sum_{k \ge \Lambda} k^{r-1}\Pr(|\bp^{-}_t| = k) = O\Bigl(\sum_{k \ge \Lambda} k^{r-5/2} e^{-\xi(\fS^-_t) k}\Bigr) 
= O(\eps^{-2r+3}) \cdot e^{-\xi(\fS^-_t) \Lambda/2}  = O((\eps^3 n)^{-1}) \cdot \eps^{-2r+3}.
\end{equation*}
Together with~\eqref{eq:SrE:sub:1} and~\eqref{eq:SrE:sub:main} this completes the proof of~\eqref{eq:SrE:sub}. 

A much simpler variant of the above calculations 
(using ${\E S_{r,n}(\JPpm)} \ge {\E S_{r,n}(\JPm)}= {\Omega(\sum_{1 \le k \le \Lambda}k^{r-1}\Pr(|\bp^{-}_t| = k))}$
and ${\sum_{1 \le k \le \eps^{-2}} k^{r-5/2}} = {\Omega(\eps^{-2r+3})}$, say) yields the crude lower bound
${\E S_{r,n}(\JPpm)} = {\Omega(\eps^{-2r+3})}$.
(This also follows directly from~\eqref{eq:SrE:sub} if we allow ourselves to impose an upper
bound on $\eps$ that depends on $r$.)
This lower bound, and the upper bound in~\eqref{eq:SrE:sub} (applied with $2r$ in place of $r$)
imply that ${\E S_{2r,n}(\JPpm)} \le {b_r \eps^{-3} \bigl(\E S_{r,n}(\JPpm)\bigr)^2}$.  
By Lemma~\ref{lem:Srvar:sub} we have $\Var S_{r,n}(\JPpm) \le n^{-1} \E S_{2r,n}(\JPpm)$, so~\eqref{eq:SrVar:sub} follows.
\end{proof}

\newoddpage

\section{Proofs of the main results}\label{sec:proof}
In this section we prove our main results for the size of the largest component, the number of vertices in small components, and the susceptibility. 
As discussed in the proof outline of Section~\ref{sec:po:main}, we shall establish these by adapting Erd{\H o}s--R{\'e}nyi proof strategies to the Achlioptas process setting, exploiting the setup and technical work of Sections~\ref{sec:setup}--\ref{sec:cpl2} 
(which are, of course, the meat of the proof). 
One non-standard detail is that we study~$G_i=G_{n,i}^{\cR}$
via the auxiliary random graphs~$J_i$ and~$\JPpm_i$ (see Lemmas~\ref{lem:cond}, \ref{lem:cond'} and~\ref{lem:cpl}),
using the following two key facts: (i)~that $J_i=J(\fS_i)$ has the same component size distribution as $G_i$ conditioned on the parameter list~$\fS_i$ defined in~\eqref{def:fS}, and (ii)~that we can whp sandwich~$J_i$ between two `Poissonized' random graphs~$\JPpm_i$, i.e., $\JPm_i \subseteq J_i \subseteq \JPp_i$ with $\JPpm_i=\JP(\fS^{\pm}_t)$, where $t=i/n$ and $\fS^{\pm}_t$ is as in Definition~\ref{def:cpl}.
For technical reasons our arguments require that the parameter list~$\fS_i$ is $t$-nice in the sense of Definition~\ref{def:nice}, which by Lemma~\ref{lem:nice} fails with probability at most $\Pr(\neg\cN) = O(n^{-99})$.

In Section~\ref{sec:small} we focus on the number of vertices in small components, and prove Theorem~\ref{thm:Nk}. 
In Section~\ref{sec:L1} we turn to the size of the largest component, and prove Theorems~\ref{thm:L1sub} and~\ref{thm:L1}. 
Finally, in Section~\ref{sec:Sj} we consider the susceptibility, and prove Theorem~\ref{thm:Sj}.
Note that Theorems~\ref{thm:L1rho}, \ref{thm:rhok} and~\ref{thm:sjfkt} have already been proved in Sections~\ref{sec:cpl}--\ref{sec:cpl2}; 
indeed, as noted there, in the light of Corollaries~\ref{cor:Nk}--\ref{cor:rhosr}, these results are immediate from Theorems~\ref{thsurv-simple}--\ref{asyfull-simple} 
and Lemma~\ref{lem:SjE:bp}, respectively.

\subsection{Small components}\label{sec:small} 
In this subsection we prove Theorem~\ref{thm:Nk}, i.e., estimate the number $N_k(i)$ of vertices in components of size~$k$ after~$i$ steps.  
Our arguments use the following three ideas:
(i)~that the random variable $N_k(i)$ is typically close to its expected value,  
(ii)~that we can approximate $\E N_k(i)$ using the `idealized' branching process~$\bp_{t}$, $t=i/n$, and
(iii)~that we have detailed results for the point probabilities of~$\bp_{t}$.

We start with a conditional concentration result.
The key observation is that, for any graph, adding or deleting an edge changes the number of vertices in components of size $k$ (at least $k$) by at most $2k$. 
\begin{lemma}\label{lem:Nk:C} 
Let $t\in [t_0,t_1]$, and let~$\fS$ be a $t$-nice parameter list, and define~$J=J(\fS)$ as in Definition~\ref{def:J}.
Then, with probability at least $1-n^{-\omega(1)}$, we have $|N_k(J)-\E N_k(J)| \le k (\log n)n^{1/2}$ and $|N_{\ge k}(J) - \E N_{\ge k}(J)| \le k (\log n)n^{1/2}$ for all $1 \le k \le n$.  
\end{lemma}
\begin{proof}
By definition of $J=J(\fS)$ and \eqref{eq:Qkr:tail'}, the probability space $\Omega=\Omega(\fS)$ on which the random graph~$J(\fS)$ is defined consists of 
\[
M:=\sum_{k,r \ge 0} rQ_{k,r}(\fS)\le \sum_{k,r \ge 0} r Be^{-b(k+r)}n =O(n)
\] 
independent random variables, each corresponding to the uniform choice of a random vertex from $V_L$.
Furthermore, changing the outcome of one variable can be understood as (i)~first removing one edge and (ii)~then adding one edge. 
From the observation before the lemma, it follows that $|N_{k}(J)(\omega_1)-N_{k}(J)(\omega_2)| \le 4k$ whenever $\omega_1,\omega_2 \in \Omega$ differ in the outcome of a single random variable. 
An analogous remark applies to~$N_{\ge k}(J)$. 
By McDiarmid's bounded-differences inequality~\cite{McDiarmid1989} we thus have 
\[
\Pr\bb{|N_k(J)-\E N_k(J)| \ge k (\log n)n^{1/2}} \le 2 \exp\left(-\frac{2\bigl(k(\log n)n^{1/2}\bigr)^2}{M (4k)^2}\right) = n^{-\omega(1)} .
\]
An analogous bound holds for $N_{\ge k}(J)$, and a union bound over all $1 \le k \le n$ completes the proof. 
\end{proof}
\begin{theorem}[Concentration of $N_k$ and $N_{\ge k}$]
\label{thm:NkC}
Set $D_{\bp} := D_{\cT}+2$, where $D_{\cT}>0$ is as in Theorem~\ref{thm:cpl:btT},
and define $\bp_{t}$ as in Section~\ref{sec:BPI}. 
Then, with probability at least $1-O(n^{-99})$, the following hold for all $i_0\le i\le i_1$ and all $k \ge 1$:
\begin{align}
\label{eq:NkC}
 \bigl|N_{k}(i)-\Pr(|\bp_{i/n}|=k) n\bigr|  &\le k (\log n)^{D_{\bp}}n^{1/2},\\
\label{eq:NgekC}
 \bigl|N_{\ge k}(i)-\Pr(|\bp_{i/n}| \ge k) n\bigr|  &\le k (\log n)^{D_{\bp}}n^{1/2}.
\end{align}
\end{theorem}
We have not tried to optimize the $k(\log n)^{D_{\bp}} n^{1/2}$ error term in~\eqref{eq:NkC}--\eqref{eq:NgekC}, which suffices for our purposes.
\begin{proof}
For $k>n$ the statement is trivial since $0 \le N_{k}(i),N_{\ge k}(i) \le n$.
For $i_0\le i\le i_1$, let $\cE_i$ be the event that \eqref{eq:NkC}--\eqref{eq:NgekC} hold for all $1 \le k \le n$, so
our goal is to estimate the probability that $\cE= \bigcap_{i_0 \le i \le i_1} \cE_i$ fails.
Given a parameter list $\fS$, let $\cE_i(\fS)$ be the that event \eqref{eq:NkC}--\eqref{eq:NgekC} hold for all $1 \le k \le n$, with $N_{k}(i)$ and~${N_{\ge k}(i)}$ replaced by~${N_{k}(J(\fS))}$ and~$N_{\ge k}(J(\fS))$.
Let $\fS_i$ be the random parameter list defined in \eqref{def:fS}, and let $\cN_i$ be the event that $\fS_i$ is $(i/n)$-nice.
By Lemma~\ref{lem:cond'} we have
\begin{equation}\label{eq:transfer}
\Pr(\neg\cE_i \text{ and } \cN_i) 
\le  \max_{\text{$(i/n)$-nice $\fS$}}\Pr(\neg\cE_i(\fS)) .
\end{equation}
Since $D_{\bp} > \max\{D_{\cT},1\}$, Lemma~\ref{lem:Nk:C} and Theorem~\ref{thm:Nk:E}
give $\Pr(\neg\cE_i(\fS)) \le n^{-\omega(1)}$ when $\fS$ is $(i/n)$-nice.
This, Lemma~\ref{lem:nice}, and a union bound over the $O(n)$ values of $i$ completes the proof.
\end{proof}
\begin{remark}
By combining the concentration result Theorem~\ref{thm:NkC} with the branching process results from Section~\ref{sec:bpresults} (see Theorems~\ref{thsurv-simple}--\ref{asyfull-simple}) and sprinkling (see Section~\ref{sec:SP}), we can easily prove that whp ${L_1(\tc n+\eps n)} \sim {\Pr(|\bp_{\tc+\eps}|=\infty)n}$ for $n^{-1/6+o(1)} \le \eps \le \eps_0$, say.
In Section~\ref{sec:L1:super} we shall use a more involved second moment argument to relax this assumption to the optimal condition $n^{-1/3} \ll \eps \le \eps_0$ (exploiting the sandwiching and domination arguments of Sections~\ref{sec:dom} and~\ref{sec:mom}).
\end{remark}
We now combine the concentration result above with the branching process results from Section~\ref{sec:bpresults}.
\begin{proof}[Proof of Theorem~\ref{thm:Nk}]
Let $\beta:=1/10$ and $\tau := n^{-\beta/3}$.
By Theorem~\ref{asyfull-simple} and Corollary~\ref{cor:Nk}, we have
\begin{equation}\label{eq:rhokasy}
\rho_k(t) =  \Pr(|\bp_t|=k) = (1+O(1/k)) \indic{k \in \cSR} k^{-3/2} \theta(t) e^{-\psi(t) k} 
\end{equation}
uniformly over all $k\ge 1$ and $t\in [\tc-\eps_0,\tc+\eps_0]$, where the functions $\theta$ and $\psi$ are analytic. 
As discussed in Section~\ref{sec:period}, if $k \not\in \cSR$, then $N_k(i)=0$ holds with probability one for all $i \ge 0$. 
Due to the indicator in the above bound on~$\rho_k(t)$, 
it follows that~\eqref{eq:Nk} holds trivially for~$k \not\in \cSR$.

We now focus on the main case $k \in \cSR$. 
Aiming at comparing the additive errors in Theorem~\ref{thm:NkC} with the above bound on~$\rho_k(t)$, 
note that by Theorem~\ref{asyfull-simple} we have $a:=\psi''(\tc)>0$, $\psi(\tc)=\psi'(\tc)=0$ and $b:=\theta(\tc)>0$.  
Hence, after decreasing~$\eps_0$ if necessary, we crudely have $\psi(t) \le a\eps^2$ and $\theta(t) \ge b/2>0$ for $t\in [\tc-\eps_0,\tc+\eps_0]$. 
When $1 \le k \le n^{\beta}$ and $a\eps^2k\le\beta\log n$, we thus have 
\[
 \frac{k n^{1/2}}{k^{-5/2}e^{-\psi(t)k}n} \le k^{7/2}e^{a\eps^2 k}n^{-1/2} \le n^{9\beta/2-1/2} = o(n^{-\beta/3-1/1000}) ,
\]
say.
Let $K_0$ be a constant such that for all $k\ge K_0$ the $O(1/k)$ error term in~\eqref{eq:rhokasy} is at most $0.1$ in magnitude.
Since $\theta(t) \ge b/2>0$, it follows that $k (\log n)^{D_{\bp}}n^{1/2} = o(\tau/k) \cdot \rho_k(t)n$ for $k \in \cSR \setminus [K_0]$ and $t\in [\tc-\eps_0,\tc+\eps_0]$. 
Using Theorem~\ref{thm:NkC}, this establishes~\eqref{eq:Nk} for $k \in \cSR \setminus [K_0]$. 

We now turn to the remaining case $k \in \cK := \cSR \cap [K_0]$ of~\eqref{eq:Nk}. 
By Lemma~\ref{lem:rhokS1}, for each $k\in \cK$ we have $\rho_k(\tc)>0$. 
Since each $\rho_k(t)$ is continuous, after decreasing~$\eps_0$ if necessary,
there is some constant $c>0$ such that $\rho_k(t)\ge c$ for all $k\in \cK$ and $t\in[\tc-\eps_0,\tc+\eps_0]$.  
Hence $k (\log n)^{D_{\bp}}n^{1/2} = o(\tau/k) \cdot \rho_k(t)n$ for~$k \in \cK$ and~$t\in [\tc-\eps_0,\tc+\eps_0]$, 
which by Theorem~\ref{thm:NkC} completes the proof of~\eqref{eq:Nk}. 

Finally, we omit the similar (but simpler) proof of~\eqref{eq:Ngek}. 
\end{proof}

\subsection{Size of the largest component}\label{sec:L1} 
In this subsection we prove Theorems~\ref{thm:L1sub} and~\ref{thm:L1}, i.e., estimate the size $L_1(i)$ of the largest component after~$i=\tc n \pm \eps n$ steps. 
Our arguments use the following three ideas:
(i)~that we can typically sandwich~$G_i$ between two Poissonized random graphs~$\JPpm_i$ from Section~\ref{sec:dom:sw},
(ii)~that we have $L_1(\JPm_i) \le L_1(G_i) \le L_1(\JPp_i)$ by sandwiching and monotonicity, and
(iii)~that we can estimate the typical size of $L_1(\JPpm_i)$ by first and second moment arguments combined with `sprinkling', exploiting that the component size distribution has an exponential cutoff after size~$\eps^{-2}$.

\subsubsection{The subcritical case (Theorems~\ref{thm:L1sub} and~\ref{thm:L1})}\label{sec:L1:sub}
In this subsection we estimate the size of the $r$-th largest component in the subcritical case $i=\tc n-\eps n$ (for constant~$r$). 
Before giving the technical details, let us first sketch the high-level proof structure of Theorem~\ref{thm:L1sub}, ignoring the difference between~$G_i$ and~$\JPpm_i$ for simplicity.
Note that for any $r \ge 1$ we have 
\begin{equation}\label{eq:Lr:mon}
\begin{split}
\Pr\bb{L_r(i) \notin (\Lambda^-,\Lambda^+)} &\le \Pr\bb{L_1(i) \ge \Lambda^+} + \Pr\bb{L_r(i) \le \Lambda^- \text{ and } L_1(i) < \Lambda^+} \\
& \le \Pr\bb{N_{\ge \Lambda^+}(i) \ge \Lambda^+} + \Pr\bb{N_{\ge \Lambda^-}(i) \le r\Lambda^+} .
\end{split}
\end{equation}
With the exponential decay of Lemma~\ref{lem:NkE:sub} in mind, the basic idea is now to pick $\Lambda^- \approx \Lambda^+$ such that,
roughly speaking,
$\E N_{\ge \Lambda^-}(i) \gg \Lambda^+ \gg \E N_{\ge \Lambda^+}(i)$. 
By Markov's inequality this will give
\begin{equation*}
\Pr\bb{N_{\ge \Lambda^+}(i) \ge \Lambda^+} \le \frac{\E N_{\ge \Lambda^+}(i)}{\Lambda^+} = o(1) .
\end{equation*}
Furthermore, $X=N_{\ge \Lambda^-}(i)-N_{\ge \Lambda^+}(i)$ will satisfy $\E X \approx \E N_{\ge \Lambda^-}(i) \gg \Lambda^+ \gg \E N_{\ge \Lambda^+}(i)$. From Chebychev's inequality and
the variance estimate of Lemma~\ref{lem:Nkvar:sub}, we will then obtain
\begin{equation}\label{eq:Lr:mon:LB}
\Pr\bb{N_{\ge \Lambda^-}(i) \le r\Lambda^+} \le \Pr\bb{X \le r \Lambda^+} \le \Pr\bb{X \le \E X/2} \le \frac{\E X \cdot O(\Lambda^+)}{(\E X)^2} = \frac{O(\Lambda^+)}{\E N_{\ge \Lambda^-}(i)} = o(1) .
\end{equation}
The proofs below make the outlined argument precise.
\begin{proof}[Proof of Theorem~\ref{thm:L1sub}]%
For $x=x(n)$ satisfying $1 \le x \le \log \log \log(\eps^3 n)$, say, set
\[
 \Lambda^{\pm} := \psi(\tc-\eps)^{-1}\Bigl(\log (\eps^3 n) -\tfrac{5}{2}\log \log(\eps^3 n) \pm x\Bigr) .
\]
Since $\psi(\tc-\eps)=\Theta(\eps^2)$ and $\eps^3n \to \infty$,
routine calculations yield
$\eps^{-2} \ll \Lambda^-\sim \Lambda^+ \ll \min\{n^{2/3},n^{1/3}/\eps\}$ (recall that $a_n \ll b_n$ means $a_n = o(b_n)$, 
cf.~Remark~\ref{rem:notation}). 
Furthermore, by the choice of $\Lambda^{\pm}$ we have
\[
 e^{-\psi(\tc-\eps)\Lambda^{\pm}} = \Theta\left(\frac{(\log(\eps^3n))^{5/2}e^{\mp x}}{\eps^3n}\right)
 = \Theta\left( \eps^2n^{-1}(\Lambda^\pm)^{5/2} e^{\mp x} \right).
\]
Similar to the argument for~\eqref{eq:transfer}, using Lemma~\ref{lem:cond'}
and writing $t=i/n$ we obtain 
\begin{equation*}
\Pr\bb{L_r(i) \notin (\Lambda^-,\Lambda^+) \text{ and } \cN_i} \le \max_{\text{$t$-nice $\fS$}} \Pr\bb{L_r(J(\fS)) \notin (\Lambda^-,\Lambda^+)} .
\end{equation*}
Combining the sandwiching of Lemma~\ref{lem:cpl} with the idea of~\eqref{eq:Lr:mon}, and writing $\JPpm=\JP(\fS^{\pm}_{t})$ for brevity,
using monotonicity we arrive at 
\begin{equation}\label{eq:transfer3}
\Pr\bb{L_r(i) \notin (\Lambda^-,\Lambda^+) \text{ and } \cN_i} \le n^{-\omega(1)} + \max_{\text{$t$-nice $\fS$}} \Bigl[\Pr\bb{N_{\ge \Lambda^+}(\JPp) \ge \Lambda^+} + \Pr\bb{N_{\ge \Lambda^-}(\JPm) \le r\Lambda^+}\Bigr].
\end{equation}
By Lemma~\ref{lem:NkE:sub} and the fact that $\Lambda^-\sim\Lambda^+$, for $\eps$ small enough we have
\begin{equation}\label{eq:ENkpm}
\E N_{\ge \Lambda^{\pm}}(\JPpm) = \Theta(1) \cdot \frac{e^{-\psi(\tc-\eps)\Lambda^{\pm}}n}{\eps^2(\Lambda^{+})^{5/2}} \cdot \Lambda^+ \pm n^{-\omega(1)} = \Theta(e^{\mp x} \Lambda^{+}) . 
\end{equation}
By Markov's inequality, it follows that
\begin{equation}\label{eq:MICI:1}
 \Pr\bb{N_{\ge \Lambda^+}(\JPp) \ge \Lambda^+}   \le \frac{\E N_{\ge \Lambda^+}(\JPp)}{\Lambda^+}  =  O(e^{-x}).
\end{equation}
Moreover, from \eqref{eq:ENkpm}, for $x$ sufficiently large (depending on the constant $r$) we have
$\E N_{\ge \Lambda^{+}}(\JPm) \le\Lambda^+$
and $\E N_{\ge \Lambda^{-}}(\JPm) \ge 4r\Lambda^+$. Let $X=N_{\ge \Lambda^{-}}(\JPm) - N_{\ge \Lambda^{+}}(\JPm)$.
Then $\E X \ge \E N_{\ge \Lambda^{-}}(\JPm)/2 \ge 2r\Lambda^+$.
By Lemma~\ref{lem:Nkvar:sub} we have
\begin{equation*}
  \Var X\le \E X \bb{ \E N_{\ge \Lambda^{+}}(\JPm) + \Lambda^+ } \le 2 \Lambda^+ \E X.
\end{equation*}
Proceeding analogously to~\eqref{eq:Lr:mon:LB}, using Chebychev's inequality, the variance bound above, and~\eqref{eq:ENkpm}, it follows that
\begin{equation}\label{eq:MICI:2}
 \Pr\bb{N_{\ge \Lambda^-}(\JPm) \le r\Lambda^+}  \le  \Pr\bb{X \le \E X/2}
 \le  \frac{O(\Lambda^+)}{\E N_{\ge \Lambda^-}(\JPm)} =  O(e^{-x}).
\end{equation}
The result follows from~\eqref{eq:transfer3}, the bounds \eqref{eq:MICI:1} and \eqref{eq:MICI:2},
and Lemma~\ref{lem:nice}.
\end{proof}
Next we estimate the sizes of the~$r$ largest components in \emph{every} subcritical step by a similar (but more involved) argument. 
For Theorem~\ref{thm:L1} the idea is to consider the graphs $G_{m_j}$ at a \emph{decreasing} sequence of intermediate steps $m_j=(\tc-\eps_j)n$, where $\eps_1^3n = \omega$;
we index from $j=1$ since $\eps_0$ plays a different role -- as usual it is a (small) constant upper bound on values
of $\eps=|i/n-\tc|$ that we consider.
We shall show that typically $L_1(m_j) \le \Lambda^+_j$ and $N_{\ge \Lambda^-_{j+1}}(m_{j+1}) \ge r \Lambda^+_j$ for all $j \ge 1$ with $\eps_j \le \eps_0$.
For steps $m_{j+1} \le i \le m_j$ we then argue similarly to~\eqref{eq:Lr:mon}: by monotonicity we have 
\begin{equation}\label{eq:L1:sub:mon:UB}
L_1(i) \le L_1(m_{j}) \le \Lambda^+_{j}, 
\end{equation}
which together with $N_{\ge \Lambda^-_{j+1}}(i) \ge N_{\ge \Lambda^-_{j+1}}(m_{j+1}) \ge r \Lambda^+_{j}$ then in turn implies 
\begin{equation}\label{eq:L1:sub:mon:LB}
L_r(i) \ge \Lambda^{-}_{j+1}. 
\end{equation}
The next proof implements this strategy, using parameters~$\eps_j \approx \eps_{j+1}$ and~$\Lambda^+_j \approx \Lambda^-_{j+1}$ that make the corresponding error probabilities summable. 
\begin{proof}[Proof of Theorem~\ref{thm:L1} (subcritical phase)]%
For concreteness, let
\begin{equation}
\label{def:xi}
\xi  = \xi(n) := (\log \omega)^{-2/3},
\end{equation}
so that $\xi \to 0$ as $n \to \infty$. 
To ensure $\eps_j \le \eps_0$, we define $j_0=j_0(n,\omega,\xi,\eps_0)$ as the smallest $j \in \NN$ such that $\omega^{1/3}n^{-1/3}(1+\xi)^{j-1} \ge \eps_0$. 
For all $j \ge 1$ we set 
\begin{align}
\label{def:lambdaj}
\eps_j & :=
\begin{cases}
	\omega^{1/3}n^{-1/3}(1+\xi)^{j-1} , & ~~\text{if $j < j_0$}, \\
	\eps_0  , & ~~\text{if $j \ge j_0$}, 
\end{cases}\\
\label{def:Lj}
\Lambda^{\pm}_j & := (1 \pm \xi) \psi(\tc-\eps_j)^{-1}\log (\eps_j^3 n) ,\\
\label{def:mj}
m_j & :=(\tc-\eps_j)n .
\end{align}
Since $\eps_j^3n\ge \eps_1^3n=\omega\to\infty$, as in the proof of Theorem~\ref{thm:L1sub} we have 
$\eps_j^{-2} \ll \Lambda^{\pm}_j \ll \min\{n^{2/3},n^{1/3}/\eps_j\}$. 
Moreover, since $\psi(\tc-\eps)=\Theta(\eps^{-2})$, by choice of $\Lambda_j^\pm$ we have
\begin{equation}\label{eq:LJpmasy}
 \eps_j^{-2}(\Lambda_j^\pm)^{-5/2}e^{-\psi(\tc-\eps_j)\Lambda_j^\pm n}n
 = \Theta\bb{ \eps_j^{-2}(\eps_j^{-2}\log(\eps_j^3n))^{-5/2} (\eps_j^3n)^{-(1\pm\xi)}n }
 = \Theta\bb{ (\log(\eps_j^3n))^{-5/2} (\eps_j^3n)^{\mp\xi}}.
\end{equation}

Since $\psi(\tc)=\psi'(\tc)=0$ and $\psi''(\tc)>0$, the Mean Value Theorem implies that for all
$\eps,\eps' \in [\eps_{j}, \eps_{j+1}]$ with $j \ge 1$ we have
\[
|\psi(\tc-\eps)-\psi(\tc-\eps')| \le |\eps_{j+1}-\eps_{j}| \cdot O(\eps) = O(\xi \eps) \cdot O(\eps)= O(\xi) \cdot  \psi(\tc-\eps) ,
\]
where the implicit constant does \emph{not} depend on~$j$. 
It follows that there exists a universal constant $d>0$ such that for all $\eps_{j} \le \eps \le \eps_{j+1}$ with $j \ge 1$ we have
\begin{align}
\label{def:Lj1:UB}
\Lambda^{+}_{j} & \le (1 +d \xi) \cdot \psi(\tc-\eps)^{-1}\log (\eps^3 n) ,\\
\label{def:Lj1:LB}
\Lambda^{-}_{j+1} & \ge (1 -d \xi) \cdot \psi(\tc-\eps)^{-1}\log (\eps^3 n) .
\end{align}
In view of these bounds and the proof strategy outlined above, it thus suffices to prove that whp the following event~$\cE$ holds: 
\begin{equation}
\label{def:Lj1:suff}
L_1(m_j) \le \Lambda^+_j \quad \text{and} \quad N_{\ge \Lambda^-_{j+1}}(m_{j+1}) \ge r \Lambda^+_j \quad \text{for all $1 \le j < j_0$.}
\end{equation}
Indeed, arguing as for~\eqref{eq:L1:sub:mon:UB}--\eqref{eq:L1:sub:mon:LB} above, if $\cE$ holds
then $\Lambda_{j+1}^-\le L_r(i)\le L_1(i)\le \Lambda_j^+$ for all $m_{j+1}\le i\le m_j$ with $1 \le j < j_0$, which,
in view of \eqref{def:Lj1:UB}--\eqref{def:Lj1:LB} implies~\eqref{eq:sub:Lj} with $\tau=(\log \omega)^{-1/2} \gg d\xi$, say.

As we shall see, the proof of $\Pr(\neg\cE)=o(1)$ is similar to the proof of Theorem~\ref{thm:L1sub}, but here we have more elbow room. 
By Lemma~\ref{lem:NkE:sub} and the estimate \eqref{eq:LJpmasy},
recalling that $\eps_j^3n\ge \eps_1^3n=\omega\to\infty$ and $\omega^{\xi/3} \to \infty$, there is a constant $D>0$
such if $n$ is large enough, then for all $1\le j\le j_0$ and all $(m_j/n)$-nice parameter lists $\fS$,
setting $\JPpm_{m_j}=\JP(\fS^{\pm}_{m_j/n})$ we have
\begin{equation} \label{eq:NKj:UB}
\frac{\E N_{\ge \Lambda^+_j}(\JPpm_{m_j})}{\Lambda^+_j} \le \frac{D}{\bigl(\log(\eps_j^3n)\bigr)^{5/2} (\eps_j^3n)^\xi} \le \frac{1}{(\eps_j^3n)^{\xi/3}} \to 0
\end{equation}
and
\begin{equation} \label{eq:NKj:LB}
 \frac{\E N_{\ge \Lambda^-_j}(\JPm_{m_j})}{\Lambda^-_j} \ge \frac{(\eps_j^3n)^\xi}{D \bigl(\log(\eps_j^3n)\bigr)^{5/2}} \ge (\eps_j^3n)^{\xi/3} \to\infty.
\end{equation}
Recalling that $L_1(\JPp_{m_j}) \ge \Lambda^+_j$ implies $N_{\ge \Lambda^+_j}(\JPp_{m_j}) \ge \Lambda^+_j$,
by \eqref{eq:NKj:UB} and Markov's inequality we have
\begin{equation}\label{eq:MICI2:1}
 \max_{\text{$(m_j/n)$-nice $\fS$}} \Pr\bigl(L_1(\JPp_{m_j}) \ge \Lambda^+_j\bigr)  \le \frac{1}{(\eps_j^3n)^{\xi/3}}.
\end{equation}

Arguing as for~\eqref{eq:MICI:2}, given a parameter list $\fS$ which is $(m_j/n)$-nice,
let $\JPm_{m_j}=\JP(\fS^-_{m_j/n})$, and let
\[
X_j := N_{\ge \Lambda_j^-}(\JPm_{m_j}) - N_{\ge \Lambda_j^+}(\JPm_{m_j}).
\]
Then by Lemma~\ref{lem:Nkvar:sub} and the crude final estimate in \eqref{eq:NKj:UB}, for $n$ large enough we have
\[
   \Var X_j\le \E X_j \bb{ \E N_{\ge \Lambda_j^{+}}(\JPm_{m_j}) + \Lambda_j^+ } \le 2 \Lambda_j^+ \E X_j.
\]
Since $\Lambda_{j-1}^+\sim \Lambda_j^+\sim \Lambda_j^-$, the estimate \eqref{eq:NKj:LB} easily implies
$\E X_j \ge \E N_{\ge \Lambda^-_j}(\JPm_{m_j})/2 \ge 2r\Lambda_{j-1}^+$ (for $n$ large). Hence
by Chebychev's inequality we have
\begin{equation}\label{eq:NLj-}
 \Pr\bigl(N_{\ge \Lambda^-_j}(\JPm_{m_j}) \le r\Lambda^+_{j-1}\bigr) \le \Pr\bb{X_j\le \E X_j/2} \le
 \frac{4\Var X_j}{(\E X_j)^2} \le \frac{8\Lambda_j^+}{\E X_j} \le \frac{16\Lambda_j^+}{\E N_{\ge \Lambda_j^{-}}(\JPm_{m_j})}
 \le \frac{17}{(\eps_j^3n)^{\xi/3}},
\end{equation}
say, using \eqref{eq:NKj:LB} in the last step.
Arguing as for \eqref{eq:transfer3} (using sandwiching and the idea of~\eqref{eq:Lr:mon}),
writing $\cN=\bigcap_{i_0\le i\le i_1}\cN_i$ for the event that every $\fS_i$ is $(i/n)$-nice,
from \eqref{eq:MICI2:1} and \eqref{eq:NLj-} we conclude that
\begin{equation*}
\Pr(\neg\cE \cap \cN) 
 \le \sum_{1 \le j \le j_0} \biggl[n^{-\omega(1)} + \frac{18}{(\eps_j^3n)^{\xi/3}}\biggr]
 \le \sum_{1 \le j \le j_0} \biggl[n^{-\omega(1)} + \frac{18}{\omega^{\xi/3} (1+\xi)^{(j-1)\xi}}\biggr],
\end{equation*}
recalling the definition of $\eps_j$ in the last step. The main term is a geometric progression with ratio
$(1+\xi)^{-\xi}=1-\Theta(\xi^2)$, so the sum is $O(\omega^{-\xi/3}\xi^{-2})=o(1)$ by choice of $\xi$.

This completes the proof since $\Pr(\neg\cN)=o(1)$ by Lemma~\ref{lem:nice}.
\end{proof}
Note that in the above proof we can allow for $r \to \infty$ at some slow rate (e.g., $r=\omega^{\xi/4}$ works readily).

\subsubsection{The supercritical case (Theorem~\ref{thm:L1})}\label{sec:L1:super} 
In this subsection we estimate the 
size of the largest component in the supercritical phase. 
We first outline the proof structure of Theorem~\ref{thm:L1} for step~$i=\tc n + \eps n$, ignoring several technicalities.
Given $\xi=o(1)$, let $i^*=\tc n + (1-\xi)\eps n$. 
From the variance estimate of Lemma~\ref{lem:NkE:sup} and the assumption $\eps^3n\to\infty$, we can eventually pick $\xi =o(1)$ and $(\xi\eps)^{-2} \ll \Lambda \ll \min\{n^{2/3}, n^{1/3}/\eps\}$ such that for $m \in \{i,i^*\}$ Chebychev's inequality yields
\begin{equation}\label{eq:L1:sup:SMM}
\Pr\bb{|N_{\ge \Lambda}(m)-\E N_{\ge \Lambda}(m)| \ge \xi \eps n} \le \frac{\Var N_{\ge \Lambda}(m)}{(\xi \eps n)^2} \le \frac{O(e^{-d_1 \eps^2\Lambda} + (\eps^{3}n)^{-1/3})}{\xi^2} = o(1) .
\end{equation}
Lemma~\ref{lem:NkE:sup} and continuity of $\Pr(|\bp_{t}|=\infty)$ thus suggest that for $m \in \{i,i^*\}$ whp 
\begin{equation}\label{eq:L1:sup:SMM:NL}
N_{\ge \Lambda}(m) \approx \E N_{\ge \Lambda}(m) \approx \Pr(|\bp_{m/n}|=\infty) n \approx \Pr(|\bp_{\tc+\eps}|=\infty) n = \Theta(\eps n).
\end{equation}
Now the upper bound for $L_1(i)$ is immediate by the standard observation that 
\begin{equation}\label{eq:L12}
 L_1(G)+L_2(G) \le N_{\ge \Lambda}(G)+2\Lambda
\end{equation}
for any graph $G$ and any $\Lambda\ge 1$.
Indeed, using $\Lambda = o(\eps n)$ we should whp have 
\begin{equation}\label{eq:L1L2:spr}
L_1(i) \le L_1(i) + L_2(i) \le N_{\ge \Lambda}(i) + 2 \Lambda \approx \Pr(|\bp_{\tc+\eps}|=\infty) n .
\end{equation}
For the lower bound we use `sprinkling', exploiting that $N_{\ge \Lambda}(i^*) \ge x= \Theta(\eps n)$ by~\eqref{eq:L1:sup:SMM:NL}. 
Applying Lemma~\ref{lem:SP}, using $\Delta_{\Lambda,x,\xi} = O(n^2/(\xi \Lambda x)) = o(\xi \eps n)$, $x/\Lambda = \Theta(\eps n/\Lambda) = \omega(1)$ and $i^*+\Delta_{\xi,\Lambda,x} \le i$ we expect that whp 
\[
L_1(i) \ge L_1(i^*+\Delta_{\xi,\Lambda,x}) \ge (1-\xi)N_{\ge \Lambda}(i^*) \approx \Pr(|\bp_{\tc+\eps}|=\infty) n ,
\]
which together with~\eqref{eq:L1L2:spr} also suggests $L_2(i) = o(L_1(i))$. 
Similar to the subcritical proof, for concentration in \emph{every} supercritical step, we shall use (a rigorous version of) the above line of reasoning for a carefully chosen \emph{increasing} sequence of intermediate steps~$m_j=(\tc+\eps_j)n$, relating $N_{\ge \Lambda}(m_{j-1})$ and $L_1(m_{j})$ via sprinkling. 
\begin{proof}[Proof of Theorem~\ref{thm:L1} (supercritical phase)]%
For concreteness and brevity, let
\begin{align}
\label{def:xi:sup}
\xi = \xi(n) & := (\log \omega)^{-1} ,\\
\label{def:rhox:sup}
\varphi(x) & :=  \Pr(|\bp_{\tc+x}|=\infty) .
\end{align}
Define $j_0=j_0(n,\omega,\xi,\eps_0)$ as the smallest $j \in \NN$ such that $\omega^{1/6}n^{-1/3}(1+\xi)^{j-1} \ge \eps_0$. 
For all $j \ge 1$ we set 
\begin{align}
\label{def:lambdaj:sup}
\eps_j & :=
\begin{cases}
	\omega^{1/6}n^{-1/3}(1+\xi)^{j-1} , & ~~\text{if $j < j_0$}, \\
	\eps_0  , & ~~\text{if $j \ge j_0$}, 
\end{cases}\\
\label{def:Lj:sup}
\Lambda_j & := \eps_j^{-2}\bigl(\log(\eps_j^3n)\bigr)^3  ,\\
\label{def:mj:sup}
m_j & :=(\tc+\eps_j)n .
\end{align}
Since $\eps_j^3n \ge \eps_1^3n =\omega^{1/2} \to \infty$, routine calculations yield
$\eps_j^{-2} \ll \Lambda_j \ll \min\{n^{2/3},n^{1/3}/\eps_j\}$.
By Theorem~\ref{thsurv-simple} there is a constant $c>0$ such that we have
\begin{equation}\label{eq:rho:ineq:sup}
 \varphi(\eps) \ge c \eps \quad \text{for all $0 \le \eps \le \eps_0$.}
\end{equation} 
Furthermore, $\varphi'$ is bounded on $[0,\eps_0]$, so for all $\eps_{j-1} \le \eps \le \eps_{j+1}$ with $j \ge 2$ the Mean Value Theorem yields
\[
|\varphi(\eps_{j\pm 1})-\varphi(\eps)| \le |\eps_{j+1}-\eps_{j-1}| \cdot O(1) = O(\xi \eps) = O(\xi) \cdot \varphi(\eps) ,
\]
using \eqref{eq:rho:ineq:sup} in the last step. Here the implicit constant does \emph{not} depend on~$j$. 
It follows that there exists a universal constant $d>0$ such that for all $\eps_{j-1} \le \eps \le \eps_{j+1}$ with $j \ge 2$ we have
\begin{equation}
\label{eq:Lj1:imv}
 (1-d \xi) \cdot \varphi(\eps) \le \varphi(\eps_{j\pm1}) \le (1 +d \xi) \cdot \varphi(\eps) .
\end{equation}
Note for later that, since $\eps_j^3n\ge \eps_1^3n=\omega^{1/2}$, for all $1\le j\le j_0$ we have
\begin{equation}\label{eq:Ljvphi}
 \frac{\Lambda_j}{\varphi(\eps_j) n} = \Theta\left( \frac{\Lambda_j}{\eps_j n} \right) = \Theta\left( \frac{(\log(\eps_j^3n))^3}{\eps_j^3n} \right)
 = O\left( \frac{(\log(\omega^{1/2}))^3}{\omega^{1/2}}\right) \le \omega^{-1/4},
\end{equation}
if $n$ is large enough.

Let  $\cE$ be the event that 
\begin{equation}\label{def:cE:sup}
 (1-\xi) \cdot \varphi(\eps_j)n \le N_{\ge \Lambda_j}(m_j) \le (1+\xi) \cdot \varphi(\eps_j)n \quad \text{for all $1 \le j \le j_0$.}
\end{equation}
To later use `sprinkling', for $c$ as in~\eqref{eq:rho:ineq:sup}, we define 
\begin{equation}
\label{def:xj}
x_j := c \eps_j n/2 .
\end{equation}
Recalling Lemma~\ref{lem:SP}, we now define $\cS_j := \cS_{m_j,\Lambda_j,x_j,\xi}$ and $\cS := \bigcap_{1 \le j \le j_0} \cS_j$.

\begin{claim}
If $\cE \cap \cS$ holds, then so do~\eqref{eq:super:L1}--\eqref{eq:super:L2}.
\end{claim}

To establish the claim, suppose that $\cE$ and $\cS$ hold, and let $i$ be such that $\eps=i/n-\tc$ satisfies $\eps^3n\ge \omega$ and $\eps\le \eps_0$.
Then from the definition of $j_0$ and the fact that $\eps_2^3n=(1+\xi)^3\omega^{1/2}\le\omega$, there is some $2\le j< j_0$
such that $m_j\le i\le m_{j+1}$.
From \eqref{eq:Ljvphi} we have $\Lambda_{j+1} \le \omega^{-1/4}\varphi(\eps_{j+1}) = o(\xi) \cdot \varphi(\eps_{j+1})$, say. 
So, using \eqref{eq:L12}, monotonicity, \eqref{def:cE:sup} and~\eqref{eq:Lj1:imv}, we have
\begin{equation}\label{eq:L1L2:sup}
\begin{split}
L_1(i)  \le L_1(i) + L_2(i) & \le N_{\ge \Lambda_{j+1}}(i) + 2\Lambda_{j+1} \le N_{\ge \Lambda_{j+1}}(m_{j+1}) + 2\Lambda_{j+1} \\
& \le (1+2\xi) \cdot \varphi(\eps_{j+1})n \le \bigl(1+(d+3)\xi\bigr) \cdot \varphi(\eps)n.
\end{split}
\end{equation}
From the lower bound in \eqref{def:cE:sup} and~\eqref{eq:rho:ineq:sup}, we have $N_{\ge \Lambda_{j-1}}(m_{j-1}) \ge c \eps_{j-1} n/2 = x_{j-1}$.
Since $\eps_{j-1}^3 n \ge \omega^{1/2} \to \infty$, using~\eqref{def:xi:sup} it is not difficult to check that the parameter
$\Delta_{\Lambda,x,\xi}=\Theta(n^2/(\xi\Lambda x))$ appearing in Lemma~\ref{lem:SP} satisfies 
\[
\Delta_{\Lambda_{j-1},x_{j-1},\xi}  = \frac{\Theta(n^2)}{\xi \Lambda_{j-1} x_{j-1}} = \frac{O(\eps_{j-1} n)}{\xi (\log \omega)^3} = o(\xi \eps_{j-1} n) < (\eps_j-\eps_{j-1}) n = m_j - m_{j-1} .
\]
Since the `sprinkling event'~$\cS_{j-1}=\cS_{m_{j-1},\Lambda_{j-1},x_{j-1},\xi}$, holds, from~\eqref{def:cE:sup} and~\eqref{eq:Lj1:imv} we thus deduce that 
\begin{equation}\label{eq:L1:sup}
\begin{split}
L_1(i) \ge L_1(m_j) & \ge  L_1(m_{j-1}+\Delta_{\Lambda_{j-1},x_{j-1},\xi}) \ge (1-\xi) \cdot N_{\ge \Lambda_{j-1}}(m_{j-1})\\
& \ge (1-2\xi) \cdot \varphi(\eps_{j-1})n \ge \bigl(1-(d+3)\xi\bigr) \cdot \varphi(\eps)n.
\end{split}
\end{equation}
Combining~\eqref{eq:L1L2:sup} and~\eqref{eq:L1:sup}, we have
\begin{equation}\label{eq:L2:sup}
L_2(i) \le 2(d+3)\xi \cdot \varphi(\eps)n \le 4 (d+3)\xi \cdot L_1(i) .
\end{equation}
Together,~\eqref{eq:L1L2:sup}--\eqref{eq:L2:sup} readily establish~\eqref{eq:super:L1}--\eqref{eq:super:L2} with $\tau=(\log \omega)^{-1/2} \gg 4 (d+3)\xi$, say, completing the proof of the claim.

Having proved the claim, it remains only to show that $\Pr(\neg\cE)=o(1)$ and $\Pr(\neg\cS)=o(1)$. 
As in previous subsections, by a simple application of our conditioning and sandwiching results (Lemmas~\ref{lem:cond'} and~\ref{lem:cpl}),
writing $\JPpm_{m_j}=\JP(\fS^{\pm}_{m_j/n})$ it follows that
\begin{equation}\label{eq:prEN:sup}
\Pr(\neg\cE \cap \cN) \le \sum_{1 \le j \le j_0}\biggl[n^{-\omega(1)} + 2 \max_{\text{$(m_j/n)$-nice $\fS$}} \: \max_{\JPpm_{m_j} \in \{\JPm_{m_j},\JPp_{m_j}\}} \Pr\bigl(|N_{\ge \Lambda_j}(\JPpm_{m_j})-\varphi(\eps_j)n| > \xi \varphi(\eps_j)n\bigr) \biggr] .
\end{equation}
To avoid clutter, we henceforth tacitly assume that~$\fS$ is $(m_j/n)$-nice. 
Since $\eps_j^2 \Lambda_j = (\log(\eps^3_j n))^3$ and $\eps_j^3 n \ge \omega^{1/2} \to \infty$, it is not difficult to check that
\begin{equation}\label{eq:lK:cond}
 e^{-d_1\eps_j^2 \Lambda_j} + (\eps_j^{3} n)^{-1/3} 
 \le \frac{2}{(\eps_j^3 n)^{1/3}} = o(\xi) ,
\end{equation}
where the constant $d_1>0$ is as in Lemma~\ref{lem:NkE:sup}.   
Recalling $\varphi(\eps_j)=\Pr(|\bp_{\tc+\eps_j}|=\infty) \ge c \eps_j$, see~\eqref{def:rhox:sup} and~\eqref{eq:rho:ineq:sup}, using~\eqref{eq:lK:cond} and Lemma~\ref{lem:NkE:sup} we infer that, say, 
\begin{equation*}
\Pr\Bigl(\bigl|N_{\ge \Lambda_j}(\JPpm_{m_j})-\varphi(\eps_j)n\bigr| > \xi \varphi(\eps_j)n\Bigr) \le \Pr\Bigl(\bigl|N_{\ge \Lambda_j}(\JPpm_{m_j})-\E N_{\ge \Lambda_j}(\JPpm_{m_j})\bigr| \ge c\xi\eps_j n/2\Bigr) .
\end{equation*}
Similar to~\eqref{eq:L1:sup:SMM}, using Chebychev's inequality, \eqref{eq:lK:cond}, and the variance estimate of Lemma~\ref{lem:NkE:sup} it now follows that there is a constant $C>0$ such that 
\begin{equation}\label{eq:Nk:sup:2}
\Pr\bigl(|N_{\ge \Lambda_j}(\JPpm_{m_j})-\varphi(\eps_j)n| \ge \xi \varphi(\eps_j)n\bigr) \le \frac{\Var N_{\ge \Lambda_j}(\JPpm_{m_j})}{(c\xi\eps_j n/2)^2} \le \frac{C}{\xi^2(\eps_j^3n)^{1/3}}.
\end{equation}
Substituting~\eqref{eq:Nk:sup:2} and $\eps_j^3n = \omega^{1/2}(1+\xi)^{3(j-1)}$ into~\eqref{eq:prEN:sup},
using $\sum_{\ell\ge 0} (1+\xi)^{-\ell}=O(\xi^{-1})$ we obtain
\begin{equation}\label{eq:prEN:sup:1}
\Pr(\neg\cE \cap \cN) \le n^{-\omega(1)} + \sum_{j \ge 1}\frac{2C}{\xi^2\omega^{1/6}(1+\xi)^{j-1}} \le n^{-\omega(1)} + \frac{O(1)}{\omega^{1/6}\xi^3} = o(1) .
\end{equation}
Since $\Pr(\neg\cN)=o(1)$ by~Lemma~\ref{lem:nice}, this establishes $\Pr(\neg\cE)=o(1)$. 
It remains to prove $\Pr(\neg\cS)=o(1)$. 
Since $\eps_j^3 n \ge \omega^{1/2} \to \infty$, using~\eqref{def:Lj:sup} and~\eqref{def:xj} it is routine to see that, say,
\[
\frac{x_j}{\Lambda_j} = \frac{\Theta(\eps^3_j n)}{\bigl(\log(\eps_j^3n)\bigr)^3} \ge \bigl(\log(\eps_j^3n)\bigr)^2 .
\]
Using Lemma~\ref{lem:SP} and $\eps_j^3 n = \omega^{1/2}(1+\xi)^{3j}$, for some constant $\eta>0$ we thus obtain
\begin{equation*}
\Pr(\neg\cS) \le \sum_{1 \le j \le j_0} \Pr(\neg \cS_{m_j,\Lambda_j,x_j,\xi}) \le \sum_{1 \le j \le j_0}\biggl[ \exp(-\eta x_j/\Lambda_j) + n^{-\omega(1)} \biggr] \le \sum_{1 \le j \le j_0}\frac{1}{\eps_j^3 n} + n^{-\omega(1)} = o(1), 
\end{equation*}
arguing as for~\eqref{eq:prEN:sup:1} in the final step. This was all that remained to complete the proof of Theorem~\ref{thm:L1}.
\end{proof}

\subsection{Susceptibility}\label{sec:Sj} 
In this subsection we prove Theorem~\ref{thm:Sj}, i.e., estimate the ($r$th order) susceptibility $S_r(i)=S_{r,n}(G_i)$ in the subcritical case $i=\tc n-\eps n$ (see~\eqref{eq:def:SrG} for the definition of the modified parameter~$S_{r,n}$). 
Similar to Section~\ref{sec:L1}, our arguments use the following two ideas:
(i)~that we typically have $S_{r,n}(\JPm_i) \le S_{r,n}(G_i) \le S_{r,n}(\JPp_i)$ by sandwiching and monotonicity, and
(ii)~that we can estimate the typical value of $S_{r,n}(\JPpm_i)$ by a second moment argument.

Ignoring the difference between $G_i$ and $\JPpm_i$ (and some other technical details), for Theorem~\ref{thm:Sj} the basic line of reasoning is as follows.
Using the variance estimate of Lemma~\ref{lem:SrE:sub} and the assumption $\eps^3n \to \infty$, the idea is that we can eventually pick $\xi = o(1)$ such that Chebychev's inequality intuitively gives
\begin{equation*}
\Pr\bb{|S_{r,n}(i)-\E S_{r,n}(i)| \ge \xi \E S_{r,n}(i)} \le \frac{\Var S_{r,n}(i)}{\bigl(\xi \E S_{r,n}(i)\bigr)^2} \le \frac{O(1)}{\xi^2 \eps^3 n} = o(1) .
\end{equation*}
To prove bounds in \emph{every} subcritical step, analogous to Section~\ref{sec:L1} we consider a \emph{decreasing} sequence of intermediate steps $m_j=(\tc-\eps_j) n$. 
Using (a rigorous version of) the above reasoning we show that typically
$\LS^{-}_{r,j} \le S_{r,n}(m_j) \le \LS^{+}_{r,j}$ for suitable $\LS^{\pm}_{r,j}$, which by monotonicity (see Remark~\ref{rem:Snr}) then translates into bounds for every step.
\begin{proof}[Proof of Theorem~\ref{thm:Sj}]
Fix $r\ge 2$. For concreteness, let, 
\begin{equation}\label{def:Sj:xi}
\xi = \xi(n) := \omega^{-1/4} ,
\end{equation}
so that $\xi \to 0$ as $n \to \infty$. 
To ensure $\eps_j \le \eps_0$, we define $j_0=j_0(n,\omega,\xi,\eps_0)$ as the smallest $j \in \NN$ such that $\omega^{1/3}n^{-1/3}(1+\xi)^{j-1} \ge \eps_0$. 
For all $j \ge 1$ we set
\begin{align}
\label{def:Sj:lambdaj}
\eps_j & :=
\begin{cases}
	\omega^{1/3}n^{-1/3}(1+\xi)^{j-1} , & ~~\text{if $j < j_0$}, \\
	\eps_0  , & ~~\text{if $j \ge j_0$}, 
\end{cases}\\
\label{def:Sjmj}
m_j & :=(\tc-\eps_j)n ,\\
\label{def:SjLj}
\LS^{\pm}_{r,j} & := \bigl(1 \pm (\eps_j^3n)^{-1/4}\bigr)\bigl(B_r \pm a_r\bigl(\eps_j + (\eps_j^3n)^{-1/3}\bigr)\bigr)\eps_j^{-2r+3} ,
\end{align}
where $B_r>0$ is the constant defined in~\eqref{eq:sjfkt:Dr} and $a_r>0$ is the constant in Lemma~\ref{lem:SrE:sub}.
It is routine to see that there is a constant $A_r>0$ such that for all
$1\le j\le j_0$ and all $\eps_j \le \eps \le \eps_{j+1}$ we have 	
\begin{align}
\label{def:Sr:UB}
\max\Bigl\{\bigl|\LS^{+}_{r,j}-B_r \eps^{-2r+3}\bigr|, \: \bigl|\LS^{-}_{r,j+1}-B_r \eps^{-2r+3}\bigl|\Bigr\} \le A_r \bigl(\eps + (\eps^3 n)^{-1/4}\bigr) B_r \eps^{-2r+3} .
\end{align}

Let $\cE$ be the event that
\begin{equation}\label{def:cE:Sj}
\LS^{-}_{r,j} \le S_{r,n}(m_j) \le \LS^{+}_{r,j}  \quad \text{for all $1 \le j \le j_0$.}
\end{equation}
Since $S_r(i)=S_{r,n}(G_i)$ is monotone (see Remark~\ref{rem:Snr}), the event $\cE$ implies that for all steps $m_{j+1} \le i \le m_j$ with $1 \le j < j_0$ we have have $\LS^{-}_{r,j+1} \le S_{r}(i) \le \LS^{+}_{r,j}$. 
Recalling the definition of the steps~$m_j$, 
in view of~\eqref{def:Sr:UB} it thus is immediate that~$\cE$ implies~\eqref{eq:Sjsub:uniform}. 

Since $\Pr(\neg\cN)=o(1)$ by Lemma~\ref{lem:nice}, it remains only to show that $\Pr(\neg\cE \cap \cN)=o(1)$.
As usual (for example, as for~\eqref{eq:prEN:sup}), by conditioning (Lemma~\ref{lem:cond'}),
sandwiching (Lemma~\ref{lem:cpl}) and monotonicity (Remark~\ref{rem:Snr}), writing $\JPpm_{m_j}=\JP(\fS^{\pm}_{m_j/n})$, we conclude that
\begin{equation}\label{eq:prEN:Sr}
\Pr(\neg\cE \cap \cN) \le \sum_{1 \le j \le j_0}\biggl[n^{-\omega(1)} + \max_{\text{$(m_j/n)$-nice $\fS$}} \Bigl[\Pr\bigl(S_{r,n}(\JPm_{m_j}) < \LS^{-}_{r,j}\bigr) + \Pr\bigl(S_{r,n}(\JPp_{m_j}) > \LS^{+}_{r,j}\bigr)\Bigr]\biggr] .
\end{equation}
Since $\eps_j^3n\ge \omega\to\infty$, if $n$ is large enough the assumptions of Lemma~\ref{lem:SrE:sub} are satisfied.
In particular, by the expectation bound~\eqref{eq:SrE:sub} and the choice of $\LS_{r,j}^\pm$ we infer that
\begin{equation}\label{eq:prCh:Sr:0}
 \LS^{+}_{r,j} \ge \bigl(1+(\eps_j^3n)^{-1/4}\bigr) \E S_{r,n}(\JPp_{m_j})   \quad \text{and} \quad 
 \LS^{-}_{r,j} \le \bigl(1-(\eps_j^3n)^{-1/4}\bigr) \E S_{r,n}(\JPm_{m_j}) .
\end{equation}
The variance bound~\eqref{eq:SrVar:sub} of Lemma~\ref{lem:SrE:sub} states that
\[
\frac{\Var S_{r,n}(\JPpm_{m_j}) }{\bb{\E S_{r,n}(\JPpm_{m_j})}^2} \le \frac{b_r}{\eps_j^3n}.
\]
By Chebychev's inequality and~\eqref{eq:prCh:Sr:0}, it follows for large $n$ that
\[
 \Pr\bigl(S_{r,n}(\JPm_{m_j}) < \LS^{-}_{r,j}\bigr) +  \Pr\bigl(S_{r,n}(\JPp_{m_j}) >\LS^{+}_{r,j} \bigr) \le 
\frac{2b_r}{(\eps_j^3 n)^{1/2}} \le \frac{1}{(\eps_j^3 n)^{1/3}} .
\]
Substituting this bound into~\eqref{eq:prEN:Sr} and using $\eps_j^3n = \omega (1+\xi)^{3(j-1)}$, it follows that 
\[
 \Pr(\neg\cE \cap \cN) \le n^{-\omega(1)} + \sum_{j \ge 1}\frac{1}{\omega^{1/3} (1+\xi)^{j-1}} \le n^{-\omega(1)} + \frac{O(1)}{\omega^{1/3} \xi} = o(1),
\]
completing the proof. 
\end{proof}

\newoddpage

\section{Open problems and extensions}\label{sec:open}

Our proof methods exploited that, via a two-round exposure argument, we could construct the random graphs we study using many independent (uniform) random vertex choices. This allowed us to bring branching process comparison arguments into play.
Although we have not checked the details, we believe that these methods adapt without problems
to, for example, the vertex immigration random graph model introduced by Aldous and Pittel~\cite{AldousPittel2000}, 
and its generalization proposed by Bhamidi, Budhiraja and Wang~\cite{BBW11,BBW12b,BBW12a},
where in each time-step either (i)~components of bounded size immigrate into the vertex set, or (ii)~an edge connecting two randomly chosen vertices is added. 
Indeed, for these models the key observation is that, similar to the present paper, near the critical point we can again partition the vertex set into $V_S \cup V_L$ in such a way that we can construct the random graph by (a)~joining components from~$V_S$ to a certain number of uniformly random vertices from~$V_L$ and (b)~adding uniformly random edges to~$V_L$. 
(To account for the fact that the vertex set grows over time, here~$V_L$ contains all components at some suitable time $t_0 \in (0,\tc)$, and~$V_S$ contains all new vertices and components which arrive by time $t \in (t_0,t_1)$, where $t_1>\tc$ and $t_1-t_0$ is small enough that the graph induced by~$V_S$ stays `subcritical'.)  
This makes it plausible that the methods of this paper can again be used to analyze the phase transition in these models.

In the light of the above discussion, the following \emph{open problems} might be more interesting for further work 
than simply adapting the methods used here to other models.
In (1)--(3) below, we consider only bounded-size rules.
\begin{enumerate}
	\item[(1)] {\bfseries Size of the second largest `subcritical' component:} 
	Show that, for $\eps=\eps(n)$ satisfying $\eps \to 0$ and $\eps^3 n \to \infty$ as $n \to \infty$, 
	the size of the second largest `supercritical' component whp satisfies $L_2(\tc n + \eps n) \sim a\eps^{-2}\log(\eps^3n)$, where the constant~$a = \Psi''(\tc)>0$ is as in Theorem~\ref{thm:rhok}. 
	Since Theorem~\ref{thm:L1} implies that the largest `subcritical' component whp satisfies $L_1(\tc n - \eps n) \sim a\eps^{-2}\log(\eps^3n)$, this would establish the `symmetry rule' (also called `discrete duality') that is well-known for Erd\H os--R\'enyi random graphs (see, e.g., Section~3 in~\cite{Luczak1990} or Section~5.6 in~\cite{JLR}); it  
	would also be consistent with the small component size distribution~\eqref{eq:intro:smallcpt} established in this paper.
	
	\item[(2)] {\bfseries Analyticity away from critical~$\tc$:} 
	Show that the asymptotic form~\eqref{eq:rhok} of the function $\rho_k(t)$ appearing in~\eqref{eq:Nk:pto} remains valid for any bounded time interval (excluding~$0$), 
	not just close to the critical time~$\tc$.
	One can also ask similar questions about the function $\rho(t)$ appearing in~\eqref{eq:L1:pto}. For example, Janson and Spencer~\cite{JS} were interested (for the Bohman--Frieze rule) in whether $\rho(t)$ is analytic (or, as they asked it, smooth) 
for any~$t \in [\tc,\infty)$, not just for time $t \in [\tc,\tc+\eps_0)$ as shown in this~paper.
	
	\item[(3)] {\bfseries Central limit theorem for size of the largest `supercritical' component:} 
	Show that, for $\eps=\eps(n)$ satisfying $\eps = O(1)$ and $\eps^3 n \to \infty$ as $n \to \infty$, the size of the largest `supercritical' component $L_1(\tc n + \eps n)$ satisfies a central limit theorem~(CLT). 
This is well-known for Erd\H os--R\'enyi random graphs (see, e.g.,~\cite{PittelWormald,BR2012RW} and the references therein); it would also complement the law of large numbers established in this paper. 

	\item[(4)] {\bfseries Size of the largest component in `explosive' size rules:} 
	Analyze, for fixed $\eps>0$ or $\eps=\eps(n) \to 0$, the qualitative behaviour of the rescaled size of the largest `supercritical' component $L_1(\tc n + \eps n)/n$ for `explosive' (unbounded) size rules such as the product rule, the sum rule, or the dCDGM rule (defined in~\cite{DRS,RW,RWapcont,dCDGM}). 
	As discussed in the introduction, see also Figure~\ref{fig:L1plots}, these rules seem to have an extremely steep growth, which most likely differs from the linear growth of bounded-size rules established in this paper (we believe that the corresponding scaling limits~$\rho(t)$ have an infinite right-hand derivative at the critical time~$\tc$, see also Section~\ref{sec:bg}).

\end{enumerate}

For the \emph{duality problem}~(1), similar to~\cite{BJR,JR2011}, we expect that taking out the giant component we obtain an instance of the random hypergraph model $J(\tilde{\fS})$ that is close enough to a natural dual `subcritical' version, which can be coupled to the supercritical branching process conditioned on not surviving.
Then it ought to be possible to prove results for the small component sizes that are similar to what we have below~$\tc$, though the technical challenges seem formidable. 
In work in preparation~\cite{RWapip}, 
we use a combinatorial multi-round exposure argument to prove a weaker result: that whp $L_2(\tc n + \eps n) = O(\min\{\eps^{-2},1\}\log n)$ for $n^{-1/3}(\log n)^{1/3} \ll \eps = O(1)$.

For the \emph{time interval problem}~(2), we speculate that variants of the methods of this paper might extend by some kind of step-by-step argument, but we did not investigate this closely as the present paper was already long enough, and the near-critical behaviour in any case seems the most interesting. In~\cite{RWapip} we exploit the PDE approach of Section~\ref{sec:AP} (among other ideas) to prove, for any~$t \in (\tc,\infty)$, that $\rho_k(t)$ decays exponentially in~$k$ and that $\rho(t)$ is analytic.  

For the \emph{CLT problem}~(3), we speculate that for fixed~$\eps \in (0,\eps_0)$ it might be possible to adapt the differential equation method based approach of Seierstad~\cite{Seierstad} (together with ideas of this paper and~\cite{RWapip}) to establish asymptotic normality after suitable rescaling, but we have not investigated this closely as our main focus is the more challenging $\eps=\eps(n) \to 0$ case.
Indeed, it seems that a CLT for $\eps=\eps(n) \to 0$ with $\eps^3 n \to \infty$ requires new ideas that go beyond~\cite{Seierstad} and the recent random walk based CLT approach~\cite{BR2012RW}.

The \emph{`unbounded' size rules problem}~(4) is conceptually perhaps the most important one, and it will most likely further stimulate the development of new tools and techniques in the area. 
Based on the partial results from~\cite{RWapsubcr}, we believe that it would be key to understand the effect of the edges which are added close to the `critical point' where the susceptibility diverges (e.g., if they have a similar effect to the addition of random edges).
An alternative approach might be to analyze the behaviour of the infinite system of differential equations derived in~\cite{RWapunique}, which however is not known to have a unique solution. For this one may perhaps need to augment the system by further typical properties of the associated random graph process; see also Section~3 in~\cite{RWapunique}.

\newoddpage

\small
\bibliographystyle{plain}

\normalsize

\newoddpage

\begin{appendix}

\section{Appendix}\label{sec:apx}

\subsection{Transferring results from $4$-vertex rules to Achlioptas processes}\label{apx:tfer} 
In this appendix we briefly present one possible way of transferring results from $4$-vertex processes to the original Achlioptas process (where in each step the two edges $e_1,e_2$ are chosen independently and uniformly at random from all edges not yet present, say).
Fixing some rule~$\cR$, the Achlioptas process $(G^{\cR}_{n,i})_{0 \le i \le 9n}$ is uniquely determined by the sequence of potential edges $\ve = (e_{1,i},e_{2,i})_{1 \le i \le 9n}$ offered during the first $9n$ steps. 
In the Achlioptas process any valid sequence $\ve$ occurs with probability 
at most 
\[
\prod_{0 \le i < 9n}\frac{1}{\left(\binom{n}{2}-i\right)^2} = \prod_{0 \le i < 9n} \frac{4}{n^4\left(1-1/n-2i/n^2\right)^2} \le \frac{\left(4/n^4\right)^{9n}}{\left(1-19/n\right)^{18n}} \le e^{400} \left(\frac{4}{n^4}\right)^{9n} 
\]
for $n \ge n_0$. 
Mapping $\vv_i=(v_1, \ldots, v_4)$ to the pairs $e_{1,i}=\{v_{i,1},v_{i,2}\}$ and $e_{2,i}=\{v_{i,3},v_{i,4}\}$, in the $4$-vertex process any edge sequence $\ve = (e_{1,i},e_{2,i})_{1 \le i \le 9n}$ occurs with probability exactly~$(4/n^4)^{9n}$. 
It follows that if an event~$\cE$ fails with probability at most~$\pi$ in the $4$-vertex process, then~$\cE$ fails with probability at most~$e^{400}\pi= O(\pi)$ in the Achlioptas process (tacitly assuming that the event~$\cE$ does \emph{not} depend on any graphs~$G^{\cR}_{n,i}$ with $i > 9n$, which of course holds in this paper).
Since our main results only concern events that fail with negligible probability $\pi \to 0$, 
this formally justifies the fact that we may treat the original Achlioptas processes as a $4$-vertex process.  
(Similar reasoning applies to other variations.)

\subsection{Cauchy--Kovalevskaya ODE and PDE theorems}\label{apx:CK}
In this appendix we present two `easy-to-apply' versions of the 
Cauchy--Kovalevskaya theorem, 
which are optimized for the (combinatorial) applications in this paper. 
These show that, under suitable regularity conditions,
certain systems of ODEs or PDEs have analytic solutions.

We first consider first-order PDEs, with $\bx=(x_1, \ldots, x_n)\in \CC^n$. 
Our starting point is the following standard version of the Cauchy--Kovalevskaya Theorem,
taken from pages~15--16 in~\cite{Petrovsky}. 
This states that a first order PDE has an analytic \emph{local} solution provided (i)~the time-derivative
of the function $u$ to be solved for is
given by an analytic function of $u$ and its space-derivatives as in~\eqref{eq:CK:PDE} below,
and (ii)~the initial data~\eqref{eq:CK:PDE:init} is analytic.  
Similar statements hold for more general PDEs, but we shall not need this.
%
\begin{lemma}[Cauchy--Kovalevskaya for first-order PDEs]\label{lem:CK:PDE:simple}
Let $n \ge 1$, let $t_0\in \CC$ and let $\bx_0 \in \CC^n$.
Suppose that the function $f:\CC^n \to \CC$ is analytic in some neighbourhood of~$\bx_0$,
and that $F:\CC^{2n+2} \to \CC$ is analytic in some neighbourhood of
$\bigl(t_0, \bx_0, f(\bx_0), \frac{\partial f}{\partial x_1}(\bx_0), \ldots, \frac{\partial f}{\partial x_n}(\bx_0)\bigr)$. 
Then there exists a neighbourhood $\cN$ of $(t_0,\bx_0)$ in $\CC^{n+1}$ and an analytic function
$u: \cN \to \CC$ which satisfies
\begin{align}
\label{eq:CK:PDE}
\frac{\partial}{\partial t} u(t,\bx) &= F\biggl(t,\bx,  \: u(t,\bx),   \: \frac{\partial}{\partial x_1}u(t,\bx),  \ldots,  \frac{\partial}{\partial x_n}u(t,\bx)\biggr) \hbox{\quad and} \\ 
\label{eq:CK:PDE:init}
u(t_0,\bx) & = f(\bx) .
\end{align} 
\noproof
\end{lemma}
Standard results also give uniqueness in this case (among analytic solutions). For our application, local existence as above
is not quite enough; we would like existence in a neighbourhood $\cN$ of a certain compact (`space')
domain rather than just of a point. Fortunately, this follows by a compactness argument.
Given $\by=(y_1,\ldots,y_n) \in \CC^n$ and $\br=(r_1,\ldots,r_n)$ with all $r_i>0$, we write
\[
 \cB(\by,\br) := \{ \bx\in \CC^n: \: |x_i-y_i|< r_i, 1\le i\le n \}
\]
for the \emph{polycylinder} (or polydisc)   
in $\CC^n$ with centre~$\by$ and polyradius~$\br$. 
With $t_0\in \CC$ fixed, for $r > 0$ we write
\[
 \cT(r) := \{ t\in \CC: \: |t-t_0| < r \}.
\]

\begin{theorem}[Convenient Cauchy--Kovalevskaya for first-order PDEs]\label{thm:CK:PDE}
Suppose that $n\ge 1$, $t_0\in \CC$, $\eps>0$, and
$0<a_i<b_i$ for $i=1,\ldots,n$.
Let\begin{align*}
\cT&:=\cT(\eps),\\
\cX_0&:=\cB((0,\ldots,0),\ba), \hbox{\quad and}\\
\cX_1&:=\cB((0,\ldots,0),\bbb).
\end{align*}
Suppose that the functions $f:\cX_1 \to \CC$ and $F:\cT \times \cX_1 \times \CC^{n+1} \to \CC$ are analytic.
Then there is a $\delta>0$ and an analytic function $u: \cT_0 \times\cX_0 \to \CC$ which
satisfies~\eqref{eq:CK:PDE}--\eqref{eq:CK:PDE:init}, where $\cT_0 := \cT(\delta)$.
Furthermore, the Taylor series of $u$ around $(t_0,0, \ldots, 0)$ converges (to $u$) 
in the domain~$\cT_0 \times \cX_0$.  
\end{theorem}
\begin{proof}
Let $\bcX_0\subset \cX_1$ be the closure of $\cX_0$, i.e., the set $\{\bx\in \CC^n: \: |x_i|\le a_i, i=1,\ldots,n\}$.
For any point $\pp \in \bcX_0$,
by Lemma~\ref{lem:CK:PDE:simple} there is an $r_\pp>0$ such that, defining
\[ 
 \cT_\pp:=\cT(r_\pp), \qquad  \cX_\pp:=\cB(\pp,(r_\pp, \ldots, r_\pp)) \qquad \text{and} \qquad \cN_\pp := \cT_\pp\times \cX_\pp,
\]
the following holds: (i)~we have $\cT_\pp \subseteq \cT$ and $\cX_\pp \subseteq \cX_1$,
and (ii)~there exists an analytic function $u_\pp: \cN_\pp \to \CC$ which satisfies~\eqref{eq:CK:PDE}--\eqref{eq:CK:PDE:init} for all $(t,\bx) \in \cN_\pp$ (with~$u$ replaced by~$u_\pp$). 

Suppose that $\cN_\pp$ and $\cN_\qq$ intersect; we claim that then $u_\pp$ and $u_\qq$ agree on $\cN_\pp\cap\cN_\qq$.
To see this, first note that $\cN_\pp\cap \cN_\qq$ is of the form $\cT(r)\times \cD$ for some
$r>0$ and some open $\cD\subset \CC^n$. Suppose that $(t,\by)\in \cN_\pp\cap \cN_\qq$.
Since $\cD$ is open, some open polycylinder $\cB:=\cB(\by,\br)$ is contained in $\cD$, so 
\[
 (t,\by) \in \cT(r)\times\cB \subset \cN_\pp\cap \cN_\qq.
\]
Since $u_\pp$ and $u_\qq$ are analytic in the polycylinder $\cT(r)\times \cB$, by
the complex version of the Taylor series expansion (see, e.g., Theorem~1.18 in~\cite{Range})
they both have Taylor series around~$(t_0,\by)$ which converge in this domain.
By construction, $u_\pp$ and $u_\qq$ and satisfy the initial condition~\eqref{eq:CK:PDE:init}
and the time-derivative equation~\eqref{eq:CK:PDE} for all $(t,\bx) \in \cT(r)\times \cB$.
These properties together uniquely determine all partial derivatives of $u_\pp$ and $u_\qq$ at the point~$(t_0,\by)$ (this observation also forms the basis of the Cauchy--Kovalevskaya theorem).
Thus $u_\pp$ and $u_\qq$ have the same Taylor expansion around $(t_0,\by)$ and hence agree in $\cT(r)\times\cB$
and in particular at $(t,\by)$.

The collection $\{\cX_\pp\}$ of polycylinders forms an open cover of $\bcX_0$. By compactness, there
is a finite subcover: $\bcX_0 \subset \bigcup_{\pp \in P} \cX_\pp$ with $P$ finite.
Let $\delta := \min_{\pp\in P} r_\pp>0$, and set $\cT_0:=\cT(\delta)$.
Let $\cN := \cT_0\times\cX_0$. Then
\[
 \cN\subseteq \bigcup_{\pp \in P} (\cT_0 \times \cX_\pp)
\subseteq \bigcup_{\pp \in P} (\cT_\pp \times \cX_\pp) =
 \bigcup_{\pp \in P}  \cN_\pp.
\]
Define $u:\cN \to\CC$ by $u(t,\bx):=u_\pp(t,\bx)$
for any $\pp\in P$ such that $(t,\bx)\in \cN_\pp$. This definition makes sense by the claim above.
Then $u$ is analytic: for any $(t,\bx)\in \cN$, we have $(t,\bx)\in \cN_\pp$ for some $\pp\in P$,
and since $\cN_\pp$ is open and $u_\pp$ agrees with $u$ in $\cN_\pp$, $u$ is analytic at $\pp$.
Similarly, $u$ satisfies \eqref{eq:CK:PDE}--\eqref{eq:CK:PDE:init} since the $u_\pp$ do.
This completes the proof of the first statement.

The second statement follows: since $u$ is analytic in the polycylinder $\cN$ centered at $(t_0,0,\ldots,0)$,
by e.g., Theorem~1.18 in~\cite{Range} its Taylor series about $(t_0,0,\ldots,0)$
converges in $\cN$.
\end{proof}

Turning to the ODE case, the following folklore theorem (see, e.g., Corollary~2 in Section 6.11 of~\cite{ODEGG}) states that 
functions $u_1, \ldots, u_s$ which satisfy a finite system of ODEs 
are real-analytic if their derivatives~\eqref{eq:CK:ODE} are based on real-analytic equations; 
the technical condition~\eqref{eq:CK:ODE:I} ensures 
that~\eqref{eq:CK:ODE} makes sense.  
\begin{theorem}[Cauchy--Kovalevskaya for ODEs]\label{thm:CK:ODE}
Let $s \ge 1$. 
Suppose that $\cT \subseteq \RR$ is an open interval, that $\cI \subseteq \RR^s$ is an open set, and that $F_j: \cT \times \cI \to \RR$ is real-analytic for $1 \le j \le s$. 
Suppose that the functions $u_1,\ldots,u_s$ from $\cT$ to $\RR$ satisfy 
\begin{gather}
\label{eq:CK:ODE:I}
 \Bigl(u_1(t), \ldots, u_s(t)\Bigr) \in \cI \hbox{\quad and} \\
\label{eq:CK:ODE}
\frac{d}{d t} u_j(t) = F_j\Bigl(t, \: u_1(t), \ldots, u_s(t)\Bigr) 
\end{gather}
for all $t\in \cT$. Then $u_1,\ldots,u_s$ are real-analytic in~$\cT$. \noproof
\end{theorem}

\subsection{Palm theory for the Poisson process}\label{apx:palm}

In this appendix we present two elementary instances of palm theory for the Poisson process, 
which provide methods for calculating the mean of certain random sums.
In Lemma~\ref{lem:palm} below we write, as usual, $[N]=\{1,2,\ldots,N\}$ and $[0]=\emptyset$. 
The symmetry assumption~\eqref{eq:sym} holds for functions that are invariant under relabellings. 
In the right hand side of~\eqref{eq:palm}, we intuitively think of $N+1, \ldots, N+s$ either (i)~as `extra' elements that are added to the random set~$[N]$, or (ii)~as special elements of the `enlarged' random set~$[N+s]$. 

\begin{lemma}\label{lem:palm}
Let $N \sim \Po(\lambda)$ with $\lambda \in [0,\infty)$. 
Given $s \ge 1$, let $f$ 
be a measurable random function, independent of~$N$, defined on the product of~$(\NNP)^s$ and finite subsets of~$\NNP$. 
Assume that, for all $m \ge s$ and $x_1, \ldots, x_s \in [m]$, we have
\begin{equation}\label{eq:sym}
 \E\bigl(f(x_1, \ldots, x_s,[m]  \setminus \{x_1, \ldots, x_s\})\bigr) = \E\bigl(f(m-s+1, \ldots, m, [m-s])\bigr).
\end{equation} 
Then 
\begin{equation}\label{eq:palm}
\E\Bigl(\sideset{}{^*}\sum_{(x_1, \ldots, x_s) \in [N]^s} f(x_1, \ldots, x_s,[N]  \setminus \{x_1, \ldots, x_s\})\Bigr) = \lambda^s \E\bigl(f(N+1, \ldots, N+s, [N])\bigr) , 
\end{equation}
where $\sum^*$ means that we are summing over $s$-tuples with distinct $x_i$. 
\end{lemma}
\begin{proof}
The argument is elementary: after conditioning on $N \ge s$ it suffices to rewrite terms, exploiting symmetry of~$f$ and the identity $\Pr(N=m) \binom{m}{s}s! = \lambda^s\Pr(N=m-s)$. 
More precisely, by the assumed independence, we see that the left hand side of \eqref{eq:palm} may be written as   
\begin{equation*}
\begin{split}
& \sum_{m \ge s} \Pr(N=m) \sideset{}{^*}\sum_{(x_1, \ldots, x_s) \in [m]^s} \E\bigl(f(x_1, \ldots, x_s,[m]  \setminus \{x_1, \ldots, x_s\})\bigr) \\
& \qquad = \sum_{m \ge s} \Pr(N=m) \binom{m}{s}s! \E\bigl(f(m-s+1, \ldots, m, [m-s])\bigr) ,
\end{split}
\end{equation*}
which by our above discussion equals $\lambda^s \E(f(N+1, \ldots, N+s, [N]))$. 
\end{proof}
We shall also use the following simple variant (again thinking of~$f$ as being symmetric w.r.t.\ the labels); 
the proof is very similar to Lemma~\ref{lem:palm} and thus omitted.
\begin{lemma}\label{lem:palm:2}
For $i \in [2]$, let $N_i \sim \Po(\lambda_i)$ be independent random variables. 
Let~$f$ be a measurable random function, independent of~$N_1$ and $N_2$, defined on the product of~$(\NNP)^2$ and finite subsets of~$\NNP \times \NNP$. 
Assume that, for all $m_1,m_2 \ge 1$, $x \in [m_1]$ and $y \in [m_2]$, we have 
${\E(f(x, y,[m_1]  \setminus \{x\}, [m_2] \setminus\{y\}))} = {\E(f(m_1, m_2,[m_1-1],[m_2-1]))}$. 
Then 
\begin{equation*}
\E\Big(\sum_{x \in [N_1],\,y \in [N_2]} f(x, y,[N_1]  \setminus \{x\}, [N_2] \setminus\{y\})\Bigr) = \lambda_1\lambda_2 \E\bigl(f(N_1+1, N_2+1, [N_1], [N_2])\bigr) .
\end{equation*}
\noproof
\end{lemma}

\subsection{Branching processes}\label{sec:BP}

The branching process results stated in Section~\ref{sec:bpresults}, namely Theorems~\ref{thsurv-simple}--\ref{asyfull-simple-pm},
are, in essence, proved in a separate paper~\cite{BPpaper} with Svante Janson. 
The reason for the split is that the proofs use very
different methods from those used in the present (already fairly long) paper;  
they are pure branching-process theory, with no random graph theory involved.
As formulated in Section~\ref{sec:bpresults}, however, these results involve rather complicated definitions
from Section~\ref{sec:cpl}.
Although our only aim is to analyze the specific branching processes~$\bp_t$ and~$\bp_t^\pm$ defined in Section~\ref{sec:cpl},
to avoid the need to repeat the full definitions in~\cite{BPpaper}, in Section~\ref{ss:BPsetup}
we review, and somewhat generalize, them. More precisely,
we gather together the properties of these processes (or rather,
the offspring distributions defining them) needed for the analysis into a formal definition, 
which is of course tailored to our context.
Then, in Section~\ref{ss:BPresults}, we state two results that, as we show,
imply Theorems~\ref{thsurv-simple}--\ref{asyfull-simple-pm}. 
The statements of these results are complicated by the parameter $\per$; in Section~\ref{apx:period} we show
that the general case may be deduced from the special case $\per=1$ proved in~\cite{BPpaper}.  

\subsubsection{Setup and assumptions}\label{ss:BPsetup}

Throughout this appendix,
we consider branching processes of the following general form, formally defined in Definition~\ref{def:bp:new}.
Each generation consists of some number of particles of type $L$ and some number of type $S$.
Particles of type $S$ never have children.
Given a probability distribution
$(Y,Z)$ on $\NN^2$, $\bp^1_{Y,Z}$ is the Galton--Watson process starting with a single particle of type $L$ (in generation~$0$)
in which each particle of type $L$ has $Y$ children of type $L$ and $Z$ of type $S$, independent of other particles in its generation
and of the history. Given a second probability distribution $(Y^0,Z^0)$ on $\NN^2$,
$\bp_{Y,Z,Y^0,Z^0}$ is the branching process defined in the same way, except that the first generation consists of $Y^0$ particles
of type $L$ and $Z^0$ of type $S$.

\begin{definition}
A \emph{branching process family} $(\bp_t)_{t\in (t_0,t_1)} = (\bp_{Y_t,Z_t,Y^0_t,Z^0_t})_{t\in (t_0,t_1)}$
is simply a family of branching processes as above, one for each real number $t$ in some interval $(t_0,t_1)$.
\end{definition}

Note that the branching process family $(\bp_{Y_t,Z_t,Y^0_t,Z^0_t})_{t\in (t_0,t_1)}$ is fully specified by the interval
$(t_0,t_1)$ and the distributions of $(Y_t,Z_t)$ and of $(Y^0_t,Z^0_t)$ for each $t$. We shall often describe
properties of these distributions via their probability generating functions. 
The next definition encapsulates those properties of the `idealized' branching process
$\bp_t$ defined in \eqref{def:bp:t} that we shall need.

\begin{definition}\label{def:bpprops}
Let $t_0<\tc<t_1$ be real numbers, and let $\per$ and $K$ be non-negative integers.
The branching process family
$(\bp_{Y_t,Z_t,Y^0_t,Z^0_t})_{t\in (t_0,t_1)}$
is \emph{$\tc$-critical with period $\per$ and offset $K$} if 
the following hold:
\begin{romenumerate}
  \item\label{def:bp:analytic} There exist $\delta>0$ and $R > 1$ with $(\tc-\delta,\tc+\delta) \subseteq (t_0,t_1)$ such that the functions 
	\begin{equation}\label{fdef}
 \gf(t,\alpha,\beta) := \E\bigl(\alpha^{Y_t}\beta^{Z_t}\bigr) \qquad\text{and}\qquad
 \gf^0(t,\alpha,\beta) := \E\bigl(\alpha^{Y_t^0}\beta^{Z_t^0}\bigr) 
\end{equation}
are defined for all real $t$ with $|t-\tc| < \delta$ and all complex $\alpha$, $\beta$ with $|\alpha|,|\beta| < R$.
Furthermore, these functions have analytic extensions to the complex domain
\[
 \fD_{\delta,R}:=\bigl\{(t,\alpha,\beta) \in \CC^3 : \: |t-\tc|<\delta \text{ and } |\alpha|,|\beta|<R\bigr\}.
\]
  \item\label{def:bp:lattice}  For each $t\in (t_0,t_1)$, with probability~$1$ we have 
\begin{gather}
\label{YZsupnew}
 (Y_t,Z_t) \in (\per\NN)^2 
\qquad\text{and}\qquad
 (Y_t^0,Z_t^0) \in (\per\NN)^2 \cup (\{0\}\times [K]).
\end{gather}
\item\label{def:bp:crit} We have
\begin{equation}
\label{means:Y}
 \E Y_{\tc} =1, 
\qquad \E Y^0_{\tc} > 0,  
\qquad\text{and}\qquad 
\left. \frac{\mathrm{d}}{\dt} \E Y_{t} \right|_{t=\tc} > 0 .
\end{equation}
\item\label{def:bp:nondeg} There exists some $k_0\in\NN$ such that
\begin{equation}\label{cnondeg}
\min\Bigl\{
\Pr\bb{Y_{\tc}=k_0,\,Z_{\tc}=k_0}, \: 
\Pr\bb{Y_{\tc}=k_0+\per,\,Z_{\tc}=k_0}, \:
\Pr\bb{Y_{\tc}=k_0,\,Z_{\tc}=k_0+\per}
\Bigr\} > 0 .
\end{equation}
\end{romenumerate}
\end{definition}

As we shall show in a moment, the results in Section~\ref{sec:cpl}--\ref{sec:cpl2} show that the branching process family $(\bp_t)_{t\in (t_0,t_1)}$
defined in \eqref{def:bp:t} is $\tc$-critical with period $\per$ and offset $K$,
where $\per$ is the period of the rule $\cR$, defined in Section~\ref{sec:period}
(see Lemma~\ref{lem:allowed}), and $K$ is the cut-off of $\cR$. 
We also consider `perturbed' distributions that differ from these `idealized' ones slightly.

\begin{definition}\label{def:dtype}
Let $(\bp_t)_{t\in (t_0,t_1)}$ be a $\tc$-critical branching process family with period~$\per$ and offset~$K$,
let~$\delta$, $R$ and $k_0$ be as in Definition~\ref{def:bpprops}.
Given $t,\eta \ge 0$ with $|t-\tc|< \delta$, 
we say that the branching process $\bp_{Y,Z,Y^0,Z^0}$ is \emph{of type $(t,\eta)$}
(with respect to $(\bp_t)$, $\delta$, $R$, and $k_0$) if the following hold: 
\begin{romenumerate}
  \item\label{def:bp2:analytic} Writing $\fN := \{(\alpha,\beta) \in \CC^2: \: |\alpha|,|\beta| < R\}$, 
the expectations
	\begin{equation}\label{tfdef}
 \tg(\alpha,\beta):=\E\bigl(\alpha^Y\beta^Z\bigr) \qquad\text{and}\qquad
 \tg^0(\alpha,\beta) := \E\bigl(\alpha^{Y^0}\beta^{Z^0}\bigr) 
\end{equation}
are defined (i.e., the sums converge absolutely) for all $(\alpha, \beta) \in \fN$.
  \item\label{def:bp2:support}  With probability~$1$ we have 
\begin{gather}
\label{dYZsupnew}
 (Y,Z) \in (\per\NN)^2 
\qquad\text{and}\qquad
 (Y^0,Z^0) \in (\per\NN)^2 \cup (\{0\}\times [K]).
\end{gather}
\item\label{def:bp2:close} 
For all $(\alpha,\beta) \in \fN$ we have
\begin{equation}
\label{dtype1}
\bigl|\tg(\alpha,\beta)-\gf(t,\alpha,\beta)\bigr|  \le \eta
\qquad\text{and}\qquad
\bigl|\tg^0(\alpha,\beta)-\gf^0(t,\alpha,\beta)\bigr| \le \eta .
\end{equation}
\end{romenumerate}
\end{definition}

Note that when $\per=1$ (the main case we are interested in), the offset $K$ plays no role in 
Definitions~\ref{def:bpprops} and~\ref{def:dtype}, 
so we may take $K=0$. 
Definition~\ref{def:dtype} says that, in some precise sense, the distributions of $(Y,Z)$ and of $(Y^0,Z^0)$
are `$\eta$-close' to those of $(Y_t,Z_t)$ and $(Y_t^0,Z_t^0)$, respectively. 
We shall only consider cases where 
\[
0\le \eta\le |t-\tc|.
\]
Note that our definition of `type $(t,\eta)$' is with reference to a branching process family $(\bp_t)$, as well as 
some additional constants. This branching process family will always be clear from context, so we shall
often omit referring to it explicitly; we shall always omit reference to the additional constants.

\begin{lemma}\label{lem:bp:satisfy}
Let $(\bp_t)_{t\in (t_0,t_1)}$ be the branching process family defined in~\eqref{def:bp:t}. 
Then $(\bp_t)$ is $\tc$-critical
with period $\per$ and offset~$K$, where $\per$ is defined in Lemma~\ref{lem:allowed},
and~$K$ is the cut-off size in the bounded-size rule~$\cR$.
Furthermore, there exist constants $\delta,C>0$ 
such that for any $t\in (t_0,t_1)$ with $Cn^{1/3}\le |t-\tc|\le \delta$
and any $t$-nice parameter list~$\fS$, 
the branching processes $\bp_t^\pm=\bp_t^\pm(\fS)$ defined in Definition~\ref{def:bp2pm}
are of type~$(t,Cn^{1/3})$ with respect to $(\bp_t)_{t\in (t_0,t_1)}$.
\end{lemma}
\begin{proof}
Let $\delta$ be as in Theorem~\ref{thm:f:mgf}, and let $R$ be the smaller
of the radii $R$ appearing in Theorems~\ref{thm:f:mgf} and~\ref{thm:fPT:mgf}.

Considering first $(\bp_t)_{t\in (t_0,t_1)}$, the analyticity condition~\ref{def:bp:analytic} in Definition~\ref{def:bpprops}
is satisfied by Theorem~\ref{thm:f:mgf}.
Condition~\ref{def:bp:lattice} holds by Lemma~\ref{lem:inlattice}, 
the criticality condition~\ref{def:bp:crit} holds by Lemma~\ref{lem:Zt:E}, 
and the non-degeneracy condition~\ref{def:bp:nondeg} holds by Lemma~\ref{lem:YtZtlb}.

We now turn to $\bp_t^\pm(\fS)$ as defined in Definition~\ref{def:bp2pm}.  
Condition \ref{def:bp2:analytic} in Definition~\ref{def:dtype} is an immediate consequence of the uniform upper bound~\eqref{eq:fPT:mgf:ff0} from Theorem~\ref{thm:fPT:mgf}. 
Condition \ref{def:bp2:support} on the support holds by Lemma~\ref{lem:inlattice:dom},
and the `$\eta$-close' condition \ref{def:bp2:close} holds by~\eqref{eq:fPT:mgf:f}--\eqref{eq:fPT:mgf:f0} of Theorem~\ref{thm:fPT:mgf}, provided~$C$ is chosen large enough. 
\end{proof}

\subsubsection{Results}\label{ss:BPresults}

In this subsection we state two results, Theorems~\ref{asyfull} and~\ref{thsurv} below, which imply
the results in Section~\ref{sec:bpresults}. For each, the case $\per=1$ is proved in~\cite{BPpaper};
in Section~\ref{apx:period} we show how to reduce the general case to $\per=1$. For comparison to~\cite{BPpaper}, note that Definitions 3.2 and 3.3 there are exactly the $\per=1$ cases of Definitions~\ref{def:bpprops} and~\ref{def:dtype}.

We start with the tail asymptotics of the branching process, which simplifies when~$\per=1$: then~\eqref{eq:asyfull} holds for all~$k \ge 1$ without the indicator. 
\begin{theorem}\label{asyfull}
Let $(\bp_t)_{t\in (t_0,t_1)}$ be a $\tc$-critical branching process family with period $\per$ and offset $K$.
Then there exist constants $\eps_0,c>0$  
and analytic functions $\theta$, $\psi$
on the interval $I=[\tc-\eps_0,\tc+\eps_0]$
such that
\begin{equation}\label{eq:asyfull}
 \Pr(|\bp|=k) = (1+O(1/k)+O(\eta)) \indic{k\equiv 0\mathrm{\ mod\ }\per}k^{-3/2} \theta(t) e^{-\xi_{Y,Z} k}
\end{equation}
uniformly over all $k>K$, $t\in I$, $0\le\eta\le c|t-\tc|$ and
all branching processes $\bp=\bp_{Y,Z,Y^0,Z^0}$ of type~$(t,\eta)$ 
(with respect to~$(\bp_t)$),
where the constant $\xi_{Y,Z}$, which depends on the distribution of $(Y,Z)$, satisfies
\[
 \xi_{Y,Z} =  \psi(t) +O(\eta|t-\tc|).
\]
Moreover, $\theta>0$, $\psi\ge 0$, $\psi(\tc)=\psi'(\tc)=0$, and $\psi''(\tc)>0$.
\end{theorem}
The condition~$k>K$ in Theorem~\ref{asyfull} is needed only to account for the possibility that 
that for some small~$k$ which are not multiples of~$\per$ we may have $\Pr(|\bp|=k)>0$;
this can only happen if $(Y_t^0,Z_t^0)=(0,z)$ for some $z \in [K]$.
As discussed in Section~\ref{ss:BPsetup}, in the most important case with period~$\per=1$ we may take~$K=0$.
The case $\per=1$ of Theorem~\ref{asyfull} follows immediately from Theorem 3.4 in~\cite{BPpaper}, noting that here we have $\eta=O(|t-\tc|)$ by assumption, so the additional $O(\eta^2)$ error term there can be absorbed into the $O(\eta|t-\tc|)$ error term here.

We next turn to the survival probability of the branching process.

\begin{theorem}\label{thsurv}
Let $(\bp_t)_{t\in (t_0,t_1)}$ be a $\tc$-critical branching process family with period $\per$ and offset $K$. 
Then there exist constants~$\eps_0,c>0$  
with the following properties.
Firstly, the survival probability $\rho(t)={\Pr(|\bp_t|=\infty)}$ is zero for $\tc-\eps_0\le t\le \tc$, and is positive
for $\tc<t\le \tc+\eps_0$.
Secondly, $\rho(t)$ is analytic on $[\tc,\tc+\eps_0]$; more precisely,
there are constants $a_i$ with $a_1>0$ such that 
\begin{equation}\label{rexp}
 \rho(\tc+\eps) = \sum_{i=1}^\infty a_i\eps^i .
\end{equation}
for $\eps \in [0,\eps_0]$.  
Thirdly, for any $t$, $\eta$ with $|t-\tc|\le\eps_0$ and $\eta\le c|t-\tc|$,
and any branching process $\bp=\bp_{Y,Z,Y^0,Z^0}$ of type~$(t,\eta)$ (with respect to~$(\bp_t)$),
the survival probability $\trho=\Pr(|\bp|=\infty)$ is zero if $t\le \tc$, 
and is positive and satisfies 
\[
 \trho = \rho(t)+O(\eta)
\]
if $t>\tc$. 
Moreover, analogous statements hold for the survival probabilities
$\rho_1(t)=\Pr(|\bp^1_t|=\infty)$ and $\trho_1=\Pr(|\bp^1_{Y,Z}|=\infty)$. 
\end{theorem}

The $\per=1$ case of Theorem~\ref{thsurv} is exactly Theorem 4.5 of~\cite{BPpaper}.
In the light of Lemma~\ref{lem:bp:satisfy},
Theorems~\ref{thsurv-simple} and~\ref{thsurv-simple-pm} follow immediately
from Theorem~\ref{thsurv}, and Theorems~\ref{asyfull-simple} and~\ref{asyfull-simple-pm} from Theorem~\ref{asyfull}
and the discussion in Remark~\ref{rem:period}. Hence all that remains is to reduce Theorems~\ref{asyfull} and \ref{thsurv} to the case $\per=1$.

\subsubsection{Reduction to the special case $\per=1$ and $K=0$}\label{apx:period} 
The proofs of Theorems~\ref{asyfull} and~\ref{thsurv} in the key case $\per=1$ and $K=0$ are given in a companion paper~\cite{BPpaper} written with Svante Janson. 
In this subsection we outline how both theorems follow from these key special~cases; the argument is purely technical and requires no new ideas.

Turning to the details for Theorem~\ref{asyfull}, 
suppose that we ave given a $\tc$-critical branching process family $(\bp_t)_{t\in (t_0,t_1)}=
(\bp_{Y_t,Z_t,Y_t^0,Z_t^0})_{t\in (t_0,t_1)}$ with period $\per > 1$ and offset $K \ge 0$,
and a branching process $\bp=\bp_{Y,Z,Y^0,Z^0}$ of type $(t,\eta)$ 
with respect to this family.
(As discussed in Sections~\ref{ss:BPsetup}--\ref{ss:BPresults}, for period $\per = 1$ we make take offset $K=0$, and so there is nothing to show.) 
We shall modify these branching processes in two steps into ones corresponding 
to the case~$\per=1$, $K=0$. 
Of course, in each step we need to check that our branching processes 
satisfy Definitions~\ref{def:bpprops} and~\ref{def:dtype} (so we can apply Theorem~\ref{asyfull} to them),
and that the conclusion of Theorem~\ref{asyfull} for the new distributions implies the conclusion of Theorem~\ref{asyfull} for the old distributions (possibly after decreasing the corresponding constant~$c>0$). 

We start with a simple 
auxiliary claim for the distributions~$(Y_t^0,Z_t^0)$ and~$(Y^0,Z^0)$ fixed above. 
\begin{claim}\label{cl:analytic}
For each $k\ge 1$ 
the function $f_k(t) := \Pr(Y^0_t=0,\,Z^0_t=k)$ is defined for real $t$ with $|t-\tc|< \delta$, and satisfies $|\Pr(Y^0=0,\,Z^0=k) - \Pr(Y_t^0=0,\,Z_t^0=k)\bigr| \le \eta$. 
Furthermore, $f_k$ has an analytic extension to the complex domain $\fD_\delta := \{t \in \CC: \: |t-\tc|< \delta\}$.
\end{claim}
\begin{proof}
Since $\gf^0(t,\alpha,\beta)$ has an analytic extension to~$\fD_{\delta,R}$ 
by Definition~\ref{def:bpprops}, and $f_k(t) = \gf^0_{\beta^k}(t,0,0)/k!$,
it follows that $f_k(t)$ has an analytic extension to $\fD_{\delta}$. 
Furthermore, for any real $t$ with $|t-\tc|< \delta$, 
using standard Cauchy estimates (with center $a=(0,0)$ and multiradius $r=(1,1)$; see, e.g., Theorem~1.6 in~\cite{Range}) we obtain 
\begin{equation*}
\begin{split}
\bigl|\Pr(Y^0=0,\,Z^0=k) - \Pr(Y_t^0=0,\,Z_t^0=k)\bigr| & = \bigl|\tg^0_{\beta^k}(0,0) - \gf^0_{\beta^k}(t,0,0) \bigr| / k!\\
& \le \sup_{\alpha,\beta \in \CC\::\: |\alpha|,|\beta| \le 1} \bigl|\tg^0(\alpha,\beta) - \gf^0(t,\alpha,\beta) \bigr| \le \eta,
\end{split}
\end{equation*}
where we used~\eqref{dtype1} from Definition~\ref{def:dtype} for the last inequality (recall that $R > 1$). 
\end{proof}

For Theorem~\ref{asyfull} we first deduce the case $\per \ge 1$, $K > 0$ from the case  $\per \ge 1$, $K=0$. 
Recall that~$Z_t^0$, and also~$Z^0$, need not always be a multiple of $\per$. However, from
condition~\eqref{YZsupnew} of Definition~\ref{def:bpprops} (and its analogue in Definition~\ref{def:dtype}),
the only possible exceptions are values $(Y_t^0,Z_t^0)=(0,k)$ with $k \in [K]=\{1, \ldots, K\}$,
and similarly for~$Z^0$.
We modify the distribution of~$(Y_t^0,Z_t^0)$ by simply setting
this random variable to be equal to~$(0,\per)$, say, whenever it takes a value~$(0,k)$ with~$k \in [K]$.
We modify~$(Y^0,Z^0)$ in an analogous way.
It is easy to see that the resulting branching processes satisfy the conditions in Definitions~\ref{def:bpprops} and~\ref{def:dtype}.
Indeed, the key assumption is the analytic extension of the probability generating function~$\gf^0$,
but~$\gf^0$ has changed
only by the addition of the finite sum $\sum_{k \in [K]}(\beta^{\per}-\beta^k)\Pr(Y^0_t=0,Z^0_t=k)$, 
which is has an analytic extension to $\fD_{\delta,R}$ by Claim~\ref{cl:analytic}.
Since the distribution of $Y_t^0$ has not changed 
the new distribution still satisfies the criticality condition~\eqref{means:Y}.  
We next check that the new~$(Y^0,Z^0)$ is of type~$(t,C\eta)$ with respect to the new~$(Y_t^0,Z_t^0)$
for some constant~$C \ge 1$. 
Considering how $\tg^0(\alpha,\beta)-\gf^0(t,\alpha,\beta)$ changes when we modify the distributions, 
and using Claim~\ref{cl:analytic} to compare the relevant point probabilities, 
this is easily seen to follow from the fact that the original $(Y^0,Z^0)$ is of type~$(t,\eta)$ with respect
to the original $(Y_t^0,Z_t^0)$. 
In terms of the conclusion of Theorem~\ref{asyfull}, 
since $Y^0=0$ implies that the process stops immediately, and so~$|\bp|=Z^0$, 
we have only affected the value of $\Pr(|\bp|=k)$ for $k \in [K]$,
which does not alter the conclusion of Theorem~\ref{asyfull}.
To sum up, since $C \eta \le c|t-\tc|$ is equivalent to $\eta \le (c/C) \cdot |t-\tc|$, 
the conclusion of Theorem~\ref{asyfull} (with constant~$c$) 
for the modified distributions with $K =0$  
implies the conclusion of Theorem~\ref{asyfull} (with $c$ replaced by the constant~$c/C$) for the original distributions with $K>0$, as~claimed.

After this first change, for Theorem~\ref{asyfull} it remains to deduce the case $\per > 1$, $K = 0$ from the case~$\per=1$, $K=0$.
If the distributions $(Y_t^0,Z_t^0)$ and $(Y^0,Z^0)$ as well as $(Y_t,Z_t)$ and $(Y,Z)$ are all supported 
on~$(\per\NN)^2$, then, in the branching process, individuals are born in groups of size~$\per$ (both in
the first generation and later on). Thus we may describe the same random tree
differently as a branching process, by treating each such group as an individual.
The new branching process $\bp'$ deterministically satisfies 
\begin{equation}
\label{eq:bpmod}
|\bp|=\per |\bp'| .
\end{equation}
To check that it satisfies the conditions in Definitions~\ref{def:bpprops} and~\ref{def:dtype},
we now relate the initial generation and later offspring distributions of~$\bp$ 
to those of~$\bp'$.

For the initial generation, we simply divide $Y_t^0$, $Z_t^0$, $Y^0$ and~$Z^0$ by~$\per$, 
which preserves all relevant conditions in Definitions~\ref{def:bpprops} and~\ref{def:dtype}.
Indeed, the only condition that requires some argument is the analytic extension condition for~$\gf^0$:
the key point is that the original~$\gf^0$ has an analytic extension to the polydisk~$\fD_{\delta,R}$. 
By standard results for complex analytic functions, this extension is given by a single power series around~$(\tc,0,0)$ which 
converges in the entire polydisk (see, e.g., Theorem~1.18 in~\cite{Range}).
Since, for every~$t$, $Y_t^0$ and~$Z_t^0$ are both supported on $\per\NN$,
all powers of~$\alpha$ and~$\beta$ appearing in this power series are multiples of~$\per$.
Substituting~$\alpha^{1/\per}$ and~$\beta^{1/\per}$ thus gives a corresponding power series
for the new distributions, converging in $\fD_{\delta,R^{\per}}$. 

For the later offspring distributions $(Y_t,Z_t)$ and $(Y,Z)$, the operation is to take the sum of~$\per$ independent copies of the distribution divided by~$\per$. It is not hard to check that this preserves
the assumptions, after increasing~$\eta$ by a constant factor~$C \ge 1$. 
Firstly, the mean of~$Y_t$ is unaffected and the mean of $Y_t^0$ is simply divided by $\per > 1$, so the criticality condition~\eqref{means:Y} still holds. 
Secondly, the new `idealized' probability generating functions~$\hf$ and~$\hf^0$
satisfy  
\begin{equation}
\label{eq:bpmod:hf}
\hf(t,\alpha,\beta)=\gf(t,\alpha^{1/\per},\beta^{1/\per})^{\per} 
\qquad\text{and}\qquad
 \hf^0(t,\alpha,\beta) = \gf^0(t,\alpha^{1/\per},\beta^{1/\per}) ,
\end{equation}
so, arguing as above, they extend analytically to $\fD_{\delta,R^{\per}}$.
Thirdly, the new `perturbed' probability generating functions~$\htf$, $\htf^0$ satisfy
\begin{equation}
\label{eq:bpmod:htf}
\htf(\alpha,\beta) = \tg(\alpha^{1/\per},\beta^{1/\per})^{\per}
\qquad\text{and}\qquad
 \htf^0(t,\alpha,\beta) = \tg^0(t,\alpha^{1/\per},\beta^{1/\per}) ,
\end{equation}
so they are defined and (complex) analytic in ${\widehat\fN} := \{(\alpha,\beta) \in \CC^2: \: |\alpha|,|\beta| < R^{\per}\}$.
Now, since the~`$\eta$-close' condition~\eqref{dtype1} holds for the original distributions, using the form of~\eqref{eq:bpmod:hf}--\eqref{eq:bpmod:htf} it is easy to see that, after replacing~$\eta$ with $C \eta \ge \eta$, \eqref{dtype1} again holds for the new distributions. 
Furthermore, if the original distribution satisfies the non-degeneracy condition~\eqref{cnondeg} with $\per> 1$ and $k_0 \in \NN$, then it is not difficult to check that the new distribution satisfies~\eqref{cnondeg} with $\per=1$ and the same constant $k_0 \in \NN$ (when we sum the~$\per > 1$ independent copies of the modified distribution, we just take the value~$(k_0,k_0)/\per$ for all the first~$\per-1$ copies, and then consider the values $(k_0,k_0)/\per$, $(k_0+\per,k_0)/\per$, and $(k_0,k_0+\per)/\per$ for the last copy). 
To sum up, the new distributions associated to~$\bp'$ satisfy Definitions~\ref{def:bpprops} and~\ref{def:dtype}, and are of type~$(t,C\eta)$ for some~$C \ge 1$. 
Recalling~\eqref{eq:bpmod}, the conclusion of Theorem~\ref{asyfull} with constant~$c$ for the modified distributions with $\per=1$ and $K=0$ easily implies the conclusion of Theorem~\ref{asyfull} (with~$c$ replaced by the constant $c/C$) 
for the original distributions with~$\per>1$ and~$K = 0$, as~claimed.

Finally, the same arguments allow us to deduce Theorem~\ref{thsurv} from the special case~$\per=1$, $K=0$.
Indeed, we modify the branching process in two steps, as above, which preserves the assumptions
of the theorem (as we have just shown). Since we have only altered outcomes with finite size, 
conclusions about the survival probability thus carry over from the modified branching processes to the original~ones.

\newoddpage 

\section{Glossary of notation}\label{sec:notation}
\vspace{0.25em}
\begin{enumerate}[topsep=0.6em,itemsep=0.25em,leftmargin=3.0cm,labelsep=0.75em]

\item[$\NN$, $\NNP$] natural numbers with and without~$0$ 

\item[$L_j(G)$] size of the $j$th largest component in the graph~$G$ 

\item[$N_k(G)$, $N_{\ge k}(G)$] number of vertices in components with exactly/at least $k$-vertices in the graph~$G$ 

\item[$S_r(G)$] $r$th order susceptibility of the graph~$G$; see~\eqref{eq:def:Sr}

\item[$S_{r,n}(G)$] modified $r$th order susceptibility of the graph~$G$; see~\eqref{eq:def:SrG}

\item[$C_v(G)$] (vertex set of) the component of a graph $G$ containing a vertex~$v$

\item[$C_W(G)$] the union of $C_v(G)$ over $v\in W$

\end{enumerate}

\vspace{0.6em}
\noindent\textbf{For bounded-size rules $\cR$:}
\begin{enumerate}[topsep=0.6em,itemsep=0.25em,leftmargin=3.0cm,labelsep=0.75em]

\item[$\cR$] decision rule; see Section~\ref{sec:intro:def}

\item[$\cR(\vc) = \{j_1,j_2\}$] indices of the vertices joined by $\cR$ when presented with vertices $v_1,\ldots,v_\ell$ in components
of size $c_1,\ldots,c_\ell$; see Sections~\ref{sec:intro:def} and~\ref{sec:DEM} 

\item[$K$] cut-off in the bounded-size rule~$\cR$

\item[$\cC=\cC_K$] set $\{1,2,\ldots,K,\omega\}$ of `observable' component sizes, where $\omega$ means size~$>K$

\item[$G_i = G^{\cR}_{n,i}$] random graph after $i$ steps of the process (often with $i=tn$)

\item[$t$] time parameter (often corresponding to $i/n$)

\item[$\tc$] critical time; see~\eqref{eq:def:tc}

\item[$\cSR$] the set of possible component sizes; see Section~\ref{sec:period}

\item[$\per$] period of the rule; see Section~\ref{sec:period}

\end{enumerate}

\vspace{0.6em}
\noindent\textbf{For the graph $G_i = G^{\cR}_{n,i}$ after $i=tn$ steps:}
\begin{enumerate}[topsep=0.6em,itemsep=0.25em,leftmargin=3.0cm,labelsep=0.75em]

\item[$\cF_i$] $\sigma$-algebra corresponding to information revealed by step $i$; see Section~\ref{sec:DEM}

\item[$L_j(i)$] size $L_j(i)=L_j(G_i)$ of the $j$th largest component after $i$ steps of the process

\item[$N_k(i)$, $N_{\ge k}(i)$] number $N_k(i)=N_k(G_i)$ and $N_{\ge k}(i) = N_{\ge k}(G_i)$
 of vertices in components with exactly/at least $k$-vertices after~$i$ steps of the process

\item[$S_r(i)$] $r$th order susceptibility $S_r(i)=S_r(G_i)$ after~$i$ steps of the process; see~\eqref{eq:def:Sr}

\item[$\rho(t)$] scaling limit of $L_1$, i.e., limit of $L_1(tn)/n$; see~\eqref{eq:L1:pto}

\item[$\rho_k(t)$] scaling limit of $N_k$, i.e, limit of $N_k(tn)/n$; see~\eqref{eq:Nk:pto}

\item[$s_r(t)$] scaling limit of $S_r$, i.e., limit of $S_r(tn)/n$; see~\eqref{eq:Sj:pto}

\item[$\psi(t)$] rate function in the decay of the component size distribution at time $t$ (step $tn$); see Theorems~\ref{thm:rhok} and~\ref{asyfull-simple-pm}

\end{enumerate}

\vspace{0.6em}
\noindent\textbf{For the two-round exposure near around~$\tc$:}
\begin{enumerate}[topsep=0.6em,itemsep=0.25em,leftmargin=3.0cm,labelsep=0.75em]

\item[$t_0$, $t_1$] times with $t_0 < \tc < t_1$: the main focus of this paper are times $t \in [t_0,t_1]$; see~\eqref{def:sigma}--\eqref{def:t0t1t2} 

\item[$i_0$, $i_1$] steps $i_0=t_0n$ and $i_1=t_1n$: we reveal information about the steps $i_0<i \le i_1$ via a two-round exposure argument; see Section~\ref{sec:exp} and~\eqref{def:i0i1i2} 

\item[$V_S$, $V_L$] sets of vertices in {\bf S}mall and {\bf L}arge (size $>K$) components \emph{at step $i_0$}; see Section~\ref{sec:setup}

\item[$H_i$] the `marked graph' after $i$ steps; see Section~\ref{sec:exp}

\item[$Q_{k,r}(i)$] number of $(k,r)$-components after~$i$ steps of the process; see Section~\ref{sec:exp} (note that the definition for $k \ge 1$ and $k=0$ differs slightly, see also Sections~\ref{sec:DEM:VL}--\ref{sec:VS:finite})

\item[$q_{k,r}(t)$] scaling limit of $Q_{k,r}$, i.e., limit of $Q_{k,r}(tn)/n$; see Sections~\ref{sec:VS:finite} and~\ref{sec:VS}

\item[$\fS_i$] random `parameter list' generated by the random graph $G^{\cR}_{n,i}$ after $i$ steps; see \eqref{def:fS}

\item[$J_i=J(\fS_i)$] random graph constructed using the random `parameter list'~$\fS_i$; see Section~\ref{sec:exp}

\end{enumerate}

\vspace{0.6em}
\noindent\textbf{For the branching process comparison arguments:}
\begin{enumerate}[topsep=0.6em,itemsep=0.25em,leftmargin=3.0cm,labelsep=0.75em]

\item[$\eps$] generally used for $|t-\tc|$ 

\item[$\Psi=(\log n)^2$] a convenient cut-off size 

\item[$\fS$] an arbitrary `parameter list'; see Definition~\ref{def:paramlist}

\item[$J=J(\fS)$] random graph constructed using~$\fS$; see Definition~\ref{def:J}

\item[$\JP=\JP(\fS)$] Poissonized random graph constructed using $\fS$; see Definition~\ref{def:F}

\item[$t$-nice] condition on a parameter list $\fS$ that essentially says it behaves similarly to the list $\fS_i$ corresponding to $G_{n,i}^{\cR}$, where $i=tn$; see Definition~\ref{def:nice}
  
\item[$\cQ_{k,r}$] the set of $(k,r)$-components/hyperedges in $\JP$; see Section~\ref{sec:nep:fg}

\item[$\bp_t$] idealized branching process that approximates the neighbourhoods of the random graphs $G^{\cR}_{n,tn}$ and $\JP(\fS_{tn})$; see Section~\ref{sec:BPI:distr}

\item[$\fS^{\pm}_t$] perturbed variants of a given $t$-nice parameter list $\fS$ (see Definition~\ref{def:nice} for the definition of $t$-nice) which typically satisfy $J(\fS^{-}_t) \subseteq J(\fS) \subseteq J(\fS^{+}_t)$; see Definition~\ref{def:cpl} and Lemma~\ref{lem:cpl}

\item[$\bp_t^{\pm}=\bp_t^{\pm}(\fS)$] dominating branching processes that approximate (from above and below) the neighbourhoods of the random graphs $\JP(\fS^{\pm}_t)$; see Section~\ref{sec:dom:dom}

\item[$\Lambda$] cut-off size beyond which a component is `large', used for ending exploration arguments (often $n^{2/3}$, but not always); distinct from the cut-off $K$ in the bounded-size rule~$\cR$
  
\item[$n^{2/3}$] convenient cut-off size at which we abandon certain domination arguments; see Theorems~\ref{thm:cpl:LB}--\ref{thm:ENgekD} and Lemma~\ref{lem:Nkvar:sup}

\end{enumerate}

\end{appendix}

\end{document}